\documentclass[a4paper,english]{smfbook}

\usepackage{amsmath,amssymb,bbm,epsfig,mathrsfs}
\usepackage[all]{xy}

\newtheorem{theorem}{Theorem}[section] 
\newtheorem{proposition}[theorem]{Proposition}
\newtheorem{lemma}[theorem]{Lemma}
\newtheorem{corollary}[theorem]{Corollary}
\newtheorem{definition}[theorem]{Definition}
\newtheorem{remark}[theorem]{Remark}
\newtheorem{example}[theorem]{Example}

\author{Thierry L\'evy}
\address{CNRS, \'Ecole Normale Sup\'erieure -- 45, rue d'Ulm, F-75005 Paris}
\email{levy@dma.ens.fr}
\urladdr{http://www.dma.ens.fr/~levy/}
\title[Markovian holonomy fields]{Two-dimensional Markovian holonomy fields}
\alttitle{Champs d'holonomie markoviens bidimensionnels}

\def\A{{\mathcal A}}
\def\Ad{{\rm Ad}}
\def\Aut{{\rm Aut}}
\def\Ax#1{{\rm A}$_{#1}$}
\def\AxD#1{{\rm D}$_{#1}$}
\def\BRB{{\mathcal{R}}}
\def\BRBm{{\sf BB}}
\def\BS{{\mathscr B}}
\def\build#1_#2^#3{\mathrel{\mathop{\kern 0pt#1}\limits_{#2}^{#3}}}
\def\C{{\mathcal C}}
\def\CEdge{{\sf CE}}
\def\CK{{\mathbb C}}
\def\CLoop{{\sf CL}}
\def\Conj{{\rm Conj}}
\def\Const{{\rm Const}}
\def\CS{{\mathscr C}}
\def\Curve{{\sf C}}
\def\D{{\rm D}}
\def\DF{{\sf DF}}
\def\E{{\mathbb E}}
\def\e{{\sf e}}
\def\Edge{{\sf E}}
\def\epsilon{{\varepsilon}}
\def\F{{\mathbb F}}
\def\f{{\sf f}}
\def\fr{{\rm fr}}
\def\G{{\mathbb G}}
\def\Gen{{\mathscr G}}
\def\GG{{\sf Gr}}
\def\gg{{\mathfrak g}}
\def\GPath{{\sf A}}
\def\H{{\mathcal{H}}}
\def\HF{{\sf HF}}
\def\hol{{\rm mon}}
\def\Hol{{\rm Mon}}
\def\HS{{\mathscr H}}
\def\I{{\mathcal I}}
\def\irrep{{\rm Irr}}
\def\id{{\rm id}}
\def\Lab{{\rm Lab}}
\def\LE{{\rm LE}}
\def\Leb{{\rm Leb}}
\def\Loop{{\sf L}}
\def\lra{\longrightarrow}
\def\m{{\sf m}}
\def\M{{\mathcal M}}
\def\MF{{\sf MF}}
\def\NK{{\mathbb N}}

\def\NMF{{\sf NMF}}
\def\NRBm{{\sf NRB}}
\def\NBRBm{{\sf NBB}}
\def\O{{\mathcal O}}
\def\p{{\sf p}}
\def\P{{\mathbb P}}
\def\Path{{\sf P}}

\def\PPath{{\sf PP}}
\def\PPP{{\Xi}}
\def\PCurve{{\sf PC}}
\def\pt{\slash\!\slash}
\def\R{{\sf{R}}}
\def\Ram{{\rm Ram}}
\def\RB{{\mathcal{R}}}
\def\RBm{{\sf RB}}
\def\rg{{\sf g}}
\def\RK{{\mathbb R}}
\def\RL{{\sf RL}}

\def\Sk{{\rm Sk}}
\def\Spl{{\sf Spl}}
\def\Sy{{\mathfrak S}}
\def\Sym{{\rm Sym}}
\def\U{{\sf U}}
\def\V{{\mathbb V}}
\def\v{{\sf v}}
\def\vol{{\rm vol}}
\def\X{{\mathcal X}}
\def\Y{{\sf F}}
\def\Z{{\mathbb Z}}
\def\zz{{\mathfrak z}}

\makeindex

\begin{document}

\frontmatter

\begin{abstract}
We define a notion of Markov process indexed by curves drawn on a compact surface and taking its values in a compact Lie group. We call such a process a two-dimensional Markovian holonomy field. The prototype of this class of processes, and the only one to have been constructed before the present work, is the canonical process under the Yang-Mills measure, first defined by Ambar Sengupta \cite{SenguptaAMS} and later by the author \cite{LevyAMS}. The Yang-Mills measure sits in the class of Markovian holonomy fields very much like the Brownian motion in the class of L\'evy processes. We prove that every regular Markovian holonomy field determines a L\'evy process of a certain class on the Lie group in which it takes its values, and construct, for each L\'evy process in this class, a Markovian holonomy field to which it is associated. When the Lie group is in fact a finite group, we give an alternative construction of this Markovian holonomy field as the monodromy of a random ramified principal bundle. This is in heuristic agreement with the physical origin of the Yang-Mills measure as the holonomy of a random connection on a principal bundle.
\end{abstract}

\begin{altabstract}
Ce travail est consacr\'e \`a la d\'efinition et \`a l'\'etude d'une classe de processus stochastiques index\'es par des chemins trac\'es sur une surface, qui prennent leurs valeurs dans un groupe de Lie compact et qui satisfont une propri\'et\'e d'ind\'ependance conditionnelle analogue \`a la propri\'et\'e de Markov. Nous appelons ces processus des champs d'holonomie markoviens bidimensionnels. L'exemple fondamental de cette sorte de processus est le processus canonique sous la mesure de Yang-Mills, qui a \'et\'e construite d'abord par Ambar Sengupta \cite{SenguptaAMS} puis plus tard par l'auteur \cite{LevyAMS}. C'est aussi le seul champ d'holonomie markovien qui ait \'et\'e construit avant ce travail. Le processus canonique sous la mesure de Yang-Mills est assez exactement aux champs d'holonomie markoviens ce que le mouvement brownien est aux processus de L\'evy. Deux de nos principaux r\'esultats affirment qu'\`a tout champ d'holonomie markovien suffisamment r\'egulier est associ\'e un processus de L\'evy d'une certaine classe sur le groupe de Lie dans lequel il prend ses valeurs et r\'eciproquement que pour tout processus de L\'evy dans cette classe il existe un champ d'holonomie markovien auquel il est associ\'e. Dans le cas particulier o\`u le groupe de Lie consid\'er\'e est un groupe fini, nous parvenons \`a r\'ealiser ce champ d'holonomie markovien comme la monodromie d'un fibr\'e principal ramifi\'e al\'eatoire. Ceci nous rapproche de l'interpr\'etation originelle de la mesure de Yang-Mills, issue de la th\'eorie quantique des champs, comme mesure de probabilit\'es sur l'espace des connexions sur un fibr\'e principal.
\end{altabstract}

\altkeywords{Champ al\'eatoire, processus de Markov, processus de L\'evy, groupe de Lie compact, graphe sur une surface, chemin rectifiable, rev\^etement ramifi\'e, Yang-Mills, th\'eorie de jauge sur r\'eseau, limite continue}

\maketitle
\setcounter{tocdepth}{3}
\tableofcontents

\mainmatter

\chapter*{Introduction}

The elementary theory of Markov processes establishes a correspondence between several types of objects among which transition semigroups and stochastic processes. These stochastic processes can take their values in fairly general spaces, but they are usually indexed by a subset of the real numbers, for the Markov property relies on the distinction between past and future. In the present work, we investigate a correspondence between certain transition semigroups and another kind of stochastic processes, where the notions of past and future are replaced by the notions of inside and outside. The processes that we consider are indexed by curves, or rather loops, drawn on a surface, and they take their values in a compact Lie group. We call them (two-dimensional) Markovian holonomy fields. They are Markovian in the following sense: if some piece of a surface is bounded by a finite collection of loops, then the values of the process on loops located inside this piece and outside this piece are independent, conditionally on the value of the process on the finite collection of loops which bounds this piece.

\section{A $1$-dimensional analogue}

Let us start by discussing the $1$-dimensional analogues of Markovian holonomy fields, which are just Markov processes looked at from a slightly unusual point of view. Let us choose a transition semigroup $P=(P_t)_{t\geq 0}$ on some state space $\X$. For each $t\geq 0$, $P_t(x,dy)$ is a transition kernel on $\X\times \X$. Under suitable assumptions, we can associate to $P$ a homogeneous Markov process with values in $\X$, which we denote by $X$. This Markov process is not really a single stochastic process, it is rather a collection of processes, essentially one for each initial condition at a specific time. In fact, if we consider $X$ restricted to segments, we can say that to each segment $[a,b]\subset \RK$ and each initial condition $x \in \X$ we associate a process $(X_t)_{t\in[a,b]}$ with values in $\X$ such that $X_a=x$ almost surely. Among  the structure implied by the fact that $[a,b]$ is a subset of $\RK$, what we really use is the topological structure of this interval, its orientation and our ability to measure the distance between any two of its points. Of course, in the present $1$-dimensional setting, this structure suffices to characterize the interval up to translation, and the last sentence may seem pointless. Its content should however become clearer in the 2-dimensional setting. Let us push the abstraction a little further and try to define, for all compact $1$-dimensional manifold $M$, a process $(X_t)_{t\in M}$ with values in $\X$. As we have just observed, we need an orientation of $M$ and a way to measure distances. Moreover, if $M$ is not connected, we want the restrictions of our process to the connected components of $M$ to be independent. So, let $M$ be an oriented compact connected Riemannian $1$-dimensional manifold. There are not so many choices: $M$ is either homeomorphic to a segment or to a circle, it has a certain positive total length, and this information characterizes it completely up to orientation-preserving isometry.  If $M$ is a segment of length $L$, it is isometric to $[0,L]$ and there is no difficulty in defining the process $(X_t)_{t\in M}$ given an initial condition. Before turning to the case of the circle, let us interpret the Markov property of $X$ in terms of these $1$-dimensional manifolds. 

Let $M_1$ and $M_2$ be two manifolds as above, isometric to segments. Let $M_1\cdot M_2$ denote the manifold obtained by identifying the final point of $M_1$ with the initial point of $M_2$. It is still homeomorphic to a segment. Choose an initial condition $x\in \X$. We are able to construct two stochastic processes indexed by $M_1\cdot M_2$. Firstly, we can take $M_1\cdot M_2$ as a segment on its own and simply consider the process $(X_t)_{t\in M_1\cdot M_2}$ with initial condition $x$. But we can also proceed as follows. For all segment $M$ and all $x\in\X$, let ${\mathcal L}(x,M)$ denote the distribution of the process $(X_t)_{t\in M}$ with initial condition $x$. Let us also denote by ${\mathcal L}(x,M,dy)$ the disintegration of ${\mathcal L}(x,M)$ with respect to the value of $X$ at the final point of $M$. Then the probability measure $\int_\X {\mathcal L}(x,M_1,dy) {\mathcal L}(y,M_2)$
is the distribution of a process indexed by the disjoint union of $M_1$ and $M_2$ which takes the same value at the final point of $M_1$ and the initial point of $M_2$. It can thus be identified with the distribution of a process indexed by $M_1\cdot M_2$. It is exactly the content of the Markov property of $X$ that the two measures that we have considered are equal: 
\begin{equation}\label{surgery intervals}
\int_\X {\mathcal L}(x,M_1,dy) {\mathcal L}(y,M_2) = {\mathcal L}(x,M_1\cdot M_2).
\end{equation} \index{Kolmogorov-Chapman equation}
This relation is an instance of the Kolmogorov-Chapman equation. This example illustrates in the simplest possible way the fact that the Markov property can be nicely formulated in terms of surgery of manifolds, in this case concatenation of intervals. Manifolds of dimension $1$ undergo another kind of surgery, when the two endpoints of a single interval are glued together (see Figure \ref{surg}). If we try to mimic (\ref{surgery intervals}), we are tempted to define the distribution of a process indexed by the circle $S^1$ of length $L$, seen as the interval $[0,L]$ of which the endpoints have been identified, by 
\begin{equation}
{\mathcal L}(S^1) = \int_\X {\mathcal L}(x,[0,L],dx),
\end{equation}
which unfortunately is meaningless. Still, this formula is consistent with the fact that a circle has no boundary, so that there is no initial condition to specify. What we are pretending to define here is a bridge, a probability measure on closed trajectories in $\X$, from the transition semigroup $(P_t)_{t\geq 0}$. 

\begin{figure}[h!]
\begin{center}
\scalebox{1}{\includegraphics{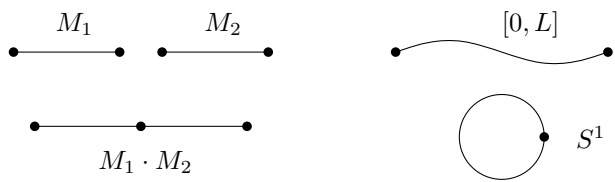}}
\end{center}
\caption{ The surgery of $1$-dimensional manifolds.}\label{surg}
\end{figure}

Without aiming at the greatest possible generality, let us describe a situation in which this is possible. We assume that $\X$ is a smooth finite-dimensional manifold, for example a vector space or a Lie group, which carries a Borel probability measure $\mu$ which is stationary for the semi-group $P$. Finally, we assume that for all $t>0$ and all $x\in \X$, the measure $P_t(x,dy)$ has a continuous density with respect to $\mu$, which we denote by $y\mapsto Q_t(x,y)$. In this situation, it is possible to define bridges of the Markov process $X$ between any two points of $\X$. Hence, for each segment $[0,L]$, it is possible to define the finite measure ${\mathcal L}(x,[0,L],y)$ which is the conditional distribution of $(X_t)_{t\in[0,L]}$ given $X_0=x$ and $X_L=y$, multiplied by the real number $Q_L(x,y)$. With this definition, ${\mathcal L}(x,[0,L],y)$ is not a probability measure in general but the relation ${\mathcal L}(x,[0,L],dy)={\mathcal L}(x,[0,1],y) \mu(dy)$ holds. We can then define the process indexed by a circle of length $L$ by setting
\begin{equation}
{\mathcal L}(S^1) = \int_\X {\mathcal L}(x,[0,L],x)\mu(dx).
\end{equation}
The identification of a process indexed by circle $S^1$ with a process indexed by $[0,L]$ requires the choice of a base point on $S^1$ but the stationarity of $\mu$ implies that the resulting definition of ${\mathcal L}(S^1)$ is independent of this choice.

Let us summarize this discussion of the $1$-dimensional case. Starting with a Markovian transition semigroup on $\X$ with good properties, we have been able to associate to each compact oriented $1$-dimensional Riemannian manifold, endowed with boundary conditions if it is homeomorphic to a segment, a stochastic process with values in $\X$ indexed by the points of this manifold. This collection of processes exhibits a Markovian behaviour with respect to the operations of cutting or concatenation of manifolds.

\section{The transition kernels of Markovian holonomy fields}

A $2$-dimensional Markovian holonomy field is a $2$-dimensional analogue of the object just described. It is a collection of stochastic processes, one for each compact surface endowed with boundary conditions and some way of measuring areas. For each such surface, the process is indexed by a set of loops drawn on this surface. Moreover, these processes satisfy Markovian properties with respect to the operations of cutting surfaces along curves or gluing them along boundary components. 

Before explaining this more carefully, let us discuss the important fact that loops can be concatenated when they have the same origin. We are interested in stochastic processes which satisfy a property of additivity with respect to the concatenation of loops. This requires that they take their values in a group and since this group will usually not be assumed to be Abelian, and denoted multiplicatively, we will rather call this a property of multiplicativity. Let us give a precise definition. If $M$ is a $2$-dimensional manifold and $m$ is a point of $M$, let $\Loop_m(M)$ be a set of loops on $M$ based at $m$. We will discuss later which loops exactly we wish to consider. To each loop $l\in \Loop_m(M)$ we can associate the inverse loop $l^{-1}$, which is simply $l$ traced backwards. Also, to each pair of loops $l_1,l_2 \in \Loop_m(M)$ we can associate their concatenation which we denote by $l_1l_2$. Let $G$ be a group, which plays the role of the space $\X$ above. A stochastic process $(H_l)_{l\in \Loop_m(M)}$ with values in $G$ is said to be multiplicative if
\begin{align} \index{concatenation}
\forall l\in \Loop_m(M)&, \; H_{l^{-1}} = H_l^{-1}  \mbox{ a.s.}, \label{i: inv} \\
\forall l_1,l_2 \in \Loop_m(M) &, \; H_{l_1l_2}=H_{l_2} H_{l_1} \mbox{ a.s. } \label{i: multi}
\end{align}
We will also explain later why we chose to reverse the order on the right-hand side of this equality. For the moment, let us describe the transitions of a Markovian holonomy field. 

Just as an interval has two extremities, a surface has a boundary which is topologically a disjoint union of circles. Let us consider a surface, with a certain number of boundary components. In contrast with the $1$-dimensional case, even if the surface is oriented, its boundary components are indistinguishable from a topological point of view: any permutation of the boundary components can be realized by an orientation-preserving homeomorphism. So, let us arbitrarily declare that some of these components are incoming and the other are outgoing. Then we have a picture of our surface as realizing a cobordism, that is, a topological transition, between two sets of circles (see Figure \ref{transtop}). If the surface is oriented, then we orient incoming boundary components negatively and outcoming boundary components positively. This matters because, according to (\ref{i: inv}), the boundary conditions on the incoming components are associated with oriented loops: to each oriented incoming boundary component we associate an element of $G$. 

\begin{figure}[h!]
\begin{center}
\scalebox{0.9}{\includegraphics{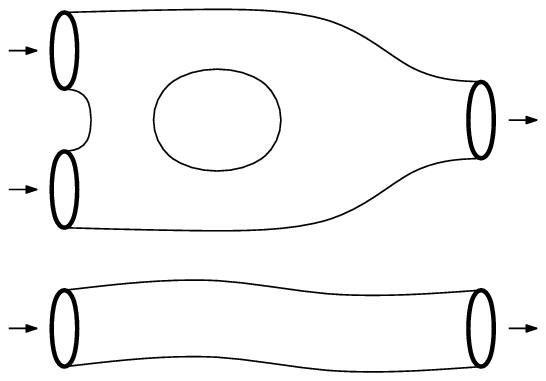}}
\caption{A surface seen as a topological transition between two collections of circles}\label{transtop}
\end{center}
\end{figure}

The comparison with the $1$-dimensional case indicates that we are lacking a measure of time on our surface. The analogue of the ability to measure distances is the ability to measure areas. On a surface, this requires much less than a Riemannian metric, only a Borel measure with suitable regularity properties. We have now isolated the necessary structure: a Markovian holonomy field takes its values in a group which we denote by $G$, and its transitions occur along surfaces with distinguished incoming and outgoing boundaries endowed with a measure of area. The associated transition kernels describe the distribution of the process on the outgoing circles conditionally on its values on the incoming ones. 
\index{cobordism}

There is a lot more variety of situations than in the $1$-dimensional case, but fortunately for us, this variety has been very well understood for about a century and is easy to describe. Up to homeomorphism, a connected compact surface is characterized by the fact that it is orientable or not, by the number of connected components of its boundary, and by a single other topological invariant called its genus which can be any non-negative integer if the surface is orientable and any positive integer if it is not. Moreover, the only invariant of a smooth measure of area under diffeomorphisms is its total area. 
It turns out that, when one deals with orientable and non-orientable surfaces at the same time, the genus is not the most convenient way to label the possible topological types of surfaces. We prefer to work with what we call the reduced genus, which is simply the genus if the surface is not orientable, and twice the classical genus if it is orientable. The main advantage of the reduced genus is that it is additive with respect to the operation of connected sum.

Let us denote by $\M(G^{p-q})$ the space of probability measures on $G^{p-q}$. Then the transition kernel of a Markovian holonomy field can be listed as follows:
\begin{equation}
P^\pm_{p,g,t}(x_1,\ldots,x_q, dx_{q+1},\ldots,dx_p) : G^q\to \M(G^{p-q}),
\end{equation}
where the sign indicates the orientability of the surface, $g$ is its reduced genus, $t$ its total area and $p$ the number of its boundary components, of which $q$ are incoming and $p-q$ outgoing. 

If $q=0$, then the transition kernel is a measure on $G^p$, if $q=p$ it is a function on $G^p$ and if $p=0$, it is a real number. In order to avoid the problems that we have encountered with the circle in the $1$-dimensional case, we will make fairly strong assumptions. Firstly, we will assume that $G$ is a compact Lie group. Such a group carries a unique probability measure invariant by translations, which we denote by $dx$. We will also assume that the transition kernels of the Markovian holonomy field that we consider can be put under the form
\begin{equation}
P^\pm_{p,g,t}(x_1,\ldots,x_q, dx_{q+1},\ldots,dx_p)=Z^\pm_{p,g,t}(x_1^{-1},\ldots,x_q^{-1},x_{q+1},\ldots,x_p)  dx_{q+1}\ldots dx_p 
\end{equation}
for some functions $Z^\pm_{p,g,t}$ which are called the partition functions of the holonomy field. The exponents that we have introduced take care of the issue of orientation. They restore the symmetry between the boundary components of a surface, so that the partition functions are invariant by permutation of their arguments.

The possibility of gluing together boundary components of one or two surfaces leads to infinitely many variants of the Kolmogorov-Chapman equation which give as many relations between the partition functions. For instance, consider a cylinder of area $s$ with one incoming circle and one outgoing circle. A cylinder has genus (both classical and reduced) $0$ and the transition kernel associated to this surface is $P^+_{2,0,s}(x,dy)$. By gluing the incoming circle of this cylinder along an outgoing circle of an arbitrary surface, we do not change this surface up to homeomorphism, we only increase its area by $s$. Hence, we have the following relation between transition kernels:
\begin{eqnarray}
\int_G P^\pm_{p,g,t}(x_1,\ldots,x_q, dx_{q+1},\ldots,dx_{p-1},dx) P^+_{2,0,s}(x,dx_p) =&& \nonumber \\ &&\hskip -4cm P^{\pm}_{p,g,t+s}(x_1,\ldots,x_q, dx_{q+1},\ldots,dx_p).\label{i: glue1}
\end{eqnarray}
As another example, let us consider an orientable surface with at least one incoming and one outgoing circle. If we glue these circles one along the other in such a way that the result is still orientable, then we obtain a surface with two less boundary components, a classical genus increased by $1$, hence a reduced genus increased by $2$, and the same area. This example is reminiscent of the situation where we obtained a circle by identifying the two endpoints of a segment. It is thus not surprising that in this case, the Markov property is best expressed in terms of the partition functions, rather than the transition kernels.
It reads
\begin{equation}\label{i: glue2}
\int_G Z^+_{p,g,t}(x_1,\ldots,x_{p-2},x,x^{-1}) \; dx= Z^+_{p-2,g+2,t}(x_1,\ldots,x_{p-2}).
\end{equation}
If we identify an outgoing circle with an incoming one in a non-orientable surface, then the reduced genus of the surface is also increased by $2$. In this case, the Kolmogorov-Chapman equation is
\begin{equation}\label{i: glue3}
\int_G Z^-_{p,g,t}(x_1,\ldots,x_{p-2},x,x^{-1}) \; dx= Z^-_{p-2,g+2,t}(x_1,\ldots,x_{p-2}).
\end{equation}
A more unusual topological operation consists in identifying a boundary component of a surface with itself by an orientation-preserving involution (see Figure \ref{fig gluings} page \pageref{fig gluings}). Up to homeomorphism, this operation is equivalent to gluing a M\"{o}bius band along this boundary component. The resulting surface is always non-orientable and its reduced genus is increased by one. The corresponding relation for partition functions is
\begin{equation}\label{i: glue4}
\int_G Z^\pm_{p,g,t}(x_1,\ldots,x_{p-1},x^2) \; dx= Z^-_{p-1,g+1,t}(x_1,\ldots,x_{p-1}).
\end{equation}

The equalities (\ref{i: glue1}), (\ref{i: glue2}), (\ref{i: glue3}) and (\ref{i: glue4}) essentially exhaust the types of relations that hold between the partition functions. Using these relations, we will prove that the whole set of partition functions of a Markovian holonomy field is determined by just those associated to a disk, a M\"{o}bius band and a three-holed sphere of arbitrary area. Indeed, any surface can be built from these three elementary bricks by a finite number of gluings. In fact, if the Markovian holonomy field is regular enough, then we will prove that all its transition kernels are  determined by the sole transition kernels associated to disks. These transition kernels $(P^+_{1,0,t}(dx))_{t>0}$  constitute a one-parameter family of probability measures on $G$ which turns out to form a continuous convolution semi-group, hence the collection of $1$-dimensional marginals of a L\'evy process on $G$. \\

We are now able to give an idea of two of our main results. After giving an axiomatic definition of a $2$-dimensional Markovian holonomy field with values in a compact Lie group $G$ (Definition \ref{def MHF}), we will prove, under a suitable regularity assumption, that there is a classical L\'evy process with values in $G$ associated with each Markovian holonomy field (Proposition  \ref{hol levy}). This L\'evy process is characterized by the fact that its $1$-dimensional distributions are the transition kernels of the holonomy field associated to disks. Moreover, all partition functions of the holonomy field can be expressed in terms of these $1$-dimensional distributions (Proposition \ref{explicit Z}). Then, we will prove that for all L\'evy process which satisfies some regularity properties, there exists a Markovian holonomy field to which this L\'evy process is associated (Theorem \ref{levy HF}). We do not settle in this work the question of the unicity of a Markovian holonomy field with a given associated L\'evy process. 

\section{Markovian holonomy fields and the Yang-Mills measure}

The Yang-Mills measure is to Markovian holonomy fields what the Brownian motion is to L\'evy processes. It is indeed a Markovian holonomy field, whose associated L\'evy process is the Brownian motion on $G$. It is also the only Markovian holonomy field to have been constructed before the present work. The Yang-Mills measure has been the object of mathematical work since around 1990. It has been first constructed on an arbitrary compact surface by Ambar Sengupta \cite{SenguptaAMS}. The author has later given a different construction of essentially the same measure in \cite{LevyAMS}. One of the by-products of the present work is to provide another construction of the Yang-Mills measure, which really is very close to that given in \cite{LevyAMS}.
Yet, we would like to emphasize several aspects in which it differs from the previous ones. 
\index{Brownian motion!on $G$}

The first difference is a slight shift of point of view which we have already illustrated. Instead of considering a surface, choosing boundary conditions and constructing a process indexed by some class of loops on this specific surface, we now consider as one single object a whole collection of processes indexed by loops on all possible surfaces. The advantage of this view is that is leads quite naturally to an axiomatic characterization of Markovian holonomy fields. The one that we propose is inspired by the classical definition of a Markov process and by the axiomatic definition of a topological quantum field theory. 

Another difference lies in the class of loops that we consider on a surface. Indeed, on a surface $M$ endowed with boundary conditions, a measure of the area, and a base point $m$, the Yang-Mills measure produces a $G$-valued stochastic process $(H_l)_{l\in \Loop_m(M)}$ for some class of loops $\Loop_m(M)$, of which we had promised a discussion earlier. In the two previous constructions of the Yang-Mills measure mentioned above, the loops to which one was able to attach a random variable were finite concatenations of very special curves, for example segments of submanifolds in the author's construction. On the other hand, it was proved in \cite{LevyAMS} that the mapping $l\mapsto H_l$ is continuous in $L^1$ norm on the class of loops considered there endowed with the topology of uniform convergence and convergence of the length. This suggested that it should be possible to associate a random variable at least to each loop of finite length. This is indeed what we achieve in the present construction, thus giving more coherence between the regularity property of the stochastic process and the set of loops on which it is defined. 
Let us point out that it is not necessary to be able to measure lengths in order to decide if a loop has finite length. Indeed, a diffeomorphism does not alter the fact that a curve has finite or infinite length. Thus, the definition of $\Loop_m(M)$ as the set of loops on $M$ with finite length based at $m$ does not require the choice of a Riemannian metric on $M$, a smooth structure is more than enough. It is natural to wonder if one could define the Yang-Mills measure or any other Markovian holonomy field for a significantly larger class of loops. We believe that this is not possible with only the techniques used in this work, and probably very difficult anyway. We shall discuss this in relation with the Brownian motion indexed by loops defined below.

The third important difference between this work and the existing constructions of the Yang-Mills measure concerns the group in which the stochastic process takes its values, which we have already denoted by $G$. Since the Yang-Mills measure is associated with the Brownian motion on $G$, it was natural to assume that it was a connected group. In the present work, the Brownian motion is replaced by a L\'evy process, thus a process with jumps. This opens the possibility of considering non-connected compact Lie groups and in particular finite groups. In the case of finite groups, the L\'evy processes have a very simple structure and we are able to give a completely geometrical picture of Markovian holonomy fields, at least the ones that we are able to construct. We will discuss this in more detail later after recalling briefly the physical and geometrical motivation for the study of Markovian holonomy fields.

\section{The real Brownian motion indexed by loops}
\index{Brownian motion!indexed by loops}

In this paragraph, we present an example of a stochastic process which does not exactly fit into our definition of Markovian holonomy fields, but is in a sense simpler and should absolutely be kept in mind as a fundamental example. We will also take this example as an opportunity to finish our introductory discussion of the role of loops with finite length.

Consider the plane $\RK^2$ endowed with the Lebesgue measure denoted by $dx$. Let $W:L^2(\RK^2,dx)\to {\mathcal G}$ be a white noise, that is, an isometry into a Hilbert space of centred real Gaussian random variables.
Let $l:[0,1]\to \RK^2$ be a continuous loop. For each point $x \in \RK^2-l([0,1])$, the topological index of $l$ with respect to $x$ is an integer denoted by ${\sf n}_l(x)$ and defined, if we identify $\RK^2$ with $\CK$, by
\begin{equation}\label{nl}
{\sf n}_l(x)=\frac{1}{2i\pi} \oint_{\tilde l} \frac{dz}{z-x},
\end{equation}
where $\tilde l$ is any piecewise smooth loop which is uniformly close enough to $l$, for example closer than the distance of $x$ to the range of $l$. Let us say that the loop $l$ {\em winds mildly} if its range is negligible and ${\sf n}_l\in L^2(\RK^2,dx)$.
In this case, it is legitimate to define 
\begin{equation} \label{def W}
B_l=W({\sf n}_l).
\end{equation}

\begin{definition} The stochastic process $\{B_l : l \mbox{ winds mildly }\}$ is called the Brownian motion indexed by loops on $\RK^2$.
\end{definition}

To each subset $D$ of the plane, we may associate a $\sigma$-field ${\mathcal F}_D$ which we define by
\begin{equation}
{\mathcal F}_D=\sigma\left(B_l : l \mbox{ winds mildly, }l([0,1])\subset D\right).
\end{equation}

If $J$ is a Jordan curve in $\RK^2$, then we denote respectively by ${\rm int}(J)$ and ${\rm ext}(J)$ the bounded and unbounded connected components of $\RK^2\setminus J$, which we call the interior and exterior of $J$. 
A basic property of the white noise $W$ is that it associates independent random variables to functions with disjoint supports. This implies the following property of the Brownian motion indexed by loops. 

\begin{proposition} Let $J_1$ and $J_2$ be two Jordan curves with disjoint interiors. Then the $\sigma$-fields ${\mathcal F}_{{\rm int}(J_1)}$ and ${\mathcal F}_{{\rm int}(J_2)}$ are independent.
\end{proposition}

\begin{proof}If $l$ is a loop whose range is contained in ${\rm int}(J_1)$, then the support of ${\sf n}_{l}$ is also contained in ${\rm int}(J_1)$. The same holds for $J_2$ and the result follows.  \end{proof}

Let us now prove a Markov property. 

\begin{proposition} Let $J$ be a Jordan curve with negligible range. The $\sigma$-fields ${\mathcal F}_{\overline{{\rm int}(J)}}$ and ${\mathcal F}_{\overline{{\rm ext}(J)}}$ are independent conditionally on $B_J$. 
\end{proposition}

\begin{proof}By the Jordan curve theorem, a Jordan curve with negligible range winds mildly. Let $L$ be the line in $L^2(\RK,dx)$ generated by ${\sf n}_J={\bf 1}_{{\rm int}(J)}$. Let $H_{\rm{in}}$ (resp. $H_{\rm{out}}$) be the closed linear subspace of $L^2(\RK,dx)$ generated by ${\sf n}_l$ for $l$ which winds mildly and $l([0,1])\subset \overline{{\rm int}(J)}$ (resp. $l([0,1])\subset \overline{{\rm ext}(J)}$). The inclusions $H_{\rm{in}} \subset \{ f\in L^2(\RK,dx) : {\rm supp}(f)\subset {\rm int}(J) \} \mbox{ and } H_{\rm{out}} \subset \{ f\in L^2(\RK,dx) : f \mbox{ is constant on } {\rm int}(J)\}$
are straightforward. They are actually equalities but we do not need this fact. In particular, $H_{\rm in}\cap H_{\rm out}=L$. Moreover, with the notation $\ominus$ for the orthogonal complement, we have 
$$H_{\rm in}\ominus L \subset \{ f\in L^2(\RK,dx) : {\rm supp}(f)\subset {\rm int}(J) \mbox{ and } \int_J f(x) \; dx =0\},$$
$$H_{\rm out}\ominus L \subset \{ f\in L^2(\RK,dx) : f =0 \mbox{ on } {\rm int}(J)\}.$$
In particular, the orthogonality relation $H_{\rm in}\ominus L \perp H_{\rm out}\ominus L$ holds and the result follows. \end{proof}

It is illuminating to discuss the role of loops of finite length with the example of the Brownian motion indexed by loops in mind. A loop with finite length admits a Lipschitz continuous parametrization. Hence, its range has Hausdorff dimension $1$, unless it is constant. In any case, its range is negligible. The fact that its topological index is square-integrable is not at all obvious. It is granted by a generalization of the isoperimetric inequality discovered by T. Banchoff and W. Pohl. We denote the length of $l$ by $\ell(l)$.

\begin{theorem}[Banchoff-Pohl] Let $l:[0,1]\to \RK^2$ be a Lipschitz continuous loop. Then
$$4\pi \int_{\RK^2} {\sf n}_l(x)^2 \; dx \leq \ell(l)^2.$$
\end{theorem}
\index{Banchoff-Pohl inequality}

The original reference for this theorem is the article of T. Banchoff and W. Pohl \cite{BanchoffPohl}. They prove the inequality for a loop of class $C^2$. An elementary proof of the inequality for rectifiable curves can be found in a paper by A. Vogt \cite{Vogt}.

Of course, there are many loops with infinite length whose topological index is square-integrable, for instance fractal Jordan curves or simply loops of infinite length whose range is contained in a line. It would probably be difficult to characterize the set of loops which wind mildly in a way which differs significantly from its definition. Nevertheless, 
on the scale of roughness given by the $p$-variation, as defined by L.C. Young \cite{Young}, the space of rectifiable loops is the largest which contains only loops which wind mildly. Recall that the $p$-variation of a loop $l:[0,1]\to \RK^2$ is defined as the supremum over all subdivisions $\{t_0\leq \ldots \leq t_r\}$ of $[0,1]$ of the quantity $\left(\sum_i \parallel l(t_{i+1})-l(t_i) \parallel^p\right)^{\frac{1}{p}}$. A loop has finite length if and only if it has finite $1$-variation. A loop with finite $p$-variation for $p<2$ has negligible range.

\begin{proposition} There exists a loop $l:[0,1]\to \RK^2$ such that $l$ has finite $p$-variation for all $p>1$ and $\int_{\RK^2} {\sf n}_l(x)^2 \; dx=+\infty$.
\end{proposition}

\begin{figure}[h!]
\begin{center}
\scalebox{0.5}{\includegraphics{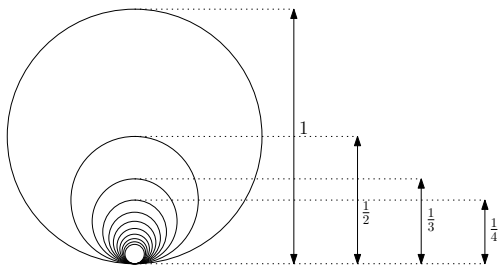}}
\caption{This loop has finite $p$-variation for all $p>1$ but its topological index is not square-integrable.}
\end{center}
\end{figure}

\begin{proof}For each $n\geq 1$, let $l_n$ be the loop based at the origin which goes once along the circle of radius $\frac{1}{n}$ through the origin, tangent to the horizontal axis and contained in the upper half-plane. Assume that $l_n$ is parametrized at constant speed by an interval of length $2^{-n}$. Let $l$ be the uniform limit of the finite concatenations $l_1 \ldots l_n$ as $n$ tends to infinity. This limit exists because the radii of the circles $l_n$ tend to $0$ as $n$ tends to infinity. For all $p>1$, the $p$-variation of $l$ raised to the power $p$ is equal, up to some constant, to $\sum_{n\geq 1} n^{-p}$, hence it is finite. On the other hand, the squared $L^2$ norm of ${\sf n}_l$ is $\pi \sum_{n\geq 1} n^2 ( n^{-2}-(n+1)^{-2})=+\infty$. \end{proof}

\section{Markovian holonomy fields and gauge fields}

The original motivation for the study of processes indexed by loops is issued from theoretical physics, indeed from quantum field theory and more precisely from quantum gauge theories. Let us explain this with the example of electrodynamics. The classical theory of electrodynamics, as established in the second half of the nineteenth century, is summarized by the Lorentz law and Maxwell's equations. Maxwell's equations relate the electric and the magnetic field to the density of electric charge and the density of electric current in space. In order to derive these equations from a principle of least action,  which is usually the first step in the procedure of quantization of a physical theory, it is convenient to express the electric and magnetic fields in terms of a scalar potential and a vector potential. These potentials are not uniquely defined by the fields and this indeterminacy is called the gauge symmetry of the theory. It turns out that the geometric nature of the pair formed by the scalar and vector potentials is that of a connection on a principal bundle with structure group $U(1)$ over the space-time. At the level of rigour of this discussion, we do not make a serious mistake by identifying this object with a differential $1$-form on the space-time. This $1$-form is usually denoted by $A$ and called the gauge field. The exterior differential of the gauge field is  a mixture of the electric and magnetic fields called the electromagnetic field and denoted by $F$. The relation $F=dA$ implies the equality $dF=0$, which is equivalent to the two homogeneous Maxwell equations. The two inhomogeneous equations can be put under the form $*d*F=J$, where $*$ is the Hodge operator on space-time associated with the Minkowski metric, and $J$ is a differential $1$-form built from the densities of charge and current.
\index{connection!on a principal bundle}

Most importantly for us, the gauge field, the object in terms of which the classical electrodynamics is best described, is a differential $1$-form on space-time. The most natural way to evaluate a $1$-form is to integrate it along paths. In the case of electrodynamics, the gauge symmetry of the theory implies that any two gauge fields which differ by a total differential describe the same physics. This indicates that the natural gauge-invariant functionals of the gauge field are in fact its integrals along loops.

If instead of doing quantum mechanics we prefer to do statistical mechanics, then we turn the Minkowski space-time into a Euclidean space-time and put on the space of gauge fields the Gibbs measure corresponding to the action which gives back Maxwell's equations through the least action principle. In an empty space-time, this action is called the Yang-Mills action and it is essentially the squared $L^2$ norm of the electromagnetic field. The natural way to study random gauge fields is to do so through their integrals along loops, and this constitutes indeed a stochastic process indexed by loops. The Markov property of such a process, in this physical context, reflects the following property of locality of the Yang-Mills action: if the space-time is partitioned into several regions, then each region contributes to the action by a quantity which can be computed from the values of the field inside this region only. 

In general, the random object that we are studying is thus an analogue of the electromagnetic field, or rather of the gauge field formed by the scalar potential of the electric field and the vector potential of the magnetic field. Let us give an idea of the physical meaning of the gauge field $A$. This field interacts with particles which carry an analogue of electric charge. In fact, the $1$-form $A$ takes its values in the Lie algebra of a Lie group $G$ and the charge of a particle is mathematically a linear action of $G$ on some vector space in which the wave function of the particle takes its values. For example, in the case of electrodynamics, the group is $U(1)$, the wave functions take their values in $\CK$ and for a particle of charge $ne$, where $-e$ is the charge of the electron, the group $U(1)$ acts on $\CK$ be the representation $e^{i\theta} \cdot z = e^{in\theta}z$. The exponential of the integral of the gauge field along a certain loop, as an element of $U(1)$, describes the modification of the phase of the wave function of a particle which travels along $l$. More generally, let $l:[0,1]\to M$ be a loop. Assume that $G$ is a Lie group and that $A$ is a differential $1$-form on $M$ with values in the Lie algebra of $G$. Then, under fairly weak regularity assumptions, the differential equation \index{charge of a particle}
\begin{equation}\label{i: holonomy}
\left\{\begin{array}{ll} h_0=1 \cr \dot h_t h_t^{-1} = -A(\dot l_t) , \;  t \in[0,1] \end{array}\right.
\end{equation}
has a unique solution $h:[0,1]\to G$. The element $h_1$ determined by (\ref{i: holonomy}) is called the holonomy of $A$ along $l$ and this is why we call our processes holonomy fields. The action of this element $h_1$ of $G$ on the left on the vector space in which the wave function of a particle takes its values determines how the state of a particle is modified when it travels along $l$. The reversed order in the right-hand side of (\ref{i: multi}) is due to the fact that the state of a particle travelling along the concatenation of $l_1$ and $l_2$ is modified first by the transformation due to the displacement along $l_1$ and then by the transformation due to the displacement along $l_2$. 
\index{concatenation} \label{ordre inverse}

\section{Finite groups, gauge fields and ramified coverings}

When the group $G$ is finite, the interpretation that a Markovian holonomy field is a reflection of a probability measure on a space of connections or differential $1$-forms with values in the Lie algebra of $G$ becomes awkward. Indeed, this Lie algebra is the null vector space. Topologically, a principal bundle with finite structure group is a covering and it carries a unique connection, which is flat. If $M$ is simply connected, the holonomy along any loop is the unit element of $G$. On the other hand, there exist Markovian holonomy fields with values in finite groups, which are non-trivial processes indexed by loops, even on the sphere $S^2$.

When $G$ is finite, the correct geometric picture is the following: a Markovian holonomy field with values in $G$ is the monodromy of a random ramified $G$-bundle. By a ramified $G$-bundle, we mean a ramified covering whose regular fibres are endowed with a free transitive action of $G$, or equivalently, a principal $G$-bundle over the complement of a finite set. There is still a unique connection on a ramified principal bundle and this connection is flat at each point which is not a ramification point, but each ramification point acts like a macroscopic amount of curvature concentrated at a single point. For instance, if $n\geq 2$ is an integer, then the mapping $z\mapsto z^n$ from $\CK$ to itself is naturally a ramified $\Z/n\Z$-bundle. The group $\Z/n\Z=\{e^{\frac{2i k \pi}{n}} : k\in\{0,\ldots,n-1\}\}$ acts by multiplication on $\CK^*$, freely and transitively on the fibres of the covering map $z\mapsto z^n$. Any loop which goes once positively around a disk which contains $0$ has monodromy $e^{\frac{2i  \pi}{n}}$, no matter how small this disk. This is consistent with the picture of a concentration of curvature at $0$. 
\index{connection!on a principal bundle}

The idea that the Yang-Mills measure, or a Markovian holonomy field, with values in a connected Lie group is a probability measure on a space of connections looked at through its holonomy is only a guide for the intuition and has no firm rigorous ground. On the contrary, the fact that a large class of Markovian holonomy fields with values in a finite group $G$ can be realized as the monodromy of a random ramified $G$-bundle is a theorem that we will prove. In the intuitive picture of the Yang-Mills measure, the curvature of the random connection is supposed to have the distribution of a white noise. The correct distribution of the ramified $G$-bundles which correspond with a given Markovian holonomy field can be roughly described as follows: first choose the ramification locus by throwing a Poisson point process on the surface with intensity the measure of area, then give a weight to every ramified $G$-bundle with this ramification locus which depends on its monodromy at each ramification point and the L\'evy process associated to the Markovian holonomy field. 

\section{Organization}

The present work consists of five chapters. In Chapters 1 and 2, we develop the tools and prove most of the technical results that we use in our study of Markovian holonomy fields. The first chapter covers the topology of surfaces and their surgery, the topological space of paths on a surface, the fundamental notion of graph on a surface, which we treat both topologically and combinatorially, and finishes with a discussion of Riemannian metrics. The second chapter introduces the space of multiplicative functions of paths with its measurable structures and its uniform measure. The last section is devoted to a study of the free group of loops in a graph in relation with this uniform measure.

In Chapter 3, we define Markovian holonomy fields and their discrete analogues. We prove the first central result of this work (Theorem \ref{main existence}) which encapsulates in an abstract way and extends the procedure which allowed us, in our previous construction of the Yang-Mills measure, to take the continuous limit of a discrete gauge theory. In our present language, we prove that every regular discrete Markovian holonomy field can be extended in a unique way to a regular Markovian holonomy field. 

In Chapter 4, we prove that a regular Markovian holonomy field with values in a compact Lie group $G$ determines a classical L\'evy process in $G$, which in turn determines completely the partition functions of the holonomy field (Propositions \ref{hol levy} and \ref{explicit Z}). We then prove that each L\'evy process of a wide class can be obtained in this way (Theorem \ref{levy HF}). Whether or not two distinct Markovian holonomy fields can have the same associated L\'evy process is a natural question which we do not settle here.

In Chapter 5, we prove that when the group $G$ is finite, the Markovian holonomy field constructed in the previous chapter is the monodromy process of a random ramified covering (Theorem \ref{holo mono}). In fact, most of the chapter is devoted to the construction of this random ramified covering. \\

The choice that we have made of concentrating to the extent possible the technical results in the first two chapters has the obvious drawback that the results exposed there often lack their real motivation, and that a linear reading of these two chapters may not be very rewarding. We hope that this is compensated by the fact that the study of Markovian holonomy fields themselves is much more straightforward than it would be if one had to constantly interrupt the exposition to prove technical results. In order to allow as much as possible the possibility of jumping from a section to another, we have included an index of notation which should be helpful in locating the first occurrence of a notation or a symbol. 

\section{Bibliography}

The original paper of C.N. Yang and R. L. Mills where they introduced non-Abelian gauge theories appeared in 1954 \cite{YangMills}. An explicit description of the Yang-Mills measure on a lattice was given by A. Migdal in 1975 \cite{Migdal}. The importance of the $2$-dimensional Yang-Mills theory in relation with the geometry of the moduli space of flat connections was emphasized by E. Witten in two papers published in the early 1990's \cite{Witten1,Witten2}.  

An introduction to some aspects of the physics of gauge theories accessible to a mathematician can be found in the book by J. Baez and J. Muniain \cite{BaezMuniain}. The book by D. Bleecker \cite{Bleecker} is a mathematical exposition of the theory of connections on a principal bundle with a strong bias towards gauge theories. The mathematical reference on connections is  the book by S. Kobayashi and K. Nomizu \cite{KobayashiNomizu}.


The mathematical study of the Yang-Mills measure was developed around 1985 by S. Albeverio, R. H\o{}egh-Krohn and H. Holden \cite{AlbeverioHKH1,AlbeverioHKH2}, L. Gross \cite{Gross1,Gross2}, C. King and A. Sengupta \cite{GKS}, B. Driver \cite{Driver1,Driver2}. The idea that one could construct a variant of the Yang-Mills measure associated to a L\'evy process which is not necessarily the Brownian motion was already present, twenty years ago, in the work of Albeverio et al. \cite{AlbeverioHKH2}. A. Sengupta gave the first construction of the measure on the sphere $S^2$ in 1992 \cite{SenguptaS2} and on an arbitrary compact surface in 1996 \cite{SenguptaAMS}.  The author gave another construction on an arbitrary surface in 2003 \cite{LevyAMS}, later refined to take non-trivial principal bundles into account \cite{LevyNTB}. In a different direction, around 1990, D. Fine has investigated  the Yang-Mills measure in \cite{Fine1,Fine2} with a focus on the geometrical structure of the space of connections on a principal bundle. 
\index{Brownian motion!on $G$}
\index{connection!on a principal bundle}

The fact that gauge theories with finite structure groups are related to random ramified coverings had already been observed by A. D'Adda and P. Provero \cite{dAddaProvero}. Ramified coverings are also important in relation with the large $N$ limit of the 2-dimensional $U(N)$ Yang-Mills theory, for a reason which is not directly related to the content of the present work. This relation has been revealed by D. Gross, in collaboration with W. Taylor \cite{GrossTaylor} and A. Matytsin \cite{GrossMatytsin}. It has recently been investigated by the author \cite{LevyAIM}.

\chapter{Surfaces and graphs}

In this chapter, we introduce the tools of topology and geometry of surfaces that we use in the rest of this work. We set up the notation, collect the necessary classical results and prove less classical ones. After a short review of compact surfaces, we describe their surgery and study in some detail the paths and graphs drawn on them. In particular, we describe carefully the boundary of a face of a graph. Then we define the group of reduced loops based at a point in a graph and recall why it is free. In the next chapter, we will prove the existence of sets of generators of this group with specific properties. Finally, we discuss Riemannian metrics on surfaces in relation with our problem.

\section{Surfaces}\label{sec: surfaces}
\subsection{Classification of surfaces}
\label{subsec: classif}

Let us start by recalling the definition of a surface.

\begin{definition} A {\em topological compact surface} is a Hausdorff compact topological space in which every point admits a neighbourhood homeomorphic to $\RK^2$ or to $\RK_+ \times \RK$.

A {\em smooth compact surface}, or simply a compact surface, is a topological compact surface equipped with a structure of smooth 2-dimensional manifold with boundary. 
\end{definition}

The distinction between topological and smooth surfaces is not essential, as the following result shows.

\begin{theorem} Any topological compact surface is homeomorphic to a smooth compact surface. Moreover, two smooth compact surfaces are diffeomorphic if and only if they are homeomorphic. 
\end{theorem}

The classification theorem for compact surfaces is thus the same for smooth and topological surfaces, and we describe it now. We warn the reader that we are using a slightly unorthodox convention about the genus of a surface.

Let us describe two infinite series of surfaces. The first series is built from the torus, which is the Cartesian product of two circles. For each even integer $g\geq 0$ and each integer $p\geq 0$, let $\Sigma^+_{p,g}$ be the surface obtained by removing $p$ pairwise disjoint open disks from the connected sum of $\frac{g}{2}$ tori. For $g=0$, the surface $\Sigma^+_{p,0}$ is a sphere with $p$ holes. The second series is built from the projective plane, which is the quotient of the unit sphere of $\RK^3$ by the group of isometries $\{{\rm id},-{\rm id}\}$. For each integer $g\geq 1$ and each $p\geq 0$, let $\Sigma^-_{p,g}$ be the surface obtained by removing $p$ pairwise disjoint open disks from the connected sum of $g$ projective planes.

Recall that a smooth compact surface is {\em orientable} if it carries a non-vanishing differential 2-form. We say that a topological compact surface is orientable if a smooth compact surface to which it is homeomorphic is orientable. 

\begin{theorem}\label{classif} Any connected orientable topological compact surface is homeomorphic to one and exactly one of the surfaces $\{\Sigma^+_{p,g} : p,g\geq 0, g \mbox{ even}\}$. Any connected non-orientable compact surface is homeomorphic to one and exactly one of the surfaces $\{\Sigma^-_{p,g} : p\geq 0, g\geq 1\}$. Any oriented smooth compact surface admits an orientation-reversing diffeomorphism.
\end{theorem}

We call the integer $g$ which appears in this classification the {\em genus} of a surface. For orientable surfaces, it is twice the number which is usually called the genus. The advantage of our convention is illustrated by Proposition \ref{connected sum}. We denote the genus of a surface $M$ by ${\rg}(M)$ and the number of connected components of its boundary by  $\p(M)$
\index{surface! genus}
\index{genus|see{surface}}

With this notation, we have $\rg(\Sigma^\pm_{p,g})=g$ and $\p(\Sigma^\pm_{p,g})=p$. Let us define a binary operation $\wedge$ on $\{+,-\}$ by setting $+\wedge +=+$ and $+\wedge -=-\wedge +=-\wedge -=-$. If $M_1$ and $M_2$ are two compact topological surfaces, we denote by $M_1 \# M_1$ the connected sum of $M_1$ and $M_2$ which is the surface obtained by removing a small disk from $M_1$ and $M_2$ and gluing the two resulting surfaces along the boundaries of these disks. Of course, this surface is defined up to homeomorphism only. 

\begin{proposition}\label{connected sum} Let $\Sigma^\epsilon_{p,g}$ and $\Sigma^{\epsilon'}_{p',g'}$ be two surfaces of the list described above. Then
\begin{equation}\label{somme connexe}
\Sigma^\epsilon_{p,g}\# \Sigma^{\epsilon'}_{p',g'}=\Sigma^{\epsilon\wedge \epsilon'}_{p+p',g+g'}.
\end{equation}
\end{proposition}

\begin{proof}It is clear that $\Sigma^\epsilon_{p,g}\# \Sigma^{\epsilon'}_{p',g'}$ has $p+p'$ boundary components. Moreover, the connected sum of two manifolds of the same dimension is orientable if and only if both manifolds are. The only non-trivial point is that the value of the reduced genus is correct when exactly one of the two surfaces 
$\Sigma^\epsilon_{p,g}$ and $\Sigma^{\epsilon'}_{p',g'}$ is orientable. In this case, since the operation of connected sum is commutative and associative, this boils down to the fact that the connected sum of a projective plane and a torus is homeomorphic to the connected sum of three projective planes. This is a classical result, proved as Lemma 7.1 in \cite{Massey}. 
\end{proof}

It is useful to keep in mind that if $M$ is a non-orientable compact topological surface, then the connected sum of $M$ with the torus $\Sigma^+_{0,2}$ is homeomorphic to the connected sum of $M$ with the Klein bottle $\Sigma^-_{0,2}$.

The fundamental groups of surfaces are most easily described by generators and relations. We denote by $\langle x_1,\ldots,x_n | r_1,\ldots,r_m \rangle$ the group generated by $x_1,\ldots,x_n$ subject to the relations $r_1,\ldots,r_m$.

\begin{theorem}\label{pi 1 surf} The fundamental groups of compact surfaces are, up to isomorphism, the following.\\
1. $\pi_1(\Sigma^+_{0,0})=\{1\}$ and for all $g\geq 1$, 
$$\pi_1(\Sigma^+_{0,2g})=\langle a_1,b_1,\ldots,a_g,b_g | [a_1,b_1] \ldots [a_g,b_g]=1\rangle.$$
2. For all $g\geq 0$ and all $p\geq 1$, $\pi_1(\Sigma^+_{p,2g})$ is free of rank $2g+p-1$.\\
3. For all $g\geq 1$, $\pi_1(\Sigma^-_{0,g})=\langle a_1,\ldots,a_g | a_1^2 \ldots a_g^2=1\rangle$.\\
4. For all $g\geq 0$ and all $p\geq 1$, $\pi_1(\Sigma^-_{p,g})$ is free of rank $g+p-1$.\\
\end{theorem}
\index{surface!classification}

It follows from this theorem that a compact surface is not characterized up to homeomorphism by its fundamental group. For example, a sphere with three holes and a torus with one hole both have a fundamental group which is free of rank 3. However, {\em closed} surfaces, that is, surfaces without boundary, are indeed characterized by their fundamental group.

If $M$ is a closed compact surface, then the reduced genus of $M$ is the minimal number of generators in a presentation of the fundamental group of $M$. On the other hand, if $M$ has a non-empty boundary, then its fundamental group is free of rank $\rg(M)+\p(M)-1$.

The boundary of a compact surface is a finite union of circles. If a surface is oriented, then every connected component of $\partial M$ carries an induced orientation, such that the surface stays on the left of a person walking along the boundary in the positive direction.

\begin{definition} Let $M$ be a compact surface. We denote by $\BS(M)$ the set of connected components of the boundary of $M$, each taken twice, once with each orientation. If $M$ is oriented, we denote by $\BS^+(M)$ the subset of $\BS(M)$ formed by the oriented connected components of $\partial M$ which bound $M$ positively.
\end{definition}
\index{BAB@$\BS(M)$}

Any diffeomorphism of a compact surface induces a diffeomorphism of its boundary. We need to know which diffeomorphisms of the boundary can be obtained in this way. For this, observe that if the boundary of an oriented surface $M$ has $p$ connected components, then, among the $2^p$ distinct orientations of $\partial M$, the $2$ distinct orientations of $M$ determine $2$ preferred orientations. We say that a diffeomorphism of $\partial M$ is orientation-preserving if it preserves these orientations and orientation-reversing if it exchanges them. Of course, if $p\geq 2$, then there exist diffeomorphisms of $\partial M$ which are neither orientation-preserving nor orientation-reversing.

\begin{theorem}\label{diffeo bord} Let $M$ be a smooth compact surface. If $M$ is non-orien\-ta\-ble, then any diffeomorphism of $\partial M$ can be extended to a diffeomorphism of $M$. If $M$ is orientable, then any orientation-preserving (resp. orientation-reversing) diffeomorphism of $\partial M$ can be extended to an orientation-preserving (resp. orientation-reversing) diffeomorphism of $M$. 
\end{theorem}

\subsection{Surgery of surfaces}
\label{s:surgery}

Surfaces undergo natural operations of surgery such as cutting along a curve or gluing one or two boundary components. When one performs gluings and wants to keep track of where they have occurred, one ends up with surfaces which carry marks. On the other hand, when one cuts a surface along one or several curves, a convenient way of keeping track of what has been done is to maintain an involution of the set of boundary components of the current surface.

\begin{definition} A {\em marked surface} is a pair $(M,\CS)$, where $M$ is a smooth
compact surface and ${\mathscr C}$ is a finite collection of pairwise disjoint
oriented smooth $1$-dimensional submanifolds of the interior of $M$, such that an oriented curve
belongs to $\CS$ if and only if the same curve with the opposite orientation belongs to $\CS$. The elements of $\CS$ are called {\em marks}.

The marked surface $(M,{\mathscr C})$ is said to be {\em oriented} if every connected component of $M$
is orientable and oriented.
\end{definition}
\index{MAM@$(M,\CS)$}\index{CCC@$\CS$}
\index{surface!marked}

Let us emphasize that on a marked surface, even an oriented one, the marks do not carry a preferred orientation. The group $\Z/2\Z$ acts on $\CS\cup \BS(M)$ by reversing the orientation, and we denote this action by $b\mapsto b^{-1}$.

\begin{definition} A {\em tubular pattern} is a triple $(M,\CS,\tau)$ where $(M,\CS)$ is a marked surface and $\tau$ is an involution of $\BS(M)$ which commutes to the orientation reversal, that is, which satisfies $\tau(b^{-1})=\tau(b)^{-1}$. If $\CS=\varnothing$, the tubular pattern is said to be {\em split}.
\end{definition}
\index{pattern!tubular}

Let $(M,\CS,\tau)$ be a tubular pattern. Choose $b\in\BS(M)$. If $\tau(b)$ is a component of $\partial M$ distinct from $b$, then $b$ is meant to be identified with $\tau(b)$ by an orientation-preserving diffeomorphism. If $\tau(b)=b$, then $b$ is not meant to be glued or altered in any way. Finally, if $\tau(b)=b^{-1}$, then $b$ is meant to be glued on itself according to an orientation-preserving involution. Of course, this way of encoding the possible gluing operations is purely conventional.

We define now the basic operation of surgery, which is the operation of gluing. 

\begin{definition}\label{def gluing} Let $(M,\CS,\tau)$ and $(M',\CS',\tau')$ be two tubular patterns. A smooth mapping $f:M'\to M$ is called an {\em  elementary gluing} if one of the following sets of conditions is satisfied.

1. The mapping $f$ is the quotient map which identifies $b'$ with $\tau'(b')$ by an orientation-preserving diffeomorphism for some $b'\in \BS(M')$ such that $\{b',{b'}^{-1}\} \neq \{\tau'(b'),\tau'(b')^{-1}\}$. Moreover, $f(\CS' \cup \{b',{b'}^{-1}\})=\CS$ and, on $\BS(M')-\{b',{b'}^{-1},\tau'(b'),\tau'({b'})^{-1}\}$, $\tau \circ f = f \circ \tau'$.  Such a gluing is called {\em binary}, and the pair of curves $\{f(b'),f({b'}^{-1})\}$ is called its {\em joint}.

2. The mapping $f$ is the quotient map which identifies the points of $b'$ by pairs according to an orientation-preserving smooth involution for some $b'\in\BS(M')$ such that $\tau(b')={b'}^{-1}$. Moreover, $f(\CS' \cup \{b',{b'}^{-1}\})=\CS$ and, on $\BS(M')-\{b',{b'}^{-1}\}$, $\tau \circ f = f \circ \tau'$. In this case, the gluing is called {\em unary} and the pair of curves $\{f(b'),f({b'}^{-1})\}$ is called its {\em joint}.

A {\em gluing} is a map which can be written as the composition of several elementary gluings. A gluing is {\em complete} if the involution of the set of boundary components of the target surface is the identity.
\end{definition}
\index{binary|see{gluing}}
\index{unary|see{gluing}}
\index{gluing}
\index{surface!surgery}
\index{splitting}

Up to homeomorphism of the underlying surfaces and disregarding the markings, performing a unary gluing along a boundary component is equivalent to gluing a M\"{o}bius band along this boundary component. The result of this operation is never orientable. 

\begin{figure}[h!]
\begin{center}
\scalebox{0.8}{\includegraphics{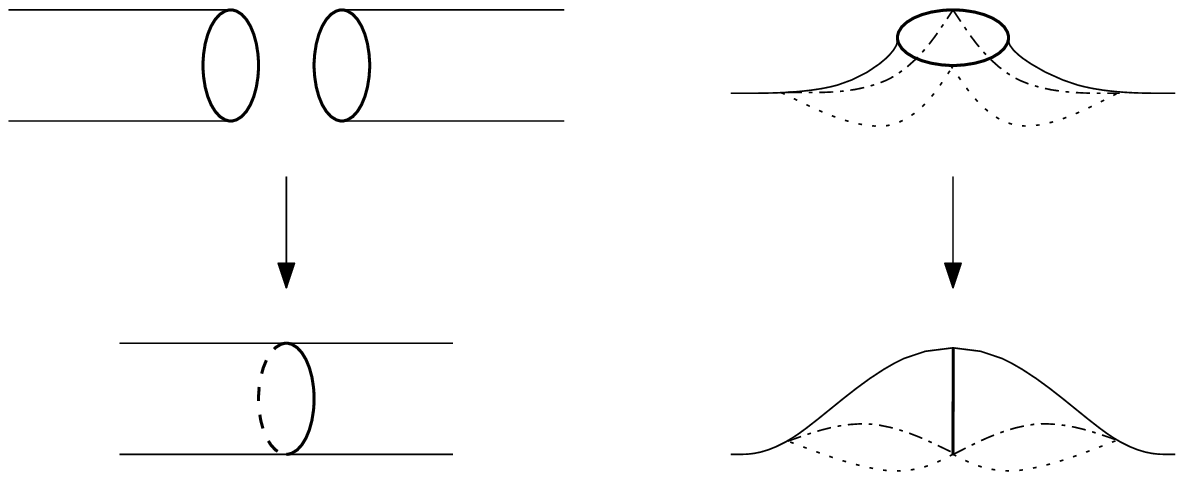}}
\caption{Binary and unary gluings.}\label{fig gluings}
\end{center}
\end{figure}

The other basic surgery operation is that of splitting. It is really the same thing as a gluing, looked at in the other direction.

\begin{proposition}\label{def splitting} Let $(M,\CS,\tau)$ be a tubular pattern with $\CS\neq \varnothing$. Choose $\{l,l^{-1}\} \subset \CS$. Then there exists a tubular pattern $(\Spl_l(M),\Spl_l(\CS),\Spl_l(\tau))$, and an elementary  gluing $f:\Spl_l(M)\to M$ such that the joint of $f$ is $\{l,l^{-1}\}$. Moreover, this gluing is unique up to isomorphism: if $(M',\CS',\tau')$ and $f':M'\to M$ satisfy the same properties, then there exists a diffeomorphism $\psi : \Spl_l(M)\to  M'$ such that $\psi(\Spl_l(\CS)))=\CS'$, $\psi\circ \tau = \tau'\circ \psi$ and $f'\circ \psi = f$.
\end{proposition}
\index{SASpl@$\Spl_l(M)$}

This result is intuitively obvious and yet it lacks concise rigorous proof. Considering that E. Moise, in the reference book \cite{Moise}, defines the splitting operation by a picture (Figure 21.1), we feel excused for not trying what would be a lengthy and uninstructive proof. Instead, let us make some comments on this result. 

The surface $\Spl_l(M)$ can be defined as the topological space underlying the universal completion of the metric space $(M\setminus l,d)$, where $d$ is the restriction to $M\setminus l$ of an arbitrary Riemannian distance on $M$. The universal property of the completion actually provides us with a continuous gluing map.

It is easy to determine from the pair $(M,\{l,l^{-1}\})$ if the gluing is binary or unary. In fact, the gluing is unary if and only if $l$ does not admit an orientable neighbourhood, which is equivalent to the fact that for every neighbourhood $U$ of $l$ there exists  a neighbourhood $V\subset U$ such that $V\setminus l$ is connected. Another equivalent statement is that $l$ admits a compact neighbourhood which is homeomorphic to a M\"{o}bius band of which $l$ is an equator. 

When these equivalent properties do not hold and the gluing is binary, there is an issue of orientation about the way in which the two boundary components of $\Spl_l(M)$ which are glued together are identified. If $M$ is orientable, then $\Spl_l(M)$ is also orientable, because a gluing performed on a non-orientable surfaces always results in another non-orientable surface, and the identification must be orientation-reversing. If $M$ is non-orientable, then two situations arise. Either $\Spl_l(M)$ is orientable, in which case the identification must be orientation-preserving, or $\Spl_l(M)$ is non-orientable. In this last case, any identification of the two boundary components is convenient. This is not in contradiction with the uniqueness part of the statement, since thanks to Theorem \ref{diffeo bord}, any diffeomorphism of the boundary of a non-orientable compact surface can be extended to the whole surface.\\

Let $(M,\CS)$ be a marked surface. By successively applying Proposition \ref{def splitting} to the marks of $M$, one eventually gets a tubular pattern with no marks, from which one can reconstruct $(M,\CS)$. 

\begin{proposition}\label{complete splitting} Let $(M,\CS)$ be a marked surface. Let $(M,\CS,{\rm id})$ be the associated tubular pattern. There exists a tubular pattern $(M',\varnothing,\tau')$ and a smooth mapping $f:M'\to M$ which is a complete gluing in the sense of Definition \ref{def gluing}. This gluing is unique up to isomorphism.
The pattern $(M',\varnothing,\tau')$ is called a {\em split tubular pattern} of $(M,\CS)$. 
\end{proposition}

\begin{definition}\label{split genus} Let $(M,\CS)$ be a marked surface. Let $(M',\varnothing,\tau')$ be a split tubular pattern of $(M,\CS)$. Assume that $M'$ is connected. We define the {\em split reduced genus} of $(M,\CS)$ as the reduced genus of $M'$. We denote it by ${\sf sg}(M,\CS)$.
\end{definition}
\index{surface!split genus}
\index{split genus|see{surface}}

\section{Curves and paths}
\subsection{Definitions} In the theory of Markovian holonomy fields, curves and paths on a surface play the role of points of a time interval for classical Markov processes. We warn the reader that the words {\em curve} and  {\em path} are not interchangeable in this work: a path is a curve with finite length.

\begin{definition}
Let $M$ be a topological compact surface.  A {\em parametrized curve} on $M$ is a continuous
curve $c:[0,1]\to M$ which is either constant on $[0,1]$ or constant on no open sub-interval of $[0,1]$. The set of parametrized curves is denoted by $\PCurve(M)$.

Two parametrized curves on $M$ are said to be {\em equivalent} if they differ by
an increasing homeomorphism of $[0,1]$. An equivalence class is simply called a {\em
curve} and the set of curves on $M$ is denoted by $\Curve(M)$. 

A {\em continuous loop} is a curve whose endpoints coincide.  A
continuous loop is said to be {\em simple} if it is injective on $[0,1)$.
\end{definition}
\index{curve}

If $c$ is a curve, then we denote respectively by $\underline{c}$ and $\overline{c}$ the starting and finishing point of $c$.  We denote its inverse by $c^{-1}$. It is defined as the class of the parametrized curve $t\mapsto c(1-t)$, which does not depend on the particular parametrization of $c$. The concatenation of curves is defined in the usual way. It is only partially defined on $\Curve(M)$ but associative whenever this makes sense. 
\index{concatenation}

The space $\Curve(M)$ is too large for many of our purposes. Let us define another space of curves which we call {\em paths}. Let $M$ be a smooth compact surface endowed with a Riemannian metric. Let $c:[0,1]\to M$ be a Lipschitz
continuous curve. Then the derivative of $c$ is defined almost-everywhere and its norm is
bounded above. We are going to consider curves whose speed is also bounded {\em below} by a
positive constant. Since the range of a curve is compact, this notion would be independent of the choice
of the Riemannian metric on any smooth manifold.

\begin{definition} \label{def path} Let $M$ be a smooth compact surface. A {\em parametrized path} on $M$ is a continuous
curve $c:[0,1]\to M$ which is either constant or Lipschitz continuous with speed bounded below by
a positive constant. The set of parametrized paths is denoted by $\PPath(M)$.

Two parametrized paths on $M$ are said to be {\em equivalent} if they differ by
an increasing bi-Lipschitz homeomorphism of $[0,1]$. An equivalence class is simply called a {\em
path} and the set of paths on $M$ is denoted by $\Path(M)$. 

A {\em loop} is a path whose endpoints coincide. The set of  loops is denoted by $\Loop(M)$. A
loop is said to be {\em simple} if it is injective on $[0,1)$.
\end{definition}
\index{PApath@$\PPath(M),\Path(M)$}
\index{LAloop@$\Loop(M)$}
\index{path}
\index{loop}
\index{loop!simple}

We use for paths the same notation as for curves. If $c$ is a path, we denote its endpoints by $\underline{c}$ and $\overline{c}$, and its inverse by $c^{-1}$. The concatenation of paths is also associative whenever it is defined. When $M$ is endowed with a specific Riemannian metric, we usually identify
$\Path(M)$ with the subset of $\PPath(M)$ consisting of paths which are parametrized at constant
speed. 

While the inclusion $\PPath(M)\subset \PCurve(M)$ does not determine an inclusion $\Path(M)\subset \Curve(M)$, because we are not using the same equivalence relation on parametrized curves and parametrized paths, it is true that a path, as a set of parametrized curves, is a subset of a unique curve. Moreover, two parametrized paths which are equivalent as parametrized curves are also equivalent as parametrized paths. Hence, there is a natural injection $\Path(M)\subset \Curve(M)$ which we use without further comment.

Let us define a relation on $\Loop(M)$ by saying that two loops $l_{1},l_{2}$ are related if and only if
there exists $c,d\in \Path(M)$ such that $l_{1}=cd$ and $l_{2}=dc$. It is not difficult to check
that this is an equivalence relation.

\begin{definition} Let $M$ be a smooth compact surface. A {\em cycle} is an equivalence class of loops for the relation on $\Loop(M)$ just defined. We call {\em non-oriented cycle} a pair $\{l,l^{-1}\}$
where $l$ is a cycle. We say that a cycle is {\em simple} if one of its representatives (hence all) are simple loops. 
\end{definition}
\index{cycle}

A cycle is simply a loop from which one has forgotten the origin. It is important to observe that an oriented $1$-dimensional submanifold of $M$ determines a simple cycle. Another definition derived from that of loops and that will be useful is the following.

\begin{definition} A path $l\in \Path(M)$ is called a {\rm lasso} if there exists a path $s$ and
a simple loop $m$ such that $c = sms^{-1}$. 
\end{definition}
\index{lasso}
\index{lasso!spoke}
\index{lasso!meander}

\begin{lemma} Let $l$ be a lasso. There exist a unique path $s$ and a unique simple loop $m$
such that $l=sms^{-1}$. The path $s$ is called the {\em spoke} of $l$ and the simple loop $m$ the {\em meander} of $l$.
\end{lemma}
\index{spoke|see{lasso}}
\index{meander|see{lasso}}

\begin{proof}Endow $M$ with a Riemannian metric. Assume that $l$ is parametrized at constant speed by $[0,1]$. Then the meander of $l$ is the restriction of $l$ to the largest interval of the form $[\frac{1}{2}-t,\frac{1}{2}+t)$ on which $l$ is injective. \end{proof}

We will use the topology of uniform convergence on $\Curve(M)$. 

\begin{definition} Let $M$ be a compact surface endowed with a Riemannian metric, whose
Riemannian distance we denote by $d$. Let $c_{1},c_{2}$ be two curves of $M$. We define
$$d_{\infty}(c_{1},c_{2})=\inf_{\rm param} \sup_{t\in [0,1]} d(c_{1}(t),c_{2}(t)),$$
where the infimum is taken over all parametrizations of $c_{1}$ and $c_{2}$.
\end{definition}

The distance $d_{\infty}$ depends on the Riemannian metric chosen on $M$. However, the
topology on $\Curve(M)$ does not.

\begin{lemma} Let $M$ be a compact surface. The distances on $\Curve(M)$ associated with any two Riemannian
metrics on $M$ are equivalent.
\end{lemma}

\begin{proof}Since $M$ is compact, the Riemannian distances on $M$ determined by any two Riemannian metrics
are equivalent. \end{proof}

On $\Path(M)$, we will use a topology which is stronger than the trace of 
the uniform topology. We use an analogue of the topology of convergence in
$1$-variation of Lipschitz continuous paths, for which a sequence of Lipschitz continuous
paths in a Euclidean space converges if it converges uniformly and the sequence of the derivatives
of the paths converges in $L^1$. For the moment, we introduce a metric on $\Path(M)$ which depends
on a Riemannian metric on $M$ and is apparently weaker than the distance in $1$-variation.

When $c$ is a path on a Riemannian surface, we denote by $\ell(c)$ its length.

\begin{definition} Let $M$ be a compact surface endowed with a Riemannian metric. Let
$c_{1},c_{2}$ be two paths on $M$. We define
$$d_{\ell}(c_{1},c_{2})=d_{\infty}(c_{1},c_{2})+|\ell(c_{1})-\ell(c_{2})|.$$
\end{definition}

It is not obvious, and we will prove in the next section, that the topology induced by
$d_{\ell}$ on $\Path(M)$ does not depend on the Riemannian metric on $M$. Moreover, we will
prove that this topology can be metrized by a distance for which $\Path(M)$ is a complete metric
space.

The topology on $\Path(M)$ induced by the distance $d_\ell$ is the only one that we consider in this work and it is always the one to which we refer when we say that a sequence of paths converges. We will often add a condition on the endpoints of the paths which we consider.

\begin{definition} \label{def cv fe} Let $(c_{n})_{n\geq 0}$ be a sequence of paths on $M$. Let $c$ be a path on
$M$. We say that $(c_{n})_{n\geq 0}$ {\em converges to $c$ with fixed endpoints} if \\
1. $d_{\ell}(c_{n},c)\build{\lra}_{}^{} 0$ as $n\to\infty$,\\
2. for all $n\geq 0$, $\underline{c_{n}}=\underline{c}$ and $\overline{c_{n}}=\overline{c}$.
\end{definition}
\index{path!convergence with fixed endpoints}

When $M$ is endowed with a Riemannian metric, we will also make use of piecewise geodesic paths.

\begin{definition}\label{def gpath}
Let $M$ be a compact surface endowed with a Riemannian metric $\gamma$. We define $\GPath_{\gamma}(M)$ as the subset of
$\Path(M)$ containing the piecewise geodesic paths, that is, the finite concatenations of segments
of geodesics. 
\end{definition}
\index{AAA@$\GPath{\gamma}(M)$}
\index{path!piecewise geodesic}

The letter $\GPath$ stands for {\it affine}, instead of the letter $G$ which will be used for
a lot of other things. We claim that $\GPath_\gamma(M)$ is dense in $\Path(M)$. Indeed, there is an obvious way to approximate an arbitrary path by piecewise geodesic ones. This definition will be useful at a later time.  

\begin{definition} \label{Dn} Let $M$ be a compact surface endowed with a Riemannian metric $\gamma$. Consider $c\in \Path(M)$, identified with a path parametrized at constant speed. Let $n\geq 0$ be an integer. Assume that $2^{-n}\ell(c)$
is smaller than the injectivity radius of $M$. For each $k\in \{0,\ldots,2^{n}-1\}$, let
$\sigma_{n,k}$ be the minimizing geodesic which joins $c(k2^{-n})$ to $c((k+1)2^{-n})$. Then
define $\D_{n}(c)$ by
$$\D_{n}(c)=\sigma_{n,0}\ldots \sigma_{n,2^n-1}.$$
\end{definition}
\index{DDD@$\D_{n}(c)$}
\index{path!dyadic approximation}

\begin{proposition}\label{A dense} Let $M$ be a compact surface endowed with a Riemannian metric $\gamma$. For all path $c\in \Path(M)$, the sequence $(\D_n(c))$ defined for $n$ large enough converges to $c$ with fixed endpoints.
In particular, the space $\GPath_\gamma(M)$ is dense in $\Path(M)$ for the convergence with fixed endpoints.
\end{proposition}

\begin{proof}Let $n$ be large enough for the path $\D_n(c)$ to be defined. It has the same endpoints as $c$ by construction. Let us parametrize it in such a way that for each $k\in \{0,\ldots,2^{n}-1\}$, the restriction of $\D_n(c)$ to $[k 2^{-n},(k+1)2^{-n}]$ is the minimizing geodesic which joins $c(k2^{-n})$ to $c((k+1)2^{-n})$. It is straightforward that 
$$\sup_{t\in[0,1]} d(c(t),\D_n(c)(t)) \leq 2^{-n+1} \ell(c).$$
Hence, $\D_n(c)$ converges uniformly towards $c$. Since the length is lower semi-continuous with respect to pointwise convergence, this implies that $\liminf \ell(\D_n(c))\geq \ell(c)$. On the other hand, $\ell(\D_n(c))\leq \ell(c)$, hence $\ell(\D_n(c))$ converges to $\ell(c)$. \end{proof}

\subsection{The complete metric space of rectifiable paths}\label{complete rectif}

The goal of this section is to prove that the topology that we have introduced on $\Path(M)$ does not depend on a particular choice of a Riemannian metric on $M$ and can be metrized by a complete distance. Let us start by a negative result.

\begin{lemma} Let $M$ be a compact surface endowed with a Riemannian metric. The metric
space $(\Path(M),d_{\ell})$ is not complete. 
\end{lemma}

\begin{proof}Let $c:[-\frac{1}{4},\frac{5}{4}]\to M$ be a segment of minimizing geodesic parametrized at
constant speed. For each
$n\geq 1$, define $c_{n}:[0,1]\to M$ by $c_{n}(t)=c(t+\frac{1}{n}\sin (2\pi n t))$. For all $n\geq 1$,
$d_{\infty}(c_{n},c_{|[0,1]})=\frac{1}{n}\ell(c_{|[0,1]})$. Moreover, for all $n\geq
1$, $\ell(c_{n})=4 \ell(c_{|[0,1]})$. Hence, the sequence $(c_{n})_{n\geq 1}$ is a
Cauchy sequence for $d_{\ell}$ which converges uniformly to $c_{|[0,1]}$. Its only possible limit is
$c_{|[0,1]}$, but $\ell(c_{n})$ does not converge to $\ell(c_{|[0,1]})$.\end{proof}

The main result of this section is the following.

\begin{proposition}\label{main topology} Let $M$ be a compact surface.\\
1. The topologies induced on $\Path(M)$ by the distances $d_{\ell}$ associated to any two
Riemannian metrics on $M$ are the same.\\
2. Endow $M$ with a Riemannian metric. There exists a metric $d_{1}$ on $\Path(M)$ which induces the
same topology as $d_{\ell}$ and such that $(\Path(M),d_{1})$ is a complete metric space.
\end{proposition}
\index{path!convergence|)}

In order to prove this theorem, we define a new distance on $\Path(M)$, which is the analogue of
the distance in $1$-variation between Lipschitz continuous paths. Let $TM$ denote the tangent bundle of $M$. The Levi-Civita connection of $\gamma$ determines a splitting of the tangent bundle to $TM$ as $T(TM)=T^V(TM)\oplus T^H(TM)$.
The vertical part $T^V(TM)$ is the kernel of the derivative of the bundle map $\pi : TM\lra M$. It is canonically
identified with $TM$ by associating to $X\in T_{m}M$ the vector $\frac{d}{dt}_{|t=0}(Y+tX) \in T_{Y}(T_{m}M)$. The
horizontal part $T^H(TM)$ is mapped isomorphically onto $TM$ by the
differential of $\pi$ and the reciprocal mapping can be defined as follows. Consider $X, Y\in T_mM$. Let $c:(-1,1)\to M$ be a smooth curve such that $c(0)=m$ and $\dot c(0)=X$. Let $Y(t)$ be the unique vector field along $c$ such that $Y(0)=Y$ and $\nabla_{\dot c(t)}Y(\cdot)=0$ for all $t$. Then the element of $T^H_Y(T_mM)$ which is sent to $X$ by $T\pi$ is $\dot Y(0)$. 
\index{connection!Levi-Civita}

Since the tangent space to $TM$ at each vector $X\in T_mM$ splits into the direct sum of two subspaces isomorphic to $T_mM$, there is a natural Riemannian metric on $TM$, which we denote by $\gamma\oplus \gamma$. The corresponding Riemannian distance on $TM$ can be described as follows : if $m$ and $n$ are close enough on $M$ to be joined by a unique minimizing geodesic and if $X\in T_{m}M$, $Y\in T_{n}M$, then
$$d_{TM}(X,Y)=\left(d(m,n)^2+\parallel \pt_{[m,n]}(X)-Y\parallel^2\right)^{\frac{1}{2}},$$
where $\pt_{[m,n]}$ denotes the parallel transport along the unique minimizing geodesic from $m$ to
$n$.

\begin{definition} Let $M$ be a compact surface endowed with a Riemannian metric $\gamma$. Let
$c_{1},c_{2}$ be two paths on $M$. We define 
$$d_{1}(c_{1},c_{2})=\inf_{\rm param.} \left(\sup_{t\in [0,1]} d(c_{1}(t),c_{2}(t))  + \int_{0}^{1} d_{TM}(\dot c_{1}(t),\dot c_{2}(t))\right),$$
where the infimum is taken over all parametrizations of $c_{1}$ and $c_{2}$.

We define also
$$\overline{d_{1}}(c_{1},c_{2})=\sup_{t\in [0,1]} d(c_{1}(t),c_{2}(t)) + \int_{0}^{1} d_{TM}(\dot c_{1}(t),\dot c_{2}(t)),$$
where $c_{1}$ and $c_{2}$ are parametrized at constant speed.
\end{definition}

It is clear that the inequalities $d_{\ell}\leq d_{1}\leq \overline{d_{1}}$
hold. We are going to prove that these three metrics induce the same topology on $\Path(M)$. The
main result is the following.

\begin{proposition}\label{dl -> d1} Let $M$ be a compact surface endowed with a Riemannian metric.
Let $c$ be a path on $M$ and $(c_{n})_{n\geq 1}$ a sequence of paths such that
$d_{\ell}(c_{n},c)$ tends to $0$. Then $\overline{d_{1}}(c_{n},c)$ tends to $0$.
\end{proposition}

Let us start with two preparatory lemmas.

\begin{lemma}\label{dl unif} Let $M$ be a compact surface endowed with a Riemannian metric.
Let $c$ be a path on $M$ and $(c_{n})_{n\geq 1}$ a sequence of paths such that
$d_{\ell}(c_{n},c)$ tends to $0$. Then, the paths $c_{n}$ and $c$ being parametrized at constant
speed, the uniform convergence holds:
$$\sup_{t\in [0,1]} d(c_{n}(t),c(t)) \build{\longrightarrow}_{n\to\infty}^{} 0.$$
\end{lemma}

\begin{proof}Let us parametrize $c$ and $c_n$ for all $n\geq 1$ at constant speed. Let us also choose for all $n$ a parametrization $\widetilde{c_n}$ of $c_n$ such that the uniform convergence  $\sup_{t\in [0,1]} d(\widetilde{c_n}(t),c(t)) \to 0$ holds as $n$ tends to infinity. Consider $t\in [0,1]$. Since ${\widetilde{c_n}}_{|[0,t]}$ and ${\widetilde{c_n}}_{|[t,1]}$ converge uniformly respectively to $c_{|[0,t]}$ and $c_{|[t,1]}$, we have 
$$\liminf \ell({\widetilde{c_n}}_{|[0,t]})\geq \ell(c_{|[0,t]})=t\ell(c) \mbox{ and } \liminf \ell({\widetilde{c_n}}_{|[t,1]})\geq \ell(c_{|[t,1]})=(1-t)\ell(c).$$
Since $\ell(\widetilde{c_n})$ tends to $\ell(c)$, this implies that $\ell({\widetilde{c_n}}_{|[0,t]})$ tends to $t\ell(c)$ as $n$ tends to infinity. This convergence holds for all $t\in [0,1]$ and, since the functions $t\mapsto \ell({\widetilde{c_n}}_{|[0,t]})$ are non-decreasing, a classical result ensures that the convergence is uniform:
$$\sup_{t\in [0,1]} |\ell(\widetilde{c_n}_{|[0,t]}) - t\ell(c) | \build{\lra}_{n\to \infty}^{} 0.$$
Now, for all $t\in [0,1]$, $\widetilde{c_n}(t)=c_n\left(\frac{\ell(\widetilde{c_n}_{|[0,t]})}{\ell(c_n)}\right)$. Since $c_n$ is $\ell(c_n)$-Lipschitz continuous, we have thus
$$d(\widetilde{c_n}(t),c_n(t))\leq | \ell(\widetilde{c_n}_{|[0,t]}) - t\ell(c_n) | \leq | \ell(\widetilde{c_n}_{|[0,t]}) - t\ell(c) | + t |\ell(c_n)-\ell(c)|.$$
The result follows easily. \end{proof}

The second lemma is the Euclidean version of Proposition \ref{dl -> d1}.

\begin{lemma}\label{dl d1 eucl} Let $N\geq 1$ be an integer. Consider $\RK^N$ endowed with its usual
Euclidean structure, with norm $\Vert \cdot \Vert$. Let $f$ and $(f_{n})_{n\geq 1}$ be
measurable functions from $[0,1]$ to the unit sphere
of $\RK^N$. Assume that the primitives of $f_{n}$ converge uniformly to the primitive of $f$ as
$n$ tends to infinity, that is, 
$$\sup_{t\in [0,1]} \left\Vert \int_{0}^{t} f_{n}(s)\; ds - \int_{0}^{t} f(s)\; ds \right\Vert 
\build{\longrightarrow}_{n\to\infty}^{} 0.$$
Then $f_{n}$ converges in $L^1$ towards $f$, that is,
$$\int_{0}^{1} \Vert f_{n}(t) -f(t) \Vert \; dt \build{\longrightarrow}_{n\to\infty}^{} 0.$$ 
 \end{lemma}

The difficulty of this lemma is that the assumptions do not imply that the sequence $(f_{n})$
converges almost-everywhere to $f$. For example, if $f$ is constant, $f_{n}$ can be constant
except on a small interval, which wanders around $[0,1]$, and inside which $f_{n}$ oscillates
rapidly around the value of $f$.\\

\begin{proof}Since all functions take their values in a bounded subset of $\RK^N$, it suffices to prove that
the sequence $(f_{n})_{n\geq 1}$ converges in measure to $f$, that is, denoting by $\Leb$ the
Lebesgue measure on $[0,1]$, to prove that
$$\forall \epsilon>0, \; \lim_{n\to \infty} \Leb(\{t\in [0,1] : \Vert f_{n}(t)- f(t)\Vert
>\epsilon\}) = 0.$$
According to Lebesgue's differentiation theorem,
\begin{equation}\label{leb diff}
\frac{1}{2h}\int_{t-h}^{t+h} f(\tau)\; d\tau \build{\lra}_{h\to 0}^{} f(t) \; \mbox{ for
a.e. } t\in (0,1).
\end{equation}
Let $p,q\geq 1$ be two integers. Set
$$C_{p,q}=\left\{t\in[0,1] : \forall h\leq \frac{1}{p}, \left\Vert \frac{1}{2h}\int_{t-h}^{t+h}
f(\tau)\; d\tau - f(t) \right\Vert\leq\frac{1}{q}\right\}.$$
The relation (\ref{leb diff}) is equivalent to the fact that for all $q\geq 1$, $\Leb(\bigcup_{p\geq
1} C_{p,q})=1$. Hence, for all $\alpha>0$ and all $q\geq 1$, there exists $p\geq 1$ such that
$\Leb(C_{p,q})\geq 1-\alpha$.

Let us fix $\epsilon>0$. Then, let us choose two reals $\alpha>0$, $r>0$ and an integer $q\geq 1$.
Let $p(q,\alpha)\geq 1$ be an integer such that $\Leb(C_{p(q,\alpha),q})\geq 1-\alpha$. Set
$h=\frac{1}{p(q,\alpha)}$. Let $n_{0}(r)\geq 1$ be an integer such that, 
$$\forall n\geq n_{0}(r),\; \sup_{t\in [0,1]} \left\Vert \int_{0}^{t} f_{n}(s)\; ds - \int_{0}^{t} f(s)\;
ds \right\Vert\leq \frac{1}{2r}.$$  
Choose $n\geq n_{0}(r)$ and $t\in C_{p(q,\alpha),q}$. Then
$$\left\Vert \frac{1}{2h}\int_{t-h}^{t+h} f_{n}(\tau) \; d\tau - f(t)\right\Vert
<\frac{1}{q}+\frac{1}{2hr}.$$
Since for all $s\in [0,1]$, $\Vert f_{n}(s)\Vert=\Vert f(s)\Vert=1$, we have
$$\forall \tau \in [t-h,t+h], \; \Vert f_{n}(\tau) - f(t)\Vert>\frac{\epsilon}{2} \Rightarrow
1-\langle f_{n}(\tau),f(t)\rangle >\frac{\epsilon^{2}}{8}.$$ 
Hence,
\begin{eqnarray*}
\frac{1}{2h}\Leb\left(\left\{\tau \in [t-h,t+h] : \Vert f_{n}(\tau)-f(t)\Vert
>\frac{\epsilon}{2}\right\}\right)&&\\
&& \hskip -2cm \leq  \frac{8}{\epsilon^2}\frac{1}{2h} \int_{t-h}^{t+h}\langle
f(t)-f_{n}(\tau), f(t)\rangle \; d\tau\\
&& \hskip -2cm\leq  \frac{8}{\epsilon^2} \left(\frac{1}{q}+\frac{1}{2hr}\right).
\end{eqnarray*}
The same inequality holds when $f_{n}$ is replaced by $f$. Hence, 
\begin{equation}\label{en t}
\frac{1}{2h}\Leb\left(\left\{\tau \in [t-h,t+h] : \Vert f_{n}(\tau)-f(\tau)\Vert
>\epsilon \right\}\right)\leq \frac {16}{\epsilon^2} \left(\frac{1}{q}+\frac{1}{2hr}\right).
\end{equation}
This inequality holds for every $t\in C_{p(q,\alpha),q}$. Consider a subset $T$ of
$C_{p(q,\alpha),q}$ such that any two distinct points of $T$ are at distance at least $h$. Take
$T$ to be maximal for inclusion among all subsets of $C_{p(q,\alpha),q}$ with this property. Then
by the assumption of separation of the points of $T$, $T$ has less than $\frac{1}{h}+1$ points and
by the maximality of $T$, the intervals $[t-h,t+h]$ with $t\in T$ cover $C_{p(q,\alpha),q}$. By
applying (\ref{en t}) at the points of $T$, we find
$$ \Leb(\{t \in C_{p(q,\alpha),q} : \Vert f_{n}(t)-f(t)\Vert >\epsilon\})\leq
\frac{32(1+h)}{\epsilon^2}\left(\frac{1}{q}+\frac{1}{2hr}\right).$$
Since $\Leb(C_{p(q,\alpha),q})\geq 1-\alpha$, and since $h=\frac{1}{p(q,\alpha)}\leq 1$, we have
finally proved that for all $\alpha>0$, $r>0$ and all $q\geq 1$, there exists $n_{0}(r)$ such that
$$\forall n\geq n_{0}(r), \; \Leb(\{t\in [0,1]: \Vert f_{n}(t)-f(t)\Vert>\epsilon\})\leq
\frac{64}{\epsilon^2}\left(\frac{1}{q}+\frac{p(q,\alpha)}{2r}\right)+\alpha.$$
By choosing $\alpha$ sufficiently small and $q$ sufficiently large, then $r$ such that
$\frac{p(q,\alpha)}{2r}$ is sufficiently small, this proves that the left-hand side of this
inequality can be made arbitrarily small by choosing $n$ sufficiently large. This is exactly the
desired convergence. \end{proof}

Let us now prove Proposition \ref{dl -> d1}.\\

\begin{proof}[Proof of Proposition \ref{dl -> d1}] Let us parametrize the paths $(c_{n})_{n\geq
1}$ and $c$ at constant speed. For each $n\geq 1$, set 
$U_{n}=\sup_{t\in [0,1]} d(c_{n}(t),c(t))$. By Lemma \ref{dl unif}, $U_{n}$ tends to $0$ as $n$
tends to infinity. Hence, we need to prove
that $\int_{0}^{1} d_{TM}(\dot c_{n}(t),\dot c(t))\; dt$ tends to $0$. Let us choose $n$ large
enough for $U_{n}$ to be smaller than the injectivity radius of $M$. Then
\begin{eqnarray*}
\int_{0}^{1} d_{TM}(\dot c_{n}(t),\dot c(t)) &=& \int_{0}^{1} \left(d(c_{n}(t),c(t))^2 + \left\Vert
\pt_{[c_{n}(t),c(t)]} \dot c_{n}(t) - \dot c(t)\right\Vert^2 \right)^{\frac{1}{2}} \; dt\\
& \leq  & U_{n} + \int_{0}^{1} \left\Vert \pt_{[c_{n}(t),c(t)]} \dot c_{n}(t)
- \dot c(t)\right\Vert \; dt.
\end{eqnarray*}
Nash's embedding theorem grants the existence of an isometric embedding of $M$ in a Euclidean
space. Let $i:M\to \RK^N$ be such an embedding. We denote its differential by $di$ and, using the
natural identification $T\RK^N\simeq \RK^N\times \RK^N$, we see $di$ as a map from $TM$ to
$\RK^N$. For all $X\in TM$, we have $\Vert di(X)\Vert_{\RK^N}=\Vert X \Vert$. Hence,
\begin{eqnarray*}
\left\Vert \pt_{[c_{n}(t),c(t)]} \dot c_{n}(t) - \dot c(t)\right\Vert &=& \left\Vert (di\circ
\pt_{[c_{n}(t),c(t)]})(\dot c_{n}(t)) - di(\dot c(t))\right\Vert_{\RK^N}\\
&&\hskip -2.5cm \leq \left\Vert (di\circ \pt_{[c_{n}(t),c(t)]})(\dot c_{n}(t)) - di(\dot c_{n}(t))\right\Vert_{\RK^N}
+ \left\Vert di(c_{n}(t)) - di(c(t))\right\Vert_{\RK^N}.
\end{eqnarray*}
Since $M$ is compact, any smooth function on $M\times M$ which vanishes on the diagonal
is dominated by the Riemannian distance. Let us apply this observation to a smooth function which
near the diagonal is defined by $(m,n)\mapsto \sup\{ \Vert (di\circ \pt_{[m,n]})(X)-di(X)\Vert_{\RK^N} :
X\in T_{m}M, \Vert X \Vert = 1\}$.  We find a constant $K>0$ such that for all $m,n\in M$ and all $X\in T_{m}M$, 
$$\Vert (di\circ \pt_{[m,n]})(X)-di(X)\Vert_{\RK^N}\leq K d(m,n) \Vert X \Vert.$$
Since the lengths of the paths $c_{n}$ converge, $L=\sup \{\ell (c_{n}) : n\geq 1\}$ is finite.
Hence, for $n$ large enough, we have
$$\int_{0}^{1} d_{TM}(\dot c_{n}(t),\dot c(t))\; dt \leq (1+KL) U_{n}
 + \int_{0}^{1} \left\Vert di(\dot c_{n}(t)) - di(\dot c(t))\right\Vert_{\RK^{N}}\; dt.$$
It suffices to prove that the last integral converges to $0$. As a mapping between metric spaces, $i$ is $1$-Lipschitz continuous. Hence, $i(c_{n})$ converges
uniformly to $i(c)$ as $n$ tends to infinity.  Since $i$ is a Riemannian isometry, $i(c_{n})$
and $i(c)$ are also parametrized at constant speed for all $n\geq 1$, respectively $\ell(c_{n})$
and $\ell(c)$. For all $n\geq 1$, define $f_{n}:[0,1]\to \RK^N$ by $f_{n}=\frac{1}{\ell(c_{n})}
di(\dot c_{n})$. Define also $f:[0,1]\to \RK^N$ by $f=\frac{1}{\ell(c)} di(\dot c)$.
These functions take their values in the unit sphere of $\RK^N$. Since $i(c_{n})$ converges
uniformly to $i(c)$ as $n$ tends to infinity and $\ell(c_{n})$ tends to $\ell(c)$, the primitives
of $(f_{n})_{n\geq 1}$ converge uniformly to the primitive  of $f$. By Lemma \ref{dl d1 eucl}, this
implies that $(f_{n})_{n\geq 1}$ converges in $L^1$ to $f$. Using again the fact that $\ell(c_{n})$
converges to $\ell(c)$, we find that the derivative of $i(c_{n})$
converges in $L^1$ to the derivative of $i(c)$:
\begin{equation}\label{L1 eucl}
\int_{0}^{1} \left\Vert di(\dot c_{n}(t)) - di(\dot c(t)) \right\Vert_{\R^N}
\build{\longrightarrow}_{n\to\infty}^{} 0.
\end{equation}
This is the expected convergence.
\end{proof}

\begin{lemma} \label{d1 d1} Let $M$ be a compact surface. The topology on $\Path(M)$ induced by
the distances $d_{1}$ associated to any two Riemannian metrics on $M$ are the same.
\end{lemma}

\begin{proof} Consider two Riemannian metrics $\gamma$ and $\gamma'$ on $M$. We will denote with a prime the
quantities associated with $\gamma'$.

Let $c$ be a path and $(c_{n})_{n\geq 1}$ a sequence of paths such that $d_{1}(c_{n},c)$, and thus
also $d_{\ell}(c_{n},c)$, tend to $0$ as $n$ tends to infinity. Let us parametrize $c$ and
each
path $c_{n}$ at constant speed with respect to $\gamma$. By Proposition \ref{dl -> d1}, we have
$$ \sup_{t\in [0,1]} d(c_{n}(t),c(t))  + \int_{0}^{1} d_{TM}(\dot c_{n}(t),\dot c(t)) 
\build{\longrightarrow}_{n\to\infty}^{} 0.$$
Set $L=\sup \{\ell (c_{n}) : n\geq 1\}\geq \ell(c)$. On the compact subset $B_{L}(TM)=\{X\in TM:
\Vert X \Vert_{\gamma} \leq L\}$ of $TM$, the distances $d_{TM}$ and $d_{TM}'$ are
equivalent. Moreover, the distances $d$ and $d'$ on $M$ are also equivalent. It follows that
$$ \sup_{t\in [0,1]} d'(c_{n}(t),c(t)) + \int_{0}^{1} d_{TM}'(\dot c_{n}(t),\dot c(t)) 
\build{\longrightarrow}_{n\to\infty}^{} 0$$
for some parametrization of $c$ and the paths $c_{n}$. Hence, $d_{1}'(c_{n},c)$ tends to $0$. \end{proof}

\begin{lemma}\label{d1 complete} Let $M$ be a compact surface endowed with a Riemannian metric. The metric space
$(\Path(M),\overline{d_{1}})$ is complete.
\end{lemma}

\begin{proof}Let $(c_{n})_{n\geq 1}$ be a Cauchy sequence of $\Path(M)$ for the distance
$\overline{d_{1}}$. Let us parametrize all these paths at constant speed. They form a Cauchy
sequence for the uniform distance between continuous mappings from $[0,1]$ to $M$, so they converge
uniformly to some continuous mapping $c:[0,1]\to M$.

Let us use Nash's theorem again to find an isometric embedding $i:M\to \RK^N$. Since $i$ is $1$-Lipschitz
continuous, the sequence $(i(c_{n}))_{n\geq 1}$ of paths in $\RK^N$ converges uniformly to $i(c)$.

The sequence $(c_{n})_{n\geq 1}$ is in particular Cauchy for the distance $d_{\ell}$, so that the
sequence $(\ell(c_{n}))_{n\geq 1}$ converges to some real $l$. Set $L=\sup \{\ell (c_{n}) : n\geq
1\}<+\infty$. The restriction to the compact set $B_{L}(TM)=\{X \in TM :
\Vert X \Vert \leq L\}$ of the smooth mapping $di:TM \to \RK^{N}$ is Lipschitz continuous. Hence,
the sequence $(i(c_{n}))_{n\geq 1}$ of paths in $\RK^N$ is also a Cauchy sequence for the $L^1$
distance of the derivatives. Hence, the derivatives $di(\dot c_{n})$ converge in $L^1$ to some
function $f:[0,1]\to \RK^N$ which takes its values in the sphere of radius $\lim_{n\to \infty} \ell(c_{n})=l$.
Passing the equality $\int_{0}^{t} di(\dot c_{n}(s))=i(c_{n}(t))-i(c_{n}(0))$ to the limit, we find
that $f$ is the derivative of $i(c)$. Hence, $i(c)$ is a Lipschitz continuous path parametrized at
constant speed $l$, and so is $c$. In particular, $l=\ell(c)$.

Finally, the sequence $(c_{n})_{n\geq 1}$ satisfies $d_{\ell}(c_{n},c)\to 0$ as $n$ tends to
infinity. By Proposition \ref{dl -> d1}, this implies that 
$\overline{d_{1}}(c_{n},c)$ tends to $0$ as $n$ tends to infinity.\end{proof}

Let us collect the results that we have proved and deduce Proposition \ref{main topology}.\\

\begin{proof}[Proof of Proposition \ref{dl -> d1}] Since $d_\ell\leq d_1 \leq \overline{d_1}$ and by Proposition \ref{dl -> d1}, the three distances $d_\ell$, $d_1$ and $\overline{d_1}$ induce the same topology on $\Path(M)$. By Lemma \ref{d1 d1}, this topology does not depend on the Riemannian metric on $M$. By Lemma \ref{d1 complete}, it is the topology of a complete metric space. \end{proof}

\section{Graphs}
\subsection{Graphs and the sewing of patterns}
A graph on a surface is a finite set of paths or curves called {\em edges} and which satisfy several conditions. For Markovian holonomy fields, these finite sets of paths play the role of the finite sets of points in a time interval along which one considers the finite-dimensional marginals of a Markov process. The fact that most finite collections of paths are not the set of edges of graph leads to quite a lot of technical complication: Markovian holonomy fields are stochastic processes of which only a small number of finite-dimensional marginals can be described by a simple formula.

\begin{definition}  Let $M$ be a topological compact surface. A curve on $M$ which is injective or a simple continuous loop is called a {\em continuous edge}. The set of continuous edges on $M$ is
denoted by $\CEdge(M)$. 

Let $M$ be a smooth compact surface.  A path on $M$ which is injective or a simple loop is called an {\em edge}. The set of edges on $M$ is denoted by $\Edge(M)$. 
\end{definition}
\index{edge|see{graph}}
\index{graph!edge}

When we consider an edge or a continuous edge $e$, we will often abusively denote the range of $e$  by $e$ instead of $e([0,1])$.

\begin{definition}\label{graph1} Let $M$ be a connected compact surface (resp. a topological compact surface). A {\em pre-graph} on $M$ is a triple $\G=(\V,\E,\F)$, where\\
1. $\V$ is a finite subset of $M$,\\
2. $\E$ is a non-empty finite subset of $\Edge(M)$ (resp. $\CEdge(M)$), stable by inversion, such that $\V=\bigcup_{e\in
\E}\{\underline{e},\overline{e}\}$, and such that two edges
of $\E$ which are not each other's inverse meet, if at all, only at some of their endpoints,\\
3. $\F$ is the set of the connected components of $M-\bigcup_{e\in \E}e([0,1])$.\\
The elements of $\V,\E,\F$ are called the {\em vertices}, {\em edges} and {\em faces} of the pre-graph. 

A {\em graph} on $M$ is a pre-graph which satisfies the following condition:\\
4. Each face of $\G$ is homeomorphic to an open disk of $\RK^2$.

The {\em skeleton} of a pre-graph $\G$ is the subset of $M$ defined by $\Sk(\G)=\bigcup_{e\in \E} e([0,1])$. 
The set of paths (resp. curves) that can be obtained by concatenating edges of $\G$ is denoted by $\Path(\G)$ (resp. $\Curve(\G)$). The subset of $\Path(\G)$ (resp. $\Curve(\G)$) consisting of loops is denoted by $\Loop(\G)$ (resp. $\CLoop(\G)$). 
\index{SASk@$\Sk(\G)$}
\index{graph}
\index{skeleton|see{graph}}
\index{vertex|see{graph}}
\index{GAG@$\G$}
\index{EAE@$\E$}
\index{VAV@$\V$}
\index{FAAF@$\F$}

If $M$ is homeomorphic to a sphere and $m$ is a point of $M$, we include the exceptional triple $(\{m\},\varnothing,\{M-\{m\}\})$ in the set of graphs.

A graph on a non-connected surface is defined as the data of a graph on each connected component of this surface.

Let $(M,\CS)$ be a marked surface. Let $\G$ be a graph on $M$. We say that $\G$ is a graph on $(M,\CS)$ if
each cycle of $\CS$ is represented by a loop of $\Loop(\G)$.
\end{definition}

In the terminology of Mohar and Thomassen \cite{MoharThomassen}, what we call a graph on a topological surface is a cellular embedding of a combinatorial multigraph. It was proved by Rad\'o in 1925 that every surface can be triangulated. In particular, on every topological compact surface there exists a graph. On a Riemannian surface, the proof of the fact that there exists a triangulation given in \cite{MoharThomassen} is still valid if one
uses only piecewise geodesic paths. Hence, a Riemannian surface admits a piecewise geodesic
triangulation. This triangulation is a graph and, by adding some vertices, one may assume that the
edges of this graph are geodesic. Finally, a Riemannian surface admits a graph with geodesic edges.
\index{cellular embedding of a graph}

In order to analyze a pre-graph or a graph, an effective method consists in splitting it along some of its edges. This is very similar to the surgery of smooth marked surfaces described in Section \ref{s:surgery}. The operations described here are however less regular and best defined in the category of topological surfaces. 

\begin{definition}\label{pattern} A {\em pattern} is a triple $(M,\G,\iota)$, where $M$ is a
topological surface, $\G$ is a pre-graph on $M$ and $\iota$ is an involution of the set $\E$ of
edges of $\G$ such that for all edge $e\in \E$, $\iota(e)\neq e^{-1}$,
$\iota(e^{-1})=\iota(e)^{-1}$ and $e\not\subset \partial N \Rightarrow \iota(e)=e$. A pattern
$(M,\G,\iota)$ is {\em split} if $\Sk(\G)\subset 
\partial M$. 

Two patterns $(M,\G,\iota)$ and $(M',\G',\iota')$ are isomorphic if there exists a homeomorphism
$\psi:M\to M'$ such that $\psi(\G)=\G'$ and $\psi\circ \iota=\iota' \circ \psi$. 
 \end{definition}
\index{pattern}

A pattern is meant to be sewed according to the identifications determined by its involution. Our convention here is slightly simpler than in the case of tubular patterns, in that we exclude the case $\iota(e)=e^{-1}$ which was the purely conventional encoding of unary gluings. Here, an edge $e$ is always meant to be identified by an orientation-preserving homeomorphism with $\iota(e)$. 

When $f:M'\to M$ is a continuous mapping between two surfaces and $e'$ is a continuous edge on $M'$, we denote by $f(e')$ the curve $f\circ e'$. 


\begin{definition}\label{def sewing} Let $(M,\G,\iota)$ and $(M',\G',\iota')$ be two patterns. A continuous mapping $f:M'\to M$ is called an {\em  elementary sewing} if it is the quotient map which identifies $e'$ with $\iota'(e')$ by an orientation-preserving homeomorphism for some $e'\in \E'$. Moreover, it is required that $f(\E')=\E$ and, on $\E'-\{e',{e'}^{-1}, \iota(e'),\iota(e')^{-1}\}$, $\iota \circ f = f \circ \iota'$. 

The unoriented edge $\{f(e'),f(e')^{-1}\}$ is called the {\em joint} of the elementary sewing.

A {\em sewing} is a map which can be written as the composition of several elementary sewings. A sewing is {\em complete} if the involution of the set of edges of the target surface is the identity.
\end{definition}

We have results for sewings which are similar to those we had for gluings. In particular, sewings can always be performed and a surface can always be split along an edge, provided the interior of the edge does not meet the boundary of the surface. 

\begin{proposition}\label{def unsewing} 1. Let $(M',\G',\iota')$ be a pattern. Consider $e'\in\E'$ such that $\iota(e')\neq e'$. There exists a pattern $(M,\G,\iota)$ and an elementary sewing $f:M'\to M$ such that the joint of $f$ is $\{f(e'),f(e')^{-1}\}$. Moreover, this gluing is unique up to isomorphism: if $(M'',\G'',\iota'')$ and $f'':M'\to M''$ satisfy the same properties, then there exists an isomorphism $\psi : (M,\G,\iota)\to (M'',\G'',\iota'')$ such that $\psi\circ f=f''$.

2. Let $(M,\G,\iota)$ be a pattern. Choose $\{e,e^{-1}\} \subset \E$ such that $e\cap \partial M\subset \{\underline{e},\overline{e}\}$. Then there exists a pattern $(M',\G',\iota')$ and an elementary sewing $f:M'\to M$ such that the joint of $f$ is $\{e,e^{-1}\}$. Moreover, this gluing is unique up to isomorphism: if $(M'',\G'',\iota'')$ and $f'':M''\to M$ satisfy the same properties, then there exists an isomorphism $\psi : (M',\G',\iota')\to (M'',\G'',\iota'')$ such that $f''\circ \psi = f$.
\end{proposition}

Just as Proposition \ref{def splitting}, this result is obvious at a certain intuitive level but lacks a concise proof. The first assertion relies on the fact that the result of the identification of $e$ with $\iota(e)$ is always a compact surface. This fact is explained in  \cite{MoharThomassen}, at the beginning of Section 3.1. The second assertion relies on the Jordan curve theorem and on Sch\"onfliess' theorem, which asserts that the group of homeomorphisms of $\RK^2$ acts transitively on the set of parametrized Jordan curves. A self-contained exposition of the theorems of Jordan and Sch\"onfliess and of results which are very close to the forthcoming Proposition \ref{structure pregraph} can be found in the book of B. Mohar and C. Thomassen \cite{MoharThomassen}.

Let us only discuss the second assertion when the edge $e$ is a simple loop and the equator of a M\"obius band. In this case, the surface $M'$ has one more boundary component than $M$ and this boundary component is covered by two unoriented edges $e'_1$ and $e'_2$, which we may assume to be oriented in such a way that the concatenation $e'_1 e'_2$ makes sense. In this case, $e'_1 e'_2$ is a loop which represents the new boundary component of $M'$ and the involution $\iota$ exchanges $e'_1$ and $e'_2$.

Since a gluing is a special case of a sewing, Proposition \ref{def unsewing} implies that a pre-graph can be lifted through a splitting, in a way which is unique up to homeomorphism. It also implies the following result.

\begin{lemma} Let $M$ be a compact topological surface. The group of homeomorphisms of $M$ acts transitively on the set of injective continuous edges contained in the interior of $M$.
\end{lemma}

\begin{proof}An injective continuous edge contained in the interior of $M$ determines a pre-graph on $M$. The associated split pattern is simply $M$ to which a disk has been removed. The boundary of this disk is the concatenation of two edges which are identified with each other's inverse by the involution. Hence, this split pattern does not depend, up to homeomorphism, on the edge. \end{proof}

\begin{corollary} \label{exist graph edge} Let $M$ be a topological compact surface. Let $e$ be an injective continuous edge contained in the interior of $M$. There exists a graph on $M$ of which $e$ is an edge.
\end{corollary}

\begin{proof}Let $\G$ be a graph on $M$. If $M$ is a disk and $\Sk(\G)\subset \partial M$, let us add to $\G$ a continuous edge whose interior is contained in the interior of $M$. In any other case, $\G$ contains an edge whose interior is contained in the interior of $M$. By adding vertices to $\G$ if necessary, we may assume that it has an edge, say $e_1$, contained in the interior of $M$. The image of $\G$ by a homeomorphism of $M$ which sends $e_1$ to $e$ is a graph on $M$ of which $e$ is an edge. \end{proof}

By successively applying Proposition \ref{def unsewing} in order to split all the edges of a pre-graph which are not located on the boundary, we end up with a split pattern.

\begin{proposition}\label{structure pregraph} Let $M$ be a topological compact surface. Let
$\G=(\V,\E,\F)$ be a pre-graph on $M$. Assume that each edge of $\G$ is either contained in $\partial M$ or has no interior point on $\partial M$. Endow $\E$ with the identity involution. There exists a split pattern $(M',\G',\iota)$ and a sewing $f:M'\to M$ such that $f(\E')=\E$. 
For each face $F$ of $\G$, $f^{-1}(F)$ is the interior of a connected component of $M'$ which we denote
by $M'_F$.  The sewing map $f$ applies $M'_F\setminus (\Sk(\G')\cap M'_F)$ homeomorphically onto $F$
and $\Sk(\G')\cap M'_F$ continuously onto the topological boundary of $F$.
We call $(M',\G',\iota,f)$ a split pattern of the pair $(M,\G)$.

If $(M'',\G'',\iota'')$ is another split pattern and $f'':M''\to M$ a sewing which sends $\G''$
to $\G$, then there exists an isomorphism of patterns $\psi:M'\to M''$ such that $f''\circ \psi=f$. 
Finally, if $M$ is oriented, then $M'$ can be oriented and the sewing map can be assumed to be orientation-preserving.
\end{proposition}
\index{pattern!split}

One of the simplest consequences of this result is that a pre-graph has a finite number of faces. Let us identify a simple condition under which the assumption on the edges of pre-graph made in Proposition \ref{structure pregraph} are satisfied.

\begin{lemma} \label{cycle couvert} Let $M$ be a topological compact surface. Let $\G$ be a pre-graph on $M$. Let $c$ be a subset of $M$ homeomorphic to a circle. Then $c$ is the image of a simple loop of $\CLoop(\G)$ if and only if $c\subset \Sk(\G)$. Moreover, if $c\subset\Sk(\G)$, then for each edge $e$ of $\G$, either $e$ is contained in $c$ or $e$ has no interior point on $c$.
\end{lemma}

\begin{proof}One implication in the first assertion is obvious. Assume that $c\subset \Sk(\G)$. The image of $(0,1)$ by an edge is
homeomorphic to $(0,1)$, hence it cannot contain a subset homeomorphic to a circle. Thus, there is
at least one vertex on $c$. Let us choose a continuous parametrization of $c$ by $[0,1]$,
injective on $[0,1)$, such that $c(0)=c(1)\in \V$. The set $c\cap \V$ is finite and its
complement in $c$ is a
finite union of open intervals. Let $(a,b)$ be such an interval. Each point of $(a,b)$ belongs to one
single edge of $\G$. Assume that there exists $u,v$ with $a<u<v<b$, such that $u$ and $v$ do not
belong to the same edge. Since for each given edge, the subset of $[0,1]$ covered by that edge is
closed, there must be a point between $u$ and $v$ which is covered by at least two distinct edges.
Hence, $(a,b)$ is covered by a single edge. Both $a$ and $b$ must be vertices of this edge and the
result follows.

Let $e$ be an edge which has an interior point on $c$. Let us choose a parametrization of $e$ and $t\in (0,1)$ such that $e(t)\in c$. Let $I\subset [0,1]$ be the largest segment containing $t$ such that $e(I)\subset c$, that is, the connected component of $t$ in $\{s\in[0,1] : e(s)\in c\}$. Assume first that $I=\{t\}$. In this case, since $c$ is contained in $\Sk(\G)$, $e(t)$ belongs to the closure of another edge of $\G$, hence to another edge, and it is a vertex of $\G$. This is impossible since $t\notin\{0,1\}$ by assumption. Let us now assume that $I=[a,b]$ with $a<b$. Since $e$ is an edge, the equality $e(a)=e(b)$ can occur only if $a=0$ and $b=1$, in which case $e$ is contained in $c$. Actually, in this case, $e$ is a simple loop whose range is $c$. Assume now that $e(a)\neq e(b)$. Then $e(I)$ is a subset of $c$ homeomorphic to a segment. Since $c\subset \Sk(\G)$, each endpoint of this segment belongs to the range of another edge of $\G$. Hence $e(a)$ and $e(b)$ are vertices. This forces $a=0$, $b=1$ and in particular the fact that $e$ is contained in $c$. \end{proof}

In our definition of graphs, the focus is put on edges: a graph is a set of edges which satisfies certain properties. It is important that we find a robust criterion which tells us when a pre-graph satisfies the topological condition which makes it a graph. By a robust criterion, we mean a criterion which makes it obvious that a pre-graph whose edges are close to those of a graph is also a graph. Let us apply  Proposition \ref{structure pregraph} to establish such a criterion. 

\begin{proposition} \label{equiv disks} Let $M$ be a connected topological surface. Let $\G$ be a
 pre-graph on $M$. 
The following properties are equivalent.\\
4. Each face of $\G$ is homeomorphic to an open disk of $\RK^2$.\\
4'. The skeleton of $\G$ is connected, contains $\partial M$, and there exists $v\in \Sk(\G)$ such
that any loop in $M$ based at $v$ is homotopic to a loop whose image is contained in $\Sk(\G)$.

In particular, if $\G$ is a graph, then each connected component of $\partial M$ is the image of a loop of $\CLoop(\G)$.
\end{proposition}
\index{graph!face!topology}
\index{face|see{graph}}

Of course, if $4'$ is satisfied for some $v\in \Sk(\G)$, it is satisfied for all such $v$. In the course of the proof, we use the following lemma.

\begin{lemma}\label{surj} Let $M$ be a topological surface. Let $\G$ be a topological pre-graph
on $M$. Assume that $\Sk(\G)$ is connected and contains $\partial M$. Let $v$ be a point of 
$\Sk(\G)$. Consider the quotient topological space $M/\Sk(\G)$, in which all the points of $\Sk(\G)$ are identified. Then the natural mapping $\pi_1(M,v)\to \pi_1(M/\Sk(\G),[v])$ is onto.
\end{lemma}

\begin{proof}Let $(M',\G',\iota,f)$ be a split pattern of $(M,\G)$. Since $\Sk(\G)$ contains $\partial M$,
we have $\Sk(\G')=\partial M'$. Hence, the ill-defined mapping $f^{-1}:M \to M'$ descends to a
well-defined mapping $f^{-1}:M/\Sk(\G) \to M'/\partial M'$, which is a homeomorphism. In particular,
$[v]$ admits a neighbourhood homeomorphic to a finite bunch of disks whose centres are identified,
thus a contractible neighbourhood in $M/\Sk(\G)$. Hence, any loop in $M/\Sk(\G)$ based at $[v]$ is
homotopic to a finite product of loops based at $[v]$ and whose interior does not visit $[v]$.
Choose a loop $l$ based at $[v]$ on $M/\Sk(\G)$. Assume that $l$ is homotopic to $l_1\ldots l_n$
and the interiors of $l_1,\ldots,l_n$ do not visit $[v]$.

Each loop $l_i$ corresponds via $f^{-1}$ to a loop in the space $M'_F/\partial M'_F$ for some
$F\in\F$. Such a loop can be lifted to a path $c_i$ on $M'_F$ which starts and finishes on
$\partial M'_F$ and stays in the interior of $M'_F$ in between. Since $\Sk(\G)$ contains $\partial M$,
the paths $f(c_i)$ are paths on $M$ which start and finish in $f(\partial M')=\Sk(\G)\cup \partial
M=\Sk(\G)$. Since $\Sk(\G)$ is connected, it is possible to connect their endpoints inside
$\Sk(\G)$ and thus to produce a loop in $M$ based at $v$ whose image in the quotient $M/\Sk(\G)$ is $l_1\ldots l_n$.\end{proof}

\begin{proof}[Proof of Proposition \ref{equiv disks}]  $4\Rightarrow 4'.$ Let us assume that the assumption $4$ is satisfied. Since the faces of $\G$
are homeomorphic to open disks, they contain no point of $\partial M$. Hence, the skeleton of $\G$
contains $\partial M$. Let $(M',\G',\iota,f)$ be a split pattern of $(M,\G)$. By
Proposition \ref{structure pregraph}, the interior of $M'_F$ is homeomorphic to an open disk for
each $F$, so that $M'_F$ is a closed disk. In particular, $\partial M'_{F}$ is connected. By
Proposition \ref{structure pregraph} again, it follows that the
boundary of each face is connected. If $\Sk(\G)$ was not connected, there would exist a face whose
boundary meets two distinct connected components of $\Sk(\G)$. The boundary of this face would not
be connected : this is impossible. 

For each face $F$ of $\G$,  choose a point $x_F$ in $F$. Choose $v\in\Sk(\G)$. It is well known
that any continuous loop in $M$ based at $v$ is homotopic to a loop which avoids the points
$x_F,F\in\F$. To see this, endow $M$ with a Riemannian metric. Then, any two loops which are closer
in uniform distance than the convexity radius of $M$ are homotopic to each other. In particular, any loop is homotopic to a
piecewise geodesic loop and it is possible to choose this loop such that it avoids the points
 $x_F,F\in\F$. By Proposition \ref{structure pregraph}, the skeleton of $\G$ is a retract by
deformation of $M-\{x_F:F\in\F\}$. Hence, any loop based at $v$ is homotopic in $M$ to a loop
which stays in the skeleton of $\G$.

$4'\Rightarrow 4.$ Let $(M',\G',\iota,f)$ be a split pattern of $(M,\G)$. Let $[v]$ denote the
class of $v$ in the quotient topological space $M/\Sk(\G)$. 
This class is nothing but $\Sk(\G)$. The third part of assumption 4' implies that the homomorphism $\pi_1(M,v)\to\pi_1(M/\Sk(\G),[v])$ induced by the quotient mapping is trivial. By Lemma \ref{surj} below, this homomorphism is
surjective. Hence, the assumption $4'$ implies that $M/\Sk(\G)$ is simply connected. Hence, it
implies that $M'/\partial M'$ is simply connected. The fundamental group of this space is isomorphic
 to the free product of the fundamental groups of the spaces $M'_F/\partial M'_F$. Hence, the
assumption $4'$ implies that each space $M'_F/\partial M'_F$ is simply connected. Up to
homeomorphism, there exist only two connected compact surfaces which, when all their
boundary points are identified to a single point, are simply connected : the sphere and the disk. Finally, the assumption $4'$ implies that all the surfaces $M'_F$ are homeomorphic to disks, that is, the assertion $4$. 

The last assertion follows from Lemma \ref{cycle couvert} and the fact that the skeleton of a graph covers $\partial M$. \end{proof}

\begin{corollary}\label{c: limit pregraph} Let $M$ be a topological compact surface. Let $\G=(\V,\E,\F)$ be a graph on $M$. For each $n\geq 0$, let $\G_n=(\V,\E_n,\F_n)$ be a pre-graph on $M$ equipped with a bijection $S_n:\E\to \E_n$ such that for all $e\in \E$, $S_n(e^{-1})=S_n(e)^{-1}$. Assume that for all $n\geq 0$ and all edge $e$ such that $e\subset \partial M$, $S_n(e)=e$. Assume also that for all $e\in \E$, the sequence $(S_n(e))_{n\geq 0}$ converges uniformly to $e$ with fixed endpoints.
Then, for $n$ large enough, $\G_n$ is a graph on $M$.
\end{corollary}

\begin{proof}We need to check that $\G_n$ satisfies the condition $4'$ of Proposition \ref{equiv disks} for $n$ large enough. Firstly, for all $n\geq 0$, the skeleton of $\G_n$ contains $\partial M$ because the skeleton of $\G$ does and every edge of $\G$ located on $\partial M$ is also an edge of $\G_n$. Then, for all $n\geq 0$, the skeleton of $\G_n$ is connected. Indeed, let $m$ and $m'$ be two points of $\Sk(\G_n)$. They can be respectively joined inside $\Sk(\G_n)$ to two vertices $v$ and $v'$, which are also vertices of $\G$. Since $\Sk(\G)$ is connected, there exist a curve $e_1\ldots e_k$ in $\Curve(\G)$ which joins $v$ to $v'$. The curve $S_n(e_1)\ldots S_n(e_k)$ joins $v$ to $v'$ inside $\Sk(\G_n)$.
Finally, for $n$ large enough and for all $e\in \E$, $S_n(e)$ is homotopic with fixed endpoints to $e$. By choosing a point of $\V$ as base point, we find that any loop in $\Sk(\G)$ is homotopic to a loop in $\Sk(\G_n)$. This finishes the proof. \end{proof}

\subsection{The boundary of a face}
\label{bo fa}
\index{boundary of a face|(}

Although a face in a graph is, by definition, homeomorphic to an open disk, its closure needs not be
homeomorphic to a closed disk and even when it is the case, the topological boundary of the face 
may not be homeomorphic to a circle. The boundary of a face of a graph can in fact be defined as a cycle in the graph and this is the notion which matters for us. The appropriate intuitive picture is that of someone walking
in the interior of the face, keeping her right hand on the boundary. If the surface is non-orientable, the boundary of the face is a non-oriented cycle.

\begin{figure}[h!]
\begin{center}
\scalebox{0.7}{\includegraphics{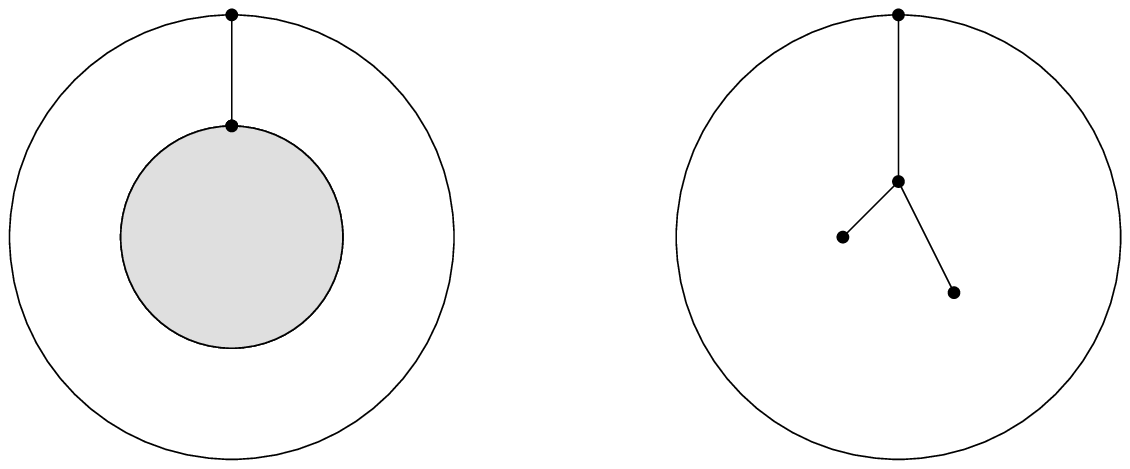}}
\caption{\small The closure of the white face of the first graph is not homeomorphic to a closed disk. The closure of the unique face of the second graph is homeomorphic to a closed disk, but its topological boundary is not homeomorphic to a circle.}  
\end{center}
\end{figure}

The following definition makes sense thanks to Proposition \ref{structure pregraph}, in particular the statement of uniqueness. 

\begin{definition}\label{graph polygon} Let $M$ be a topological compact surface. Let $\G$ be a
graph on $M$. Let $(M',\G',\iota,f)$ be a split pattern of $(M,\G)$. Let $F$ be a face of $\G$. Let $M'_F$ be the connected component of $M'$ such that $f(M'_F)=\overline{F}$.

If $M$ is oriented, then the {\em boundary} of $F$ is defined as the cycle $\partial F=f(\partial
M'_F)$ in $\CLoop(\G)$. If $M$ is not oriented, we may still orient $M'_F$ and the boundary of $F$ is defined as the unoriented cycle $\partial F=\{f(\partial M'_F),f(\partial M'_F)^{-1}\}$.

A cycle of the form $\partial F$ for some face $F$ is called a {\em facial cycle} of $(M,\G)$. 
\end{definition}
\index{DDdF@$\partial F$}
\index{graph!face!boundary (definition)}
\index{boundary of a face|see{graph}}
\index{cycle!facial}

This definition allows us to make sense of an edge adjacent to a face. 

\begin{definition} Let $M$ be a topological compact surface endowed with a graph $\G$. Let
$(M',\G',\iota,f)$ be a split pattern of $(M,\G)$. Let $F$ be a face of $\G$ and $M'_{F}$ the
corresponding connected component of $M'$. Let $e$ be an edge of $\G$. We say that the unoriented
edge $\{e,e^{-1}\}$ is adjacent to $F$ if there exists an edge $e'$ of
$\G'$ such that $e'\subset \partial M'_{F}$ and $f(e')\in\{e,e^{-1}\}$. 

If $M$ is oriented and $M'$ is oriented accordingly, we say that $e$ is adjacent to $F$ if there
exists an edge $e'$ with the same properties as above and $e'$ bounds $M'_{F}$ positively. 
\end{definition}

An unoriented edge is adjacent to a face if and only if it is contained in its topological closure. 
When $M$ is oriented, it follows from Proposition \ref{structure pregraph} that each oriented edge is
adjacent to exactly one face. It may however occur that $e$ and $e^{-1}$ are adjacent to the same
face. 

The content of the next result is that an edge which is adjacent to two distinct faces can be removed from a graph.

\begin{proposition} \label{erase edge} Let $(M,\CS)$ be a marked surface. Let $\G=(\V,\E,\F)$
be a graph on $(M,\CS)$. 

1. Let $e$ be an edge of $\G$ which is not contained in any curve of $\CS$. Assume that $e$
is adjacent to two distinct faces $F_{1}$ and $F_{2}$. Write $\partial F_{1}=ce$ and $\partial F_{2}=e^{-1}d$
for some $c,d\in\Path(\G)$. 

Then $\E\setminus\{e,e^{-1}\}$ is the set of edges of a graph on $(M,\CS)$, denoted by $\G\setminus e$, with the same
faces as $\G$,
except for the faces $F_{1}$ and $F_{2}$ which are replaced by $F=F_{1}\cup F_{2} \cup e((0,1))$.
Moreover, $\partial F=cd$.

2. Let $e$ be an edge of $\G$ which finishes at a vertex of degree $1$, that is, such that the terminal point of $e$ 
is the terminal point of no other edge of $\G$. Let $F$ be the unique face adjacent to $e$. Let $c\in \Path(\G)$ be such that $\partial F=cee^{-1}$ or $\partial F=ce^{-1}e$.

Then $\E\setminus\{e,e^{-1}\}$ is the set of edges of a graph on $(M,\CS)$, denoted by $\G\setminus e$, with the same
faces as $\G$, except for the face $F$ which is replaced by $F \cup e((0,1])$. Moreover, $\partial F=c$.
\end{proposition}

\begin{proof} The proofs of the two assertions are very similar. We prove only the first one. Let $(M',\G',\iota,f)$ be a split pattern of $(M,\G)$. By suitably orienting $M'$, we may assume that $e$ is the image by $f$ of an edge $e'_1$ which bounds $M'_{F_{1}}$ positively and $e^{-1}$ the image of an edge $e'_2=\iota(e'_1)$ which bounds $M'_{F_{2}}$ positively. Let us write $\partial M'_{F_1}=c' e'_1$ and $\partial M'_{F_2}=e'_2 d'$, where $c'$ and $d'$ are curves in $\G'$ which satisfy $f(c')=c$ and $f(d')=d$. 

Let us assume first that either $c'$ or $d'$ is not the constant curve, that is, that either $\partial M'_{F_1}\neq e'_1$ or $\partial M'_{F_2}\neq e'_2$. In this case, sewing $e'_1$ and $e'_2$ results in a new surface $M'_{F}$ which is still homeomorphic to a closed disk. 

\begin{figure}[h!]
\begin{center}
\scalebox{0.75}{\includegraphics{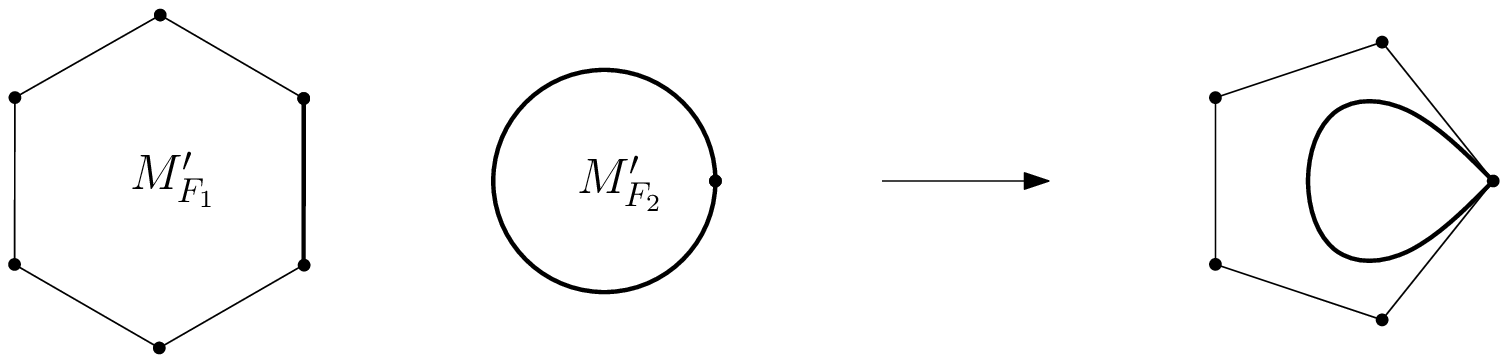}}
\caption{\small This picture illustrates the case $\partial M'_{F_1}=c'e'_1$ and $\partial M'_{F_2}=e'_2$ with $c'$ non-constant.}
\end{center}
\end{figure}

By removing the inner edge of this disk, we obtain a new split pattern $(M'',\G'',\iota'')$ with one connected component less than $M'$. The mapping $f'':M''\to M$ induced by $f$ is a complete sewing of this pattern, so that $(M'',\G'',\iota'',f'')$ is a split pattern of $(M,\G\setminus e)$. It follows on one hand that $\G\setminus e$ is a graph and on the other hand that the boundary of the new face $F$ is $f''(c'd')=f(c')f(d')=cd$. 

It remains to check that each curve of $\CS$ is represented by a loop of $\G\setminus e$. By Lemma \ref{cycle couvert}, it suffices to check that each curve of $\CS$ is contained in $\Sk(\G\setminus e)$. Consider $l\in \CS$. By the second assertion of Lemma \ref{cycle couvert}, the assumption that $e$ is not contained in any curve of $\CS$ ensures that $e$ has at most some of its endpoints on $l$. Hence, $\Sk(\G\setminus e)$ contains at least the complement of a finite set in $l$, hence $l$ itself because it is closed. This finishes the proof.

Let us now treat the case where $\partial M'_{F_1}= e'_1$ and $\partial M'_{F_2}=e'_2$. In this case, the image of $M'_{F_1}\cup M'_{F_2}$ by $f$ is a sphere of which $e$ is an equator. This sphere is a connected component of $M$ and, ignoring possible other connected components, we have $\E=\{e,e^{-1}\}$ and $\CS=\varnothing$. Hence, $\G\setminus e$ is indeed a graph, the exceptional graph with no edge and a single vertex. It has a unique face whose boundary is the constant curve at this vertex. \end{proof}

The difficulty with the definition of the boundary of a face given by Definition \ref{graph polygon} is the same that we encountered about the topological condition on the faces of a graph and that led us to state Proposition \ref{equiv disks}. It is not obvious from this definition that a small deformation of the edges of a graph cannot affect essentially the facial cycles. Since there will come a point in this work at which we will need to compare graphs with close edges, we need to be able to extract in
a fairly explicit and robust way the amount of combinatorial structure of a graph which determines
its facial cycles. 

The content of Sch\"onfliess theorem is that there is no local topological
invariant of a simple curve in a surface. Hence, in a graph, the only place where some local
structure arises is at the vertices. This structure at a given vertex is completely described by
the cyclic order of the edges which share this vertex as an endpoint. When the surface is
orientable, the information of these cyclic orders is in fact sufficient to determine completely the
facial cycles of the graph, hence, by Proposition \ref{structure pregraph}, the pair $(M,\G)$ up to
homeomorphism. When the surface is not orientable, a small amount of global information is needed
to recover the facial cycles. Before explaining this, let us describe precisely what we mean by the cyclic order of the
edges at a vertex. The characterization of this order which we establish now will be useful later.

In the next lemma, the surface is equipped with a differentiable structure because we need to consider a Riemannian metric on it. Nevertheless, the graph is allowed to have continuous edges. Also, the result is stated on a surface without boundary. If $M$ has a boundary, then the cyclic order of the edges at a vertex should be computed after a disk has been glued along each boundary component of $M$. 

\begin{lemma} \label{cyclic order}Let $M$ be a smooth surface without boundary. Let $\G$ be a graph on the topological surface underlying $M$. Let $v\in\V$ be a vertex. Let $e_1,\ldots,e_n$ be $n$ parametrized curves which represent the edges of $\G$ which share $v$ as their starting point. Let $\gamma$ be a Riemannian metric on $M$, whose Riemannian distance is denoted by $d$. Let $R$ be the injectivity radius of $\gamma$. 
Set $r_0= \min(\{R\}\cup\{d(v,e_i(\frac{1}{2})):i\in\{1,\ldots,n\}\})$. For each $r \in(0,r_0)$ and each $i\in\{1,\ldots,n\}$, define $s_i(r),t_i(r)\in [0,\frac{1}{2}]$ by
$$s_i(r)=\inf\left\{t\in\left[0,\frac{1}{2}\right] : d(v,e_i(t))= r\right\} \mbox{, } t_i(r)=\sup\left\{t\in\left[0,\frac{1}{2}\right] : d(v,e_i(t))= r\right\}.$$
 
If $M$ is not oriented, choose an orientation of the ball $B(v,r_0)$. For each $r\in(0,r_0)$, let $\omega_{first}(r)$ be the cyclic permutation of $\{e_1,\ldots,e_n\}$ corresponding to the cyclic order of the points $e_1(s_1(r)),\ldots,e_n(s_n(r))$ on the circle $C(v,r)$, oriented as the boundary of the ball $B(v,r)$. Similarly, let $\omega_{last}(r)$ be the cyclic permutation of $\{e_1,\ldots,e_n\}$ corresponding to the cyclic order of the points $e_1(t_1(r)),\ldots,e_n(t_n(r))$ on the circle $C(v,r)$. Then the following properties hold.\\
1. The cyclic order $\omega_{first}(r)$ does not depend on $r\in (0,r_0)$. We denote it simply by $\omega_{first}$. \\
2. There exists $r_1\in(0,r_0)$ such that for all $r\in (0,r_1)$, $\omega_{last}(r)=\omega_{first}$.
\end{lemma}
\index{graph!edge!cyclic order}

In the proof of this lemma, we take the following fact (which can  be deduced from Proposition \ref{structure pregraph}) for granted. On the compact cylinder $[0,1] \times S^1$, consider $n$ injective continuous curves $c_1,\ldots,c_n$ which do not meet each other. Assume that each curve starts at a point of $\{0\}\times S^1$ and finishes at a point of $\{1\}\times S^1$. Assume that no point of these curves other than their endpoints is located on the boundary of the cylinder. Then there exists an orientation-preserving  homeomorphism of the cylinder onto itself which sends each curve to a set of the form $[0,1]\times \{z\}$ for some $z \in S^1$. In particular, the cyclic order of the initial points of $c_1,\ldots,c_n$ on the circle $\{0\}\times S^1$ is the same as the cyclic order of their terminal points on the circle $\{1\}\times S^1$.\\

\begin{proof} Let us choose $r\in(0,r_0)$ and $r'\in(0,r)$. For each $i\in\{1,\ldots,n\}$, let us define 
$$u_i(r',r)=\sup\{t\in[0,s_i(r)] : d(v,e_i(t))= r'\}.$$
Thus, $c_i=e_i([u_i(r',r),s_i(r)])$ is an injective curve which joins $C(v,r')$ to $C(v,r)$ and stays in the annulus $r'\leq d(v,\cdot)\leq r$. Moreover, only the endpoints of $c_i$ lie on the boundary of the annulus. This annulus is homeomorphic to a cylinder and the curves $c_1,\ldots,c_n$ do not meet each other. 
According to the remark made before the proof, the cyclic order of the points $e_1(u_1(r',r)),\ldots,e_n(u_n(r',r))$ on the circle $C(v,r')$, which we denote by $\omega_{mixed}(r',r)$, is the same as the cyclic order of $e_1(s_1(r)),\ldots,e_n(s_n(r))$ on $C(v,r)$, which is by definition $\omega_{first}(r)$.

Set $r_1(r)=\min(\{r\}\cup \{d(v,e_i([s_i(r),\frac{1}{2}])):i\in\{1,\ldots,n\}\})$. Since the edges $e_1,\ldots,e_n$ are injective paths, $r_1(r)$ is a positive number and, for all $r'<r_1(r)$, $u_i(r',r)=t_i(r')$. Hence, for all $r\in(0,r_0)$ and $r'\in(0,r_1(r))$, 
$$\omega_{last}(r')=\omega_{mixed}(r',r)=\omega_{first}(r).$$
 Both assertions follow from this equality. \end{proof}

Let us describe informally the algorithm which one uses to computes the facial cycles of a graph. First, one has to land somewhere on the surface, near the boundary of a face and to grasp the nearest edge with either hand. Then, one walks forward to the next vertex without breaking the contact with the edge. There, one performs two operations. The first consists in changing the hand which holds the edge and turning one's body of a half-turn. One has now the vertex in one's back. The second operation consists in grasping with one's free hand the only outcoming edge at this vertex that one is not already holding and that is accessible without crossing any edge, and finally releasing the first edge. When one reaches a vertex at which there is only one outcoming edge, one turns around this vertex and walks back along the same edge, on the other side and holding it with the other hand. This process has to be iterated until one comes back to one's initial position. 

Formally, the facial cycles arise as the cycles of a certain permutation on a set which corresponds to the possible ways of our explorer holding an edge. We call this set a {\em framing} of the graph. 

\begin{definition} Let $M$ be a compact topological surface. Let $\G$ be a graph on $M$. An {\em orientation of the vertices of}  $\G$ is a collection $(U_v)_{v\in \V}$ of pairwise disjoint open subsets of $M$ such that for all vertex $v\in\V$, $U_v$ is an orientable and oriented neighbourhood of $v$.
\end{definition}
\index{graph!vertex!orientation}

Given an orientation of the vertices of $\G$, and for each edge $e$, we use the orientation of $U_{\underline{e}}$ to determine a left and a right of $e$, at least in the vicinity of $\underline{e}$. If $e$ is located on the boundary of $M$, we say that is bounds $M$ positively if $M$ is on the left of $e$.  
 
\begin{definition} Let $M$ be a topological compact surface. Let $\G$ be a graph on $M$. Let $(U_v)_{v\in \V}$ be an orientation of the vertices of $\G$. For each $e\in \E$, set
$$\fr(e)=\left\{\begin{array}{rl} \{-1,1\} & \mbox{ if } e \mbox{ is not contained in } \partial M, \\
             \{1\} & \mbox{ if } e\subset \partial M \mbox{ and } e \mbox{ bounds } M \mbox{ positively,}\\
            \{-1\} & \mbox{ if } e\subset \partial M \mbox{ and } e \mbox{ bounds } M \mbox{ negatively.}\\
                \end{array}\right.
$$
The {\em framing} of $\E$ is the subset $\fr(\E)$ of $\E\times \{-1,1\}$ defined by
$$\fr(\E)=\bigcup_{e\in\E} \{e\}\times \fr(e).$$
\end{definition}
\index{graph!framing}
\index{framing|see{graph}}

We have already mentioned that, without the assumption that $M$ is orientable, some amount of global information is needed to determine the facial cycles. 

\begin{definition} Let $M$ be a compact topological surface. Let $\G$ be a graph on $M$. Let $(U_{v})_{v\in\V}$ be an orientation of the vertices of $\G$. The {\em signature} of this orientation is the collection of signs $(\lambda_e)_{e\in \E} \in \{-1,1\}^{\E}$ defined as follows. For each edge $e$ which is a simple loop, set $\lambda_e=1$ if $e$ admits an orientable neighbourhood and $\lambda_e=-1$ otherwise. Then, for each edge $e$ such that $\underline{e}\neq \overline{e}$, consider an orientable neighbourhood $U_e$ of $e$ and set $\lambda_e=1$ if there exists an orientation of
$U_e$ compatible with the orientations of $U_{\underline{e}}$ and $U_{\overline{e}}$, and $\lambda_e=-1$ otherwise.
\end{definition}

If $M$ is orientable, then it is possible to choose an orientation of the vertices of $\G$ which is induced by an orientation of $M$. The signature of such an orientation is simply given by $\lambda_e=1$ for all $e\in\E$.

We are now ready to define the permutation on $\fr(\E)$ which determines the facial cycles.

\begin{definition}\label{fas} Let $M$ be a compact topological surface. Let $\G$ be a graph on $M$. Let $(U_{v})_{v\in\V}$ be an orientation of the vertices of $\G$. Let $\fr(\E)$ be the associated framing of $\G$. Let $(\lambda_e)_{e\in \E}$ be the signature of this orientation. 

The collection of the cyclic orders of the outcoming vertices at each vertex relatively to the orientation specified by the collection $(U_{v})_{v\in\V}$
is the the set of cycles of a unique permutation of $\E$ which we denote by $\sigma$. The involution $e\mapsto e^{-1}$ is another permutation of $\E$ which we denote by $\alpha$.

We define now three permutations $\bar\alpha$, $\bar\sigma$ and $\bar\varphi$ of $\fr(\E)$ as follows. Firstly, we set
$$\forall (e,\epsilon) \in \fr(\E),\; \bar\alpha(e,\epsilon)=(e^{-1},-\lambda_e \epsilon) \mbox{ and } \bar\sigma(e,\epsilon)=(\sigma^\epsilon(e),-\epsilon).$$
Then, we define $\bar\varphi$ by the relation $\bar\varphi\bar\alpha\bar\sigma={\rm id}$. Hence,
$$\forall (e,\epsilon) \in \fr(\E),\; \bar\varphi(e,\epsilon)=(\sigma^{-\lambda_e \epsilon}(e^{-1}),\lambda_e \epsilon).$$
\end{definition}
\index{graph!face!boundary (combinatorial description)}

It is easy to check that $\bar\sigma$ and $\bar\alpha$, hence $\bar\varphi$ take indeed their valued in $\fr(\E)$. The permutations $\bar\alpha$ and $\bar\sigma$ are both involutions. They correspond respectively to the first and second operations performed by our explorer after reaching a vertex. The permutation $\bar\varphi$ is the one whose cycles give the facial cycles of the graph.

\begin{proposition}\label{bord combi} Let $M$ be a connected topological surface. Let $\G$ be a graph on $M$. Let $(U_{v})_{v\in\V}$ be an orientation of the vertices of $\G$. Let $\fr(\E)$ be the associated framing of $\G$. Let $\bar\varphi$ be the permutation of $\fr(\E)$ defined in Definition \ref{fas}.

The range of the mapping which to each cycle $((e_1,\epsilon_1) \ldots (e_n,\epsilon_n))$ of the permutation $\bar\varphi$ associates the cycle $e_1\ldots e_n$ in $\G$ is exactly the set of the facial cycles of $\G$, taken once with each orientation.

If $M$ is orientable and oriented, and if the orientation of the sets $U_v$ is induced by the
orientation of $M$, then $\bar\varphi$ leaves $\fr(\E)\cap (\E\times \{1\})$ globally invariant.
Moreover, the set of cycles of the restriction of $\bar\varphi$ to $\fr(\E)\cap (\E\times \{1\})$
determines exactly the set of facial cycles which bound positively a face. These cycles are also those of the
permutation $\varphi=\sigma^{-1}\circ \alpha^{-1}$ on $\E$.
\end{proposition}

The best proof of this result is probably the one which the reader will make for himself by drawing pictures. Another option is to read the section 3.3 of the book by B. Mohar and C. Thomassen \cite{MoharThomassen}, although their description of
the permutations is slightly less formal than ours. This whole discussion is also a variation on the theme of {\em ribbon
graphs} or {\em maps}, which are discussed extensively in \cite{LandoZvonkin}.

Let us apply Proposition \ref{bord combi} to prove a result in the same vein as Corollary \ref{c: limit pregraph}. If $A$ and $B$ are two subsets of a same set, we use the notation $A\dotplus B=(A\cup B)\setminus (A\cap B)$. 

\begin{proposition} \label{aire faces} Let $M$ be a connected compact topological surface. Let $\G=(\V,\E,\F)$ be a graph on $M$. For each $n\geq 0$, let $\G_n=(\V_n,\E_n,\F_n)$ be a graph on $M$ equipped with a bijection $S_n:\V\to \V_n$ and a bijection $S_n:\E\to \E_n$ such that for all $e\in \E$, $S_n(\underline{e})$ is the starting point of $S_n(e)$ and $S_n(e^{-1})=S_n(e)^{-1}$. We assume that for all $n\geq 0$ and all edge $e$ such that $e\subset \partial M$, $S_n(e)=e$. We assume also that for all $e\in \E$, the sequence $(S_n(e))_{n\geq 0}$ converges uniformly to $e$. Finally, we assume that for all $n\geq 0$ and for some orientation $(U_v)_{v\in \V}$ of the vertices of $\G$ such that $S_n(v)\in U_v$ for all $n\geq 0$ and all $v\in \V$, the cyclic order of the outcoming edges at every vertex is preserved by the bijection $S_n$.

Then for all $n\geq 0$, there exists a unique bijection $S_n:\F\to \F_n$ such that for all $F\in\F$, $\partial S_n(F)=S_n(\partial F)$. Moreover, for all $F\in\F$, one has
$$\limsup_{n\to\infty } (F\dotplus S_n(F)) = \bigcap_{n\geq 0} \bigcup_{m\geq n} (F\dotplus S_m(F))\subset \Sk(\G).$$
\end{proposition}

We use the following simple lemma.

\begin{lemma}\label{faces meme bord} Let $M$ be a connected compact surface endowed with a graph $\G$. 
If $M$ is non-orientable, then two faces cannot have the same bounding unoriented cycle, and if $M$ is oriented, then two faces cannot have the same oriented bounding cycle.

More specifically, assume that there exist two faces of $\G$ whose boundaries are equal as unoriented cycles. Then $M$ is homeomorphic to a sphere and $\Sk(\G)$ is homeomorphic to a circle. In particular, after choosing an orientation of $M$, the boundaries of the two faces as oriented cycles are each other's inverse.
\end{lemma}
\index{graph!face!with the same bounding cycle}

\begin{proof}Let $F_1$ and $F_2$ denote two faces of $\G$ which share the same unoriented bounding cycle. Let $c$ be a simple loop which represents this cycle, oriented in an arbitrary way. Consider, in a split pattern $(M',\G',\iota,f)$ of $(M,\G)$, the two disks $M'_1$ and $M'_2$ corresponding to $F_1$ and $F_2$ respectively. They are bounded by the same number of edges, which is also the combinatorial length of $c$. Let $e'_{1,1},\ldots,e'_{1,n}$ and $e'_{2,1},\ldots,e'_{2,n}$ denote respectively the set of edges located on the boundary of $M'_1$ and $M'_2$, in such a way that $\partial M'_1=e'_{1,1}\ldots e'_{1,n}$ and $\partial M'_2=e'_{2,1}\ldots e'_{2,n}$. Each edge on the boundary of $M'_1$ is sent by $f$ to an edge of $\G$ which is also adjacent to $F_2$, hence is identified by $\iota$ with an edge bounding $M'_2$. 
We may assume that $\iota(e'_{1,1})=e'_{2,1}$. We may also assume that $f(e'_{1,1})=f(e'_{2,1})$ is the first edge traversed by $c$ and this characterizes fully $c$ among all representatives of the unoriented cycle $\partial F_1$. Indeed, $c$ traverses each unoriented edge of $\G$ at most once, for $\iota$ does never identify two distinct edges of the boundary of $M'_1$, or $M'_2$. Hence, $c=f(e'_{1,1},\ldots,e'_{1,n})=f(e'_{2,1},\ldots,e'_{2,n})$, so that $\iota(e'_{1,i})=e'_{2,i}$ for all $i\in\{1,\ldots,n\}$. The result follows.\end{proof}

\begin{proof}[Proof of Proposition \ref{aire faces}]  In the case where $M$ is a disk and $\Sk(\G)\subset \partial M$, $\G_n=\G$ for all $n$ and the result is true. In any other case, each face of $\G$ is bounded by at least one edge which is not contained in $\partial M$. For each face $F$ of $\G$, let us choose a point $m_F\in F$ and a point $v_F$ in the interior of an edge $e_F$ adjacent to $F$ and not contained in $\partial M$. Let us choose a continuous edge $\tilde f_F$ which crosses $\Sk(\G)$ exactly once at $v_F$, and finishes at $m_F$. We assume that $v_F$ is not the initial point of $\tilde f_F$. We denote by $f_F$ the portion of $\tilde f_F$ which joins $v_F$ to $m_F$. For $n\geq 0$ large enough, $\tilde f_F$ meets $\Sk(\G_n)$,  more precisely the edge $S_n(e_F)$ and only this edge. For such $n$, let $v_{F,n}$ be the last exit point of $\tilde f_F$ from $\Sk(\G_n)$. It is an interior point of $S_n(e_F)$. Let $S_n(f_F)$ be the portion of $\tilde f_F$ which joins $v_{F,n}$ to $m_F$.

Let us perform this construction for each face $F$, with the edges $\tilde f_F$ chosen to be pairwise disjoint. Let us define $\G'$ as the graph obtained from $\G$ by subdividing the edges $e_F$ at $v_F$ and adding the edges $f_F$. Also, for all $n\geq 0$, let $\G'_n$ be the graph obtained from $\G_n$ by subdividing the edges $S_n(e_F)$ at $v_{F,n}$ and adding the edges $S_n(f_F)$. We extend $S_n:\V\to\V_n$ to $S'_n:\V'\to \V'_n$ by setting $S_n(v_F)=v_{F,n}$, and extend also $S_n:\E\to \E_n$ to $S'_n:\E'\to \E'_n$  in the obvious way. It is not difficult to check that $v_{F,n}$ tends to $v_F$ as $n$ tends to infinity, and hence that $S'_n(f_F)$ converges uniformly to $f_F$. By considering a small ball around $v_F$, one checks also that the bijections $S'_n$ still preserve the cyclic order at each vertex, including $v_F$ and $m_F$. 

\begin{figure}[h!]
\begin{center}
\scalebox{1}{\includegraphics{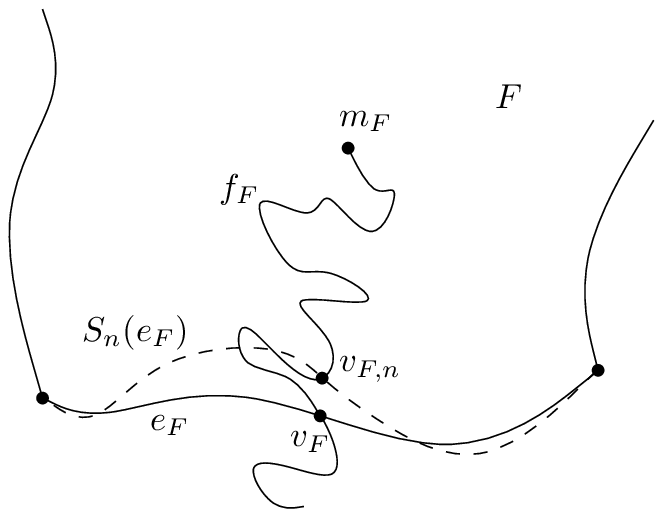}}
\caption{The construction of the graphs $\G'$ and $\G'_n$.}
\end{center}
\end{figure}

The faces of $\G$ and $\G'$ (respectively $\G_n$ and $\G'_n$) are obviously in bijective correspondence which we denote simply with a prime. For instance, for all $F\in \F$, the cycle which bounds $F'$ is deduced from $\partial F$ by the insertion of a sequence $f_F f_F^{-1}$. In particular, $F$ and $F'$ have the same topological closure. The boundary of $F'$ is the only facial cycle of $\G'$ which involves the edge $f_F$ and it involves no other edge of $\G'$ which is not contained in the skeleton of $\G$. Hence, any facial cycle of $\G'$ which involves $f_F$ must be the boundary of $F'$.

The advantage of this tedious construction is that for all $F\in \F$, $F'$ is the only face of $\F'$ whose closure contains the edge $f_F$ and in particular the vertex $m_F$.

Consider now $n\geq 1$ and a face $F$ of $\G$. Since $S_n$ preserves the cyclic order of the edges at each vertex, it follows from Proposition \ref{bord combi} that the cycle $S_n(\partial F)$ is a facial cycle of $\G_n$. By Lemma \ref{faces meme bord}, two faces of $\G_n$ cannot have the same boundary. Hence, $S_n(\partial F)$ is the boundary of a unique face of $\G_n$ which we denote by $S_n(F)$. By construction, the equality $\partial S_n(F)=S_n(\partial F)$ holds for all $F\in \F$. 

The same construction provides us with a bijection $S'_n$ between $\F'$ and $\F'_n$. Now consider $F\in \F$. By definition, $S'_n(F')$ is the face of $\G'_n$ whose boundary is the cycle $S_n(\partial F)$ in which $S'_n(f_F)S'_n(f_F)^{-1}$ has been inserted at the occurrence of the vertex $v_{F,n}$. Hence, $m_F$ belongs to the closure of $S'_n(F')$ and, by the discussion a few lines above, $S'_n(F')=S_n(F)'$. 

Now this implies that $m_F$ belongs to the closure of $S_n(F)'$, which is equal to the closure of $S_n(F)$. Since $m_F$ does not lie on the skeleton of $\G_n$, this implies that $m_F\in S_n(F)$.  

For each face $F$ of $\G$, let us choose a connected open subset $U_F$ of $F$ which contains $m_F$ and such that the closure of $U_F$ is contained in $F$. The last assumption implies that for $n$ large enough, $\Sk(\G_n)$ is disjoint from $U_F$. Hence, $U_F$ is contained in a unique face of $\G_n$, which must be $S_n(F)$.

In particular, $F\dotplus S_n(F)$ is contained in $(F\cup S_n(F))\setminus U_F$. Since $F$ nor $S_n(F)$ meet another subset of the form $U_{F_1}$ for some $F_1\in \F$, we deduce from this inclusion that, for $n$ large enough, $F\dotplus S_n(F)$ is contained in $M\setminus \bigcup_{F\in \F} U_F$. Hence, $\limsup (F\dotplus S_n(F))$ is contained in the same set. This inclusion holds for any choice of the sets $U_F$ and the result follows. \end{proof}

\index{boundary of a face|)}

\subsection{Adjunction of edges}

Proposition \ref{erase edge} gives us a way of removing some edges from a graph. We will also need a way of adding edges to a graph. The typical problem is the following: we are given a compact surface $M$ endowed with a graph $\G$, we consider a face $F$ of this graph, two vertices $v_1$ and $v_2$ on the boundary of $F$ and we would like to join them by a new edge whose interior is contained in $F$. 

If we are working in the category of topological surfaces and graphs with continuous edges, then Proposition \ref{structure pregraph} suffices to guarantee the existence of a continuous edge with the desired properties. However, if we are working with a graph with rectifiable edges and insist that the new edge be rectifiable too, then we need something more. The problem is a purely local one and we loose nothing by formulating it in the plane. The difficulty is that it seems not to be known whether a rectifiable edge can be straightened, even locally, by a bi-Lipschitz continuous homeomorphism of the plane (see for instance \cite{LuukkainenVaisala}). 

In this section and in this section only, we use the symbol $\partial$ to denote the topological boundary of a set. We denote by $\HS^1$ the 1-dimensional Hausdorff measure on $\RK^2$. 

\begin{proposition}\label{accessible} Let $K$ be a compact subset of $\RK^2$. Assume that $\partial K$ is connected and satisfies $\HS^1(\partial K)<+\infty$. Let $v$ be a point of $\partial K$. Let $m$ be a point of $\RK^2-K$. Assume that $v$ is curve-accessible from $m$, that is, that there exists a continuous curve $c:[0,1]\to \RK^2$ such that $c(0)=m$, $c(1)=v$ and $c([0,1))\cap K=\varnothing$. Then there exists an injective Lipschitz-continuous curve with the same properties as $c$.
\end{proposition}
\index{curve accessible}

We start by proving an intermediary result, whose content is that two points of the bounded connected component of the complement of a Jordan curve with finite length can be joined inside this component by a path with finite length controlled by the length of the Jordan curve.

\begin{proposition}\label{jordan longueur}  Let $U$ be a non-empty bounded connected open subset of $\RK^2$ with connected boundary. Assume that $\HS^1(\partial U)<+\infty$. For all $a,b\in U$, there exists a rectifiable path $c$ which joins $a$ to $b$ and such that $\ell(c)=\HS^1(c)\leq 100 \HS^1(\partial U)$. 
\end{proposition}

\begin{proof}Since $U$ is connected and open, it is arcwise connected. Let $\gamma$ be a continuous curve which joins $a$ to $b$ inside $U$. Set $\epsilon=\frac{1}{4} \min(d(\gamma([0,1]),\partial U),{\rm diam}(\partial U))>0$. Let $X$ be a maximal subset of $\partial U$ such that for all $x,y\in X$ with $x\neq y$, $d(x,y)\geq 2\epsilon$. The assumptions on $U$ imply that $\partial U$ is compact, so that $X$ is finite. Write $X=\{x_1,\ldots,x_n\}$. 

Let us say that a finite set of circles are in {\em generic position} if no two distinct of them are tangent and no three pairwise distinct of them have a common point. We claim that it is possible to choose $n$ positive real number $r_1,\ldots,r_n$ in the interval $(2\epsilon,4\epsilon)$ such that the boundaries of the balls $B_i=B(x_i,r_i)$ are in generic position.

Indeed, let us choose $r_1=3\epsilon$. There are only a finite number of values of $r_2$ (actually, two values) for which $\partial B_2$ is tangent to $\partial B_1$. Thus, we can choose $r_2 \in (2\epsilon,4\epsilon)$ such that $\partial B_1$ and $\partial B_2$ are not tangent. Assume that $r_1,\ldots,r_k$ have been chosen such that $\partial B_1,\ldots,\partial B_k$ are in generic position. Then there are only a finite number of values of $r_{k+1}$ for which $\partial B_1,\ldots,\partial B_{k+1}$ would not be in generic position. Hence, we can choose $r_{k+1}$ in $(2\epsilon,4\epsilon)$ such that this does not happen. 

Set $K=\bigcup_{i=1}^n \overline{B}(x_i,r_i)$. This is a compact set which contains $\partial U$, by maximality of $X$. Moreover, each connected component of $K$ meets $\partial U$, which is connected by assumption. Hence, $K$ itself is connected. Finally, $K$ does not meet $\gamma$, so that $a$ and $b$ are in the same connected component of $\RK^2\setminus K$.

Consider $x\in \partial K$. Then $x$ is on the boundary of one or two of the balls $B_1,\ldots,B_n$. In any case, $x$ admits a neighbourhood in which $\partial K$ is a simple curve composed of one or two arcs of circle. It follows that the boundary of $K$ is a compact set each point of which admits a neighbourhood homeomorphic to $\RK$. Hence, it is homeomorphic to a finite union of pairwise disjoint circles. 

\begin{figure}[h!]
\begin{center}
\scalebox{0.4}{\includegraphics{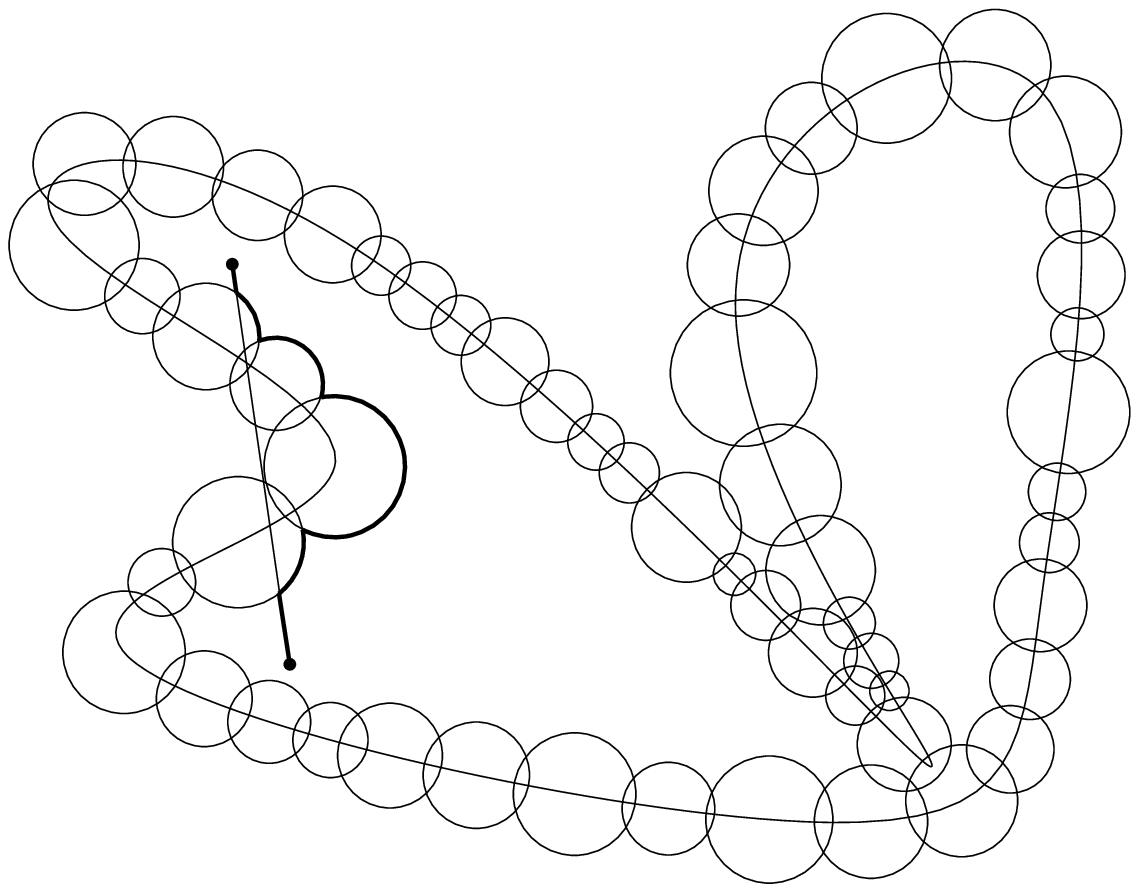}}
\caption{A path with controlled finite length between two points of the interior of a Jordan curve with finite length.}
\end{center}
\end{figure}

We have proved that $K$ is bounded by a finite collection of pairwise disjoint Jordan curves. Let us call {\em interior} of  a Jordan curve the bounded connected component of its complement. Since $K$ is bounded and connected, $\partial K$ is necessarily the union of a Jordan curve $J_0$ such that $K$ is contained in the interior of $J_0$, and a certain number of Jordan curves $J_1,\ldots,J_k$ whose interiors are disjoint and contained in the interior of $J_0$. Since $a$ and $b$ belong to the same bounded connected component of $\RK^2\setminus K$, they both lie in the interior of one of the curves $J_1,\ldots,J_k$, say $J_1$. 

Let us give a bound on the  length of $J_1$ by bounding the total length of $\partial K$. Since $2 \epsilon<{\rm diam}(\partial U)$, none of the balls $B(x_i,\epsilon)$ contains $\partial U$. Let us choose $i\in\{1,\ldots,n\}$ and set, for $y\in \RK^2$, $\pi_i(y)=d(x_i,y)$. Then $\pi_i(\partial U)$ is not contained in $(0,\epsilon)$, it is connected, and it contains $0$. Hence, $\pi_i(\partial U) \supset (0,\epsilon)$ and in fact $\pi_i(\partial U\cap B(x_i,\epsilon))\supset (0,\epsilon)$. It follows that $\HS^1(\partial U\cap B(x_i,\epsilon))\geq \epsilon$, so that $n\leq \frac{1}{\epsilon} \HS^1(\partial U)$ and finally $\HS^1(\partial K)\leq 8\pi \HS^1(\partial U)$.  

Let $s$ be the straight path from $a$ to $b$. Its length is smaller than ${\rm diam}(U)\leq{\rm diam}(\partial U)\leq \HS^1(\partial U)$. If $s$ does not meet $K$, the result is proved. Otherwise, let $a'$ and $b'$ be respectively the first and the last point at which $s$ meets $\partial K$. Let $c$ be the path obtained by concatenating the straight path from $a$ to $a'$, then an arc of $J_1$ from $a'$ to $b'$ and finally the straight path from $b'$ to $b$. The length of $c$ is bounded by $\ell(c)\leq \ell(s)+\ell(J_1) \leq 100 \HS^1(\partial U)$ as expected. \end{proof}

\begin{proof}[Proof of Proposition \ref{accessible}]  Since $\HS^1$ is a $\sigma$-additive measure, $\HS^1(K\cap B(v,r))$ tends to $\HS^1(\{v\})=0$ as $r$ tends to $0$. Let $(r_n)_{n\geq 0}$ be a decreasing sequence of positive reals such that 
\begin{equation} \label{serie converge}
\sum_{n\geq 0} (\HS^1(K\cap B(v,2r_n))+ 4\pi r_n) <+\infty.
\end{equation}
By shifting the sequence $(r_n)$ if necessary, we may assume that $B(v,2 r_0)$ does not contain $K$. Hence, for all $n\geq 0$, $K\cap \partial B(v,r_n)\neq \varnothing$. 

Let $c$ be a continuous curve which joins $m$ to $v$ and meets $K$ only at $v$. For each $n\geq 0$, let $m_n$ be the last point of the curve $c$ which is on $\partial B(v,r_n)$. Choose $n\geq 0$. Consider the compact set $K_n=(K\cap B(v,2r_n)) \cup \partial B(v,2r_n)$. It is connected, as the image of the connected set $K\cup \partial B(v,2r_n)$ by the projection on the closed convex set $B(v,2r_n)$. The points $m_n$ and $m_{n+1}$ belong to the same connected component of the complement of $K_n$ (see Figure \ref{figjfini}). Let $U_n$ denote this connected component. As a bounded connected component of the complement of a connected compact subset of $\RK^2$, $U_n$ has a connected boundary (see \cite{Wilder} p.47, where this property is called the Brouwer property of the sphere). Moreover, $\partial U_n$ is contained in $K_n$, hence $\HS^1(\partial U_n)\leq \HS^1(K\cap B(v,2r_n)) + 4\pi r_n$. So, by Proposition \ref{jordan longueur}, there exists a path $c_{n+1}$ which joins $m_n$ to $m_{n+1}$ and has length smaller than $100 \HS^1(\partial U_n)$. 
\index{Brouwer property of the sphere}

\begin{figure}[h!]
\begin{center}
\scalebox{0.8}{\includegraphics{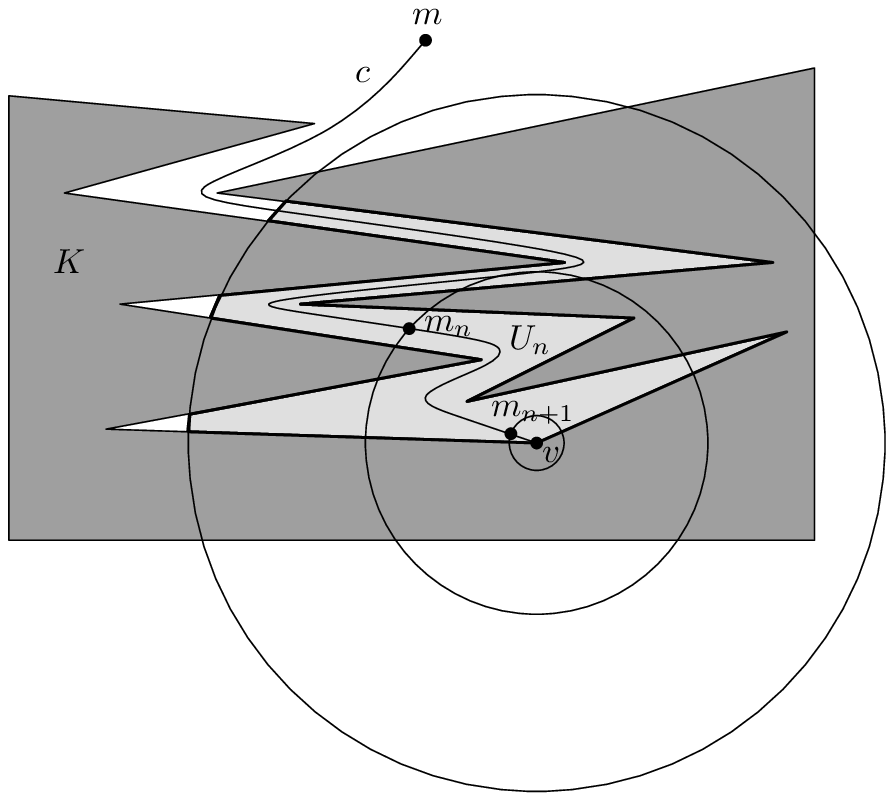}}
\caption{The path $c_{n+1}$ will join $m_n$ to $m_{n+1}$ inside $U_n$. }\label{figjfini}
\end{center}
\end{figure}

Finally, let $c_{0}$ be a path of finite length which joins $m$ to $m_0$. Set $C=\{v\}\cup \bigcup_{n\geq 0} c_n([0,1])$. It is not difficult to check that $C$ is closed. It contains $m$ and $v$, and satisfies $C\cap K=\{v\}$. By (\ref{serie converge}), $\HS^1(C)<+\infty$. However, $C$ needs not be an injective path. However, by a classical result (see \cite{David}, Proposition 14), there exists an injective Lipschitz-continuous path which joins $m$ to $v$ in $C$. Such a path is exactly what we were looking for. \end{proof}

We will make use of this result under the following form.

\begin{proposition}\label{add edge} Let $M$ be a compact surface. Let $\G$ be a graph on $M$. Let $F$ be a face of $\G$. Let $v_1$ and $v_2$ be two vertices which lie on the boundary of $F$. There exists an edge $e$ such that $\underline{e}=v_1$, $\overline{e}=v_2$ and $e((0,1))\subset F$. In particular, $\E\cup\{e,e^{-1}\}$ is the set of edges of a graph on $M$.
\end{proposition}
\index{graph!edge!adjunction}

In this statement, the assumption that the vertices lie on the boundary of the face $F$ can be understood in the topological sense or, as it is equivalent, in the sense that they are traversed by the facial cycle associated to $F$.\\

\begin{proof}Let us endow $M$ with a Riemannian metric. By the same result used at the end of the previous proof (\cite{David}, Proposition 14), it suffices to prove that there exists a compact subset $C$ of $M$ with finite $1$-dimensional Hausdorff measure which contains $v_1$ and $v_2$ and such that $C\setminus \{v_1,v_2\}\subset F$. Since $F$ is arcwise connected by paths of finite length (for instance piecewise geodesic paths), it suffices to prove that $v_1$, hence $v_2$, can be joined to at least one point of $F$ by a curve of finite length which has only its starting point outside $F$. 

For this, choose a point $n$ in $F$ and choose a continuous curve $c$ which joins $n$ to $v$ and has only its finishing point outside $F$. That such a curve exists is obvious in a split pattern of $\G$. Choose also $r>0$ such that the metric ball $B(v,r)$ is diffeomorphic to a disk and each edge starting from $v_1$  crosses the circle $\partial B(v_1,r)$. Choose a point $m$ on $c$ which $c$ traverses after its last entry time in the ball $B(v_1,r)$. 
By applying Proposition \ref{accessible} to $m$ and $(\Sk(\G)\cap B(v_1,r)) \cup \partial B(v_1,r)$ inside the ball $\overline{B}(v_1,r)$ smoothly identified with a ball in $\RK^2$, we find the desired curve with finite length. \end{proof}

\subsection{The group of loops in a graph}
\label{def RL} 

The concatenation of paths is not a group operation, even when it is restricted to a set of loops based at the same point, in which case all pairs of loops can be concatenated. The obstruction is the fact that if $c$ is a non-constant path, then there is no path $c'$ such that $cc'$ is constant. However, the path $cc^{-1}$ is equivalent to the constant path for a natural equivalence relation.

\begin{definition} Let $M$ be a compact topological compact surface. Two curves $c,c'\in\Curve(M)$ are {\em elementarily
equivalent} if there exist three curves $a,b,d$ such that $\{c,c'\}=\{ab,add^{-1}b\}$. We say that $c$ and $c'$ are {\em equivalent}
and write $c\simeq c'$ if there exists a finite chain $c=c_{0},\ldots,c_{n}=c'$ of curves such that
$c_{i}$ is elementarily equivalent to $c_{i+1}$ for each $i\in\{0,\ldots,n-1\}$.
\end{definition}
\index{curve!equivalence}
\index{path!equivalence}

This relation is an equivalence relation on $\Curve(M)$ similar to the equality of words in a free group, with the important difference that for the relation $\simeq$, there is no analog of the reduced form of a word, even if we restrict ourselves to rectifiable paths. For example, the
class of the rectifiable infinite polygonal path in the complex plane which joins
the points $0,e^{i\pi},0,2^{-1}e^{2^{-1} i\pi},$ $0, \ldots,0,2^{-n}e^{2^{-n}i\pi},0,\ldots$
contains no path of minimal length.

\begin{figure}[h!]
\begin{center}
\scalebox{1}{\includegraphics{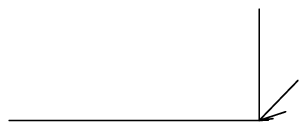}}
\caption{\small The equivalence class of this rectifiable path contains no shortest element.}
\end{center}
\end{figure}

B. Hambly and T. Lyons have defined in \cite{HamblyLyons} an
equivalence relation on rectifiable paths for which the path described above is equivalent to the
constant path equal to 0. This relation is strictly less fine than $\simeq$ and each class
contains a unique element of minimal length. We plan to investigate in a future work the
importance of this equivalence relation in the framework of the present theory. 

On a graph however, these subtleties do not arise. Until the end of this section, we work on smooth surfaces and consider paths instead of curves but all our results apply to graphs on topological surfaces.

\begin{lemma}\label{equiv graph graph} Let $M$ be a compact surface endowed with a graph $\G$. Let $c$ and $c'$ be two elements of $\Path(\G)$. Assume that $c$ and $c'$ are elementarily equivalent. Then there exist $a,b,d$ in $\Path(\G)$ such that $\{c,c'\}=\{ab,add^{-1}b\}$.
\end{lemma}

\begin{proof}Let $a,b,d$ be given by the definition of the fact that $c$ and $c'$ are elementarily equivalent. Let us assume that $c=ab$ and $c'=add^{-1}b$. Since $a$, $b$, $d$ are pieces of a path in $\G$, it suffices to show that their endpoints are vertices of $\G$ to show that they are themselves paths in $\G$. This is clear for $\underline{a}=\underline{c}$ and $\overline{b}=\overline{c}$. 

Let us say that the curve $c$ {\em backtracks} at a point $m\in M$ if there exists $t,\epsilon>0$ such that $(t-\epsilon,t+\epsilon)\subset [0,1]$ and a parametrization of $c$ such that $c(t)=m$ and for all $h\in[0,\epsilon)$, $c(t+h)=c(t-h)$. A point at which a path in $\G$ backtracks must be a vertex, hence $\overline{d}$, which is a point at which $c'$ backtracks, is a vertex. There remains to prove that $m=\overline{a}=\underline{d}=\underline{b}$ is a vertex. 

Let $\G'$ be the graph obtained by adding $m$ to the set of vertices of $\G$ and subdividing the edges of $\G$ accordingly. The graph $\G'$ has the same skeleton as $\G$. We claim that either $m$ is a backtracking point for $c$ or $c'$, or $\G'$ has at least three distinct outgoing edges at $m$. In both cases, it follows that $m$ is a vertex of $\G$.

Let $e_a^{-1}$ denote the last edge of $a$ as a path in $\G'$, and $e_d$ and $e_b$ the first edges of $d$ and $b$ as paths in $\G'$. Let us assume that $c$ does not backtrack at $m$. Then $e_a\neq e_b$. Let us assume that $c'$ does not backtrack at $m$ either. Then $ad$ and $d^{-1}b$ do not backtrack at $m$, so that $e_a\neq e_d$ and $e_d\neq e_b$. The claim is proved. \end{proof}

If $c$ is a path in a graph, we call {\em combinatorial length} of $c$ and denote by ${\rm length}(c)$ the number of edges which constitute $c$. A path is said to be {\em reduced} if it contains no sequence $ee^{-1}$ for some $e\in \E$. Equivalently, a path is reduced if it is not elementarily equivalent to a combinatorially shorter path.

\begin{corollary} Let $M$ be a compact surface endowed with a graph $\G$. Every class of equivalence of
$\Path(\G)$ contains a unique element of shortest combinatorial length, which is characterized by the fact that it is reduced.
\end{corollary}

\begin{proof}By Lemma \ref{equiv graph graph}, two paths in a graph which are equivalent differ by a finite number of insertions or erasures of sequences $ee^{-1}$, where $e$ is an edge of $\G$. 

Let us consider an equivalence class of paths for the equivalence. This class contains paths of minimal combinatorial length. These paths are necessarily reduced. Thus, it suffices to prove that the given class contains only one reduced path. Assume that there are two distinct reduced paths, say $c$ and $c'$. Consider a chain of paths $c=c_{0},\ldots,c_{n}=c'$ obtained by successive erasures and insertions of sequences $ee^{-1}$ where $e\in \E$. Assume that this chain minimizes $\max \{{\rm length}(c_{0}),\ldots,{\rm length}(c_{n})\}$ among all chains from $c$ to $c'$ and that, among those
minimizers, it also minimizes the number of intermediate paths of maximal length. Consider an
integer $k$ such that $c_{k}$ has maximal length among $c_{0},\ldots,c_{n}$. Since $c_{0}$ and
$c_{n}$ are reduced, $c_{1}$ is deduced from $c_{0}$ by an insertion and $c_{n}$ from $c_{n-1}$ by
an erasure. Thus, $k\in\{1,\ldots,n-1\}$. So, $c_{k}$ is deduced from $c_{k-1}$ by an insertion
of, say, $ee^{-1}$ and $c_{k+1}$ from $c_k$ by an erasure of $ff^{-1}$. Let us assume that $e\notin\{f,f^{-1}\}$. Then the sequence $ff^{-1}$ is already present in $c_{k-1}$ and could have been
removed before the insertion of $ee^{-1}$, thus diminishing the number of intermediate paths of
maximal length. By assumption, this is impossible. Hence, $e=f$. Moreover, for the same reason, the sequence $ee^{-1}$
removed between $c_{k}$ and $c_{k+1}$ is not present in $c_{k-1}$. It cannot be the sequence
$ee^{-1}$ inserted between $c_{k-1}$ and $c_{k+1}$, for then $c_{k-1}=c_{k+1}$ and by removing $c_{k}$ and $c_{k+1}$ from the chain, we would again diminish the number of intermediate paths of
minimal length. Hence, exactly one of the letters inserted in $c_{k}$ is removed between $c_{k}$
and $c_{k+1}$. There are two cases and in both, it appears that $c_{k-1}=c_{k+1}$. This is again
impossible. Finally, there is exactly one reduced path in each tree-equivalence class of paths.
\end{proof}

We can now define the group of reduced loops.

\begin{definition} Let $M$ be a smooth surface endowed with a graph $\G$. Let $v$ be a vertex
of $\G$. We denote by $\RL(\G)$ (resp. $\RL_{v}(\G)$) the
subset of $\Loop(\G)$ formed by reduced loops (resp. reduced loops based at $v$). 

The set $\RL_{v}(\G)$ is a group for the operation of concatenation-reduction, which to  two loops $l_1$ and $l_2$ associates the unique reduced loop equivalent to $l_1l_2$. 
\end{definition}
\index{RARL@$\RL_{m}(\G)$}
\index{loop!reduced}
\index{group of reduced loops}

It is a classical fact that the group $\RL_v(\G)$ is free. At a later stage, we will spend some effort to find families of generators of this group which satisfy special properties. For the moment, let us simply recall why it is a free group, by using spanning trees.

\begin{definition} Let $M$ be a compact surface endowed with a graph $\G$. 

A {\em spanning tree} of $\G$ is a subset $T\subset \E$ such that $T=T^{-1}$ and such that by concatenating edges of $T$, one may construct a path from any vertex to any other but no simple cycle. 

For all vertices $v_{1},v_{2}$ of $\G$, we denote by $[v_{1},v_{2}]_{T}$ the unique injective path in $T$ which joins $v_1$ to $v_2$. For each edge $e$ of $\E$, we define the loop $l_{e,T}$  by setting $l_{e,T}=[v,\underline{e}]_{T} e [\overline{e},v]_{T}$.
\end{definition}

If $e\in T$, then $l_{e,T}$ is equivalent to the constant loop. Otherwise, it is reduced and in fact a lasso. Indeed, the paths $[v,\underline e]_{T}$ and $[v,\overline e]_{T}$ can be written in a unique way as $[v,w]_{T}[w,\underline e]_{T}$ and $[v,w]_{T}[w,\overline e]_{T}$ with $[w,\underline e]_{T}\cap [w,\overline e]_{T}=\{w\}$. Then $l_{e}$ is the lasso with
spoke $[v,w]_{T}$ and meander $[w,\underline e]_{T} e [\overline e,w]_{T}$.

\begin{definition} Let $M$ be a compact surface endowed with a graph $\G$. An {\em orientation} of $\G$ is a subset $\E^+$ of $\E$ such that for all $e\in \E$, exactly one of the two edges $e$ and $e^{-1}$ belongs to $\E^+$. If $M$ is oriented, $e$ is an edge which lies on the boundary of $M$ and which bounds $M$ positively, we insist that $e\in \E^+$.

Given an orientation $\E^+$ of $\G$ and a subset $Q\subset \E$, we use the notation $Q^+=Q\cap \E^+$. 
\end{definition}
\index{EAE@$\E^+$} 

Given a graph $\G$, we set $\v(\G)=\#\V$, $\e(\G)=\frac{1}{2}\#\E$ and $\f(\G)=\#\F$. The following lemma is classical. 
\index{VAv@$\v(\G)$}
\index{FAaf@$\f(\G)$}
\index{EAe@$\e(\G)$}

\begin{lemma}\label{lg is free} Let $M$ be a compact surface endowed with a graph $\G$. Let $v$ be a vertex of $\G$.
Let $T\subset \E$ be a spanning tree of $\G$. Let $\E^+$ be an orientation of $\G$. The group $\RL_{v}(\G)$
is freely generated by the loops $\{l_{e,T}:e\in (\E\setminus T)^+\}$.
In particular, it is free of rank $\e(\G)-\v(\G)+1$. Moreover, the natural mapping $\Loop_{v}(\G)\lra
\pi_{1}(\Sk(\G),v)$ descends to a group isomorphism $i:\RL_{v}(\G)\stackrel{\sim}{\lra}\pi_{1}(\Sk(\G),v)$.
\end{lemma}

\begin{proof}If $l=e_{1}\ldots e_{n}$ belongs to $\Loop_{v}(\G)$, then $l\simeq l_{e_{1},T}\ldots
l_{e_{n},T}$. Hence, the loops $l_{e,T},e\in \E\setminus T$ generate $\RL_{v}(\G)$. Since $l_{e^{-1},T}=l_{e,T}^{-1}$ for all $e\in \E$, this implies that the loops $l_{e,T},e\in (\E\setminus T)^+$ generate $\RL_{v}(\G)$.
Now let $X$ be a group. Let $x=\{x_{e}:e\in \E\setminus T \}$ be a collection of elements of $X$ such that $x_{e^{-1}}=x_{e}^{-1}$ for all $e\in \E-T$. Complete the collection $x$ by setting $x_{e}=1$ for all $e\in T$. The mapping from $\Loop_{v}(\G)$ to $X$ which sends the loop $l=e_{1}\ldots e_{n}$ to $x_{1}\ldots x_{n}$ descends to a group homomorphism from $\RL_v(\G)$ to $X$ which sends $l_{e,T}$ to $x_{e}$ for all $e\in\E$. Thus, $\RL_{v}(\G)$ satisfies the universal property which characterizes freeness. Finally, since $T$ has $\v(\G)$ vertices, it has $\v(\G)-1$
unoriented edges. Hence, $\RL_{v}(\G)$ is free of rank $\e(\G)-\v(\G)+1$.

It is obvious that two equivalent loops are homotopic in $\Sk(\G)$. Hence, the morphism
$i$ is well defined. Let us use the letter $T$ to denote the subset $\bigcup_{e\in T} e$ of $\Sk(\G)$. This
subset is contractible and it is easy to check that $\Sk(\G)$ has the same homotopy type as $\Sk(\G)/T$, the topological
space obtained from $\Sk(\G)$ by identifying all the points of $T$. This topological space is a
bunch of circles, one for each element of $(\E\setminus T)^+$. Moreover, each loop $l_{e,T}$, composed with
the continuous projection $\Sk(\G)\lra \Sk(\G)/T$, becomes a loop which goes once around the
circle corresponding to $e$. Thus, the composition of $i$ with the isomorphism $\pi_{1}(\Sk(\G),v)\lra
\pi_{1}(\Sk(\G)/T,v)$ is an isomorphism, and $i$ is also an isomorphism. \end{proof}

\subsection{Graphs with one face}
\label{subsec: gr1}

By Proposition \ref{structure pregraph}, a graph with a single face on a connected surface determines a way of realizing this surface as the quotient of a disk by a suitable identification of its boundary. On the other hand, many non-isomorphic patterns with a single face give rise, when they are completely sewed, to homeomorphic surfaces. In this section, we discuss this fact in relation with the classical proof of the theorem of classification of surfaces (see Theorem \ref{classif}) by cut-and-paste. 

It is convenient to represent a pattern with one face by a word in a free group. This is what we explain now. For each integer $n\geq 1$, let us call {\em $n$-gon} the split pattern $(D,\G_n)$ formed by a closed disk $D$ and a graph $\G_n$ with $n$ unoriented edges on the boundary of $D$. This split pattern is unique up to homeomorphism. 

 \begin{definition}\label{admissible class}
Consider a set $X$ and let $\langle X \rangle$ denote the free group over $X$. 

1. Let $w$ be a an element of $\langle X \rangle$.  Write $w$ as a reduced word $x_1\ldots x_n$ with $x_{1},\ldots,x_{n}\in X\cup X^{-1}$. We say that $w$  is {\em admissible} if $w$ is cyclically reduced, that is, if $x_n\neq x_1^{-1}$ and each letter of $X$ appears at most twice in $w$, that is, for each $x\in X$, 
$$\#\{i\in \{1,\ldots,n\} : x_{i}\in \{x,x^{-1}\}\}\leq 2.$$
We say that an admissible word $w$ is {\em closed} if no letter appears exactly once in $w$.
\end{definition}

The fact that a word is admissible is not changed if this word is submitted to a circular permutation of its letters nor if it is replaced by its inverse. Of course, it is not changed either by changing the names of the letters: the set $X$ plays no special role and we identify two words which differ only by relabelling the letters which constitute them. We define now a correspondence between admissible words and graphs with one face.

\begin{definition} \label{mot W}
1. Let $M$ be a compact surface endowed with a graph $\G$. Assume that $\G$ has a single face and that
each vertex of $\G$ is the starting point of at least two distinct edges. Let $\E^+$ be an orientation of $\G$. Then each cycle which represents the boundary of the unique face of $\G$ is a cyclically reduced word in the letters of $\E^+$, that is, an element of $\langle \E^+\rangle$. We define
$$W(M,\G)=\{w\in \langle \E^+\rangle : w \mbox{ is a facial cycle of } \G \}.$$

2. Let $w=x_1\ldots x_{n}$ be an admissible word of length $n$. Let $(D,\G_{n})$ be an $n$-gon. Write $\E_n=\{e_1^{\pm 1},\ldots,e_{n}^{\pm 1}\}$ in such a way that $e_1\ldots e_{n}$ represents $\partial D$. Let $\iota_w$ be the involution of $\E_n$ defined as follows:
$$\forall i\in\{1,\ldots,n\}, \iota_w(e_i)=\left\{\begin{array}{cl} e_j^{\epsilon} & \mbox{ if there exists } j\neq i \mbox{ and } \epsilon=\pm 1 \mbox{ such that } x_i= x_j^\epsilon \\ e_i & \mbox{ otherwise.} \end{array}\right.$$ 
The closed compact surface obtained by completely sewing the pattern $(D,\G_{n},\iota_w)$ is said to be {\em associated} with $w$ and we denote it by $\Sigma(w)$.
\end{definition}
\index{admissible!word}

Let $M$ be a compact surface. It follows from the definitions that for all graph $\G$ with a single face on $M$ and for all $w\in W(M,\G)$, the surface $\Sigma(w)$ is homeomorphic to $M$. On the other hand, there are in general many admissible words which are not in $W(M,\G)$ whose associated surface is homeomorphic to $M$. Ignoring the precise set to which the letters of the words that we consider belong, we define a set of words as follows:
$$W(M)=\{w  \mbox{ admissible word } : M(w) \mbox{ is homeomorphic to } M\}.$$
Each word of $W(M)$ belongs to $W(M,\G)$ for some graph $\G$, for instance the graph constructed by sewing the pattern associated to this word. 

\begin{proposition}\label{graph to graph} Let $M$ be a compact surface. Let $\G_1$ and $\G_2$ be two graphs on $M$. There exists a homeomorphism $f:M\to M$ which preserves each connected component of $\partial M$, and which is orientation-preserving if $M$ is oriented, and a finite sequence of graphs $\G_{1,0},\ldots,\G_{1,r}$ such that $\G_{1,0}=\G_1$ and $\G_{1,r}=f(\G_2)$, and such that for all $i\in \{0,\ldots,r-1\}$, $\G_{1,i+1}$ is deduced from $\G_{1,i}$ by erasure of an edge in the sense of Proposition \ref{erase edge} or by adjunction of an edge in the sense of Proposition \ref{add edge}.
\end{proposition}

\begin{proof} By erasing enough edges of $\G_1$ and $\G_2$, we can transform them into two graphs with a single face and of which every vertex is the initial point of at least two distinct edges. Such graphs determine two words which we denote by $w_1$ and $w_2$.

The theorem of classification of surfaces as it is proved in \cite{Massey} asserts that, by repeated operations of cutting and pasting, $w_1$ and $w_2$ can be put under one of the standard words $[a_1,a_2]\ldots [a_{\rg-1},a_{\rg}]$ or $a_1^2 \ldots a_{\rg}^2$ if $M$ is closed of genus $\rg$, or the same words multiplied by a word of the form $d_1 c_1 d_1^{-1} \ldots d_\p c_\p d_\p^{-1}$ if $M$ has a boundary, with $c_1,\ldots,c_\p$ corresponding to the $\p$ boundary components of $M$.

The general operation of cutting and pasting, described at the level of a split pattern $(M',\G')$ of $\G$, consists in choosing two vertices $v_1$ and $v_2$ on the boundary of $M'$ and a pair of edges ${e,\iota(e)}$ which are identified by $\iota$ and separated by $v_1$ and $v_2$. One then adds to $\G'$ an edge inside $M'$ which joins $v_1$ to $v_2$, identifies $e$ and $\iota(e)$ and removes the joint of this identification (see Figure \ref{cutandpaste}). Seen on $M$, these operations can be described simply as follows:  add an edge to $\G_1$ joining $v_1$ to $v_2$, thanks to Proposition \ref{add edge}, thus creating two faces, and remove the edge $e$, thus retrieving a graph with a single face. 

\begin{figure}[h!]
\begin{center}
\scalebox{0.85}{\includegraphics{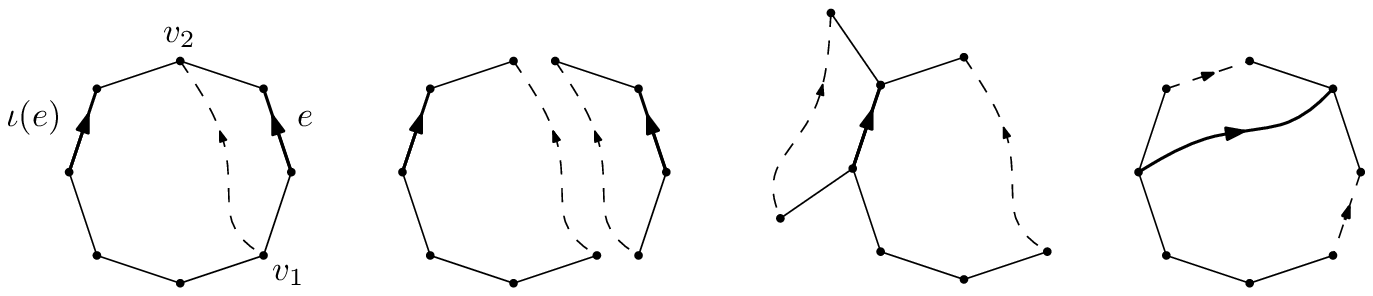}}
\caption{The basic operation of cutting and pasting}
\end{center}\label{cutandpaste}
\end{figure}

Thus, by successively adding and erasing edges to $\G_1$ and $\G_2$, we may arrive at a point where $W(M,\G_1)=W(M,\G_2)$. Then, by Proposition \ref{structure pregraph}, there exists a homeomorphism of $M$ which sends $\G_1$ to $\G_2$.  \end{proof}

\section{Riemannian metrics}
\subsection{Measured and Riemannian marked surfaces}
In the setting of Markovian holonomy fields, the scale of time is provided by a means to measure areas on each surface that one considers.

\begin{definition} Let $M$ be a smooth compact surface. A {\em measure of area} on $M$ is a smooth non-vanishing density on $M$, that is, a Borel measure which has a smooth positive density with respect to the Lebesgue measure in any coordinate chart.
\end{definition}
\index{area}

A gluing is a diffeomorphisms outside a negligible subset. Hence, a measure of area on a surface determines a measure of area on any other surface obtained by splitting this surface along a curve.

\begin{definition} Let $M$ be a smooth compact surface endowed with a measure of area denoted by $\vol$. Let $l$ be a mark on $M$ and let $\Spl_l(M)$ be the surface obtained by splitting $M$ along $l$. Let $f:\Spl_l(M) \to M$ be the associated gluing. Then $\vol \circ f$ is a measure of area on $\Spl_l(M)$ which we denote by $\Spl_l(\vol)$.
\end{definition}

On an oriented surface, a measure of area is also the same thing as a non-vanishing differential $2$-form. We are going to work with surfaces endowed with a specific measure of area, on which we will choose  Riemannian metrics. We would like these two structures to be compatible. In fact, we have the following result.

\begin{proposition} \label{nice metric} Let $(M,\CS)$ be a marked surface. Let $\vol$ be a measure of area on $M$. There exists a Riemannian metric on $M$ whose Riemannian volume is $\vol$ and such that the curves of $\CS\cup\BS(M)$ are
closed geodesics. 
\end{proposition}
\index{Riemannian metric}

Before we prove this proposition, let us state a definition.

\begin{definition}\label{riem mark surf} A {\em marked Riemannian surface} is a quadruple $(M,\vol,\gamma,\CS)$, where $(M,\vol,\CS)$ is a marked surface endowed with a smooth density and $\gamma$ is a Riemannian metric on $M$ with Riemannian volume $\vol$ and such that the curves of $\CS\cup\BS(M)$ are closed geodesics. 
\end{definition}

\begin{proof}In this proof, we denote by $\vol_{\gamma}$ the Riemannian volume of a Riemannian metric $\gamma$ on $M$. Let us first assume that $M$ is orientable and has no boundary. We write $\CS=\{l_{1}^{\pm 1},\ldots,l_{p}^{\pm 1}\}$.  

Let $\gamma_{0}$ be any Riemannian metric on $M$. Let $f$ be the smooth real function on $M$ such
that $\vol=e^{f}\vol_{\gamma_{0}}$. Set $\gamma_{1}=e^{f}\gamma_{0}$. The Riemannian volume of
$\gamma_{1}$ is $\vol$.

For each $i\in \{1,\ldots,p\}$, let $T_{i}$ denote a neighbourhood of $l_{i}$ diffeomorphic to
$[-1,1]\times \RK/2\pi\Z$ such that $l_{i}$ corresponds to $\{0\} \times \RK/2\pi\Z$. We assume that
$T_{1},\ldots,T_{p}$ are pairwise disjoint. For each $i\in \{1,\ldots,p\}$, we denote by $(r,\theta)_{i}$
the natural coordinates in $T_{i}$. 

Let $\varphi:[-1,1]\lra [0,1]$ be a smooth function such that $\varphi([-1,-\frac{3}{4}]\cup
[\frac{3}{4},1])=\{0\}$ and $\varphi([-\frac{1}{2},\frac{1}{2}])=\{1\}$. Let $\Phi$ be the smooth
real function on $M$ defined by
$$\Phi(m)=\left\{\begin{array}{l} \varphi(r) \mbox{ if } m=(r,\theta)_{i} \mbox{ for some } i\in
\{1,\ldots,p\} \cr
0 \mbox{ otherwise.}\end{array}\right.$$
For each $i\in \{1,\ldots,p\}$, consider the Riemannian metric $(dr^2 +d\theta^2)_{i}$ in $T_{i}$
and set
$$\gamma_{2,i}=\frac{\int_{T_{i}} \Phi \; \vol_{\gamma_{1}}}{\int_{T_{i}} \Phi\;
\vol_{(dr^{2}+d\theta^{2})_i}} \;( dr^{2}+ d\theta^{2})_{i}.$$
Finally, define $\gamma_{2}$ by $\gamma_{2}=(1-\Phi) \gamma_{1}+ \Phi \sum_{i=1}^{p} \gamma_{2,i}.$
The Riemannian volume of $\gamma_{2}$ coincides with $\vol$ outside the set
$\bigcup_{i=1}^{p}\{(r,\theta)_{i} \in T_{i} : |r|\leq \frac{3}{4}\}$. Moreover, the total volume
of $T_{i}$ is the same for $\vol$ and $\vol_{\gamma_{2}}$. Finally, $l_{1},\ldots,l_{p}$
are geodesic for $\gamma_{2}$.

For each $i\in \{1,\ldots,p\}$, and all $r\in [-1,1]$, set $V_{i}(r)=\int_{0}^{r}\int_{0}^{2\pi}
\vol$. It is understood that $V_{i}(r)<0$ when $r<0$. Similarly, set $V_{i,\gamma_{2}}(r)=\int_{0}^{r}\int_{0}^{2\pi}
\vol_{\gamma_{2}}$. The functions $V_{i}$ and $V_{i,\gamma_{2}}$ are both smooth, increasing, equal
to $0$ at $r=0$ and they coincide on $[-1,-\frac{3}{4}]\cup [\frac{3}{4},1]$. Define a
diffeomorphism $\rho$ of $M$ by setting
$$\rho(m)=\left\{\begin{array}{l} (V_{i}^{-1}(V_{i,\gamma_{2}}(r)),\theta)_{i} \mbox{ if }
m=(r,\theta)_{i} \mbox{ for some } i\in \{1,\ldots,p\} \cr m \mbox{ otherwise.}\end{array}\right.$$
The metric $\gamma_{3}=(\rho^{-1})^{*}\gamma_{2}$ satisfies $\vol_{\gamma_{3}}=\rho_{*}
\vol_{\gamma_{2}}$. Hence, the volume of any cylinder $[-r,r]\times \RK/2\pi\Z$, $r\in
[0,1]$ is the same for $\vol_{\gamma_{3}}$ and for $\vol$. Moreover, since $\rho$ preserves the
curves $l_{1},\ldots,l_{p}$, they are still geodesic for $\gamma_{3}$.

Let $D$ and $D_{3}$ be the two smooth functions defined on $T_{1}\cup \ldots \cup T_{p}$ such that
$\vol = D dr\wedge d\theta$ and $\vol_{\gamma_{3}}=D_{3} dr\wedge d\theta$. For each $i\in
\{1,\ldots,p\}$, define
$$A_{i}(r,\theta)=\int_{0}^{\theta} D((r,\xi)_{i}) \; d\xi \mbox{ and } A_{i,\gamma_{3}}(r,\theta)=\int_{0}^{\theta}
D_{3}((r,\xi)_{i}) \; d\xi.$$
By construction of $\gamma_{3}$, we have for each $i\in \{1,\ldots,p\}$ and all $r\in [-1,1]$ the
equality $A_{i}(r,2\pi)=A_{i,\gamma_{3}}(r,2\pi)$. It is easy to check that the mapping $\alpha$
from $M$ to itself defined by
$$\alpha(m)=\left\{\begin{array}{l} (r,A_{i}^{-1}(A_{i,\gamma_{3}}(\theta)))_{i} \mbox{ if }
m=(r,\theta)_{i} \mbox{ for some } i\in \{1,\ldots,p\} \cr m \mbox{ otherwise}\end{array}\right.$$
is a diffeomorphism. Set $\gamma_{4}=(\alpha^{-1})^*\gamma_{3}$. Then by construction,
$\vol_{\gamma_{4}}$ and $\vol$ give the same volume to any subset of $T_{i}$ which is a rectangle
in the coordinates $(r,\theta)_{i}$. Hence, they agree on $T_{1}\cup \ldots \cup T_{p}$, hence on $M$.
Since $\alpha$ preserves the curves $l_{1},\ldots,l_{p}$, they are still geodesic for
$\gamma_{4}$. Thus, $\gamma_{4}$ has the desired properties. 

Let us explain how the result extends to surfaces with boundary and non-orientable surfaces. Let $M$ be a non-orientable surface without boundary. Then there exists an orientable double of $M$, that is, an orientable surface $D(M)$ and a smooth mapping $f:D(M)\to M$ which is a covering of degree $2$. This surface $D(M)$ can for instance be constructed as the unitary frame bundle of the real line bundle $\bigwedge^2 T^*M$ for some Riemannian metric on $M$. The density $\vol$ and the marks of $M$ can be lifted through $f$. This yields an orientable marked surface $(D(M),D(\CS))$ endowed with a density $D(\vol)$ whose total area is equal to twice that of $\vol$.  The result that we have proved above applied on $D(M)$ yields a Riemannian metric $\gamma$ with Riemannian volume $D(\vol)$ and such that the curves of $D(\CS)$ are geodesics. Let $\alpha$ be the non-trivial automorphism of the covering $f:D(M)\to M$, that is, the diffeomorphism of $D(M)$ which exchanges the sheets of the covering. Then the Riemannian metric $\alpha^* \gamma$ has Riemannian volume $\alpha^* D(\vol)=D(\vol)$ and makes the curves of $\alpha^{-1}(D(\CS))=D(\CS)$ geodesics. Since the equations of geodesics are linear in the metric, the curves of $D(\CS)$ are also geodesic for the metric $\frac{1}{2} (\gamma + \alpha^* \gamma)$. This metric has also Riemannian volume $D(\vol)$. Moreover, it is invariant by $\alpha$, hence descends to a Riemannian metric on $M$ with the desired properties.

Finally, if $M$ has a boundary, then we may glue a disk along each boundary component of $M$ and extend $\vol$ to a measure of area on the surface without boundary thus obtained. \end{proof}

\subsection{Partially ordered sets of graphs}

The set of graphs on a compact surface carries a natural partial order.

\begin{definition} Let $M$ be a topological compact surface. Let $\G_1$ and $\G_2$ be two graphs on $M$. We say that 
$\G_{2}$ is finer than $\G_{1}$ and write $\G_{1}\preccurlyeq\G_{2}$ if $\Curve(\G_{1})\subset
\Curve(\G_{2})$. 
\end{definition}
\index{graph!partial order on the set of}
\index{directed set of graphs|see{graph}}

It is straightforward that $\G_2$ is finer than $\G_1$ if and only if $\E_{1}\subset \Curve(\G_{2})$. The inclusion $\E_1\subset \E_2$ implies $\G_1\preccurlyeq \G_2$ but the converse is false. 

As a poset, the set of graphs on a surface has few good properties. In particular, it is not directed, which means that it contains pairs without upper bound. 

\begin{lemma}\label{not directed} Let $M$ be a topological compact surface. The set of graphs on $M$ endowed with the partial order $\preccurlyeq$ is not directed. In other words, there exist two graphs $\G_1$ and $\G_2$ on $M$ such that no third graph $\G$ satisfies $\G_1\preccurlyeq \G$ and $\G_2 \preccurlyeq \G$.
\end{lemma}

\begin{proof}Let $U$ be an open subset of $M$ identified through a homeomorphism with the disk of $\RK^2$ centred at the origin and of radius 2. Let $e_1$ and $e_2$ be the parametrized curves defined by
$$\forall t\in [0,1], \; e_1(t)=(t,0) \mbox{ and } e_2(t)=\left(t,t^2 \sin \frac{\pi}{t}\right).$$
For all $k\geq 1$, let $A_k$ denote the open domain of $\RK^2$ delimited by the restrictions of $e_1$ and $e_2$ to $[\frac{1}{k+1},\frac{1}{k}]$. The sets $A_k,k\geq 1$ are also the bounded connected components of the complement of the union of the ranges of $e_1$ and $e_2$ in the plane.

The curves $e_1$ and $e_2$ are edges, so by Lemma \ref{exist graph edge}, there exist two graphs $\G_1$ and $\G_2$ on $M$ such that $e_1$ is an edge of $\G_1$ and $e_2$ is an edge of $\G_2$. Assume that there exists a graph $\G$ such that $\G_1\preccurlyeq \G$ and $\G_2 \preccurlyeq \G$. Then $\Sk(\G)$ would contain the union of the ranges of $e_1$ and $e_2$.  Since the range of an edge has an empty topological interior, none of the sets $A_k$ would be contained in $\Sk(\G)$. Hence, $\G$ would have infinitely many faces. We have observed after Proposition \ref{structure pregraph} that this is impossible.\end{proof}
 
This fact will be a serious problem for us at a later stage. A better-behaved substitute for the set of graphs is the set of graphs with piecewise geodesic edges.  

\begin{definition} Let $(M,\CS)$ be a marked surface endowed with a Riemannian metric $\gamma$. We define $\GG_{\gamma}(M,\CS)$ as the
set of graphs on $(M,\CS)$ with piecewise geodesic edges, that is
$$\GG_{\gamma}(M,\CS)=\{\G=(\V,\E,\F) \mbox{ graph on } (M,\CS) : \E\subset \GPath_{\gamma}(M)\}.$$
\end{definition}
\index{GAG@$\GG_{\gamma}(M,\CS)$}

The set $\GG_\gamma(M,\CS)$ can be non-empty only if the marks on $M$ are geodesic curves. We know by Proposition \ref{nice metric} that it is always possible to choose a Riemannian metric on $M$ for which this is the case.

The next result states that $\GG_\gamma(M,\CS)$ is indeed a better set of graphs than the set of all graphs.

\begin{proposition} \label{geod graph} Let $(M,\CS)$ be a marked surface endowed with a Riemannian metric for which the marks are geodesic curves. Any finite family of piecewise geodesic paths on $M$ is a subset of $\Path(\G)$ for some graph $\G$ on $(M,\CS)$ with geodesic edges.

In particular, the poset $(\GG_\gamma(M,\CS),\preccurlyeq)$ is directed.
\end{proposition}

\begin{proof}We have observed after the definition of a graph (Definition \ref{graph1}) that there exists a graph on $M$ with geodesic edges. Now, by induction on the number of curves in the finite family of curves that we consider, it
suffices to prove that, given a piecewise geodesic path $c$ and a graph $\G_{0}$ with
geodesic edges, there exists a graph $\G$ with geodesic edges such that $\Path(\G)\supset \Path(\G_{0})\cup \{c\}$.

For this, let us subdivide $c$ into a product of geodesic edges $e=e_{1}\ldots e_{m}$ in such a way that
each edge $e_{k}$ either is contained in one edge of $\G_{0}$ or has its interior contained in the
interior of a face. By adding finitely many vertices to $\G_{0}$, which means subdividing some of
its edges, we produce a new graph $\G_{1}$ which is such that each $e_{k}$ which is contained in an
edge of $\G_{0}$ is an edge of $\G_{1}$. Each other $e_{k}$ has its interior contained in a single face of $\G_{1}$. By lifting the curves $e_1,\ldots,e_m$ to a split pattern of $\G_1$, we reduce the problem to the case of a 
finite collection of geodesic segments contained in the interior of a disk with piecewise geodesic boundary. In this case, since the skeleton of $\G_1$ contains the boundary of the disk, it follows from Proposition \ref{equiv disks} that any pre-graph obtained by adding edges to $\G_1$ and whose skeleton is connected is a graph. Thus, it suffices to join one endpoint of each of the curves $e_1,\ldots,e_m$ to a point on the boundary of the disk by a geodesic segment and then to add a vertex at every point where two distinct geodesic curves meet. Hence, $\G_{1}$ can be refined into a graph $\G$ with geodesic edges such that $e_{1},\ldots,e_{m}$ belong to  $\Path(\G)$.

In order to prove that $\GG_\gamma(M,\CS)$ is directed, consider two graphs $\G_1$ and $\G_2$ with piecewise geodesic edges. The property that we have just proved applied to $\E_1\cup \E_2$ provides us with a graph which is finer than both $\G_1$ and $\G_2$. \end{proof}

\subsection{Approximation of graphs}

In this section, we prove that any graph can be approximated in a strong sense by a sequence of graphs with piecewise geodesic edges. We start by defining the lasso decomposition of a piecewise geodesic path, which is a variant of the more familiar operation of loop-erasure. Recall the definition of equivalence of paths and lassos (see
Section \ref{def RL}).

\begin{proposition} \label{lasso dec} Let $(M,\gamma)$ be a Riemannian compact surface. Let $c$ be an element of $\GPath_\gamma(M)$ such that $\underline{c}\neq
\overline{c}$. There 
exists in $\GPath_\gamma(M)$ a finite sequence of lassos $l_{1},\ldots,l_{p}$ with meanders $m_{1},\ldots,m_{p}$ and an
injective path $d$ with the same endpoints as $c$ such that\\
1. $c\simeq l_{1}\ldots l_{p} \; d$,\\
2. $\ell(c)\geq \ell(m_{1})+\ldots+\ell(m_{p})+\ell(d)$.\\
If $c$ is a loop, the same decomposition holds with the single difference that $d$ is a
simple loop. In both cases, we call $d$ the loop-erasure of $c$ and denote it by $\LE(c)$.
\end{proposition}
\index{LALE@$\LE(c)$}
\index{lasso!decomposition of a path}
\index{loop erasure}

\begin{proof}By Lemma \ref{geod graph}, we may assume that $c$ is a path in a graph. By adding vertices to this graph, we may also assume that no edge of this graph is a simple loop. Let us write $c$ as a
product of edges: $c=e_{i_{1}}\ldots e_{i_{n}}$. We proceed by induction on $n$. If $n=1$, then
$c$ is its own loop-erasure. Now assume that $n>1$. If $c$ is not reduced, that is, if it contains
at least one sequence $ee^{-1}$, then we reduce it. This can only shorten $c$. Now set $m=\min\{j>1 :
\exists k\in\{1,\ldots,j-1\}, \overline{e_{i_{j}}}=\overline{e_{i_{k}}}\}$. This is the first time at
which $c$ hits itself. Assume that $\overline{e_{i_{m}}}=\underline{e_{i_{p}}}$ with $1\leq p<m$.
By definition of $m$, $p$ is uniquely determined by
this relation. Set $t=e_{i_{1}}\ldots e_{i_{p-1}}$, $m=e_{i_{p}}\ldots e_{i_{m}}$ and $c'=t
e_{i_{m+1}}\ldots e_{i_{n}}$. Then $c\simeq tmt^{-1}c'$. By construction, $m$ is a simple loop and
$c'$ is a path shorter than $c$ with the same endpoints. Moreover, $\ell(c)=\ell(c')+\ell(m)$. The
result follows. \end{proof}

\begin{proposition} \label{approx graph epsilon} Let $(M,\vol,\gamma,\CS)$ be a marked Riemannian surface.
Let $\G=(\V,\E,\F)$ be a graph on $(M,\CS)$. Let $\epsilon>0$ be a real number.
There exists a graph $\G'=(\V',\E',\F')$ on $(M,\CS)$ with piecewise geodesic edges and two bijections ${S}:\E \to \E'$ and ${S}:\F\to \F'$, denoted by the same letter, such that the following properties hold.\\
1. $\V'=\V$.\\
2. The bijection ${S}:\E\to\E'$ commutes with the inversion and preserves the endpoints : for all $e\in \E$,
${S}(e^{-1})={S}(e)^{-1}$, $\underline{{ S}(e)}=\underline{e}$ and $\overline{{S}(e)}=\overline{e}$.\\
3. The bijection $S:\Path(\G)\to\Path(\G')$ induced by $S$ is such that for all $F\in\F$, $\partial ({S}(F))={S}(\partial F)$.\\
4. For all $e\in \E$, $d_{\ell}(e,S(e))<\epsilon$ and for all $F\in \F$, $\vol((F\cup S(F))-(F\cap S(F)))<\epsilon$.
\end{proposition}
\index{graph!approximation by piecewise geodesic graphs}

Notice that it is not possible to simply take $S(e)=\LE(\D_n(e))$ for some large $n$. Indeed, the new edges thus produced may behave badly near the vertices. For example, several edges may form a complicated spiral near a vertex that they share. In this case, it is neither certain that their loop-erased dyadic piecewise geodesic approximations do not intersect each other nor that they leave the vertex in the same cyclic order as the original edges. \\

If $x,y\in M$ and $d(x,y)$ is smaller than the injectivity radius of $(M,\gamma)$, we denote by $[x,y]$ the segment of minimizing geodesic joining $x$ to $y$.\\

\begin{proof}Let $R$ be the injectivity radius of $M$. Let $r_0\in (0,R)$ be such that the balls $B(v,r_0),v\in\V$ are pairwise disjoint, an edge $e$ meets a ball $B(v,r_0)$ only if $v$ is an endpoint of $e$, and the cyclic order of the outcoming edges at every vertex $v$ is the cyclic order of their last exit points from the ball $B(v,r)$ for all $r\in(0,r_0)$. The existence of such an $r_0$ is granted by Lemma \ref{cyclic order}. Let us choose an orientation of each ball $B(v,r_0),v\in\V$.

For all real $r\in(0,r_0)$ and all $e\in\E$, define $t_e(r)=\sup\{t\in[0,1] : d(\underline{e},e(t))= r\}$. Observe that $1-t_{e^{-1}}=\inf\{t\in[0,1] : d(e(t),\overline{e})=r\}$. Define
$$A_r(e)=[\underline{e},e(t_e(r))] e_{|[t_e(r),1-t_{e^{-1}}(r)]} [e(1-t_{e^{-1}}(r)),\overline{e}].$$
The path $A_r(e)$ is the concatenation of three injective paths which meet only at their endpoints, so that it is injective. Thus, $A_r(e)$ is an edge with the same endpoints as $e$. Let $e$ and $e'$ be two edges such that $e'\notin\{e,e^{-1}\}$. The central portions of $A_r(e)$ and $A_r(e')$ are contained in $e$ and $e'$ respectively, so that they are disjoint. By the assumption made on $r$, they do not enter any ball $B(v,r),v\in\V$. Hence, $A_r(e)$ and $A_r(e')$ meet, if at all, in one of these balls and this can occur only at one of their endpoints. 

Let $e$ be an edge. The continuity and injectivity of $e$ imply that $t_e(r)\to 0$ as $r\to 0$.
Since  $d_\ell(A_r(e),e)\leq 2 (t_e(r)+t_{e^{-1}}(r)) \ell(e)$,  this implies that $A_r(e)$ tends to $e$ as $r$ tends to $0$. Moreover, one always has the inequality $\ell(A_r(e))\leq \ell(e)$.

Let $r\in(0,r_0)$ be fixed. Let $n_0\geq 1$ be an integer such that for all $e\in\E$, $2^{-n_0}\ell(e)<r<R$. For all integer $n\geq n_0$, define
$$C_{r,n}(e)=\D_n(A_r(e)).$$
The path $C_{r,n}(e)$ is piecewise geodesic, with the same endpoints as $e$, but it may not be injective, even for large $n$. On the other hand, it coincides with $A_r(e)$ near its endpoints, more precisely, on a segment of length at least $r-2^{-n}\ell(e)$. When $n$ tends to infinity, $C_{r,n}(e)$ converges to $A_r(e)$. We claim that for all $e,e'\in\E$ such that $e'\notin\{e,e^{-1}\}$, the paths $C_{r,n}(e)$ and $C_{r,n}(e')$ intersect only at some of their endpoints for $n$ large enough. Indeed, consider for all $e\in \E$ the segment 
$$A'_r(e)=\{m\in A_r(e) : d(m,\underline{e})\geq \frac{r}{2}, d(m,\overline{e})\geq \frac{r}{2}\}.$$
Choose $n_1\geq n_0$ such that $2^{-n_1}\max\{\ell(e):e\in\E\} < \min(\frac{r}{2},\frac{1}{2} \min\{d(A_r(e),A'_r(e')) : e,e'\in\E, e'\notin\{e,e^{-1}\}\})$. Choose $e,e'\in\E$ such that $e'\notin\{e,e^{-1}\}$. For all $n\geq n_1$, $C_{r,n}(e)$ and $C_{r,n}(e')$ are respectively contained in the sets $A_r(e)\cup \{ m \in M : d(m,A'_r(e))<2^{-n} \ell(e)\}$ and $A_r(e')\cup \{ m \in M : d(m,A'_r(e'))<2^{-n} \ell(e)\}$, whose intersection is the same as the intersection of $e$ and $e'$. 

For each integer $p\geq 1$, choose $r\in(0,r_0)$ such that $d_\ell(e,A_r(e))<\frac{1}{2p}$ for all $e\in\E$. Then choose $n\geq n_1$ such that $d_\ell(A_r(e),C_{r,n}(e))<\frac{1}{2p}$ for all $e$. Set $S^0_p(e)=C_{r,n}(e)$. Then for each $e\in \E$, the sequence $(S^0_p(e))_{p\geq 1}$ converges to $e$ with fixed endpoints, and satisfies $\ell(S^0_p(e))\leq \ell(e)$. For each $e\in\E$ and each $p\geq 1$, set $S^1_p(e)=\LE(S_p^0(e))$. We claim that a subsequence of the sequence $(S^1_p(e))_{p\geq 1}$ tends to $e$ when $p$ tends to infinity. 

Indeed the sequence $(S^1_p(e))_{p\geq 1}$ is uniformly bounded in length by $\ell(e)$. Hence, the paths being parametrized at constant speed, it is relatively compact in the uniform topology and we can extract a sequence $(S^1_{p_q}(e))_{q\geq 1}$ which converges uniformly to a path $\tilde e$. The image of $\tilde e$ is contained in the image of $e$ and it joins $\underline{e}$ to $\overline{e}$. Hence, the images of $\tilde e$ and $e$ coincide. In particular, $\ell(\tilde e)\geq \ell(e)$. Using the lower semi-continuity of the length with respect to the uniform convergence, we find
$$\ell(e)\leq \ell(\tilde e)\leq \liminf \ell(S^1_{p_q}(e))\leq \sup\ell(S^1_{p_q}(e))\leq \ell(e).$$
It follows from these inequalities that $\ell(S^1_{p_q}(e))$ converges to $\ell(e)$, hence $S^1_{p_q}(e)$ to $e$, as $q$ tends to infinity. Let us choose a subsequence $(p_q)_{q\geq 0}$ such that the convergence holds for each edge $e\in\E$ and define $S_q(e)=S^1_{p_q}(e)$ for all $q\geq 0$. 

For all $e\in\E$ and all $q\geq 0$, $S_{q}(e)$ is a piecewise geodesic edge with the same endpoints as $e$. Moreover, by construction, $S_{q}(e^{-1})=S_{q}(e)^{-1}$. If $e$ is geodesic, then $S_q(e)=e$. Finally, the set $\E_q=\{S_{q}(e):e\in\E\}$ is the set of edges of a pre-graph on $(M,\CS)$. By Corollary \ref{c: limit pregraph}, $\E_q$ is in fact the set of edges of a graph on $(M,\CS)$, which we denote by $\G_q$.

By construction, the bijection $S_q$ between $\E$ and $\E_q$ preserved the cyclic order at every vertex of $\V$. Let us apply Proposition \ref{aire faces}. Since $\vol(\Sk(\G))=0$, the graph $\G_q$ satisfies all the desired properties for $q$ large enough. \end{proof}

\chapter{Multiplicative processes indexed by paths}

In this chapter, we begin the study of the class of objects to which Markovian holonomy fields belong: stochastic processes indexed by paths on a compact surface and with values in a compact Lie group which satisfy a condition of multiplicativity. We discuss the canonical space of such processes, two distinct $\sigma$-fields on it, and prove the version of Kolmogorov's theorem that is best suited to our situation. We then study a kind of uniform measure on the canonical space, thus providing the first example of what will be called in the next chapter a discrete Markovian holonomy field. We conclude by constructing a set of generators of the group of reduced loops in a graph for which we are able to determine the finite-dimensional marginal of the uniform Markovian holonomy field. 

From now on, the expression {\em compact surface} will mean {\em smooth compact surface}. 

\section{Multiplicative functions}
\label{s: mult func}
Let $G$ be a group, on which we make for the moment no assumption at all.

\begin{definition}\label{def:multiplicative} Let $M$ be a compact surface. Let $P$ be a subset of
$\Path(M)$. A function $h:P\to G$ is said to be {\em
multiplicative} if $h(c^{-1})=h(c)^{-1}$ for all $c\in P$ such that $c,c^{-1} \in P$
and $h(c_{1}c_{2})=h(c_{2})h(c_{1})$
for all $c_{1},c_{2} \in P$ such that $\overline{c_{1}}=\underline{c_{2}}$ and $c_{1}c_{2}\in P$. The
set of multiplicative functions from $P$ to $G$ is denoted by $\M(P,G)$.
\end{definition}
\index{MAMul@$\M(\Path(M),G)$}
\index{multiplicative function}

For an explanation of the reversed order in this definition, see the introduction, Section \ref{ordre inverse}.

The gauge group is the symmetry group of the physical theory from which our objects are issued. 

\begin{definition}\label{def:gauge group} The group $G^{M}$ of all mappings from $M$ to $G$ is called
the {\em gauge group} of $M$ and it acts by {\em
gauge transformations} on the space $\M(\Path(M),G)$, as follows. If $j=(j_{m})_{m\in M}$ belongs
to $G^M$ and $h$ belongs to $\M(\Path(M),G)$, then $j\cdot h$ is defined by
$$ \forall c\in \Path(M), (j\cdot h)(c)=j_{\overline{c}}^{-1} h(c) j_{\underline{c}}.$$
More generally, given a subset $P$ of $\Path(M)$, the group $G^V$ acts on $\M(P,G)$ in the same
way, where $V$ is defined as the set of endpoints of the paths of $P$.
\end{definition}
\index{gauge group}

\begin{example} Assume that $l_1,\ldots,l_n$ are $n$ simple loops based at the same point $m$ of $M$. Write $L=\{l_1,\ldots,l_n\}$. The concatenation of two loops of $L$ is never a simple loop, hence never an element of $L$. If we assume moreover that the inverse of a loop of $L$ is never in $L$, then any $G$-valued function on $L$ is multiplicative, so that $\M(L,G)=G^L$.
The action of the gauge group on $\M(L,G)$ is simply the action of $G=G^{\{m\}}$ on $G^L$ by simultaneous conjugation of each factor. This fundamental example should be kept in mind when one reads Lemma \ref{chemins lacets}. It explains the importance of this action of $G$ in our context.
\end{example}

We cannot do much if we do not make a few assumptions on $G$. For the rest of this section, we assume that it is a compact topological group. We do not assume that it is connected, so that it can in particular be finite. Also, for the moment, we do not assume that it is a Lie group. The group $G$ carries its normalized Haar measure and we denote simply by $\int_G f(x) \; dx$ the integral of a function $f:G\to \RK$ with respect to this measure. 
\index{GAG@$G$}

Let $P$ be a subset of $\Path(M)$. The object of this paragraph is to discuss two natural $\sigma$-fields on $\M(P,G)$. The simplest one is the cylinder $\sigma$-field, denoted by $\C$, which is defined as the smallest $\sigma$-field which makes the evaluation mapping $h\mapsto h(c)$ measurable for all $c\in P$. The gauge group acts by bi-measurable transformations and it makes sense to speak of measures on $\M(P,G)$ which are invariant under gauge transformations.
\index{CCC@$\C$}

It is also natural to consider a smaller $\sigma$-field which consists in events which are invariant under the action of the gauge group. In order to discuss this $\sigma$-field, let us first associate an abstract graph to each subset of $\Path(M)$.

\begin{definition} An {\em abstract graph} is a pair of finite sets $(V,E)$, whose elements are called vertices and edges, endowed with two mappings $s,t:E\to V$, called respectively  {\em source} and {\em target}.

Let $P$ be a subset of $\Path(M)$. The {\em configuration graph} of $P$ is the abstract graph $(V,E,s,t)$ defined by setting $V= \bigcup_{c\in P} \{\underline{c}\} \cup \{\overline{c}\}\subset M$,  $E=P$ and, for each $e\in E$, $s(e)=\underline{e}$ and $t(e)=\overline{e}$.
\end{definition}
\index{graph!abstract}
\index{abstract graph}

In what follows, we will make use of the diagonal adjoint action of $G$ on $G^n$ defined by $g\cdot(x_1,\ldots,x_n)=(gx_1g^{-1},\ldots,gx_ng^{-1})$.
\index{conjugation!diagonal action by}

\begin{lemma}\label{chemins lacets} Let $M$ be a compact surface. Let $P$ be a subset of $\Path(M)$. Assume that $P$ is stable by concatenation and inversion, and the configuration graph of $P$ is connected. Let $m$ be a vertex of this configuration graph. Let $c_1,\ldots,c_n$ be $n$ elements of $P$. Let $f:G^{n}\to \RK$ be a continuous function such that the function $h\mapsto f(h(c_1),\ldots,h(c_n))$ is gauge-invariant on $\M (P,G)$. Then there exists $n$ loops $l_1,\ldots,l_n$ in $P$ all based at $m$ and a continuous function $\tilde f : G^n \to \RK$ invariant under the diagonal action of $G$ such that 
$$\forall h\in \M(P,G), \; f(h(c_1),\ldots,h(c_n))=\tilde f(h(l_1),\ldots,h(l_n)).$$

\end{lemma}

\begin{proof}Consider the configuration graph of $\{c_1,\ldots,c_n\}$. If it is not connected, let us choose one vertex in each connected component not containing $m$ and add to the collection $\{c_1,\ldots,c_n\}$ a path of $P$ joining $m$ to this vertex. That such a path exists follows from the assumptions made on $P$. The collection has become $\{c_1,\ldots,c_r\}$ for some $r\geq n$. We denote the configuration graph of this enlarged collection by $(V,E)$. 

Let $T\subset E$ be a spanning tree of $(V,E)$. For each $i\in \{1,\ldots,n\}$, the path $[m,\underline{c_i}]_T c_i [\overline{c_i},m]_T$ is a loop based at $m$. When written as a product of edges, it becomes a word in $c_1,\ldots,c_r$ and this word makes sense as an element of $P$. More precisely, it is a loop of $P$ based at $m$, which we denote by $l_i$. 

Let $h$ be an element of $\M(P,G)$. Let us define a gauge transformation $j\in G^M$ by setting $j(p)=1$ if $p\notin V$ and, for all $v\in V$, $j(v)=h([m,v]_T)$. If $e$ is an edge which belongs to $T$, then it is easy to check that $(j\cdot h)(e)=1$. Hence, for all $i\in \{1,\ldots,n\}$, $(j\cdot h)(l_i)=(j\cdot h)(c_i)$ and, by the invariance property of $f$, $f(h(c_1),\ldots,h(c_n))=f(h(l_1),\ldots,h(l_n))$. 

Choose $g\in G$. Since the loops $l_1,\ldots,l_n$ are all based at $m$, the action of the gauge transformation $j$ defined by $j(p)=1$ if $p\neq m$ and $j(m)=g^{-1}$ transforms $h(l_1),\ldots,h(l_n)$ into $gh(l_1)g^{-1},\ldots,gh(l_n)g^{-1}$. Let us define $\tilde f:G^n\to \RK$ by
$$\tilde f(x_1,\ldots,x_n)=\int_G f(gx_1 g^{-1},\ldots,gx_n g^{-1}) \; dg.$$
Then $f(h(l_1),\ldots,h(l_n))=\tilde f(h(l_1),\ldots,h(l_n))$ and $\tilde f$ is invariant under the diagonal action of $G$. The result is proved. \end{proof}

This result motivates the following definition. 

\begin{definition}\label{def inv sigma field} Let $P$ be a subset of $\Path(M)$ stable by concatenation and inversion. The {\em invariant $\sigma$-field} on $\M(P,G)$, denoted by $\I$, is the smallest $\sigma$-field such that for all $m\in M$, all integer $n\geq 1$, all finite collection $l_1,\ldots,l_n$ of loops based at $m$ and all continuous function $f:G^n\to \RK$ invariant under the diagonal action of $G$, the mapping $h\mapsto f(h(l_1),\ldots,h(l_n))$ is measurable.
\end{definition}
\index{IAI@$\I$}
\index{multiplicative function!invariant $\sigma$-field}

If $P$ can be written as the disjoint union of $P_1$ and $P_2$ which are both stable by concatenation and inversion and whose configuration graphs are disjoint, then $\M(P,G)$ is canonically isomorphic to $\M(P_1,G)\times \M(P_2,G)$ and the invariant $\sigma$-filed on $\M(P,G)$ is the tensor product of the invariant $\sigma$-fields on $\M(P_1,G)$ and $\M(P_2,G)$.


Let $M$ and $M'$ be two surfaces. Let $\psi : M'\to M$ be a smooth mapping. Let $P'$ be a subset of $\Path(M')$ and $P=\psi(P')$. Then $\psi$ induces a map from $P'$ to $P$, hence a map from $\M(P,G)$ to $\M(P',G)$.

\begin{lemma}\label{glu mes} Let $M$ and $M'$ be two surfaces. Let $\psi : M'\to M$ be a smooth mapping. Let $P'$ be a subset of $\Path(M')$ and $P=\psi (P')$. Then the induced map $\psi : \M(P,G) \to \M(P',G)$ is measurable with respect to the cylinder $\sigma$-fields, and also with respect to the invariant $\sigma$-fields.
\end{lemma}

\begin{proof}Let $f:G\to \RK$ be continuous and consider $c'\in P'$. The function $h\mapsto f(\psi(h)(c'))$ on $\M(P,G)$ is equal to the function $h\mapsto f(h(\psi(c')))$, which is measurable with respect to $\C$ because $\psi(c')$ belongs to $P$. This proves the first assertion. 

Now let $f:G^n\to\RK$ be continuous and invariant under the diagonal action of $G$ by conjugation. Let $l'_1,\ldots,l'_n$ be $n$ loops of $P'$ based at the same point. Then $\psi(l_1),\ldots,\psi(l_n)$ are $n$ loops of $P$ based at the same point and the function $h\mapsto f(\psi(h)(l'_1),\ldots,\psi(h)(l'_n))=f(h(\psi(l_1),\ldots,\psi(l_n)))$     
on $\M(P,G)$ is measurable with respect to $\I$. This proves the second assertion.\end{proof}

Let us conclude this paragraph by discussing the case of a graph. Let $M$ be a surface endowed with a graph $\G$. Let $\E^+$ be an orientation of $\G$. It is plain that a multiplicative function on $\Path(\G)$ is determined by its values on the edges $\E$ or even just those of $\E^+$. More precisely, the natural surjective mapping $\M(\Path(\G),G)\to \M(\E^+,G)$ induced by the inclusion $\E^+\subset \Path(\G)$ is one-to-one. Indeed, if $c$ belongs to $\Path(\G)$, then $c=e_1^{\epsilon_1}\ldots e_n^{\epsilon_n}$ for some $e_1,\ldots,e_n \in\E^+$ and $\epsilon_1,\ldots,\epsilon_n\in\{-1,1\}$. Then, for all multiplicative function $h$, one has $h(c)=h(e_n)^{\epsilon_n}\ldots h(e_1)^{\epsilon_1}$. 

Since the interior of an edge contains no vertex, the concatenation of at least two edges is never an edge. Hence, every mapping from $\E^+$ to $G$ is multiplicative. We will often make the identifications
$$\M(\Path(\G),G)=\M(\E,G)=\M(\E^+,G)=G^{\E^+}$$
without further comment. In particular, we will sometimes use a collection $(g_e)_{e\in\E^+}$ of elements of $G$ to denote an element of $\M(\Path(\G),G)$. 

\begin{proposition}\label{gen invariant sigma} Let $M$ be a compact surface endowed with a graph $\G$. Let $v$ be a vertex of $\G$. Let $\{l_1,\ldots,l_r\}$ be a generating subset of the group $\RL_v(\G)$ of reduced loops in $\G$ based at $v$. The invariant $\sigma$-field $\I$ on $\M(\Path(\G),G)$ is generated by the functions of the form $f(h(l_1),\ldots,h(l_r))$ where $f:G^r\to \RK$ is continuous and invariant by diagonal conjugation.
\end{proposition}

\begin{proof}By definition of the invariant $\sigma$-field and by Lemma \ref{chemins lacets}, it suffices to prove that for all $l'_1,\ldots,l'_n\in \Loop_v(\G)$ and all continuous $f':G^n\to \RK$ invariant by diagonal conjugation, the function $f'(h(l'_1),\ldots,h(l'_n))$ can be put under the form $f(h(l_1),\ldots,h(l_r))$ for some invariant function $f$. This is easily done by expressing, modulo the equivalence relation on paths, the loops $l'_1,\ldots,l'_n$ as words in the generators $l_1, \ldots,l_r$. One then uses the multiplicativity of the elements of $\M(\Path(\G),G)$ and the fact that the multiplication map $G^2\to G$ is equivariant with respect to the diagonal actions of $G$ by conjugation. \end{proof}

\section{Multiplicative families of random variables}

Let $M$ be a compact surface and $P$ a subset of $\Path(M)$. A probability measure on $(\M(P,G),\C)$ determines a family of $G$-valued random variables $(H_c)_{c\in P}$ which are just the evaluation functions on $\M(P,G)$, defined by $H_c(h)=h(c)$. These random variables form a multiplicative family in the sense that $H_{c^{-1}}=H_c^{-1}$ and $H_{c_1 c_2}=H_{c_2} H_{c_1}$ almost surely whenever this makes sense. If $P$ is countable, then the converse is true since one can dismiss the negligible event on which the equalities do not hold. We prove in this section that the converse is in fact true even if $P$ is not countable. Let us recall the definition of a projective family of probability spaces.

\begin{definition}\label{proj fam} A projective family of probability spaces is the data of the following ingredients. \\
$\bullet$ A partially ordered set $(\Lambda,\preccurlyeq)$.\\
$\bullet$ For each $\lambda\in\Lambda$, a probability space $(\Omega_\lambda,\C_\lambda,m_\lambda)$.\\
$\bullet$ For each pair $(\lambda,\mu)\in\Lambda^2$ such that $\lambda\preccurlyeq\mu$, a measurable mapping $\rho_{\lambda \mu}:\Omega_\mu \to \Omega_\lambda$.\\
These ingredients are assumed to satisfy the following conditions.\\
1. The poset $(\Lambda,\preccurlyeq)$ is directed : for all $\lambda,\mu\in\Lambda$, there exists $\nu\in\Lambda$ such that $\lambda\preccurlyeq \nu $ and $\mu\preccurlyeq\nu$. \\
2. For all $\lambda,\mu,\nu\in\Lambda$ such that $\lambda\preccurlyeq\mu\preccurlyeq\nu$, one has the equality $\rho_{\lambda\mu}\circ\rho_{\mu \nu}=\rho_{\lambda\nu}$.\\
3. For all $\lambda,\mu\in\Lambda$ such that $\lambda\preccurlyeq\mu$, one has $m_\mu \circ \rho_{\lambda\mu}^{-1}=m_\lambda$.
\end{definition}
\index{projective family of probability spaces}

Let us state a general result of existence and uniqueness.

\begin{proposition}\label{proj lim gen} We keep the notation of Definition \ref{proj fam}. Let $\Omega$ denote the set-theoretic projective limit of the family $(\{\Omega_\lambda\},\{\rho_{\lambda\mu}\})$, endowed for each $\lambda\in\Lambda$ with the canonical mapping $\rho_\lambda : \Omega\to\Omega_\lambda$. Let $\C$ denote the smallest $\sigma$-field on $\Omega$ which makes all the mappings $\rho_\lambda$ measurable. 

Assume that for each $\lambda\in\Lambda$, $(\Omega_\lambda,\C_\lambda)$ is a compact metric space endowed with the Borel $\sigma$-field. Assume also that for all $\lambda,\mu\in\Lambda$ such that $\lambda\preccurlyeq\mu$, the mapping $\rho_{\lambda\mu}$ is continuous.

Then there exists a unique probability measure $m$ on $\C$ such that for all $\lambda\in\Lambda$, one has $m \circ \rho_\lambda^{-1}=m_\lambda$.
\end{proposition}
\index{Kolmogorov's extension theorem}
 
This result belongs to a wide family of theorems whose prototype is due to Kolmogorov and whose common ground is Caratheodory's extension theorem. The most common versions assert the existence and uniqueness of $m$ under less restrictive conditions on the probability spaces but more restrictive conditions on the poset. Typically, $\Lambda$ is a subset of a finite-dimensional Euclidean space and the probability spaces are Polish space. The form which we have stated is in fact fairly easy to prove due to the strong assumptions which we make on the probability spaces. We think wiser to give a proof than to refer to several places in the literature from which the reader would have to collect the various pieces of the argument. \\

\begin{proof}Let $\Pi_\Lambda$ denote the Cartesian product of the sets $\Omega_\lambda$, $\lambda\in\Lambda$. Recall that the set-theoretic projective limit of the family $(\{\Omega_\lambda\},\{\rho_{\lambda\mu}\})$ is, by definition, the set
$$\Omega=\left\{(\omega_\lambda)_{\lambda\in\Lambda} \in\Pi_\Lambda : \forall \lambda,\mu\in\Lambda, \lambda\preccurlyeq\mu \Rightarrow \rho_{\lambda\mu}(\omega_\mu)=\omega_\lambda\right\}.$$
The set $\Pi_\Lambda$, endowed with the product topology, is a compact topological space of which $\Omega$, as intersection of closed sets, is a compact topological subspace. It is endowed with the continuous coordinate mappings $\rho_\lambda :\Omega\to\Omega_\lambda$, $\lambda\in\Lambda$.  

For each $\lambda\in\Lambda$, let $S_\lambda$ denote the closed support of the probability measure $m_\lambda$. It is a non-empty compact subset of $\Omega_\lambda$. Consider $\lambda\preccurlyeq \mu$. The equality $m_\mu\circ \rho_{\lambda\mu}^{-1}=m_\lambda$ implies that $\rho_{\lambda\mu}(S_\mu)=S_\lambda$. We claim that for all $\xi\in\Lambda$, $\rho_\xi(\Omega)\supset S_\xi$. Indeed, choose $\xi\in\Lambda$  and $s_\xi\in S_\xi$. Define, for all $\mu\preccurlyeq \nu$,
$\Omega(\mu,\nu;\xi)=\left\{(\omega_\lambda)_{\lambda\in\Lambda} \in\Pi_\Lambda : \rho_{\mu\nu}(\omega_\nu)=\omega_\mu , \omega_\xi=s_\xi \right\}$. Then on one hand,  $\bigcap_{\mu\preccurlyeq \nu} \Omega(\mu,\nu;\xi) = \rho_\xi^{-1}(s_\xi)$. On the other hand, the discussion above implies that no finite intersection of the sets $\Omega(\mu,\nu;\xi)$ is empty. Since these sets are compact, their intersection is non-empty. Hence $s_\xi\in \rho_\xi(\Omega)$.

Define a collection $\C_\Lambda$ of subsets of $\Omega$ by setting $\C_\Lambda=\bigcup_{\lambda\in\Lambda} \rho_\lambda^{-1}(\C_\lambda)$. The collection $\C_\Lambda$ is not a $\sigma$-field, but it is stable by complementation, finite unions and finite intersections. Consider $A\in\C_\Lambda$. Assume that $A=\rho_\lambda^{-1}(A_\lambda)$ and $A=\rho_\mu^{-1}(A_\mu)$ for some $\lambda,\mu\in\Lambda$ and $A_\lambda\in\C_\lambda$, $A_\mu\in\C_\mu$. Let $\nu$ be such that $\lambda\preccurlyeq \nu$ and $\mu\preccurlyeq \nu$. Then 
$A=\rho_\nu^{-1}(\rho_{\lambda\nu}^{-1}(A_\lambda))=\rho_\nu^{-1}(\rho_{\mu\nu}^{-1}(A_\mu))$. Since $\rho_\nu(\Omega)\supset S_\nu$, this equality implies $\rho_{\lambda\nu}^{-1}(A_\lambda) \cap S_\nu=\rho_{\mu\nu}^{-1}(A_\mu)\cap S_\nu$. Hence,
\begin{eqnarray*}
m_\lambda(A_\lambda)=m_\nu(\rho_{\lambda\nu}^{-1}(A_\lambda))=m_\nu(\rho_{\lambda\nu}^{-1}(A_\lambda) \cap S_\nu)=&&\\
&&\hskip -2.5cm m_\nu(\rho_{\mu\nu}^{-1}(A_\mu)\cap S_\nu)=m_\nu(\rho_{\mu\nu}^{-1}(A_\mu))=m_\mu(A_\mu).
\end{eqnarray*}
We have proved that $\rho_\lambda^{-1}(A_\lambda)=\rho_\mu^{-1}(A_\mu)$ implies $m_\lambda(A_\lambda)=m_\mu(A_\mu)$. Hence, for all $A\in\C_\Lambda$, it is legitimate to call $m(A)$ the common value of all $m_\lambda(A_\lambda)$ for $\lambda\in\Lambda$ and $A_\lambda\in\C_\lambda$ such that $A=\rho_\lambda^{-1}(A_\lambda)$. Thus, we have defined $m:\C_\Lambda\to[0,1]$. It is not difficult to check that $m$ is finitely additive.

We claim that $m$ is $\sigma$-additive on $\C_\Lambda$. In order to prove this, it is sufficient to prove that if $(A_n)_{n\geq 1}$ is a decreasing sequence of elements of $\C_\Lambda$ such that $\bigcap_{n\geq 1}A_n=\varnothing$, then $\lim m(A_n)=0$. Let us choose such a sequence $(A_n)_{n\geq 1}$. Let us also choose $\epsilon>0$. For each $n$, let us write $A_n=\rho_{\lambda_n}^{-1}(A_{\lambda_n})$ for some $\lambda_n\in \Lambda$ and some $A_{\lambda_n}\in \C_{\lambda_n}$. For each $n$, the inner regularity of the measure $m_{\lambda_n}$ implies the existence of a compact subset $K_{\lambda_n}$ of $A_{\lambda_n}$ such that $m_{\lambda_n}(A_{\lambda_n}-K_{\lambda_n})\leq \epsilon 2^{-n}$. Since $\Omega$ endowed with the trace of the product topology of $\Pi_\Lambda$ is compact, the sets $K_n=\rho_{\lambda_n}^{-1}(K_{\lambda_n})$, $n\geq 0$ are compact. Since $K_n\subset A_n$ for all $n\geq 1$, the intersection $\bigcap_{n\geq 1} K_n$ is empty. Hence, there exists $N\geq 1$ such that $\bigcap_{n=1}^N K_n=\varnothing$. Now, for each $\omega\in A_N$, there exists at least one $n\in\{1,\ldots,N\}$ such that $\omega\notin K_n$, so that $\omega\in A_n-K_n$. Hence, $m(A_N)\leq \sum_{n=1}^N m(A_n-K_n) \leq \epsilon$.  This proves that $m(A_n)$ tends to 0 when $n$ tends to infinity. 

Carath\'eodory's extension theorem asserts that $m$ admits a unique $\sigma$-additive extension to the $\sigma$-field generated by $\C_\Lambda$, which is by definition the smallest $\sigma$-field on $\Omega$ such that the mappings $\rho_\lambda$ are measurable. This extension of $m$, which we still denote by $m$, satisfies $m\circ\rho_\lambda^{-1}=m_\lambda$ for all $\lambda\in\Lambda$ by definition. \end{proof}

In our setting, this theorem can be applied as follows.

\begin{proposition} \label{take proj lim}Let $P$ be a subset of $\Path(M)$. Let $\mathcal F$ be a collection of finite subsets of $P$ whose union is $P$ and which, when ordered by the inclusion, is directed. For all $J\in {\mathcal F}$, let $m_J$ be a probability measure on $(\M(J,G),\C)$. Assume that the probability spaces $(\M(J,G),\C,m_J)$ endowed with the restriction mappings $\rho_{JK}:\M(K,G)\to\M(J,G)$ defined for $J\subset K$ form a projective system. Then there exists a unique probability measure $m$ on $(\M(P,G),\C)$ such that for all $J\subset P$, the image of $m$ by the restriction mapping $\rho_J:\M(P,G)\to \M(J,G)$ is $m_J$.
 
In particular, if $(H_c)_{c\in P}$ is a collection of $G$-valued random variables such that \\
1. $\forall c\in P, c^{-1}\in P \Rightarrow H_{c^{-1}}=H_c^{-1} \; a.s.$,\\
2. $\forall c_1,c_2\in P, c_1c_2\in P \Rightarrow H_{c_1c_2}=H_{c_2}H_{c_1} \; a.s. $,\\
then there exists a unique probability measure $m$ on $(\M(P,G),\C)$ such that the distribution of the canonical process under $m$ is the same as the distribution of $(H_c)_{c\in P}$.
\end{proposition}

\begin{proof}Let $J,K$ be elements of $\mathcal F$ such that $J\subset K$. Then $\M(J,G)$ is a compact subset of $G^J$, actually a smooth compact submanifold. The evaluation mappings on $\M(J,G)$ generate both the topology and the $\sigma$-field $\C$. Hence, $\C$ is the Borel $\sigma$-field. Moreover, the restriction mapping $\M(K,G)\to \M(J,G)$ is continuous. 

Proposition \ref{proj lim gen} ensures the existence of a probability measure on the projective limit of the sets underlying our probability spaces, endowed with a certain $\sigma$-field. In the present case, the projective limit of the sets $\M(J,G)$ is easily identified with the set $\M(P,G)$ in such a way that the mappings $\rho_J:\M(P,G)\to\M(J,G)$ are simply the restrictions. Through this identification, the $\sigma$-field on which Proposition \ref{proj lim gen} constructs a measure is the usual cylinder $\sigma$-field. The first assertion follows.

Let $(H_c)_{c\in P}$ be a collection of random variables which satisfies the assumptions 1 and 2. For each finite subset $J$ of $P$, the distribution of $(H_c)_{c\in J}$ is a Borel probability measure on $G^J$, which we denote by $m_J$ and which is actually supported by $\M(J,G)$. By applying the first assertion to the collection of probability spaces $(\M(J,G),\C,m_J)$ where $J$ spans the collection of finite subsets of $P$, we find the desired probability measure $m$ on $\M(P,G)$. \end{proof}

\section{Uniform multiplicative functions on a graph}

Let $M$ be a compact surface. Let $\G=(\V,\E,\F)$ be a graph on $M$. In this paragraph, we discuss the uniform measure on $\M(\Path(\G),G)$ and some of its natural disintegrations. The disintegrations that we have in mind are associated with  random variables associated to marking curves or boundary components. Before we define the measures, let us set up some notation.

\subsection{Constraints on marked surfaces}

Let us denote by $\Conj(G)$ the set of conjugacy classes of $G$. The inversion map $x\mapsto x^{-1}$ on $G$ descends to an involution of ${\rm Conj}(G)$ which determines an action of $\Z/2\Z$ on ${\rm Conj}(G)$. 
\index{OOO@${\mathcal O}_x$}\index{CCC@$C$}
Recall that if $(M,\CS)$ is a marked surface, then $\Z/2\Z$ acts on $\CS\cup\BS(M)$ by reversing the orientation.

\begin{definition} Let $(M,\CS)$ be a marked surface. A {\em set of $G$-constraints} on $(M,\CS)$ is a $\Z/2\Z$-equivariant mapping $C:\CS\cup \BS(M)\to \Conj(G)$. 
\end{definition}
\index{surface!$G$-constraints}
\index{gc@$G$-constraints|see{surface}}
\index{constraints|see{surface}}

A set of $G$-constraints on a marked surface determines a set of $G$-constraints on any splitting of this surface. In the case of unary splittings, we need the following observation. If $\O$ is a conjugacy class of $G$, then the set of squares of elements of $\O$ is also a conjugacy class, which we denote by $\O^2$.

\begin{definition}\label{split C} Let $(M,\CS,C)$ be a marked surface endowed with a set of $G$-constraints. Consider $l\in\CS$. Let $f:\Spl_l(M)\to M$ be the elementary gluing with joint $\{l,l^{-1}\}$. The marked surface $(\Spl_l(M),\Spl_l(\CS))$ carries the set of $G$-constraints $\Spl_l(C)$ defined by $\Spl_l(C)(l')=C(f(l'))$ for all $l'\in \Spl_l(\CS)$, with the following exception: if $f$ is a unary gluing and $f(l')=l^{\pm 1}$, then $\Spl_l(C)(l')=C(l)^{\pm 2}$.
\end{definition}
\index{surface!surgery}

Consider a marked surface $(M,\CS,C)$ with $G$-constraints. Given $h\in\M(\Path(M),G)$, we say that $h$ {\em satisfies the constraints $C$} if 
\begin{equation}\label{h satisfies C}
\forall l\in \CS \cup \BS(M) , h(l)\in C(l).
\end{equation}

The set of elements of $\M(\Path(M),G)$ which satisfy the constraints $C$ is globally invariant under the action of the gauge group.

\subsection{Uniform measures}

\begin{definition}\label{def U free} Let $M$ be a compact surface. Let $\G=(\V,\E,\F)$ be a graph on $M$. Let $\E^+$ be an orientation of $\G$. The Haar measure on $G^{\E^+}$, seen as a probability measure on $(\M(\Path(\G),G),\C)$, is called the {\em uniform measure} and denoted by $\U^{\G}_{M,\varnothing}$.
\end{definition}
\index{multiplicative function!uniform}

Plainly, the uniform measure does not depend on the choice of $\E^+$. The reason for the subscript $\varnothing$ will become apparent soon. We would like now to incorporate boundary conditions and constraints along marking curves into the uniform measure $\U^{\G}_{M,\varnothing}$.

Let $\O\subset G$ be a conjugacy class. Let $n\geq 1$ be an integer.
The set 
$$\O(n)=\{(x_{1},\ldots,x_{n})\in G^{n} : x_{1}\ldots x_{n}\in \O\}$$
 is a $G^{n}$-homogeneous
space under the action 
$$(g_{1},\ldots,g_{n})\cdot (x_{1},\ldots,x_{n})=(g_{1}x_{1}g_{2}^{-1},\ldots,g_{n}x_{n}g_{1}^{-1}).$$
Let $\delta_{\O(n)}$
denote the extension to $G^n$ of the unique $G^n$-invariant probability measure on $\O(n)\subset G^{n}$.
In particular, $\O(1)=\O$ and $\delta_{\O(1)}$ is the $G$-invariant probability measure on $\O$,
which we also denote simply by $\delta_{\O}$.  For each element $x$ of $G$, we denote
by $\O_{x}$ the conjugacy class of $x$.
\index{DDdO@$\delta_{\O(n)}$}

\begin{lemma} Consider $f\in C^0(G^{n})$, $g\in C^0(G^{n-1})$, $u\in C^0(G)$ and a continuous map $h:G^{n-1}\to G$. Let $\O$
be a conjugacy class of $G$. The following relations hold. 
\begin{equation}\label{put constraint}
\int_{G^n} f\; d\delta_{\O(n)}=\int_{G^{n}}
f(x_{1},\ldots,x_{n-1},(x_{1}\ldots x_{n-1})^{-1}y)\; dx_{1}\ldots dx_{n-1} \delta_{\O}(dy).
\end{equation}
\begin{eqnarray} \nonumber
\int_{G^n} u(h(x_{1},\ldots,x_{n-1}) x_{1}\ldots x_{n} h(x_{1},\ldots,x_{n-1})^{-1})
g(x_{1},\ldots,x_{n-1}) \; \delta_{\O(n)}(dx_{1}\ldots dx_{n}) =&& \\
\label{put constraint 2} &&\hskip -8cm  \int_{G} u\; d\delta_{\O} \int_{G^{n-1}} g(x_{1},\ldots,x_{n-1})\; dx_{1}\ldots
dx_{n-1}.
\end{eqnarray}
\begin{equation}\label{contracte dO}
\int_{G^{n+1}} f(x_1x_2,x_3,\ldots,x_n) \; \delta_{\O(n+1)}(dg_1 \ldots dg_{n+1})=\int_{G^n} f \;\delta_{\O(n)}.
\end{equation}

\begin{equation}\label{inv perm circ}
\int_{G^n} f(x_{2},\ldots,x_{n},x_{1})\; \delta_{\O(n)}(dx_{1}\ldots dx_{n}) = 
\int_{G^n} f\; \delta_{\O(n)}.
\end{equation}
\begin{equation}\label{recover unif}
\int_{G} \left[\int_{G^{n}}f(x_{1},\ldots,x_{n}) \; \delta_{\O_{y}(n)}(dx_{1}\ldots dx_{n})\right]
\; dy=\int_{G^{n}} f(x_{1},\ldots,x_{n})\; dx_{1}\ldots dx_{n}.
\end{equation}
\end{lemma}

\begin{proof}The right-hand side of (\ref{put constraint}) defines a measure supported by $\O(n)$ and
invariant under the action of $G^{n}$. Hence, it is $\delta_{\O(n)}$. The relation
(\ref{put constraint 2}) follows easily from (\ref{put constraint}). 
The right-hand sides of (\ref{contracte dO}) and (\ref{inv perm circ}) are equal to 
$$\int_{G^{n}}
f(g_{1}z_{1}g_{2}^{-1},\ldots,g_{n}z_{n}g_{1}^{-1})\; dg_{1}\ldots dg_{n}$$ for all
$(z_{1},\ldots,z_{n})\in \O(n)$. The relation (\ref{contracte dO}) follows because the map $(x_1,\ldots,x_{n+1})\mapsto (x_1 x_2,x_3,\ldots,x_{n+1})$ maps $\O(n+1)$ into $\O(n)$, and (\ref{inv perm circ}) also follows because the set $\O(n)$
is stable by circular permutation of the variables. The relation (\ref{recover unif}) follows also from
this description of $\delta_{\O(n)}$ and a simple change of variables. \end{proof}

Let $(M,\CS,C)$ be a marked surface with a set of $G$-constraints, endowed with a graph $\G$. Let
us choose $q$ simple loops $l_{1},\ldots,l_{q}$ in $\Loop(\G)$ which represent the
unoriented cycles of $\CS\cup\BS(M)$, that is, such that
$\CS\cup\BS(M)=\{l_{1},l_{1}^{-1},\ldots,l_{q},l_{q}^{-1}\}$. Recall that for the sake of simplicity,
we denote in the same way a cycle and
the corresponding simple loop. We label the elements of $\E$ in such a way that $l_{i}=e_{i,1}\ldots
e_{i,n_{i}}$ for $i\in \{1,\ldots,q\}$. Let $\E^+$ be an orientation of $\G$ such that $e_{i,j}\in
\E^{+}$ for all $i\in\{1,\ldots,q\}$
and $j\in\{1,\ldots,n_{i}\}$. Let us label $e_{1},\ldots,e_{m}$ the other edges of $\E^+$. 

\begin{lemma}\label{def ugc} Let $(M,\CS,C)$ be a marked surface with $G$-constraints. Let $\G=(\V,\E,\F)$ be a graph on $(M,\CS)$. The {\em uniform measure on $\G^{\E^+}$ with $G$-constraints $C$} is defined as the probability measure
$$dg_{1}\otimes \ldots \otimes dg_{m}\otimes
\delta_{C(l_{1})(n_{1})}(dg_{1,n_{1}}\ldots
dg_{1,1})\otimes \ldots\otimes \delta_{C(l_{q})(n_{q})}(dg_{q,n_{q}}\ldots dg_{q,1}).$$
It is denoted by $\U^{\G}_{M,\CS,C}(dg)$.
The corresponding probability measure on $\M(\Path(\G),G)$, also denoted by $\U^{\G}_{M,\CS, C}$, depends neither on the choice of the simple loops which represent the marking of $M$ nor on the choice of $\E^{+}$. 
\end{lemma}
\index{UAUGM@$\U^\G_{M,\CS,C}$}
\index{multiplicative function!uniform with constraints|(}

\begin{proof}The invariance of the measures $\delta_{\O(n)}$ by cyclic permutation of the variables, granted
by (\ref{inv perm circ}), ensures that the measure $\U^{\G}_{M,\CS,C}$ does not depend on the
simple loops which we have chosen to
represent the cycles of the marking of $M$. Since the Haar measure on
$G$ is invariant by inversion, the measure induced on $\M(\Path(\G),G)$ by $\U^{\G}_{M,\CS,C}$ does not
depend on the choice of $\E^{+}$. \end{proof}

Let us state two basic properties of the measure $\U^{\G}_{M,\CS,C}$.

\begin{proposition} \label{support U}  \label{jauge unif} Recall the notation of Lemma \ref{def ugc}.\\
1. The event ${\mathcal N}=\{\exists l\in \CS \cup \BS(M), h(l)\notin C(l)\}$
satisfies $\U^{\G}_{M,\CS,C}({\mathcal N})=0$.\\
2. The action of the group $G^{\V}$ on $\M(\Path(\G),G)$ preserves the
probability measure $\U^{\G}_{M,\CS,C}$.
\end{proposition}

\begin{proof}1. By definition of $\U^{\G}_{M,\CS,C}$, the support of $\U^{\G}_{M,\CS,C}$ is contained in the closed set
$\{\forall l\in \CS \cup \BS(M), h(l)\in C(l)\}={\mathcal N}^{c}$.

2. Choose $v\in \V$ and $x\in G$. Set $j_{w}=1$ if $w\neq v$ and $j_{v}=x$. If $v$ is not located
on a curve of $\CS$, then the translation invariance of the Haar measure implies that $j$ leaves
the measure $\U^\G_{M,\CS,C}$ invariant. Assume now that $v$ is on the curve $l_{1}$ which is
represented by the cycle $e_{1}\ldots e_{n}$. Assume that $e_{1}$ is outcoming at $v$ and $e_{n}$
is incoming. Then the action of $j$ translates the variables associated to the edges adjacent to
$v$ other than $e_{1}$ and $e_{n}$ and it replaces $(e_{n},\ldots,e_{1})$ by
$(x^{-1}e_{n},\ldots,e_{1}x)$. This leaves the measure $\delta_{C(l_{1})(n)}(dg_{n}\ldots dg_{1})$
invariant. Finally, the action of $j$ preserves the measure $\U^{\G}_{M,\CS,C}$. Since the group $G^{\V}$ is generated by elements which are equal to $1$ at all vertices but one, the result follows. \end{proof}

Let us state precisely the fact that the measures $\U^{\G}_{M,\CS,C}$ disintegrate each other.

We use the notation $\CS=\{l_{1},l_{1}^{-1},\ldots,l_{q},l_{q}^{-1}\}$. This set does not include the boundary components of $M$ anymore. Let us define $\CS_0=\varnothing$ and for each $r\in\{1,\ldots,q\}$, $\CS_r=\{l_{1},l_{1}^{-1},\ldots,l_{r},l_{r}^{-1}\}$. 
For each $r\in\{1,\ldots,q\}$,  any collection $(\O_{1},\ldots,\O_{r})$ of $r$
conjugacy classes of $G$ determines the set of $G$-constraints on $\CS_r$ which maps
$l_{i}$ on $\O_{i}$ for all $i\in\{1,\ldots,r\}$. We denote this set of constraints simply by  $(\O_{1},\ldots,\O_{r})$.

The following lemma is a direct consequence of the definition of $\U^{\G}_{M,\CS,C}$ and (\ref{recover unif}). It shows that the measures $\U^{\G}_{M,\CS,(\O_{1},\ldots,\O_{r},\O_{x_{r+1}},\ldots,\O_{x_{q}})}$
provide a regular disintegration of $\U^{\G}_{M,\CS_r,(\O_{1},\ldots,\O_{r})}$ with respect to the random
variables $h({l_{r+1}}),\ldots,h({l_{q}})$. 

\begin{lemma}\label{pre axiome 3} Let $r$ be an integer between $0$ and $q-1$. Let $\O_{1},\ldots,\O_{r}$
be $r$ conjugacy classes of $G$. Let $f:G^{q-r}\lra \RK$ be a continuous function. Then
\begin{eqnarray}
\nonumber  \int_{\M(\Path(\G),G)} f(h({l_{r+1}}),\ldots,h(l_{q})) \; \U^{\G}_{M,\CS_r,(\O_{1},\ldots,\O_{r})}(dh)= &&\\
&&\hskip -2cm \int_{G^{q-r}} f(x_{r+1},\ldots,x_{q})\; dx_{r+1}\ldots dx_{q}.
\end{eqnarray} 
\begin{equation}
\U^{\G}_{M,\CS_r,(\O_{1},\ldots,\O_{r})}=\int_{G^{q-r}}
\U^{\G}_{M,\CS,(\O_{1},\ldots,\O_{r},\O_{x_{r+1}},\ldots,\O_{x_{q}})}\; dx_{r+1}\ldots dx_{q}.
\end{equation}
\end{lemma}

\subsection{Surgery of uniform measures} Let us investigate the behaviour of the uniform measures that we have just defined under the basic operations of surgery. So far, we have used the letter $f$ to denote gluing maps and also test functions on $G$. In the proof of the next result, we need to use both. This is why we change our notation for gluings.

\begin{proposition}\label{prop markov unif} Let $(M,\CS,C)$ be a marked surface with $G$-constraints endowed with a graph $\G$. Let $(M',\CS',C')$ be a splitting of $M$ and let $\psi:M'\to M$ denote the gluing map. Let $\G'$ be the graph on $M'$ obtained by lifting $\G$. Then the mapping $\psi:(\M(\Path(\G),G),\I) \to (\M(\Path(\G'),G),\I')$ induced by $\psi$ satisfies
\begin{equation} \label{markov unif}
\U^{\G'}_{M',\CS',C'} = \U^\G_{M,\CS,C} \circ \psi^{-1}.
\end{equation}
\end{proposition}

In this proposition, it is crucial that we consider the invariant $\sigma$-fields, not only for the general reason that one should consider only gauge-invariant quantities, but because the equality (\ref{markov unif}), although meaningful with cylinder $\sigma$-fields thanks to Lemma \ref{glu mes}, would simply be false. For instance, consider a binary gluing along two curves $b'_1$ and $b'_2$ with joint $b=f(b'_1)=f(b'_2)$. Then the event $\{h' : h'(b'_1)=h'(b'_2)\}$ belongs to $\C'$ and has measure zero. On the other hand, the pull-back by $\psi$ of this event is the event $\{h : h(b)=h(b)\}$ of $\C$ which has full measure. The point is that the first event does not belong to $\I'$.\\

\begin{proof} Let $f:\M(\Path(\G'),G)\to \RK$ be a continuous function invariant under the action of $G^{\V'}$. Then $f$ induces a continuous function $\tilde f:\M(\Path(\G),G)\to \RK$ and we need to prove that the following relation holds.
\begin{equation}\label{markov to prove}
\int_{\M(\Path(\G),G)} \tilde f \; d\U^{\G}_{M,\CS,C} =\int_{\M(\Path(\G'),G)} f \; d\U^{\G'}_{M',\CS',C'}.
\end{equation}

We treat separately the cases of binary and unary gluings.

{\em 1. Binary gluing} -- Let $\{b,b^{-1}\}$ be the joint of the gluing. Let $b'_1$ and $b'_2$ be the two boundary components of $M'$ which are identified by $f$, oriented in such a way that $f(b'_1)=f(b'_2)=b$. 

Let $\E_b$, $\E'_{1}$ and $\E'_{2}$ denote respectively the set of edges located on $b$, $b'_1$ and $b'_2$. These three sets are naturally identified by the gluing. Set $\E'_\circ=\E'-(\E'_1\cup \E'_2)$ and  $\E_\circ=\E - \E_b=f(\E'_\circ)$. The gluing identifies naturally $\E'_\circ$ with $\E_\circ$. With this notation, we have the partitions
\begin{equation} \label{partition de E}
\E=\E_\circ \cup \E_b \mbox{ and } \E'=\E'_\circ \cup \E'_1 \cup \E'_2.
\end{equation}

The partition of $\E$ above determines the equality $\M(\E,G)=\M(\E_\circ,G)\times \M(\E_{b},G)$, according to which we denote by $g=(g_\circ,g_{b})$ the generic element of $\M(\E,G)$. Similarly, we have the decomposition $\M(\E',G)=\M(\E'_\circ,G)\times \M(\E'_1,G)\times \M(\E'_2,G)$ and we write $g'=(g'_\circ, g'_1,g'_2)$ for the generic element of $\M(\E',G)$.  

With this notation and these identifications, the function $\tilde f$ is defined by the equality $\tilde f(g_\circ,g_b)=f(g_\circ,g_b,g_b)$.

Since each curve of $\CS$ is covered either by $\E_\circ$ or by $\E_{b}$, the decomposition of $\M(\E,G)$ above determines a decomposition of the measure $\U^{\G}_{M,\CS,C}$ as the tensor product
of two measures $U_\circ$ and $U_b$ on $\M(\E_\circ,G)$ and $\M(\E_b,G)$ respectively, each of which is invariant under the action of the gauge group $G^{\V}$. Let us assume that $b$ is the product of $n$ edges of $\E$.
Then $\M(\E_{b},G)$ can be identified with $G^{n}$ and the measure $U_b$ corresponds to
$\delta_{C(b)(n)}$ under this identification. 

Similarly, the measure $\U^{\G'}_{M',\CS',C'}$ splits as the tensor product of three measures $U'_\circ$, $U'_1$ and $U'_2$, on $\M(\E'_\circ,G)$, $\M(\E'_1,G)$ and $\M(\E'_2,G)$ respectively. The last two spaces can be identified with $G^n$ and the measures $U'_1$ and $U'_2$ correspond to $\delta_{C(b)(n)}$ under this identification. 

The measures $U_\circ$ and $U'_\circ$ correspond to each other via the identification of $\M(\E_\circ,G)$ and $\M(\E'_\circ,G)$. Hence, the equality which we need to prove is the following:

\begin{eqnarray*}
\int_{\M(\Path(\G),G)} f(g_\circ,g_b,g_b) \; U_\circ(dg_\circ) U_b(dg_b) =&& \\
&&\hskip -3cm  \int_{\M(\E_\circ,G)\times \M(\E_b,G)^2} f(g_\circ,g_1,g_2) \; U_\circ(dg_\circ) U_b(dg_1) U_b(dg_2).
\end{eqnarray*}

Let $\V'_2$ denote the subset of $\V'$ consisting of the vertices which lie on $b'_2$. The group $G^{\V'_2}$ is a compact Lie group and the invariant measure of its transitive action on the subset $C(b)(n)$ of $\M(\E'_2,G)$ is the measure $\delta_{C(b)(n)}=U'_2$. Let us denote by $j'_2$ the generic element of $G^{\V'_2}$ and by $dj'$ the Haar measure on this group. Since $f$ is invariant under the action of the gauge group, we have for all $(g'_\circ,g'_1,g'_2)$ in $\M(\Path(\G'),G)$ the equality $f(j'_2\cdot g'_\circ,g'_1,j'_2\cdot g'_2)=f(g'_\circ,g'_1,g'_2)$. Observe that some edges of $\E'_\circ$ have some of their endpoints in $\V'_2$, so that the term $g'_\circ$ is affected by the gauge transformation $j'_2$. On the contrary, the term $g'_1$ is not affected because $b'_1$ and $b'_2$ are disjoint. Hence, with the identifications made earlier,
\begin{eqnarray*}
\int_{\M(\Path(\G),G)} f(g_\circ,g_b,g_b) \; U_\circ(dg_\circ) U_b(dg_b) =&& \\
&&\hskip -2.5cm \int_{\M(\Path(\G),G)\times G^{\V'_2}} f(j'_2\cdot g_\circ,g_b,j'_2\cdot g_b) \; U_\circ(dg_\circ) U_b(dg_b) dj'_2.
\end{eqnarray*}
Since the measure $U_\circ$ is invariant under the action of $G^\V$, the last term is equal to
$$\int_{\M(\Path(\G),G)\times G^{\V'_2}} f(g_\circ,g_b,j'_2\cdot g_b) \; U_\circ(dg_\circ) U_b(dg_b) dj'_2.$$
It suffices now to prove an equality about the function $\int_{\M(\E_\circ,G)} f(g_\circ,\cdot,\cdot) \; U_\circ(dg_\circ)$ on $G^{2n}$. Indeed, let us denote this function by $u:G^{2n}\to \RK$. All we need to prove is that, for all conjugacy class $\O$, 
\begin{eqnarray*}
\int_{G^{2n}}u(g_1,\ldots,g_n,j_{1} g_1 j_2^{-1},\ldots,j_n g_n j_{1}^{-1}) \; \delta_{\O(n)}(dg_1\ldots dg_n) dj_1\ldots dj_n =&& \\
&&\hskip -2.5cm  \int_{G^{2n}} u \; d( \delta_{\O(n)} \otimes \delta_{\O(n)}).
\end{eqnarray*}
But the left-hand side is equal to 
$$\int_{G^n} \int_{G^n} u(g_1,\ldots,g_n,g_ {n+1},\ldots,g_{2n}) \delta_{\O_{g_1\ldots g_n}(n)}(dg_{n+1}\ldots dg_{2n}) \delta_{\O(n)} (dg_1\ldots dg_n),$$
hence to the right-hand side.\\

{\em 2. Unary gluing} -- Let $\{b,b^{-1}\}$ be the joint of the gluing. Let $b'$ be the component of the boundary of $M'$ such that $f(b')=b$. Let us write $b=e_1\ldots e_n$ and $b'=e'_{1,1}\ldots e'_{n,1}e'_{1,2}\ldots e'_{n,2}$ in such a way that $f(b'_{i,1})=f(b'_{i,2})=b_i$ for all $i\in\{1,\ldots,n\}$. 

Let $\E_b$ denote the set of edges located on $b$. Set $\E'_1=\{{e'_{1,1}}^{\pm 1},\ldots,{e'_{n,1}}^{\pm 1}\}$ and $\E'_2=\{{e'_{1,2}}^{\pm 1},\ldots, {e'_{n,2}}^{\pm 1}\}$. Set $\E'_\circ=\E'-(\E'_1\cup \E'_2)$  and $\E_\circ=\E-\E_b=f(\E'_\circ)$. We will identify freely $\E_\circ$ with $\E'_\circ$.

With this notation, the equalities (\ref{partition de E}) hold, as well as the subsequent decompositions of $\M(\E,G)$ and $\M(\E',G)$. The function $\tilde f$ is also still defined by the equality $\tilde f(g_\circ,g_b)=f(g_\circ, g_b,g_b)$. 

The decomposition $\U^\G_{M,\CS,C}=U_\circ \otimes U_b$ is valid just as in the binary case, but the decomposition of $\U^{\G'}_{M',\CS',C'}$ is now different. Indeed, this measure splits into the tensor product of $U'_\circ$ and a measure $U'_{12}$ on $\M(\E'_1,G)\times \M(\E'_2,G)$ which is $\delta_{C(b)^2(2n)}$ under the natural identification of $\M(\E'_1,G)\times \M(\E'_2,G)$ with $G^{2n}$. The formula which we need to prove is the following:
\begin{eqnarray*}
\int_{\M(\Path(\G),G)} f(g_\circ,g_b,g_b) \; U_\circ(dg_\circ)U_b(dg_b) =&& \\
&&\hskip -3cm  \int_{\M(\E_\circ,G)\times \M(\E'_1\cup \E'_2,G)} f(g_\circ,g_1,g_2)\; U_\circ(dg_\circ) U'_{12}(dg_1,dg_2).
\end{eqnarray*}
Let $\V'_{12}$ denote the set of vertices which lie on $b'$. By using the invariance of $f$ under the action of the subgroup $G^{\V'_{12}}$ of the gauge group and the invariance of the measure $U_\circ$ under the same action, we find just as in the binary case that the left-hand side of the equality to prove is equal to
$$\int_{\M(\E_\circ,G)\times \M(\E_b,G)\times G^{\V'_{12}}} f(g_\circ,j'_{12}\cdot g_b, j'_{12}\cdot g_b) \; U_\circ(dg_\circ)U_b(dg_b) dj'_{12}.$$
The notation here is misleading, since the two occurrences of $j'_{12}\cdot g_b$ do not denote the same object. Indeed, the two occurrences of $g_b$ in the arguments of $f$ are identified respectively with an element of $\M(\E'_1,G)$ and an element of $\M(\E'_2,G)$, on which $G^{\V'_{12}}$ acts differently. 

Now what we have to prove is really an equality about the function 
$$\int_{\M(\E_\circ,G)} f(g_\circ,\cdot,\cdot) \; U_\circ(dg_\circ)$$
 on $G^{2n}$. Let us call this function $u:G^{2n}\to \RK$. We need to prove that for all conjugacy class $\O$ in $G$ the following equality holds:
\begin{eqnarray*}
\int_{G^{3n}} u(j_0^{-1} g_1 j_1 ,\ldots, j_{n-1} g_n j_n, j_n^{-1} g_1 j_{n+1},\ldots,j_{2n-1}^{-1} g_n j_0) \; \delta_{\O(n)}(dg_1\ldots dg_n) dj_0 \ldots dj_{2n}&& \\
&& \hskip -4cm  = \int_{G^{2n}} u\; \delta_{\O^2 (2n)}.
\end{eqnarray*}
Recall that $\O^2$ is the conjugacy class constituted by the squares of the elements of $\O$. We claim that this equality holds for all continuous function $u$. Indeed, the integral
$$\int_{G^{2n}} u(j_0^{-1} g_1 j_1 ,\ldots, j_{n-1} g_n j_n, j_n^{-1} g_1 j_{n+1},\ldots,j_{2n-1}^{-1} g_n j_0) \;  dj_0 \ldots dj_{2n}$$
is the integral of $u$ with respect to the measure $\delta_{\O_{(g_1\ldots g_n)^2}(2n)}$, by the very definition of this measure as the invariant measure under the natural action of $G^{2n}$ on $\O_{(g_1\ldots g_n)^2}(2n)$. Hence, by a simple particular case of (\ref{put constraint 2}), the integral that we are computing is equal to
\begin{eqnarray*}
\int_{G^n}\left(\int_{G^{2n}} u(x) \delta_{\O_{(g_1\ldots g_n)^2}(2n)}(dx) \right)\delta_{\O(n)}(dg_1 \ldots dg_n)
=&& \\
&&\hskip -2cm  \int_{G^{2n}} u(x) \; \left( \int_G \delta_{\O_{g^2}(2n)} \delta_\O(dg) \right)(dx).
\end{eqnarray*}
The measure $\delta_{\O_g^2(2n)}$ is equal to $\delta_{O^2(2n)}$ for $\delta_\O$-almost all $g$. Hence, the measure between the brackets is $\delta_{\O^2(2n)}$. This concludes the proof. \end{proof}

Let us conclude this section by a much simpler result, which is the simplest instance of invariance under subdivision.

\begin{proposition}\label{inv subd unif} Let $(M,\CS,C)$ be a marked surface with $G$-constraints. Let $\G_1$ and $\G_2$ be two graphs on $(M,\CS)$. Assume that $\G_1\preccurlyeq \G_2$. Then the inclusion $\Path(\G_1)\subset \Path(\G_2)$ induces a measurable restriction map $r:\M(\Path(\G_2),G)\to \M(\Path(\G_1),G)$ which satisfies
$$\U^{\G_2}_{M,\CS,C} \circ r^{-1} = \U^{\G_1}_{M,\CS,C}.$$
\end{proposition}

\begin{proof}Let us choose an orientation for $\G_1$ and $\G_2$. The restriction map, seen as a map from $G^{\E_2^+}$ to $G^{\E_1^+}$, multiplies the components which correspond to the edges of $\E_2$ which constitute each edge of $\E_1$ and forgets about the components which correspond to edges which do not lie in the skeleton of $\E_1$.

Since the product of independent uniform variables on $G$ is still uniform, the only non-trivial thing to check is what happens along the marking curves or the boundary components. There, the invariance follows from (\ref{contracte dO}). \end{proof}

\index{multiplicative function!uniform with constraints|)}

\section{Tame generators of the group of reduced loops}
\label{sec structure} 

Consider a surface $(M,\varnothing,C)$ with $G$-constraints along its boundary, endowed with a graph $\G$. Our objective in this section is to exhibit a family of lassos which generates the group of reduced loops in $\G$ and to compute the distribution of the $G$-valued random variables associated with these lassos under the constrained uniform measure defined in the previous section. Moreover, we are going to do this in a way which is consistent with the partial order on the set of graphs.

\begin{definition} Let $M$ be a compact surface endowed with a graph $\G$. Let $v$ be a vertex of $\G$.
A lasso $l\in \RL_{v}(\G)$ is said to be {\em facial} if its meander represents a facial cycle of
$\G$. It is said to be {\em bounding} if its meander covers a connected component of $\partial M$. 
\end{definition}
\index{lasso!facial}
\index{lasso!bounding}

We want to prove the existence of systems of generators of $\RL_v(\G)$ which consist in one bounding lasso for each connected component of $\partial M$, one facial lasso for each face of $\G$, and as many supplementary lassos with non-contractible meander as the genus of $M$ (see Section \ref{subsec: classif}). 

Recall the notation $W(M)$ from Section \ref{subsec: gr1}. If $w$ is a word in some set of letters, we denote by $\overleftarrow{w}$ the word obtained by reversing the order of the letters of $w$.

\index{group of reduced loops!tame generators}
\begin{proposition}\label{tame generators} Let $M$ be a connected compact surface. Let $\rg=\rg(M)$ denote the reduced genus of $M$. Let $\p=\p(M)$ be the number of connected components of $\partial M$. Write $\BS(M)=\{b_1, b_1^{-1},\ldots,b_{\p},b_{\p}^{-1}\}$. If $M$ is oriented, we assume that $b_1,\ldots,b_{\p}$ bound $M$ positively. Let $\G$ be a graph on $M$. Let $v$ be a vertex of $\G$. Let $\f$ denote the number of faces of $\G$.

1. There exists, in the group $\RL_{v}(\G)$, $\rg$ lassos $a_{1},\ldots,a_{\rg}$; $\p$ bounding lassos
$c_{1},\ldots,c_{\p}$ whose meanders are equivalent to $b_1,\ldots,b_{\p}$ up to permutation; $\f$ facial lassos
$l_1,\ldots,l_{\f}$ whose meanders bound the $\f$ faces of $\G$, positively if $M$ is oriented; and there exists an word $w$ in the letters $a_1,\ldots,a_\rg$ which belongs to $W(M)$ such that the group $\RL_{v}(\G)$ admits the presentation
$$\langle a_{1},\ldots,a_{\rg},c_{1},\ldots,c_{\p},l_{1},\ldots,l_{\f} | w(a_{1},\ldots,a_{\rg})c_{1}\ldots c_{\p}=l_{1}\ldots l_{\f}\rangle$$
and, for all continuous function $f:G^{\rg+\p+\f}\to \RK$ and all set $C$ of $G$-constraints along the boundary components of $\G$, 
\begin{align*}
\int_{\M(\Path(\G),G)}&f(h(a_{1}),\ldots,h(a_{\rg}),h(c_{1}),\ldots,h(c_{\p}),h(l_{1}),\ldots,h(l_{\f}))\;
\U^{\G}_{M,\varnothing,C}(dh)
 \\
& = \int_{G^{\rg+\p+\f-1}}f(x_{1},\ldots,x_{\rg},y_{1},\ldots,y_{\p},z_{1},\ldots,z_{\f-1},z_{\f})\\
& \hskip 3cm  dx_{1}\ldots dx_{\rg}\delta_{C(b_{1})}(dy_{1})\ldots
\delta_{C(b_{\p})}(dy_{\p})dz_{1}\ldots dz_{\f-1},
\end{align*}
where we have set $z_{\f}=y_\p \ldots y_1 \overleftarrow{w}(x_{1},\ldots,x_{\rg}) (z_{\f-1}\ldots z_{1})^{-1}$.

A collection of lassos such as $\{a_1,\ldots,a_{\rg},c_1,\ldots,c_{\p},l_1,\ldots,l_{\f}\}$ will be called a {\em tame system of generators} associated with the word $w$.

2. Let $\G_1$ and $\G_2$ be two graphs on $M$ such that $\G_1\preccurlyeq \G_2$. Set $\f_1=\f(\G_1)$. Let $v$ be a vertex of $\G_1$. Let $\{a_1,\ldots,a_{\rg},c_1,\ldots,c_{\p},l_1,\ldots,l_{\f_1}\}$ be a tame system of generators of $\RL_v(\G_1)$ associated with the word $w$. Assume that the faces of $\G_1$ are labelled $\F_1=\{F^1_i : i\in \{1,\ldots,\f_1\}\}$ in such a way that for all $i\in\{1,\ldots,\f_1\}$, the meander of the lasso $l_i$ bounds $F^1_i$. For all $i\in \{1,\ldots,n\}$, let $\F_{2,i}=\{F^2\in \F_2 : F^2 \subset F^1_i\}$ be the set of faces of $\G_2$ contained in $F^1_i$ and set $\f_{2,i}=\#\F_{2,i}$.

Then there exists a set of facial lassos $\{l_{i,j} : i\in\{1,\ldots,\f_1\}, j\in \{1,\ldots,\f_{2,i}\}\}$ in $\RL_v(\G_2)$ such that for all $i,j$, the meander of the lasso $l_{i,j}$ bounds a face of $\F_{2,i}$, for all $i$ the equality $l_i=l_{i,1}\ldots l_{i,\f_{2,i}}$ holds and 
$$\{a_1,\ldots,a_{\rg},c_1,\ldots,c_{\p},l_{1,1},\ldots,l_{\f_1,\f_{2,\f_1}}\}$$
 is a tame system of generators of $\RL_v(\G_2)$ associated with the word $w$.

\end{proposition}

We will use the notation $\rg=\rg(M)$, $\p=\p(M)$, and $\f=\f(\G)$ until the end of this chapter.

By Lemma \ref{lg is free}, we know that $\RL_v(\G)$ is free of rank $\e(\G)-\v(\G)+1$. By Euler's relation for the graph $\G$, which writes $\v(\G)-\e(\G)+\f(\G)=\chi(M)=2-\rg-\p$, the rank of $\RL_{v}(\G)$
can also be written as  $\rg+\p+\f-1$. In order to find families of generators whose cardinal
decomposes naturally as $\rg+\p+\f-1$, we introduce the dual graph of $\G$.

\begin{definition} Let $M$ be a compact surface endowed with a graph $\G$. Let $(M',\G',\iota,f)$ be a split pattern of $(M,\G)$. 
Let $\widehat \V$  denote the set of connected components of $M'$, which is in canonical bijection with $\F$. Then, set $\widehat\E=\{\{e', {e'}^{-1}\}:e'\in \E'\}$. For each element $\hat e=\{e',{e'}^{-1}\}$ of $\widehat \E$, define the source $s(\hat e)$ of  $\hat e$ as the connected component of $M'$ which contains $e'$  and the target $t(\hat e)$ of $\hat
e$ is the connected component of $M'$ which contains $\iota(e')$.

The abstract graph $\widehat \G=(\widehat \V,\widehat \E,s,t)$ is called the {\em dual graph} of $\G$.
\end{definition}
\index{graph!dual}
\index{GAGH@$\widehat \G$}

Properly speaking, the dual graph of $\G$ depends on the choice of the split pattern of $\G$ and it is unique only up to an obvious notion of isomorphism. We shall in fact choose a split pattern and work with the associated dual graph.

The involution $\iota$ of the split pattern induces an involution on $\widehat \E$, which we still denote by $\iota$. This involution is similar to an orientation reversal, but one should observe that it may have fixed points. For example, if an edge $e'\in \E'$ is sent by $f$ on a boundary component of $M$, then $\iota(e')=e'$ and the dual edge $\hat e=\{e',{e'}^{-1}\}$ also satisfies $\iota(\hat e)=\hat e$.

The orbits of $\iota$ on $\widehat \E$, which we call unoriented edges of the dual graph, correspond bijectively by  $f$ with the unoriented edges of the graph $\G$, that is, the pairs $\{e, e^{-1}\}$ for $e\in \E$.

\begin{definition} Let $M$ be a compact surface endowed with a graph $\G$. Let $(M',\G',\iota,f)$ be a split pattern of $(M,\G)$. Recall that a {\em spanning tree } of the dual graph $(\widehat V,\widehat E,s,t)$ is a subset $\widehat T\subset \widehat \E$ stable by the involution $\iota$ and such that any two vertices of the dual graph are joined by a unique injective path made with edges of $\widehat T$. 

Let $\widehat T$ be a spanning tree of the dual graph. An orientation of $M'$ is said to be {\em adapted to } $\widehat T$ if for all $\hat e=\{e, e^{-1}\} \in \widehat T$, the edges $e$ and $\iota(e)$ are neither both positively oriented nor both negatively oriented with respect to this orientation of $M'$.

Assume that $M'$ is endowed with an orientation adapted to $\widehat T$. Then we define the image by $f$ of an edge $\hat e=\{e',{e'}^{-1}\}$ of $\widehat T$ as an edge of $\G$ by setting $f(\hat e)=f(e')$ if $e'$ is positively oriented as a subset of $\partial M'$ and $f(\hat e)=f({e'}^{-1})$ otherwise. 
\end{definition}
\index{spanning tree}
\index{graph!spanning tree}

It is not difficult to check that there are exactly two orientations of $M'$, among the $2^\f$ possible, which are adapted to any given spanning tree in the dual graph. Moreover, for all $\hat e\in \widehat T$, we have the equality $f(\iota(\hat e))=f(\hat e)^{-1}$ in 
$\E$.


It follows from Proposition \ref{erase edge} that $\E\setminus f(\widehat T)$ is the set of edges
of a graph $\G_{0}$ on $M$ with a single face. This graph has $\e_{0}=\e(\G)-(\f(\G)-1)$ edges, so by Euler's
relation, it still has $\v_{0}=\v(\G)$ vertices. Moreover, its skeleton contains $\partial M$ and every edge of $\G$ which lies on $\partial M$ is also an edge of $\G_0$. Let $B^+\subset \E_0$ be a collection of edges comprising exactly one edge on each connected component of $\partial M$. If $M$ is oriented, we assume that the edges of $B$ bound $M$ positively. The set of edges of $\G_{0}$ located on $\partial M$ and which do not
belong to $B=B^+\cup (B^+)^{-1}$ form a cycle-free subgraph of $\G_{0}$. Hence, it is possible to
extend this subgraph to a spanning tree $T$ of $\G_0$ such that $T\cap B=\varnothing$ (see Figure \ref{ex tbrt}).

\begin{lemma}\label{lasso bounding} With the notation above, for each edge $e\in B$, the lasso $l_{e,T}$ is a bounding lasso whose meander covers the connected component of $\partial M$ on which $e$ lies and whose spoke contains no edge lying on this connected component.
\end{lemma}

\begin{proof}Let $b$ be the connected component of $\partial M$ on which $e$ lies. If the base point $v$ is located on $b$, then $l_{e,T}$ is a simple loop which represents the cycle $b$. Otherwise, since any two vertices located on $b$ can be joined by a path in $T$ which stays in $b$, there exists a unique vertex $w$ on $b$ which is joined to the base point $v$ by a path in $T$ with no edge lying on $b$. Hence, $l_{e,T}=[v,w]_T ([w,\underline{e}]_T e [\overline{e},w]_T) [w,v]_T$. The three paths between the brackets form the meander of the lasso, which is a simple loop covering $b$. \end{proof}

Let us carry on with our construction. We have a partition $\E_{0}=T \cup B \cup R$, where $R=\E_{0}\setminus(T\cup B)$. It follows from Euler's relation that $R$ contains exactly $g'$ unoriented edges. From the fact that $\E_{0}$
has $\v(\G)$ vertices, we deduce that $T$ is in fact a spanning tree of
$\G$. Hence, Lemma \ref{lg is free} implies the first assertion of the following result.

\begin{proposition} \label{deux arbres} Let $M$ be a compact surface endowed with a graph $\G$. Let $(M',\G',\iota,f)$ be a split pattern of $(M,\G)$. Let $\widehat T$ be a spanning tree of the dual graph of $\G$. Let $B^+$ be a collection of edges comprising exactly one edge on each connected component of $\partial M$, which we assume to be positively oriented if $M$ is oriented. Set $B=B^+\cup (B^+)^{-1}$. Let $T$ be a spanning tree of $\E\setminus f(\widehat T)$ such that $T\cap B=\varnothing$. Set $R=\E\setminus(f(\widehat
T)\cup B\cup T)$. Choose an orientation $\E^+$ of  $\G$ such that $B^+\subset \E^+$. Write 
$B^+=\{e_1,\ldots,e_{\p}\}$. Let $b_1,\ldots,b_{\p}$ be the boundary components of $M$ enumerated in
such a way that for all $i\in\{1,\ldots,{\p}\}$, $b_i$ is the meander of the lasso $l_{e_i,T}$. Then $R^+$, $B^+$ and $f(\widehat T)^+$ contain respectively $\rg$, $\p$ and $\f-1$ edges and the following properties hold.

1. The group $\RL_{v}(\G)$ is freely generated by the loops $\{l_{e,T} : e \in R^+\cup
B^+ \cup f(\widehat T)^+ \}$. 

2. Under the probability measure $\U^{\G}_{M,\varnothing,C}$ on $\M(\Path(\G),G)$, the
collection of random variables $\{h(l_{e,T}) : e\in R^+ \cup B^+ \cup f(\widehat T)^+\}$ is a
collection of independent variables. For all $e\in R^+ \cup f(\widehat T)^+$, the variable $h(l_{e,T})$
is uniformly distributed on $G$ and for all $i\in\{1,\ldots,\p\}$, the variable
$h(l_{c_i,T})$ has the distribution $\delta_{C(b_i)}$.
\end{proposition}

\begin{figure}[h!]
\begin{center}
\scalebox{1}{\includegraphics{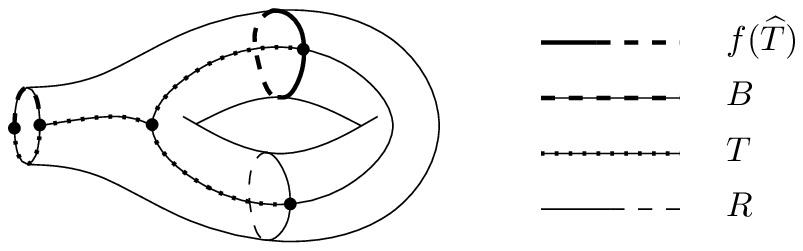}}
\caption{An example of a partition of $\E$ as $R\cup B\cup f(\widehat T)\cup T$}\label{ex tbrt}
\end{center}
\end{figure}

\begin{proof}Let us prove the second assertion. By definition of the measure $\U^{\G}_{M,\varnothing,C}$, the random variables $\{h(e) : e\in R^+ \cup f(\widehat T)^+\}$ are independent, uniformly distributed on $G$, and independent of
$\{h(e) : e\in \E^+\setminus (R^+\cup f(\widehat T)^+)\}$. Hence, the variables $\{h(l_{e,T}) : e\in R^+ \cup f(\widehat T)^+\}$ are independent, uniformly distributed and independent of the variables $\{h(l_{e,T}) : e\in B^+\}$, which do not involve any edge of $R\cup f(\widehat T)$. 

Now choose $i\in\{1,\ldots,\p\}$. Consider the loop $l_{e_i,T}$ and the corresponding bounding cycle $b_i$. By Lemma \ref{lasso bounding}, the spoke of the lasso $l_{e_i,T}$ does not involve any edge lying on the boundary component $b_i$ of $M$. However, it may involve edges located on other boundary components of $M$. 

We claim that for every subset $I\subset \{1,\ldots,\p\}$, there exists $i_0\in I$ such that for all $j\in I-\{i_0\}$, the spoke of $l_{e_j,T}$ does not involve edges located on $b_{i_0}$. Assume to the contrary that for some subset $I$ there does not exist such an $i_0$. Then there would exist $i_1,\ldots,i_{k-1} \in I$ all distinct and $i_k=i_1$ such that for all $j\in\{1,\ldots,k-1\}$, the spoke of the lasso $l_{e_{i_{j+1}},T}$ involves an edge which lies on the boundary component $b_{i_{j}}$. This would in particular imply that for each $j\in \{1,\ldots,k-1\}$ there exists a path in $T$ from $b_{i_j}$ to $b_{i_{j+1}}$ and, since $i_k=i_1$, that $T$ contains a cycle.

By relabelling the boundary components of $M$, we may assume that for all $k\in\{1,\ldots,\p\}$, the element $k$ of the subset $\{k,\ldots,\p\}\subset \{1,\ldots,\p\}$ has the property described above. Since under $\U^{\G}_{M,\varnothing,C}$, the random variable $h(b_k)$ has the distribution $\delta_{C(b_k)}$ and is independent of the variables $h(e)$ for $e\in \E$ not located on $b_k$, the variable $h(l_{e_k,T})$ itself has the distribution $\delta_{C(b_k)}$ and is independent of $h(l_{e_{k+1},T}),\ldots,h(l_{e_{\p},T})$. This implies easily the result. \end{proof}

With the proof of Proposition \ref{tame generators} in mind, the next step is to express the loops $l_{e,T}$ for $e\in f(\widehat T)$ in function of facial lassos. The exact way in which these lassos decompose into products of facial lassos depends on, and in fact encodes completely, the geometry of the tree $\widehat T$. 

Let $M'$ be endowed with an orientation adapted with $\widehat T$. Let $\hat v=M'_F$ be a vertex of $\widehat \G$ which corresponds to a face $F$. The set of edges of $\widehat \G$ whose source is $\hat v$ is in one-to-one correspondence with the set of edges of $\G'$ located on the boundary of $M'_F$ and which bound it positively. This set carries a natural cyclic order, which is the order in which the edges are traversed by a cycle bounding $M'_F$. By restriction, the set of edges of $\widehat T$ which share $\hat v$ as their source is endowed with a cyclic order.

Let us root $\widehat T$ by not only choosing a reference vertex but also by choosing among the
edges issued from this vertex which one is the first. The simplest way to do this is to choose a
vertex of $\G'$. This determines a root vertex for $\widehat T$, namely the connected component
of $M'$ which carries this vertex, and this also breaks the cyclic symmetry of the edges issued from
this connected component, which are now totally ordered.  

The object that we are now contemplating is a tree (that is, an abstract graph without simple cycle) endowed with a distinguished vertex, a total order on the edges issued from this distinguished vertex and a cyclic order on the set of edges issued from any other vertex. Such an object is called a rooted planar tree and it has a canonical representation as a set of words of integers, according to a formalism due to J. Neveu. Let us simply describe how the vertices are labelled by words of integers, that is, finite sequences of non-negative integers.

The root vertex is labelled by the empty word $\varnothing$. Let $k(\varnothing)$ be the number of children of $\varnothing$, that is, the number of vertices to which it is joined by an edge. These vertices are labelled by
words of length 1, namely $1,2,\ldots,k(\varnothing)$, according to the total order on these edges.

Then, consider a vertex labelled by a word $u=u_{1}\ldots u_{n}$. The integer $n$ is called the {\em height} of $u$ and it is denoted by $h(u)$. Let $\pi(u)=u_{1}\ldots u_{n-1}$ denote the predecessor of $u$ and let $k(u)$ denote the
number of vertices other than $\pi(u)$ to which $u$ is joined by an edge. Then the edge from $u$
to $\pi(u)$, denoted by $u\pi(u)$, breaks the cyclic symmetry among the edges issued from $u$ and
determines a total order on the $k(u)$ other edges issued from $u$. The $k(u)$ targets of these edges
are labelled $u_{1}\ldots u_{n}1, \ldots, u_{1}\ldots u_{n}k(u)$ in this order.

Thus, the choice of a rooting of $\widehat T$ and an adapted orientation of $M'$ determines a labelling of $\F$ by words of integers (see Figure \ref{neveu}). As we shall see now, it determines also a specific facial lasso for each face of $\G$ and one of the loops $l_{e,T}$ for each face distinct from the root face.
\index{planar tree}

\begin{figure}[h!]
\begin{center}
\scalebox{1}{\includegraphics{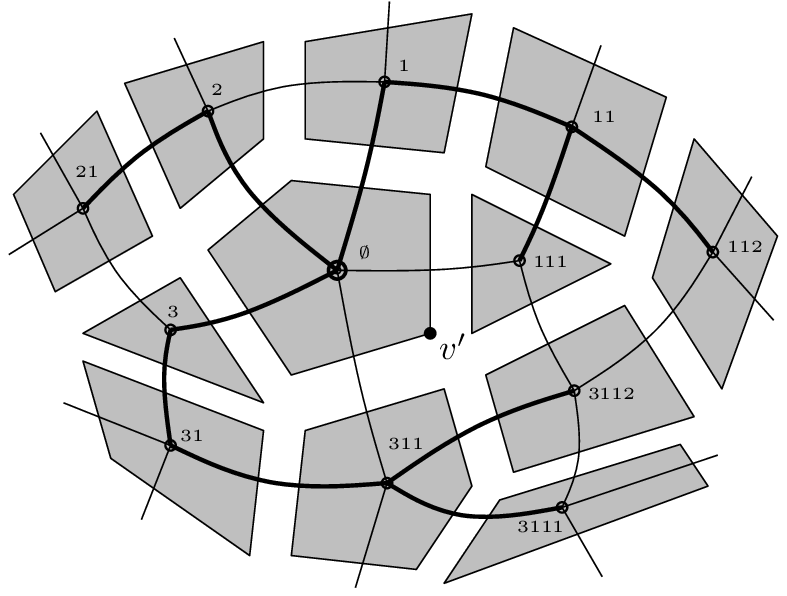}}
\caption{The labelling of the faces of a graph by words of integers. The vertex $v'$ determines the rooting of $\widehat T$.}
\label{neveu}
\end{center}
\end{figure}

\begin{definition}\label{lupiu} Let $M$ be a compact surface endowed with a graph $\G$. Let $(M',\G',\iota,f)$ be a split pattern of $(M,\G)$. Let $\widehat T$ be a spanning tree of the dual graph of $\G$. Endow $M'$ with an orientation adapted to $\widehat T$. Choose a root for $\widehat T$. Label the faces of $\G$ with words of integers accordingly.

Consider a face $F$ of $\G$ labelled by the word $u$. If $u=\varnothing$, let $\partial_{\widehat T} F$
be the unique representative of $\partial M'_{F}$ which starts at the root vertex chosen in $\V'$ and bounds positively $\partial M'_{F}$. Otherwise, consider the edge $u\pi(u)$ of $\widehat T$, which we identify with an edge of $\G'$ which bounds $M'_F$ positively. Let $\partial_{\widehat T}F$ be the unique simple loop representing $\partial M'_{F}$ which starts with this edge.  Write $\partial_{\widehat T}F$ as a product of edges $e_{1}\ldots e_{r}$. We define the facial lasso $l_{u}\in \RL_v(\G)$ by $l_{u}=l_{e_{1},T}\ldots l_{e_{r},T}$. 

Consider a face distinct from the face labelled by $u\neq \varnothing$. The edge $u\pi(u)$ of $\widehat T$ determines the edge $f(u\pi(u))$ of $\G$ and we define $l_{u,\pi(u)}=l_{f(u\pi(u)),T}$. We also define $l_{\varnothing,\pi(\varnothing)}$ as the constant loop at $v$. 
\end{definition}

Let $u$ and $u'$ be two vertices of $\widehat T$. We say that $u$ is a {\em prefix} of $u'$ if there exists a word of integers $u''$ such that $u'=uu''$, the concatenation of $u$ and $u''$. Genealogically, this can be phrased by saying that $u$ is an ancestor of $u'$.

\begin{lemma} \label{def w} Recall the notation of Definition \ref{lupiu}. Let $F$ be a face of $\G$ labelled by $u$. There exist $k(u)+1$ elements $t_{u,\pi(u)},t_{u,u1},\ldots,t_{u,uk(u)}$ of the subgroup  of $\RL_{v}(\G)$ generated by $\{l_{e,T}:e\in R\cup B\}$ such that 
\begin{equation}\label{derive integre}
l_{u}=l_{u,\pi(u)} t_{u,\pi(u)}^{-1} l_{u1,u}^{-1}t_{u,u1}^{-1}\ldots l_{uk(u),u}^{-1}t_{u,uk(u)}^{-1}.
\end{equation}

Moreover, there exists an element $t$ of the subgroup of $\RL_{v}(\G)$ generated by $\{l_v : u \mbox{ prefix of } v\}$ and $\{l_{e,T} : e\in R\cup B\}$ such that
\begin{equation}\label{derive integre 2}
l_u=l_{u,\pi(u)} t.
\end{equation}
\end{lemma}

\begin{proof}Let us write $\partial_{\widehat T}F=e_{1}\ldots e_{n}$. By definition, $e_{1}$ is the edge
of $\G$ which corresponds to the edge $u\pi(u)$ of $\widehat T$. Hence, $l_{e_{1},T}=l_{u,\pi(u)}$. Then, the list
$(e_{2},\ldots,e_{n})$ consists of the images by $f$ of the edges of $\widehat T$ which join $u$
to $u1,\ldots,uk(u)$, in this order, intermixed with some edges of $T$ and some edges of $R\cup
B$. The claimed expression for $l_{u}$ follows. In the case where $F$ is the root of $\widehat
T$, the edge $e_{1}$ does not play a special role, it is either the image by $f$ of the edge which
joins $\varnothing$ to $1$ in $\widehat T$, or an edge of $T\cup R\cup B$. 

We say that a vertex $u$ of $\widehat T$ is a {\em leaf} if $k(u)=0$. The second expression of $l_u$ reduces to the first if $u$ is a leaf. Let us now consider a vertex $u$ and assume that (\ref{derive integre 2}) holds for each vertex of which $u$ is the predecessor, that is, for $u1,\ldots,uk(u)$. For every vertex $v$, the fact that $uj$ is a prefix of $v$ for some $j\in \{1,\ldots,k(u)\}$ implies that $u$ is a prefix of $v$. Hence, by applying (\ref{derive integre}) to $u$ and then (\ref{derive integre 2}) to $u1,\ldots,uk(u)$, we find that (\ref{derive integre 2}) holds for $u$.

By induction along $\widehat T$, the second expression of $l_u$ holds for all $u\in \widehat T$. \end{proof}

\begin{corollary}\label{base avec faces} Recall the notation of Proposition \ref{deux arbres} and Definition \ref{lupiu}. The group $\RL_{v}(\G)$ is freely generated by $\{l_{e,T} : e\in R^+ \cup B^+\}$ and $\{l_{u} : u\neq \varnothing\}$. Moreover, under $\U^\G_{M,\varnothing,C}$, the random variables $\{h(l_{u}): u\neq\varnothing \}$ are mutually independent, uniformly distributed and independent of $\{h(l_{e,T}) : e\in R^+\cup B^+\}$.
\end{corollary}

\begin{proof} Let us call the height of $\widehat T$ and denote by $h(\widehat T)$ the maximal height of a vertex of $\widehat T$, which is necessarily a leaf. We claim that for all integer $n\geq 1$, $\RL_{v}(\G)$ is freely generated by the set $X(n)$ defined by
$$X(n)=\{l_{e,T} : e\in R^+ \cup B^+\} \cup \{l_{u,\pi(u)} : h(u)<n\} \cup \{l_{u} : h(u)\geq n \}.$$

For $n=1$, it is the statement which we wish to prove. For $n> h(\widehat T)$, it is the content of
Proposition \ref{deux arbres}. The fact that the claim is true for all $n$ is easily proved by descending induction.
Indeed, for all $n\geq 1$,  the set $X(n)$ is deduced from $X(n+1)$ by replacing, for all labels $u$ of $\widehat T$ such that $h(u)=n$, the loop $l_{u,\pi(u)}$ by the facial lasso $l_u$. If $X(n+1)$ is a basis of $\RL_v(\G)$, then it follows immediately from (\ref{derive integre 2}) that  $X(n)$ is another basis of $\RL_v(\G)$. 

Since the sets $\{v : u \mbox{ prefix of } v\}$ are disjoint for the distinct labels $u$ of height $n$, the assertion on the distribution of the variables $\{h(l) : l\in X(n)\}$ follows easily from (\ref{derive integre 2}) and the distribution of $\{h(l) : l\in X(n+1)\}$.  \end{proof}

At this point, we have proved most of the first assertion of Proposition \ref{tame generators}. We have exhibited a family of generators of the group $\RL_v(\G)$ which generate it freely. We know that if we add one element to this family, we get a presentation of $\RL_v(\G)$ with one relation. The next result helps us to get a relation of the form that we expect.

Lemma \ref{def w} defines $t_{u,v}$ for each ordered pair of adjacent vertices $(u,v)$ in $\widehat T$, and also an extra element $t_{\varnothing,\pi(\varnothing)}$. Now for each pair of vertices $(u,u')$ of $\widehat T$, not necessarily adjacent, let us define $t_{u,u'}=t_{v_1,v_2}\ldots t_{v_{m-1},v_m}$, where $u=v_1,\ldots,v_m=u'$ is the unique injective path in $\widehat T$ from $u$ to $u'$.

\begin{lemma} \label{relation} Endow the set of vertices of $\widehat T$ with the lexicographic order associated to the reversed order on ${\mathbb N}$. Enumerate its elements accordingly: $\F=\{u_1\leq \ldots \leq u_\f\}$. Then
\begin{equation}\label{the relation}
l_{u_1} t_{u_1,u_2} l_{u_2} t_{u_2,u_3} \ldots t_{u_{\f-1},u_\f} l_{u_\f} t_{u_\f,\varnothing} t_{\varnothing, \pi(\varnothing)}=1.
\end{equation}
Moreover, the loop $(t_{u_1,u_2}t_{u_2,u_3} \ldots t_{u_{\f-1},u_\f}t_{u_\f,\varnothing} t_{\varnothing, \pi(\varnothing)})^{-1}$ is the boundary of the unique face of $\G_{0}$. It is equal to a word in the lassos $\{l_{e,T} : e\in R^+ \cup B^+\}$ where each lasso of the set $\{l_{e,T}:e\in R^+\}$ appears exactly twice, possibly sometimes with exponent $-1$, and each lasso of the set $\{l_{e,T}:e\in B^+\}$ appears exactly once, possibly with exponent $-1$. 
If $M$ is oriented, then each lasso of the set $\{l_{e,T}:e\in B^+\}$ appears with exponent $1$.
\end{lemma}

\begin{proof}The equation (\ref{derive integre}) can be rewritten as 
$$l_{u,\pi(u)}=l_u t_{u,uk(u)} l_{uk(u),u} \ldots t_{u,u1} l_{u1,u} t_{u,\pi(u)}.$$
Let us apply this relation to $u=\varnothing$. We find
\begin{equation}\label{1=}
1=l_\varnothing t_{\varnothing, k(\varnothing)} l_{k(\varnothing),\varnothing} \ldots t_{\varnothing,1} l_{1,\varnothing} t_{\varnothing,\pi(\varnothing)}.
\end{equation}
Let us define, for all integer $n\geq 1$ and all vertex $u$ of $\widehat T$ such that $h(u)\leq n$ an element $\tilde l_u^{(n)}$ of $\RL_v(\G)$  by setting
$$\tilde l_u^{(n)} = \left\{ \begin{array}{ll} l_u & \mbox{if } h(u)<n \\ l_{u,\pi(u)} & \mbox{if } h(u)=n. \end{array}\right.$$
Also, for each $n$, let $u^n_1\leq \ldots \leq u^n_{r_n}$ be the vertices of $\widehat T$ of height at most $n$ listed in the lexicographic order corresponding to the reversed order on $\mathbb N$. Then (\ref{1=}) can be written as
$$1=\tilde l^{(1)}_{u^1_1} t_{u^1_1,u^1_2}\tilde l^{(1)}_{u^1_2} \ldots  t_{u^1_{r_1-1},u^1_{r_1}}\tilde l^{(1)}_{u^1_{r_1}} t_{\varnothing,\pi(\varnothing)}.$$

By applying (\ref{derive integre}) iteratively to the terms of the form $l_{u,\pi(u)}$ in this equality, one finds that the equality 
$$1=\tilde l^{(n)}_{u^n_1} t_{u^n_1,u^n_2}\tilde l^{(n)}_{u^n_2} \ldots  t_{u^n_{r_1-1},u^n_{r_1}}\tilde l^{(n)}_{u^n_{r_n}} t_{\varnothing,\pi(\varnothing)}$$
holds for all $n\geq 1$. Here, $t_{u,u'}$ has the meaning explained before the statement of Lemma \ref{relation}.
For $n$ larger than the height of $\widehat T$, this formula is exactly what we wanted to prove.

Finally, $(t_{u_1,u_2}t_{u_2,u_3} \ldots t_{u_{\f-1},u_\f}t_{u_\f,\varnothing} t_{\varnothing, \pi(\varnothing)})^{-1}$ is the product of the loops $l_{e,T}$ where $e$ goes around the unique face of the pattern of $\G_0$ obtained from $M'$ by sewing the edges of $\widehat T$. The result follows. \end{proof}

We are now ready to prove the main result of this section.\\

\begin{proof}[Proof of Proposition \ref{tame generators}]  Let us consider the relation given by Lemma \ref{relation}. Let us define, for all $i\in\{1,\ldots,\f\}$, 
$$l_i= (t_{u_i,u_{i+1}} \ldots t_{u_{\f-1},u_{\f}} t_{u_{\f},\varnothing}t_{\varnothing,\pi(\varnothing)})^{-1} l_{u_i} (t_{u_i,u_{i+1}} \ldots t_{u_{\f-1},u_{\f}} t_{u_{\f},\varnothing}t_{\varnothing,\pi(\varnothing)}).$$
Then the relation (\ref{the relation}) becomes
\begin{equation}\label{the relation 2}
(t_{u_1,u_2}\ldots t_{u_{\f-1},u_{\f}}t_{u_{\f},\varnothing}t_{\varnothing,\pi(\varnothing)})^{-1}=l_1\ldots l_{\f}.
\end{equation}
By Corollary \ref{base avec faces}, $\RL_v(\G)$ is freely generated by $\{l_{e,T} : e\in R^+\cup B^+\}\cup \{l_1,\ldots,l_{\f-1}\}$ and, under $\U^\G_{M,\varnothing,C}$, the variables $h(l_1),\ldots,h(l_{\f-1})$ are independent, uniformly distributed and independent of $\{h(l_{e,T}) : e\in R^+\cup B^+\}$. 

Let us write $B^+=\{b_1,\ldots,b_{\p}\}$. By the last part of Lemma \ref{relation}, there exists $\p+1$ elements $t_0,\ldots,t_{\p}$ of the subgroup generated by $\{l_{e,T} : e\in R^+\}$, $\p$ signs $\epsilon_1,\ldots,\epsilon_{\p} \in \{-1,1\}$, and a permutation $\sigma \in \Sy_{\p}$ such that 
$$(t_{u_1,u_2}\ldots t_{u_{\f-1},u_{\f}}t_{u_{\f},\varnothing}t_{\varnothing,\pi(\varnothing)})^{-1}=t_0 l_{b_{\sigma(1)},T}^{\epsilon_1} t_1 \ldots t_{\p-1} l_{b_{\sigma(\p)},T}^{\epsilon_{\p}} t_{\p},$$
with $\epsilon_1=\ldots=\epsilon_{\p}=1$ if $M$ is oriented. 

By Lemma  \ref{relation} again, the loop $t_0\ldots t_{\p}$ can be written as a word in $\{l_e : e\in R^+\}$ where each loop appears exactly twice, possibly sometimes with exponent $-1$. Let us name $a_1,\ldots,a_{\rg}$ the loops $\{l_e : e\in R^+\}$ and let $w$ denote the element of the free group of rank $\rg=\# R^+$ such that $w(a_1,\ldots,a_{\rg})=t_0\ldots t_{\p}$. Let us also define, for all $k\in\{1,\ldots,\p\}$, $c_i=(t_i \ldots t_{\p})^{-1} l_{b_{\sigma(i)},T}^{\epsilon_i} (t_i \ldots t_{\p})$. 
Then the relation (\ref{the relation 2}) becomes
\begin{equation}\label{the relation 3}
w(a_1,\ldots,a_{\rg})  c_1\ldots c_{\p}=l_1\ldots l_\f.
\end{equation}

The first assertion result now from Proposition \ref{deux arbres}, Corollary \ref{base avec faces} and the definition of $l_1,\ldots,l_{\f}$ and $c_1,\ldots,c_{\p}$.\\

2. Let us prove the second assertion. By adding vertices to $\G_1$, we do not change the group $\RL_v(\G_1)$ nor the distribution of the associated random variables, according to Lemma \ref{inv subd unif}. Hence, we may add to $\G_1$ the vertices of $\G_2$ located on $\Sk(\G_1)$ and, without loss of generality, assume that $\E_1\subset \E_2$. 

For the sake of simplicity, let us treat the case where only one face of $\G_1$ contains the interior of an edge of $\E_2\setminus \E_1$. Once the result is proved under this restrictive assumption, the general result follows by iteration.

Let us assume that only the face $F_1$ contains the interior of some edges of $\G_2$. The choice of the face $F_1$ is a simple matter of convenience, we do not use anything specific to this face. 

Let $(M',\G'_1,\iota,f)$ be a split pattern of $(M,\G_1)$. Let $M'_{F_1}$ be the closure of the preimage by $f$ of $F_1$. Let $\G'_2$ be the restriction to $M'_{F_1}$ of the lift to $M'$ of $\G_2$. Let $v_1$ be the finishing point of the spoke of the lasso $l_1$. Let us choose a vertex $v'_1$ of $\G'_2$ which is sent by $f$ on $v_1$. 

The first assertion of the proposition that we are proving applied to the graph $\G'_2$ on the disk $M'_{F_1}$, whose reduced genus is $0$, provides us with $\f_1$ facial lassos $l'_{1},\ldots,l'_{\f_{2,1}}$ based at $v'_1$ which bound the faces of $\G'_2$ and such that such that $l'_{1}\ldots l'_{\f_{2,1}}$ bounds $M'_2$. By projecting these lassos on $M$ by $f$ and conjugating them by the spoke of $l_1$, we get lassos based at $v$ which we denote by $l_{1,1},\ldots,l_{1,\f_{2,1}}$. 

Let $\Gen_1$ denote the tame set of generators of $\RL_v(\G_1)$ that we are given. Let us prove that the set $\Gen_2$ of loops obtained by replacing $l_1$ by $l_{1,1},\ldots,l_{1,\f_{2,1}}$ in $\Gen_1$ is a set which generates $\RL_v(\G_2)$. Let $c$ be a loop based at $v$ in $\G_2$. Let us split $c$ into a concatenation of paths which are either paths in $\G_1$ or concatenation of edges of $\E_2\setminus \E_1$. We get an expression of the form $c=c_1d_1 \ldots c_n d_n$, where the paths $c_1,\ldots,c_n$ are  in $\G_1$ and the paths $d_1,\ldots,d_n$ are concatenations of edges of $\E_2\setminus \E_1$. Choose $k\in\{1,\ldots,n\}$ and consider the path $d_k$. It can be lifted in a unique way to a path in $\G'_2$, which we denote by $d'_k$. Even if $d_k$ is a loop, $d'_k$ needs not be a loop. However, there exist two paths $a'_k$ and $b'_k$ in $\G'_2$ which stay on the boundary of $M'_{F_1}$ and such that $a'_k d'_k b'_k$ is a loop based at $v'_1$. Let us write $a_k=f(a'_k)$ and $b_k=f(b'_k)$, and denote the spoke of the lasso $l_1$ by $s_1$. Finally, let $f_k$ be a path in $\G_1$ from $\overline{d_k}$ to $v$.  We have the equality in $\RL_v(\G_2)$
$$c=(c_1 a_1^{-1} s_1^{-1})  [ s_1 a_1 d_1 b_1 s_1^{-1} ]  (s_1 b_1^{-1} f_1) \ldots ( f_{n-1}^{-1} c_n a_n^{-1} s_1^{-1}) [s_1 a_n d_n b_n s_1^{-1}] (s_1 b_n^{-1}).$$ 
The loops between brackets are loops of $\RL_v(\G_1)$ and the loops between square brackets are the image by $f$ of loops in $\G'_2$, conjugated by the spoke of $l_1$, hence, equal to words in the loops $l_{1,1},\ldots,l_{1,\f_{2,1}}$. 
Hence, the loops of $\Gen_2$ generate $\RL_v(\G_2)$. They satisfy the equation
$$w(a_1,\ldots,a_{\rg}) c_1 \ldots c_{\p}= l_{1,1} \ldots l_{1,\f_1} l_2 \ldots l_{\f_1}.$$
In particular, the set $\Gen_2\setminus \{l_{1,1}\}$ for instance has cardinal $\rg+\p+\f(\G_2)-1$ and generates  $\RL_v(\G_2)$ which is free of rank $\rg+\p+\f(\G_2)-1$. Hence, it is a free basis of this group (see Proposition 2.7 in Chapter 1 of \cite{LyndonSchupp}). 

There remains to determine the distribution of $\{h(l) : l\in\Gen_2\}$ under $\U^{\G_2}_{M,\varnothing,C}$. It follows from the way in which the loops $l_{1,1},\ldots,l_{1,\f_{2,1}}$ were constructed that the random variables $h(l_{1,1}),\ldots,h(l_{1,\f_{2,1}-1})$ are independent, uniformly distributed, and independent of $\sigma(h(l) : l\in\Gen_1)$. Moreover, $h(l_1)$ is independent of $\sigma(h(l) : l \in\Gen_1\setminus \{l_1\})$. Hence, the three $\sigma$-fields
$$\sigma(h(l_{1,1}),\ldots,h(l_{1,\f_{2,1}-1})) ,\; \sigma(h(l_1)) ,\; \sigma(h(l) : l \in\Gen_1\setminus \{l_1\})$$
are independent. Since $h(l_{1,1}),\ldots,h(l_{1,\f_{2,1}-1})$ and $h(l_1)$ are uniformly distributed, and $l_{1,\f_{2,1}}= (l_{1,1} \ldots l_{1,\f_{2,1}-1})^{-1}l_1$, it follows that $h(l_{1,1}),\ldots,h(l_{1,\f_{2,1}})$ are uniformly distributed and independent of $ \sigma(h(l) : l \in\Gen_2\setminus \{l_{1,1},\ldots,l_{1,\f_{2,1}}\})$. \end{proof}

\chapter{Markovian holonomy fields}

In this chapter, which is the core of this work, we define Markovian holonomy fields and their discrete analogues. We prove in full generality that the partition functions of a discrete Markovian holonomy field do not depend on the graph in which they are computed. We then prove the first main result of this work, which asserts that any discrete Markovian holonomy field which satisfies some regularity conditions can be extended in a unique way to Markovian holonomy field. 

\section{Definition} \label{section def MHF}

\begin{definition} A {\em measured marked surface with $G$-constraints} is a quadruple $(M,\vol,\CS,C)$ where $(M,\CS)$ is a marked surface, $\vol$ is a smooth non-vanishing density on $M$ and $C$ is a set of $G$-constraints on $(M,\CS)$.

Two measured marked surfaces (resp. orientable measured marked surfaces) with $G$-constraints $(M,\vol,\CS,C)$
and $(M',\vol',\CS',C')$ are {\em isomorphic} if there exists a diffeomorphism (resp. an
orientation preserving diffeomorphism) $\psi:M\to M'$ such that $\psi_{*}\vol=\vol'$, $\psi$
sends each curve of $\CS$ to a curve of $\CS'$ and, for all $l\in \CS$, $C'(\psi(l))=C(l)$.  
\end{definition}

From now on, we will make the assumption that $G$ is a compact Lie group, not necessarily connected. If $\dim G\geq1$, we endow $G$ with a bi-invariant  Riemannian metric which we normalize in such a way that the Riemannian volume of $G$ is 1. In this way, the Riemannian density coincides with the normalized Haar measure on $G$. If $G$ is finite, we endow it with the distance $d_{G}(x,y)=\delta_{x,y}$ and with the uniform probability measure. 

The set $\Conj(G)$ of conjugacy classes of $G$ inherits the quotient topology from $G$ and the
corresponding Borel $\sigma$-field. The space of $G$-constraints on a marked surface $(M,\CS)$,
which we denote by $\Const_{G}(M,\CS)$, is a subset of $\Conj(G)^{\CS\cup \BS(M)}$ and thus carries the trace
topology and $\sigma$-field. 
\index{CAconj@$\Conj(G)$}
\index{CAconst@$\Const(G)$}

Let us introduce some more notation. Let $(M,\vol,\CS,C)$ be a measured marked surface with $G$-constraints. Let
$l$ be a curve which belongs to $\CS\cup \BS(M)$. Let $x$ be an element of $G$. We define a new set of $G$-constraints $C_{l\mapsto x}$ as follows: we set $C_{l\mapsto x}(l)=\O_{x}$, $C_{l\mapsto x}(l^{-1})=\O_{x^{-1}}$ and $C_{l\mapsto x}=C$ on $\CS\cup \BS(M) - \{l,l^{-1}\}$. The main definition of this work is the following.

\index{HAHF@$\HF_{M,\vol,\CS,C}$}
\index{Markovian holonomy field!definition}
\begin{definition} \label{def MHF} A $G$-valued two-dimensional Markovian holonomy field is the data,
for
each quadruplet $(M,\vol,\CS,C)$ consisting of a marked surface endowed with a density and a set of
$G$-constraints, of a finite measure $\HF_{M,\vol,\CS,C}$ on $(\M(\Path(M),G),\I)$ such
that the following properties are satisfied.\\

\Ax{1}. For all $(M,\vol,\CS,C)$, $\HF_{M,\vol,\CS,C}\left(\exists l\in \CS \cup \BS(M), h(l)\notin
C(l)\right)=0$.\\

\Ax{2}.  For all $(M,\vol,\CS)$ and all event $\Gamma$ in the invariant $\sigma$-field of
$\M(\Path(M),G)$, the function $C\mapsto \HF_{M,\vol,\CS,C}(\Gamma)$ is a measurable function on the set $\Const_G(M,\CS)$.\\

\Ax{3}. For all $(M,\vol,\CS,C)$ and all $l\in \CS$,
$$\HF_{M,\vol,\CS-\{l,l^{-1}\},C_{|\BS(M)\cup\CS-\{l,l^{-1}\}}}=\int_{G}
\HF_{M,\vol,\CS,C_{l\mapsto x}} \; dx.$$

\Ax{4}. Let $\psi:(M,\vol,\CS,C)\to (M',\vol',\CS',C')$ be a homeomorphism such that $\vol \circ \psi^{-1}=\vol'$, $\psi(\CS)=\CS'$ and $C=C'\circ \psi$. Let $l_1,\ldots,l_n$ be loops based at the same point on $M$. Assume that their images $l'_1,\ldots,l'_n$ by $\psi$ are also rectifiable loops. Then, for all continuous function $f:G^n\to G$ invariant under the diagonal action of $G$ by conjugation,
\begin{eqnarray*}
\int_{\M(\Path(M),G)} f(h(l_1),\ldots,h(l_n)) \; \HF_{M,\vol,\CS,C}(dh) =&&\\
&&\hskip -4cm  \int_{\M(\Path(M'),G)} f(h'(l'_1),\ldots,h'(l'_n)) \; \HF_{M',\vol',\CS',C'}(dh').
\end{eqnarray*}
In particular, if $\psi$ is a diffeomorphism, then the mapping from
$\M(\Path(M'),G)$ to $\M(\Path(M),G)$ induced by $\psi$ sends the measure $\HF_{M',\vol',\CS',C'}$
to the measure $\HF_{M,\vol,\CS,C}$.\\

\Ax{5}. For all $(M_{1},\vol_{1},\CS_{1},C_{1})$ and $(M_{2},\vol_{2},\CS_{2},C_{2})$, one has the
identity 
$$\HF_{M_{1}\cup M_{2},\vol_{1}\cup \vol_{2},\CS_{1}\cup \CS_{2},C_{1}\cup C_{2}}=
\HF_{M_{1},\vol_{1},\CS_{1},C_{1}}\otimes \HF_{M_{2},\vol_{2},\CS_{2},C_{2}}.$$

\Ax{6}. For all $(M,\vol,\CS,C)$, all $l\in \CS$ and all gluing $\psi:\Spl_l(M)\to M$ along $l$, one has
$$\HF_{{\Spl}_{l}(M),\Spl_l(\vol),\Spl_l(\CS),\Spl_l(C)}=\HF_{M,\vol,\CS,C}\circ \psi^{-1}.$$

\Ax{7}. For all $(M,\vol,\varnothing,C)$ and for all $l\in \BS(M)$, 
$$\int_{G} \HF_{M,\vol,\varnothing,C_{l\mapsto x}}(1) \; dx=1.$$

A $G$-valued two-dimensional {\em oriented} Markovian holonomy field is the data, for each
quadruplet $(M,\vol,\CS,C)$ consisting of an {\em oriented} marked surface endowed with a density and
a set of $G$-constraints, of a finite measure $\HF_{M,\vol,\CS,C}$ on the measurable space $(\M(\Path(M),G),\I)$ such
that the seven properties above are satisfied.
\end{definition}

It is important to notice that the measures $\HF$ are {\em not} probability measures in
general. They are finite measures, whose masses carry a lot of information about the holonomy
field. It is actually possible that they characterize it completely, but this is a question which has yet to be answered.

Let us discuss briefly the significance of these axioms. The axioms A$_1$, A$_2$ and A$_3$ express
the fact that the measure $\HF_{M,\vol,\CS,C}$, seen as a function of the $G$-constraints, provides a regular disintegration of $\HF_{M,\vol,\varnothing,C_{|\BS(M)}}$ with respect to the value of the holonomy field along the curves of $\CS$. The simple
expression of A$_3$ is permitted by the fact that we consider finite measures rather than probability
measures. The axiom A$_4$ expresses the invariance of the field under area-preserving diffeomorphisms.
The axioms A$_5$ and A$_6$ express the Markov property of the field. Finally, A$_7$ is a normalization axiom.
Without it, if $\HF$ was a given Markovian holonomy field, then for any real $\alpha$, the measures $e^{\alpha \vol(M)} \HF_{M,\vol,\CS,C}$ would constitute another Markovian holonomy field. 

Our purpose is not to study Markovian holonomy fields in full generality. In the rest of this paper, we are going to
make strong regularity assumptions and investigate the corresponding fields. 

Recall that $d_G$ denotes a bi-invariant distance on $G$. Let $c$ and $c'$ be two paths with the same endpoints. Then, although $h_{c}$ and $h_{c'}$ are not measurable with respect with the invariant $\sigma$-field $\I$, unless $c$ and $c'$ are loops, the function $h\mapsto d_{G}(h(c),h(c'))$ is $\I$-measurable, because it can also be written as $h\mapsto d_G(1,h(c)^{-1}h(c'))=d_G(1,h(c'c^{-1}))$ and $c'c^{-1}$ is a loop.

\index{Markovian holonomy field!stochastically continuous}
\index{Markovian holonomy field!Fellerian}
\index{Markovian holonomy field!regular}

\begin{definition} \label{def SCFMHF}Let $\HF$ be a $G$-valued two-dimensional Markovian holonomy
field. \\
1. We say that
$\HF$ is {\em stochastically continuous} if, for all $(M,\vol,\CS,C)$ and for all sequence
$(c_{n})_{n\geq 0}$ of elements of $\Path(M)$ which converges with fixed endpoints to $c\in
\Path(M)$,
$$\int_{\M(\Path(M),G)} d_{G}(h(c_{n}),h(c))\; \HF_{M,\vol,\CS,C}(dh)
\build{\lra}_{n\to\infty}^{}0.$$
2. We say that $\HF$ is {\em Fellerian} if, for all  $(M,\vol,\CS)$, the function 
$$(t,C)\mapsto \HF_{M,t\vol,\CS,C}(1)$$
defined on $\RK^*_{+}\times \Const_{G}(M,\CS)$ is continuous.\\
3. We say that $\HF$ is {\em regular} if it is both stochastically continuous and Fellerian.
\end{definition}

In the definition of stochastic continuity, we use $L^1$ convergence of $G$-valued random variables.
Since $G$ is compact, this is equivalent to convergence in probability and to convergence in $L^p$ for any
$p\in[1,+\infty)$. 

From now on, the expression {\em (regular) Markovian holonomy field} will stand for
{\em two-dimensional $G$-valued (regular) Markovian holonomy field}.

\section{Discrete Markovian holonomy fields}

It is not easy to construct a Markovian holonomy field. Indeed, one has to construct a stochastic process indexed by loops. To do this, one must naturally specify the finite-dimensional marginals of this process. Thus, to each finite collection of loops, one has to associate a probability measure on some power of $G$. But unlike points on a time interval, loops on a surface may form a very complicated picture. In fact, in most cases, it is impossible to determine a probability measure from a finite set of loops. A way around this problem is to start by describing a restriction of the process to a class of loops which are nice enough, like piecewise geodesic loops, and then to extend the process by continuity. In fact, the very first step is to build a process indexed by the set of loops in a graph for every graph on a surface. This is what we call a discrete holonomy field.

\index{Markovian holonomy field!discrete}
\begin{definition} \label{def discrete MHF} A $G$-valued two-dimensional discrete Markovian holonomy field is the data,
for each measured marked surface $(M,\vol,\CS,C)$ with $G$-constraints, and each graph $\G$ on $(M,\CS)$, of a finite measure $\DF^\G_{M,\vol,\CS,C}$ on $(\M(\Path(\G),G),\I)$ such that the following properties hold.

\AxD{1}. For all $(M,\vol,\CS,C)$ and all $\G$, $\DF^\G_{M,\vol,\CS,C}\left(\exists l\in \CS \cup \BS(M), h(l)\notin
C(l)\right)=0$.\\

\AxD{2}.  For all $(M,\vol,\CS)$, all $\G$  and all event $\Gamma$ in the invariant $\sigma$-field of
$\M(\Path(\G),G)$, the function $C\mapsto \DF^\G_{M,\vol,\CS,C}(\Gamma)$ is a measurable function on $\Const_G(M,\CS)$.\\

\AxD{3}. For all $(M,\vol,\CS,C)$, all $\G$ and all $l\in \CS$,
$$\DF^\G_{M,\vol,\CS-\{l,l^{-1}\},C_{|\BS(M)\cup\CS-\{l,l^{-1}\}}}=\int_{G}
\DF^\G_{M,\vol,\CS,C_{l\mapsto x}} \; dx.$$

\AxD{4}. Consider $(M,\vol,\CS,C)$ and $(M',\vol',\CS',C')$, endowed respectively with $\G$ and $\G'$. Let $\psi:M \to M'$ be a homeomorphism. Assume that $\vol \circ \psi^{-1}=\vol'$, $\psi(\CS)=\CS'$, $C'\circ \psi=C$ and $\psi(\G)=\G'$. Then the mapping from $\M(\Path(\G'),G)$ to $\M(\Path(\G),G)$ induced by $\psi$ satisfies
$$\DF^{\G'}_{M',\vol',\CS',C'}\circ \psi^{-1}=\DF^\G_{M,\vol,\CS,C}.$$

\AxD{5}. For all $(M_{1},\vol_{1},\CS_{1},C_{1})$ and $(M_{2},\vol_{2},\CS_{2},C_{2})$, endowed respectively with two graphs $\G_1$ and $\G_2$ one has the identity 
$$\DF^{\G_1\cup\G_2}_{M_{1}\cup M_{2},\vol_{1}\cup \vol_{2},\CS_{1}\cup \CS_{2},C_{1}\cup C_{2}}=
\DF^{\G_1}_{M_{1},\vol_{1},\CS_{1},C_{1}}\otimes \DF^{\G_2}_{M_{2},\vol_{2},\CS_{2},C_{2}}.$$

\AxD{6}. For all $(M,\vol,\CS,C)$, all $\G$, all $l\in \CS$ and all gluing $\psi:\Spl_l(M)\to M$ along $l$, one has
$$\DF^{\Spl_l(\G)}_{{\Spl}_{l}(M),\Spl_l(\vol),\Spl_l(\CS),\Spl_l(C)}=\DF^\G_{M,\vol,\CS,C}\circ \psi^{-1}.$$

\AxD{7}. For all $(M,\vol,\varnothing,C)$, all $\G$ and all $l\in \BS(M)$, 
$$\int_{G} \DF^\G_{M,\vol,\varnothing,C_{l\mapsto x}}(1) \; dx=1.$$

\AxD{I}. For all $(M,\vol,\CS,C)$, all $\G_1$ and $\G_2$ graphs on $(M,\CS)$ such that $\G_1\preccurlyeq \G_2$, 
the restriction map $r:\M(\Path(\G_2),G)\to \M(\Path(\G_1),G)$ satisfies
$$\DF^{\G_2}_{M,\vol,\CS,C} \circ r^{-1}=\DF^{\G_1}_{M,\vol,\CS,C}.$$

A $G$-valued two-dimensional discrete {\em oriented} Markovian holonomy field is the data, for each
{\em oriented} measured marked surface $(M,\vol,\CS,C)$ with $G$-constraints and all graph $\G$ on $(M,\CS)$, of a finite measure $\DF^\G_{M,\vol,\CS,C}$ on $(\M(\Path(\G),G),\I)$ such
that the properties above are satisfied.
\end{definition}

The axiom \AxD{I} is specific to discrete holonomy fields. It is an axiom of consistency and is usually called the property of invariance under subdivision. 

\begin{lemma} Any (oriented) Markovian holonomy field determines a discrete (oriented) Markovian holonomy field. 
\end{lemma}

\begin{proof}For all surface $M$ endowed with a graph $\G$, the set $\Path(\G)$ is a subset of $\Path(M)$. Hence, a Markovian holonomy field determines by restriction a discrete Markovian holonomy field, except perhaps for the axiom \AxD{I}. In fact, this axiom is satisfied by the restriction of a Markovian holonomy field because the finite dimensional marginals of a stochastic process constitute a consistent system of probability measures. \end{proof}

Our main goal in this chapter is to prove a result in the other direction and to extend, when it is possible, a discrete holonomy field to a continuous one. For the moment, let us discuss briefly the discrete holonomy fields for themselves. It turns out that we have already constructed a fundamental example of discrete Markovian holonomy field.

\begin{proposition} \label{U DMHF} The family of measures $\U^\G_{M,\CS,C}$ is a discrete Markovian holonomy field. We call it the {\em uniform $G$-valued discrete holonomy field}.
\end{proposition}

\begin{proof}The measure $\U^\G_{M,\CS,C}$ is a probability measure on the cylinder $\sigma$-field of $\M(\Path(\G),G)$. By restriction, it defines a measure on the invariant $\sigma$-field. The axiom \AxD{1} is satisfied by Proposition \ref{support U}. By Proposition \ref{gen invariant sigma} and the definition of the measures $\delta_{\O(n)}$, it is possible to write $\U_{M,\CS,C}(\Gamma)$ as an expression which is explicitly measurable with respect to the $G$-constraints. Hence, the axiom \AxD{2} is satisfied. Axiom \AxD{3} is satisfied by (\ref{recover unif}). That axioms \AxD{4}, \AxD{5} and \AxD{7} hold is straightforward. Axiom \AxD{6} is satisfied thanks to Proposition \ref{prop markov unif}. The invariance property \AxD{I} holds by Proposition  \ref{inv subd unif}.
\end{proof}

This discrete Markovian holonomy field is very special in that is consists only in probability measures. 

\index{Markovian holonomy field!partition functions}
\begin{definition} Let $\DF$ be a collection of finite measures as in Definition \ref{def discrete MHF}, which does not necessarily satisfy any of the axioms listed in this definition. To each $(M,\CS,\vol,C)$ and each graph $\G$, we associate the number
$$Z^\G_{M,\vol,\CS,C}=\DF^\G_{M,\vol,\CS,C}(1),$$
which is called the {\em partition function}.
\end{definition}

We have said earlier that these numbers carry a lot of information about the field. A crucial property is that they do not depend on the graph $\G$.

\begin{proposition}\label{456 Z} Let $\DF$ be a collection of finite measures as in Definition \ref{def discrete MHF}, which satisfies the axioms \AxD{4}, \AxD{5}, \AxD{6} and \AxD{I}. Consider $(M,\vol,\CS,C)$ and two graphs $\G_1$ and $\G_2$ on $(M,\CS)$. Then
\begin{equation}\label{egalite des Z}
Z^{\G_1}_{M,\vol,\CS,C}=Z^{\G_2}_{M,\vol,\CS,C}.
\end{equation}
\end{proposition}

If there exists a graph $\G_3$ such that $\G_1\preccurlyeq \G_3$ and $\G_2\preccurlyeq \G_3$, then the equality (\ref{egalite des Z}) is a straightforward consequence of the axiom \AxD{I}. Unfortunately, by Lemma \ref{not directed}, the set of graphs on $M$ is not directed and such a graph $\G_3$ does not always exist.\\

\begin{proof}By Proposition \ref{complete splitting}, it is possible to split $M$ along each of the curves of $\CS$. Thus, there exists a measured surface $(M',\vol',\varnothing,C')$ with $G$-constraints and no marks and a gluing $f:M'\to M$ whose joint is $\CS$. We can lift $\G_1$ to a graph $\G_1'$ on $M'$ and the axiom \AxD{6} enforces the equality $Z^{\G'_1}_{M',\vol',\varnothing,C'}=Z^{\G_1}_{M,\vol,\CS,C}$. Since we can do the same for the graph $\G_2$, it suffices to prove the result when $\CS=\varnothing$. Then, the axiom \AxD{5} implies that the partition function associated to $\G_1$ (resp. $\G_2$) is the product of the partition functions associated to the connected components of $M$ endowed with the corresponding restrictions of $\G_1$ (resp. $\G_2$). 

Hence, it suffices to prove the result when $M$ is connected and $\CS=\varnothing$. In this case, the axiom \AxD{I} implies that we can remove or add edges to $\G_1$, in the sense of Propositions \ref{erase edge} and \ref{add edge}, without altering the partition function.

By Proposition \ref{graph to graph}, by such transformations we can go from  $\G_1$ to a graph which is sent to $\G_2$ by a homeomorphism of $M$. By \AxD{4}, this implies the result. \end{proof}

Let $\DF$ be a discrete holonomy field. In order to produce a Markovian holonomy field from $\DF$, the first natural step is to apply Proposition \ref{proj lim gen} to put together the measures which $\DF$ associates to a directed subset of the set of graphs on a surface. For this, we need to consider Riemannian metrics. 

Let $(M,\vol,\gamma,\CS)$ be a Riemannian marked surface (see Definition \ref{riem mark surf}). The axiom \AxD{I} and Proposition \ref{take proj lim} imply that the collection of the measures $\DF^{\G}_{M,\vol,\CS,C}$, where $\G$ spans the set $\GG_\gamma(M,\CS)$ of graphs with piecewise geodesic edges, determines a finite measure on $\M(\GPath_\gamma(M),G)$, where $\GPath_\gamma(M)$ is the set of piecewise geodesic paths on $M$.

\begin{definition} Let $\DF$ be a discrete holonomy field. Let $(M,\vol,\gamma,\CS,C)$ be a Riemannian measured marked surface with $G$-constraints. The finite measure on $\M(\GPath_\gamma(M),G)$ obtained by taking the projective limit of the measures $\DF^\G_{M,\vol,\CS,C}$ where $\G$ spans $\GG_\gamma(M,\CS)$ is denoted by $\DF^{\gamma}_{M,\vol,\CS,C}$.
\end{definition}

At this point, there are two things to do. Firstly, one needs to extend the measure $\DF^\gamma_{M,\vol,\CS,C}$ to $\M(\Path(M),G)$ for all $(M,\vol,\CS,C)$ and secondly, one needs to prove that the result of this procedure would have been the same with another Riemannian metric. The first step requires some regularity from the holonomy field.

\index{Markovian holonomy field!discrete!stochastically $\frac{1}{2}$-H\"{o}lder continuous}
\begin{definition} Let $\DF$ be a discrete holonomy field. We say that $\DF$ is {\em stochastically $\frac{1}{2}$-H\"{o}lder continuous} if the following property holds.

Let $(M,\vol,\gamma,\CS,C)$ be a Riemannian marked surface with $G$-constraints. 
Then there exists a constant $K>0$ such that for all graph $\G$ on $M$ and all simple loop $l\in \Path(\G)$ with
$\ell(l)\leq K^{-1}$ bounding a disk $D$,
$$\int_{\M(\Path(\G),G)} d_{G}(1,h(l)) \; \DF^\G_{M,\vol,\CS,C}(dh) \leq K \sqrt{\vol(D)}.$$
\end{definition}

Our unique example so far of a discrete holonomy field, the uniform holonomy field, does unfortunately not satisfy this property. Indeed, it assigns uniform random variables to arbitrary small simple loops. We will construct other more regular discrete holonomy fields later. 

The second step of the program outlined above requires another regularity condition, which is more global but less quantitative. 

\index{Markovian holonomy field!discrete!continuously area-dependent}
\index{Markovian holonomy field!discrete!Fellerian}
\index{Markovian holonomy field!discrete!regular}
\begin{definition}\label{def regular DHF} Let $\DF$ be a discrete holonomy field. 

1. We say that $\DF$ is {\em continuously area-dependent} if the following property holds.

Let $(M,\vol,\CS,C)$ be a marked surface with $G$-constraints. Let $(\G_n)_{n\geq 0}$ be a sequence of graphs on $(M,\CS)$. Let $(\psi_n)_{n\geq 0}$ be a sequence of homeomorphisms of $M$. Assume that $\psi_n(\G)=\G_n$ for all $n\geq 0$ and, for all face $F$ of $\G$, $\vol(\psi_n(F))$ tends to $\vol(F)$ as $n$ tends to infinity.  

Then, denoting by $\psi_n$ the induced map $\M(\Path(\G_n),G)\to \M(\Path(\G),G)$, we have the weak convergence
$$\DF^{\G_n}_{M,\vol,\CS,C} \circ \psi_n^{-1} \build{\Longrightarrow}_{n\to\infty}^{} \DF^{\G}_{M,\vol,\CS,C}.$$

2. We say that that $\DF$ is {\em Fellerian} if for all measured marked surface $(M,\vol,\CS)$, the mapping which to a set of $G$-constraints $C$ associates the number $Z_{M,\vol,\CS,C}$ is continuous on $\Const_G(M,\CS)$.

3. We say that $\DF$ is {\em regular} if it is stochastically $\frac{1}{2}$-H\"{o}lder continuous, continuously area-dependent and Fellerian.
\end{definition}

It is tempting to conjecture that a stochastically $\frac{1}{2}$-H\"{o}lder continuous discrete Markovian holonomy field is continuously area-dependent. At least, the two properties are not equivalent, as our unique example so far shows.
For the moment, let us state our main result concerning the construction of Markovian holonomy fields.

\begin{theorem}\label{main existence} Every regular discrete Markovian holonomy field is the restriction of a unique regular Markovian holonomy field.
\end{theorem}

The proof of this theorem occupies the rest of this chapter.

\section{An abstract extension theorem}

The core of the proof of Theorem \ref{main existence} is the next result, which we formulate in a way which is mostly independent of the context of Markovian holonomy fields. 

\begin{theorem}\label{main extension} Let $(M,\gamma)$ be a compact Riemannian surface. Let $\vol$ denote the Riemannian volume of $\gamma$. Let $(\Gamma,d)$ be a complete metric group on which translations and inversion are isometries. Let $H \in \M(\GPath_\gamma(M), \Gamma)$ be a multiplicative function. Assume that there exists $K> 0$ such that for all simple loop $l\in \GPath_\gamma(M)$ bounding a disk $D$ and such that $\ell(l)\leq K^{-1}$, the inequality $d(1,H(l))\leq K \sqrt{\vol(D)}$ holds.

Then $H$ admits a unique extension to an element of $\M(\Path(M),G)$, also denoted by $H$, such that if a sequence $(c_n)_{n\geq 0}$ of paths converges with fixed endpoints to a path $c$, then $H(c_n)\build{\lra}_{n\to\infty}^{} H(c)$.
\end{theorem}

Let us state right now the application of this theorem to the extension of holonomy fields.

\begin{corollary}\label{cor : def DF(g)} Let $\DF$ be a stochastically $\frac{1}{2}$-H\"{o}lder continuous discrete Markovian holonomy field. Let $(M,\vol,\gamma,\CS,C)$ be a Riemannian marked surface 
with $G$-constraints. There exists a unique probability measure $\DF^{(\gamma)}_{M,\vol,\CS,C}$ on $(\M(\Path(M),G),\C)$ whose image by the restriction mapping $\M(\Path(M),G)\to\M(\GPath_\gamma(M),G)$ is the measure $\DF^{\gamma}_{M,\vol,\CS,C}$ and such that for all sequence $(c_n)_{n\geq 0}$ of paths which converges with fixed endpoints to a path $c$, one has
$$\int_{\M(\Path(M),G)} d_G(h(c_n),h(c)) \; \DF^{(\gamma)}_{M,\vol,\CS,C}(dh)\build{\lra}_{n\to\infty}^{} 0.$$
\end{corollary}

\begin{proof}The canonical process  $(H_{c})_{c\in \GPath_\gamma(M)}$ on $(\M(\GPath_\gamma(G),G),\C,\DF^{\gamma}_{M,\vol,\CS,C})$ can be seen as a mapping from $\GPath_\gamma (M)$ to the set $\Gamma=L^1(\M(\GPath_\gamma(M),G),{\mathcal C},\DF^{\gamma}_{M,\vol,\CS,C};G)$. In general, the set $L^1(\Omega,\A,\P;G)$ of $G$-valued random variables on any probability space $(\Omega,\A,\P)$ is a group for the multiplication of random variables. The metric $d(H_{1},H_{2})=\E[d_{G}(H_{1},H_{2})]$ makes it a complete metric space. Both structures are compatible in that the product and inverse mappings are continuous. Moreover, the translations and the inversion map are isometries. 

The assumption of stochastic H\"{o}lder continuity ensures that the regularity condition is satisfied by the family $(H_{c})_{c\in \GPath_\gamma(M)}$. Hence, we can apply Theorem \ref{main extension}. It produces a family of random variables $(H_c)_{c\in\Path(M)}$ which is multiplicative and continuous with respect to the convergence of paths with fixed endpoints. Proposition \ref{take proj lim} applied to this family asserts that there exists a probability measure on $(\M(\Path(M),G),\C)$ under which the canonical process has the distribution of $(H_c)_{c\in\Path(M)}$. 

The uniqueness of the measure $\DF^{(\gamma)}_{M,\vol,\CS,C}$ follows from Proposition \ref{A dense} which asserts  that $\GPath_\gamma(M)$ is dense in $\Path(M)$ for the convergence with fixed endpoints.  \end{proof}

The rest of this section is devoted to the proof of Theorem \ref{main extension}. A basic tool for this proof is the lasso decomposition of a piecewise geodesic path (see Proposition \ref{lasso dec}).

\begin{proposition}\label{estim 1} Under the assumptions of Theorem \ref{main extension}, there
exists a constant $K> 0$ such that for every loop $l\in \GPath_\gamma(M)$ with $\ell(l)\leq K^{-1}$, $d(1,H(l))\leq K \ell(l)$.
\end{proposition}

\begin{proof}Assume first that $l$ is a simple loop. If $\ell(l)$ is small enough, then $l$ bounds a disk,
which we denote by $D$. Moreover, a local isoperimetric inequality holds on $M$: if $\ell(l)$ is
small enough, say $\ell(l)<L$, then $\sqrt{\vol(D)}\leq K_{1} \ell(l)$ for some constant $K>0$.
Then, if $K$ denotes the constant given by the assumptions of Theorem \ref{main extension},
$d(1,H(l))\leq K \sqrt{\vol(D)}\leq K K_{1}\ell(l)$. The result is proved in this case. 

Let us now treat the general case. Let us apply Proposition \ref{lasso dec} to find the lasso
decomposition $l=l_{1}\ldots l_{p} d$ of $l$. By the multiplicativity of $H$, which is part of the
assumptions of Theorem \ref{main extension}, $H(l)=H(d)H(l_{p})\ldots H(l_{1})$. Since the distance
$d$ on $\Gamma$ is invariant by left translations and inversion, we have
$$\forall x,y\in \Gamma, \; d(1,xy)=d(x^{-1},y) \leq d(1,x^{-1})+ d(1,y)=d(1,x)+d(1,y).$$
Hence, $d(1,H(l))\leq d(1,H(d))+\sum_{i=1}^{p} d(1,H(l_{i}))\leq K
\ell(d)+K\sum_{i=1}^{p}\ell(l_{i})$. By Proposition \ref{lasso dec}, $\ell(d)+
\sum_{i=1}^{p}\ell(l_{i})\leq \ell(l)$ and the result follows. \end{proof}

This result tells us that if $l$ is a simple piecewise geodesic loop close to the constant loop, then $H(l)$ is close to the image by $H$ of the constant loop. Our next generalizes this statement to the case of a piecewise geodesic path which is close to a geodesic segment.

\begin{proposition} \label{estim 2} Under the assumptions of Theorem \ref{main extension}, there
exists a constant $K> 0$ such that the following property holds.
Let $s\in \GPath_\gamma(M)$ be a segment of minimizing geodesic. Let $c$ be a piecewise geodesic path with
the same endpoints as $s$. Assume that $\ell(c)\leq K^{-1}$. Then
$$d(H(c),H(s))\leq K \ell(c)^{\frac{3}{4}} |\ell(c)-\ell(s)|^{\frac{1}{4}}.$$ 
\end{proposition}

The assumptions imply that $s$ is shorter than $c$. If we take $\ell(s)=0$, then $c$
is a loop and we recover the estimate of Proposition \ref{estim 1}.

In order to prove this proposition, we use a decomposition result similar to the lasso
decomposition and the following isoperimetric inequality. Needless to say, the constant $\frac{\pi}{3\sqrt{2}}$ which appears in this Euclidean case is not optimal.

\begin{lemma} \label{isop} Let $R,r\geq 0$ be real numbers. Let $J$ be a
rectifiable Jordan curve of length $2R+r$ in the Euclidean plane such that a piece of this curve is a segment
of length $R$. Then the area $A$ of the bounded component of $\RK^2\setminus J$ satisfies the
inequality 
$$A \leq \frac{\pi}{3\sqrt{2}}r^{\frac{1}{2}}(R+r)^{\frac{3}{2}}.$$

Let $M$ be a compact Riemannian surface. There exists a constant $K>0$ such that the following
property holds. Let ${ J}$ be a rectifiable Jordan curve of length $2R+r<K^{-1}$ such that
a piece of this curve is a segment of minimizing geodesic of length $R$. Then the area $A$ of the
smallest disk bounded by ${J}$ satisfies the inequality 
$$A \leq K r^{\frac{1}{2}}(R+r)^{\frac{3}{2}}.$$
\end{lemma}
\index{isoperimetric inequality}

\begin{figure}[ht!]
\begin{center}
\scalebox{1}{\includegraphics{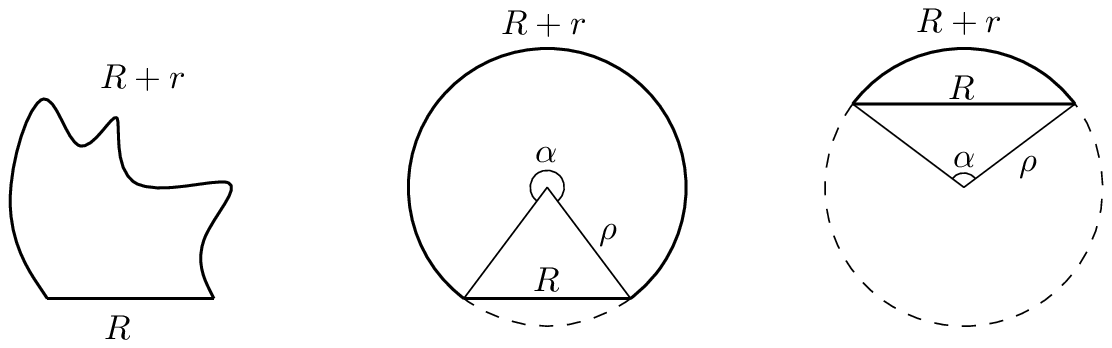}}
\caption{The generic case and two optimal cases.}
\end{center}
\end{figure}

\begin{proof}The Riemannian case can easily be deduced from the Euclidean case by working in normal
coordinates at one end point of the geodesic part of $J$. The compactness of $M$ ensures
that the resulting distortions of lengths and areas are bounded.

Let us consider the Euclidean case. Under the constraints on ${J}$, $A$ is maximal when ${
J}$ is the concatenation of a segment of length $R$ and an arc of circle of length $R+r$. In this case, let $\rho$ be the
radius of this circle and $\alpha \in [0,2\pi)$ the angle under which the arc of circle is seen from the centre of
the circle. Then $R+r=\rho \alpha$ and $R=2\rho \sin \frac{\alpha}{2}$. Now one has the relations
\begin{eqnarray}
&& A=\frac{\rho^2}{2}(\alpha-\sin \alpha)\leq \frac{\rho^2\alpha^3}{12}=\frac{\alpha}{12}(R+r)^2, \label{la 1}\\
&& \frac{\sin \frac{\alpha}{2}}{\frac{\alpha}{2}}=\frac{R}{R+r}. \label{la 2}
\end{eqnarray}
One checks easily that for all $x\in [0,\pi]$,  $\frac{\sin x}{x}\leq 1-\frac{x^2}{\pi^2}$. Hence, (\ref{la 2}) implies $\alpha\leq 2\pi \sqrt{\frac{r}{R}}$, and since $\alpha \in [0,2\pi)$, we have $\alpha\leq 2\pi \left(\sqrt{\frac{r}{R}} \wedge 1\right)$. Combining this with (\ref{la 1}), we find
$$A\leq \frac{\pi}{6} (R+r)^{\frac{3}{2}} \sqrt{r} \left(\sqrt{1+\frac{r}{R}} \wedge \sqrt{1+\frac{R}{r}}\right) \leq \frac{\pi}{3\sqrt{2}} r^{\frac{1}{2}} (R+r)^{\frac{3}{2}}.$$
This is the expected result. \end{proof}

The generalization of the lasso decomposition that we need is the following. 

\begin{proposition} \label{lasso fin} Let $M$ be a Riemannian compact surface. Let $s$ be a geodesic segment on $M$ and $c$ an injective piecewise
geodesic path with
the same endpoints as $s$. It is
possible to decompose $c$ and $s$ as concatenations $c=c_{1}\ldots c_{p}$ and $s=s_{1}\ldots
s_{p}$ in such a way that, for each $k=1,\ldots,p$, $c_{k}=s_{k}$ or $c_{k}s_{k}^{-1}$ is a simple
loop. In particular, if we set $l_{k}=(c_{1}\ldots c_{k-1}) c_{k}s_{k}^{-1}(c_{1}\ldots c_{k-1})^{-1}$,
then $l_{k}$ is a lasso and $cs^{-1}\simeq l_{1}\ldots l_{p}$. 
\end{proposition}


\begin{proof}Consider the set $c([0,1])\cap s([0,1])$. It is a reunion of isolated points and closed
subintervals of $s([0,1])$. Let $V$ be the reunion of these isolated points and the end points of
these intervals. The set $V$ contains the two end points of $s$. Set $n=\# V-1$. Then $V$ dissects
both $c$ and $s$ into $n$ edges: $c=e_{1}\ldots e_{n}$ and $s=f_{1}\ldots f_{n}$. 

For each $i\in \{1,\ldots, n\}$, define $j(i)$ by the relation $\overline{e_{i}}=\overline{f_{j(i)}}$.
Since
$c$ is injective, the mapping $j$ from $\{1,\ldots,n\}$ to itself is one-to-one. Hence, it is a
permutation. We look for the records in the sequence $j(1),j(2),\ldots, j(n)$. Define $I=\{i, 1\leq i
\leq n : j(i)=\max(j(1),\ldots,j(i))\}$ and write $I=\{i_{1}\leq \ldots\leq i_{p}\}$. Set
$J=j(I)=\{j(i_{1})\leq \ldots \leq j(i_{p})\}$. Observe that $i_{p}=n=j(i_{p})$. Set
$i_{0}=j(i_{0})=0$. For each $k=1,\ldots,p$, set $c_{k}=e_{i_{k-1}+1}\ldots
e_{i_{k}}$ and $s_{k}=f_{j(i_{k-1})+1}\ldots f_{j(i_{k})}$. By construction, $c=c_{1}\ldots c_{p}$
and $s=s_{1}\ldots s_{p}$. 

Choose $k\in \{1,\ldots,p\}$ and consider $c_{k}=e_{i_{k-1}+1}\ldots e_{i_{k}}$. Assume first that
$i_{k}=i_{k-1}+1$. Then $c_{k}=e_{i_{k}}$. Either this edge is contained in $s$, in which case
$c_{k}=e_{i_{k}}=f_{j(i_{k})}=s_{k}$, or it meets $s$ only at its endpoints, which are also those
of $s_{k}$. In this case, $c_{k}s_{k}^{-1}$ is a simple loop. Assume now that $i_{k}\geq
i_{k-1}+2$. We claim that
any point of $c_{k}$ other than one of its end points which is located on $s$ is in fact located on
one of the edges $f_{1},\ldots,f_{j(i_{k-1})}$. Indeed, if this was not the case, there would
exist $l\in\{i_{k-1}+1,\ldots,i_{k}-1\}\neq \varnothing$ such that $\overline{e_{l}}$ is located on
$s$ between $\overline{e_{i_{k}}}$ and $\overline{s}$. But then we would have $j(l)>j(i_{k})$ which
contradicts the definition of $i_{k}$. Hence, in this case, $c_{k}s_{k}^{-1}$ is a simple loop.\end{proof}

We are now ready to prove Proposition \ref{estim 2}.\\

\begin{proof}[Proof of Proposition \ref{estim 2}] Assume that $c$ and $s$ are shorter than
the bound $K^{-1}$ of Proposition \ref{estim 1}. Let $\LE(c)$ be the loop-erasure of $c$. By the
properties of the lasso decomposition of $c$ (Proposition \ref{lasso dec}) and Proposition \ref{estim 1}, 
$$ d(H(c),H(\LE(c)))\leq K(\ell(c)-d(c(0),c(1)))=K |\ell(c)-\ell(s)|.$$
Now we are reduced to consider $c'=\LE(c)$ which is an injective path. Let
$c'=c'_{1}\ldots c'_{p}$ and $s=s_{1}\ldots s_{p}$ be the decomposition given by
Proposition \ref{lasso fin}.
We have 
$$d(H(c'),H(s))\leq \sum_{i=1}^{p} d(1,H(c'_{i}s_{i}^{-1})) \leq K \sum_{i=1}^{p}
\sqrt{A_{i}},$$
where $A_{i}$ is the area enclosed by $c'_{i}s_{i}^{-1}$. By Lemma \ref{isop} and since $\ell(s)\leq \ell(c')\leq \ell(c)$,
$$A_{i}\leq K \ell(c_{i})^{\frac{3}{2}} |\ell(c_{i})-\ell(s_{i})|^{\frac{1}{2}}.$$
By H\"{o}lder inequality, it follows that
$$d(H(c'),H(s))\leq K \ell(c)^{\frac{3}{4}}|\ell(c)-\ell(s)|^{\frac{1}{4}}.$$
Hence, 
\begin{eqnarray*}
d(H(c),H(s)) &\leq& K \left(|\ell(c)-\ell(s)|+\ell(c)^{\frac{3}{4}}|\ell(c)-\ell(s)|^{\frac{1}{4}}\right) \\
&\leq& 2K \ell(c)^{\frac{3}{4}} (\ell(c)-\ell(s))^{\frac{1}{4}}.
\end{eqnarray*}
 \end{proof}

The estimate given by Proposition \ref{estim 2} will allow us to associate an element
of $\Gamma$ to every element of $\Path(M)$. Recall the definition of the dyadic piecewise geodesic approximation of a path (Definition \ref{Dn}). For a given path $c$, $\D_{n}(c)$ is in general only defined for $n$ larger than a certain integer
$n_{0}(c)$. Nevertheless, by Proposition \ref{A dense}, the sequence $(\D_{n}(c))_{n\geq n_{0}(c)}$ converges to $c$ in
1-variation. 

\begin{proposition} \label{cauchy} Let $c\in \Path(M)$. Under the assumptions of Theorem \ref{main extension}, the sequence $(H(\D_{n}(c)))_{n\geq n_{0}(c)}$
is a Cauchy sequence in $\Gamma$.
\end{proposition}

\begin{proof}Let $K$ be the constant given by Proposition \ref{estim 2}. Let $n_{1}(c)$ be an integer such
that $2^{-n_{1}(c)}\ell(c)<K^{-1}$. Choose $m\geq n\geq \max(n_{0}(c),n_{1}(c))$. 
Write $\D_{n}(c)=\sigma_{0}\ldots \sigma_{2^n-1}$ according to
the notation of Definition \ref{Dn}. Write also $\D_{m}(c)=\eta_{0}\ldots\eta_{2^n-1}$, where
for each $k\in \{0,\ldots,2^n-1\}$, $\eta_{k}=\D_{m-n}(c_{|[k2^{-n},(k+1)2^{-n}]})$.
By Proposition \ref{estim 2}, for all $k\in \{0,\ldots,2^n-1\}$, $d(H(\eta_{k}),H(\sigma_{k}))\leq K
\ell(\eta_{k})^{\frac{3}{4}}|\ell(\eta_{k})-\ell(\sigma_{k})|^{\frac{1}{4}}.$
Hence, by H\"{o}lder inequality, and since $\ell(\D_{n})(c)\leq \ell(\D_{m}(c))\leq \ell(c)$,
$$d(H(\D_{n}(c)),H(\D_{m}(c)))\leq K \ell(c)^\frac{3}{4} |\ell(c)-\ell(\D_{n}(c))|^\frac{1}{4}.$$
The result follows now from the fact that $\ell(\D_{n}(c))$ converges to $\ell(c)$. \end{proof}

By Proposition \ref{cauchy} and the assumption that $(\Gamma,d)$ is complete, it is now
legitimate to set the following definition.

\begin{definition} For each $c\in \Path(M)-\GPath_\gamma(M)$, we define $H(c)$ by
$$H(c)=\lim_{n\to \infty} H(\D_{n}(c)).$$
\end{definition}

\begin{proposition} The mapping $H:\Path(M)\lra \Gamma$ thus defined is continuous at fixed
endpoints.
\end{proposition}

Notice that we have not proved yet that $H$ is continuous even on $\GPath_\gamma(M)$.\\

\begin{proof}Take $c\in \Path(M)$ and consider a sequence $(c_{n})_{n\geq 0}$ in $\GPath_\gamma(M)$ converging to
$c$ with fixed endpoints. We claim that $H(c_{n})$ converges to $H(c)$. An elementary argument
shows that this implies the continuity of $H$ with fixed endpoints on $\Path(M)$.

Choose $\epsilon>0$. Choose an integer $m\geq 0$ such that $d(H(\D_{m}(c)),H(c))<\frac{\epsilon}{2}$. 
Such an integer exists by definition of $H(c)$. Now for each $n\geq 0$ and each $k\in \{0,\ldots,
2^{m}\}$, let $\eta_{n,k}$ be the geodesic segment joining $c_{n}(k 2^{-m})$ to
$c(k2^{-m})$. If $k\neq 2^m$, let us also denote by $c_{n,k}$ the portion of $c_{n}$
parametrized by the interval $[k2^{-m},(k+1)2^{-m}]$. Observe that $\eta_{n,0}$ and $\eta_{n,2^m}$
are constant paths. The simple equivalence 
$$c_{n}=c_{n,0}\ldots c_{n,2^{m}}\simeq (\eta_{n,0}^{-1} c_{n,0}\eta_{n,1})\ldots
(\eta_{n,k}^{-1}c_{n,k}\eta_{n,k+1})\ldots (\eta_{n,2^m-1}c_{n,2^m-1}\eta_{n,2^m})$$
implies the following inequality:
$$d(H(c_{n}),H(\D_{m}(c)))\leq \sum_{k=0}^{2^m-1}
d(H(\eta_{n,k}^{-1}c_{n,k}\eta_{n,k}),H(\sigma_{k})),$$
where $\D_m(c)=\sigma_0 \ldots \sigma_{2^m-1}$ is the decomposition given by the definition of $\D_m(c)$. 
The path $\eta_{n,k}^{-1}c_{n,k}\eta_{n,k}$ is piecewise geodesic and shares the same endpoints as
the segment $\sigma_{k}$. Hence we can apply Proposition \ref{estim 2} to find
$$d(H(\eta_{n,k}^{-1}c_{n,k}\eta_{n,k}),H(\sigma_{k}))\leq \! K \!
\left(\ell(c_{n,k})+2d_{\infty}(c_{n},c)\right)^\frac{3}{4}\! \left(\ell(c_{n,k})-\ell(\sigma_{k})+2d_{\infty}(c_{n},c)\right)^{\frac{1}{4}}\! .$$
By H\"{o}lder inequality again, 
$$d(H(c_{n}),H(\D_{m}(c)))\leq  K\!  \left(\ell(c_{n}) +
2^{m+1}d_{\infty}(c_{n},c)\right)^{\frac{3}{4}}\left(\ell(c_{n})-\ell(c)+2^{m+1}d_{\infty}(c_{n},c)\right)^{\frac{1}{4}}.$$
Since $\ell(c_{n})$ converges to $\ell(c)$ and $d_{\infty}(c_{n},c)$ tends to 0, the right hand
side tends to zero as $n$ tends to infinity. For $n$ large enough, it is smaller than
$\frac{\epsilon}{2}$. For such $n$, $d(H(c_{n}),H(c))<\epsilon$. \end{proof}

\begin{remark} The factors $2^{m+1}$ in the last expression are very unpleasant, because they give
the feeling
that $H$ is not uniformly continuous on $\Path(M)$. In fact, the last proof reveals that, on $\Path(M)$
endowed with the distance $d_{\ell}$, $H$ is uniformly continuous on subsets where
$\ell\circ \D_{n}$ converges uniformly to $\ell$. It is likely that a much better result can be
achieved by considering the stronger distance $d_{1}$ on $\Path(M)$. Ideally, one could expect $H$ to be
$\frac{1}{4}$-H\"{o}lder continuous on $(\Path(M),d_{1})$. I have not been able to prove or disprove
this statement.
\end{remark}

\section{Extension of discrete holonomy fields}

Let $(M,\vol,\gamma,\CS,C)$ be a Riemannian marked surface with $G$-constraints. Starting from a discrete Markovian holonomy field $\DF$ satisfying a regularity condition, we have constructed a measure $\DF^{X,(\gamma)}_{M,\vol,\CS,C}$ on $(\M(\Path(M),G),\C)$ (see
Corollary \ref{cor : def DF(g)}). The construction of this measure involves a Riemannian metric and we must now
prove that the result is independent of this choice. We start by identifying
the distribution of $(H_{c})_{c\in \Path(\G)}$ under $\DF^{X,(\gamma)}_{M,\vol,\CS,C}$ for an
arbitrary graph $\G$. 

\begin{proposition} \label{ident dist} Let $\DF$ be a stochastically $\frac{1}{2}$-H\"{o}lder continuous and continuously area-dependent discrete Markovian holonomy field. Let $(M,\vol,\gamma,\CS,C)$ be a Riemannian marked surface
with $G$-constraints. Let $\G=(\V,\E,\F)$ be a graph on $(M,\CS)$. Recall the notation $\DF^{(\gamma)}_{M,\vol,\CS,C}$ from Corollary \ref{cor : def DF(g)}.\\
1. The distribution of $(H_{e})_{e\in \E}$ under $\DF^{(\gamma)}_{M,\vol,\CS,C}$ is $\DF^{\G}_{M,\vol,\CS,C}$.\\
2. The measure $\DF^{(\gamma)}_{M,\vol,\CS,C}$ does not
depend on the Riemannian metric $\gamma$. We denote it by $\DF_{M,\vol,\CS,C}$.\\
3. The distribution of $(H_{e})_{e\in \E}$
under $\DF_{M,\vol,\CS,C}$ is $\DF^{\G}_{M,\vol,\CS,C}$.
\end{proposition}

\begin{proof}1. For each $n\geq 1$, let $\G_n$ be the graph produced by Proposition \ref{approx graph epsilon} with $\epsilon=n^{-1}$. Let $S_{n}:\E\lra \E_n$ denote the corresponding bijection. By the stochastic continuity of the process $(H_{c})_{c\in \Path(M)}$ under
$\DF^{(\gamma)}_{M,\vol,\CS,C}$, which follows from Theorem
\ref{main extension}, the distribution of $(H_{e})_{e\in \E}$ is the limit of the distributions
of the families $(H_{S_{n}(e)})_{e\in \E}$ as $n$ tends to infinity. 

Since $S_n$ preserves the cyclic order at each vertex of $\G$, there exists for each $n$ a homeomorphism of $M$ which preserves $\CS$ and which sends $\G$ on $\G_n$ and which induces the bijection $S_n$. Let $\psi_n$ be such a homeomorphism. Since for each face $F$ of $\G$, the boundary of $\psi_n(F)$ is $\psi_n(\partial F)=S_n(\partial F)=\partial S_n(F)$, the face $\psi_n(F)$ is $S_n(F)$.  

Moreover, the distribution of the family $(H_{S_{n}(e)})_{e\in \E}$ under $\DF^{(\gamma)}_{M,\vol,\CS,C}$ is the distribution of the same family under $\DF^{\G_n}_{M,\vol,\CS,C}$, hence the distribution of $(H_{e})_{e\in \E}$ under $\DF^{\G_n}_{M,\vol,\CS,C}\circ \psi_n^{-1}$, where $\psi_n$ denotes the map induced by $\psi_n$ from $\M(\Path(\G_n),G)$ to $\M(\Path(\G),G)$.

By the fourth assertion of Proposition \ref{approx graph epsilon}, $\vol(\psi_n(F))$ tends to $\vol(F)$ as $n$ tends to infinity for all $F\in \F$. Hence, the assumption that $\DF$ is continuously area-dependent implies that the distribution of the family $(H_{S_{n}(e)})_{e\in \E}$ under $\DF^{(\gamma)}_{M,\vol,\CS,C}$ converges weakly to $\DF^\G_{M,\vol,\CS,C}$ as $n$ tends to infinity.

2. Let $\gamma$ and $\gamma'$ be two Riemannian metrics on $(M,\vol,\CS)$. By
Lemma \ref{geod graph} and the property that we have just proved, the distributions of $(H_{c})_{c\in
\GPath_{\gamma}(M)}$ under $\DF^{X,(\gamma)}_{M,\vol,\CS,C}$ and
$\DF^{X,(\gamma')}_{M,\vol,\CS,C}$ agree. Since
$\GPath_{\gamma}(M)$ is dense in $\Path(M)$ for the convergence with fixed endpoints (see Proposition \ref{A dense}),
the continuity property granted by Theorem \ref{main extension} implies that
$\DF^{X,(\gamma)}_{M,\vol,\CS,C}=\DF^{X,(\gamma')}_{M,\vol,\CS,C}$.

3. This property follows immediately from the first two. \end{proof}

We can now finish the proof of the main theorem of this section.\\

\begin{proof}[Proof of Theorem \ref{main existence}] Let $\DF$ be a regular discrete Markovian holonomy field. By applying Corollary \ref{cor : def DF(g)} and Proposition \ref{ident dist}, we get for all measured marked surface with $G$-constraints $(M,\vol,\CS,C)$ a finite measure $\DF_{M,\vol,\CS,C}$ on $(\M(\Path(M),G),\C)$ which by restriction produces a measure on the invariant $\sigma$-field. The total mass of the measure $\DF_{M,\vol,\CS,C}$ is the common value of the masses of the measures $\DF^\G_{M,\vol,\CS,C}$ for all graph $\G$ on $(M,\CS)$, which we have denoted by $Z_{M,\vol,\CS,C}$. In particular, this mass is finite.

Now, we check that the seven axioms of Definition \ref{def MHF} are satisfied. We choose a measured marked surface with $G$-constraints $(M,\vol,\CS,C)$. We endow $(M,\CS,C)$ with a Riemannian metric $\gamma$ and use without further comment the fact, granted by Proposition \ref{ident dist}, that $\DF_{M,\vol,\CS,C}=\DF^{(\gamma)}_{M,\vol,\CS,C}$. \\

A$_1$. Let $\mathcal N$ denote the event $\{\exists l\in \CS \cup \BS(M), h(l)\notin C(l)\}$.
Let $\G$ be a graph on $(M,\CS)$. By Proposition \ref{ident dist} and the axiom \AxD{1} for $\DF$,
$$\DF_{M,\vol,\CS,C}({\mathcal N})=\DF^{\G}_{M,\vol,\CS,C}({\mathcal N})=0.$$

A$_2$. The set of bounded measurable functions from $(\M(\Path(M),G),\I)$ to $\RK$ whose integral against $\DF_{M,\vol,\CS,C}$ depends measurably of the $G$-constraints $C$ is a vector space which contains the constant functions and is stable by uniformly bounded monotone limit. Thus, by a monotone class argument, and by definition of the invariant $\sigma$-field $\I$, in order to show that this space contains all bounded functions measurable with respect to $\I$, it suffices to show that it contains all functions of the form $h\mapsto f(h(l_1),\ldots,h(l_n))$ where $l_1,\ldots,l_n$ are loops based at the same point and $f:G^n\to G$ is continuous and invariant under the diagonal action of $G$ by conjugation. 

Let us choose $l_1,\ldots,l_n$ and $f$ as above. For each $i\in\{1,\ldots,n\}$ and all $m\geq 1$ large enough, let $\D_m(l_i)$ denote the dyadic piecewise geodesic approximation of $l_i$ of order $m$. Let us define a function $F$ and a sequence of functions $F_m$ on the set $\Const_{G}(M,\CS)$ of $G$-constraints on $(M,\CS)$ as follows : 
$$F(C)=\int_{\M(\Path(M),G)} f(h(l_1),\ldots,h(l_n)) \; \DF_{M,\vol,\CS,C}(dh) $$
and, for all $m\geq 1$,
$$F_m(C)=\int_{\M(\Path(M),G)} f(h(\D_m(l_1)),\ldots,h(\D_m(l_n))) \; \DF_{M,\vol,\CS,C}(dh).$$
Our goal is to prove that the function $F$ is measurable. For all $C\in \Const_{G}(M,\CS)$, the fact that $\D_m(l_i)$ converges to $l_i$ with fixed endpoints implies, according to the conclusion of Theorem \ref{main extension}, that $H_{\D_m(l_i)}$ converges in probability to $H_{l_i}$, so that $F_m(C)$ tends to $F(C)$ as $m$ tends to infinity. Hence, it suffices to prove that $F_m$ is measurable for $m$ large enough. In fact we claim that $F_m$ is continuous as soon as it is defined. Indeed, for each $m$, $F_m(C)$ can be computed in a piecewise geodesic graph. There, the dependence in $C$ can be made explicit and all functions in the integrand are continuous. The result follows.\\

A$_3$. Just as in the last point, it suffices to check the equality when it integrates a function of the form $h\mapsto f(h(l_1),\ldots,h(l_n))$, where $l_1,\ldots,l_n$ are piecewise geodesic and $f$ is invariant by diagonal conjugation. Hence, it suffices to prove the equality for $\DF^{\G}_{M,\vol,\CS,C}$ for any graph $\G$ with piecewise geodesic edges. In this last case, the property follows from the axiom \AxD{3} satisfied by $\DF$.\\

A$_4$. The metric $\gamma'=(\psi^{-1})^*\gamma$ is a Riemannian metric on $(M',\vol',\CS')$ in the sense of Proposition \ref{nice metric}. We need to prove that the image measure of $\DF^{(\gamma)}_{M,\vol,\CS,C}$ by the mapping induced by $\psi$ is $\DF^{(\gamma')}_{M',\vol',\CS',C'}$. Again, we may restrict ourselves to functions of the form $h\mapsto f(h(l_1),\ldots,h(l_n))$, where $l_1,\ldots,l_n$ are piecewise geodesic, hence to the discrete measures associated to graphs with piecewise geodesic edges. Let $\G$ be a graph on $(M,\CS)$ with piecewise geodesic edges. Then $\G'$, the graph constituted by the images by $\psi$ of the edges of $\G$, is a graph on $(M',\CS')$ with piecewise geodesic edges and, by the axiom \AxD{4} for $\DF$, the distribution of $(H_e)_{e\in\E}$ under $\DF^{\G}_{M,\vol,\CS,C}$ is the same as the distribution of $(H_{e'})_{e'\in\E'}$ under $\DF^{\G'}_{M',\vol',\CS',C'}$. The property follows.\\

\Ax{5}, \Ax{6} and \Ax{7} follow respectively from the axiom \AxD{5}, \AxD{6} and \AxD{7} satisfied by $\DF$ and the same approximation argument as in the previous points.\\

The fact that the new Markovian holonomy field $\DF$ is stochastically continuous is a part of the conclusion of Theorem
\ref{main extension}. Finally, the Feller property follows from the Feller property of $\DF$. \end{proof}

\chapter{L\'evy processes and Markovian holonomy fields}

In this chapter, we apply the extension theorem proved in the previous chapter to construct a whole family of Markovian holonomy fields. Before that, we study the partition functions of an arbitrary regular Markovian holonomy field and prove that they are completely determined by a L\'evy process on the group $G$ with some nice properties, essentially a continuous density. We then construct a Markovian holonomy field for each such L\'evy process. The case of the Brownian motion on a connected Lie group yields the Yang-Mills measure.

\section{The partition functions of a Markovian holonomy field}
\label{par fun mar hol}

In this section, we establish some fundamental properties satisfied by the masses of the finite
measures which constitute a Markovian holonomy field. 

To start with, we describe the isomorphism classes of connected surfaces with $G$-constraints on
the boundary. If $M$ is oriented, we denote by $\BS^+(M)$ the subset of $\BS(M)$ which consists in
the curves which have the orientation induced by that of $M$. 

\begin{proposition}\label{classif surf mark} Let $(M,\vol,\varnothing,C)$ and $(M',\vol',\varnothing,C')$ be two connected measured marked surfaces with $G$-constraints. If $M$ and $M'$ are orientable, we assume that they are oriented. They are isomorphic if and only if the following
conditions hold simultaneously.\\
1. $M$ and $M'$ are homeomorphic.\\
2. $\vol(M)=\vol'(M')$.\\
3. If $M$ and $M'$ are oriented, there exists a bijection $\psi:\BS^+(M)\to\BS^+(M')$ such that
$C=C'\circ \psi$ on $\BS^+(M)$.\\
3'. If $M$ and $M'$ are non-orientable, there exists a $\Z/2\Z$-equivariant bijection
$\psi:\BS(M)\to \BS(M')$ such that $C=C'\circ \psi$.
\end{proposition} 
\index{surface!classification}

We use this result to associate to every Markovian holonomy field a family of functions of one or
several variables in ${\rm Conj}(G)$.  

Let $\HF$ be a  Markovian holonomy field. Let $g$ and $p$ be two non-negative integers, with $g$ even.
Let $t$ be a positive real number. Recall the notation of Section \ref{sec: surfaces}. Let $M$ be a surface homeomorphic to $\Sigma^+_{p,g}$, endowed with a density $\vol$ of total area $t$. Let $b_{1},\ldots,b_{p}$
denote the positively oriented
connected components of $\partial M$. Let $x_{1},\ldots,x_{p}$ be $p$ elements of $G$. Let $C$ be
the unique set of $G$-constraints on $(M,\varnothing)$ such that, for all $i\in\{1,\ldots,p\}$,
$C(b_{i})=\O_{x_{i}}$. By Proposition \ref{classif surf mark} and the axiom \Ax{4}, the number $\HF_{(M,\vol,\varnothing,C)}(1)$
depends on $M$, $\vol$ and $C$ only through $g,p,t$, the unordered list
$\O_{x_{1}},\ldots,\O_{x_{p}}$ and the fact that $M$ is orientable. Hence, it is legitimate to set
$$Z_{p,g,t}^{+}(x_{1},\ldots,x_{p})=\HF_{(M,\vol,\varnothing,C)}(1).$$
This defines a symmetric function $Z^+_{p,g,t}$ of $p$ conjugacy classes of $G$. If the Markovian field is Fellerian, then this function  is continuous
with respect to $(t,x_{1},\ldots,x_{p})$. If $p=0$, $Z^+_{0,g,t}$ is just a number, namely the
total mass of the measure
$\HF_{(M,\vol,\varnothing,\varnothing)}$ where $M$ is the connected sum of $g$ tori endowed with a
density of total area $t$.

Similarly, let $g$ and $p$ be two integers, respectively positive and non-negative. Let $M$ be a surface homeomorphic to $\Sigma^-_{p,g}$, endowed with a density $\vol$ of total area $t$. Let
$b_{1},\ldots,b_{p}$ denote the disjoint connected components of $\partial M$ endowed with an
arbitrary orientation. Let $x_{1},\ldots,x_{p}$
be $p$ elements of $G$. Let $C$ be the unique set of $G$-constraints on $(M,\varnothing)$ such
that, for all $i\in\{1,\ldots,p\}$, $C(b_{i})=\O_{x_{i}}$. When $\HF$ is not oriented, we define
$$Z_{p,g,t}^{-}(x_{1},\ldots,x_{p})=\HF_{(M,\vol,\varnothing,C)}(1).$$
Again, if $p=0$, $Z^-_{0,g,t}$ is just the total mass of the measure
$\HF_{(M,\vol,\varnothing,\varnothing)}$ where $M$ is the connected sum of $g$ projective planes
endowed with a density of total area $t$.

\index{ZZZ@$Z^+_{p,g,t},Z^-_{p,g,t}$}
\begin{definition} Let $\HF$ be a Markovian holonomy field. The functions 
$$Z^\epsilon_{p,g,t} : G^{p}\to \RK^*_{+},$$
where $(\epsilon,p,g,t)$ spans $(\{+\}\times \NK \times 2\NK \times \RK^*_{+})\cup(\{-\} \times \NK \times \NK^* \times \RK^*_{+})$, are called the {\em partition functions} of
the field $\HF$. 
\end{definition}
\index{Markovian holonomy field!partition function|(}

In the rest of this section, we fix a Markovian holonomy field $\HF$ and study its partition
functions. They are infinitely many but the Markov property of the field implies that they satisfy an
infinite set of relations and that they are in fact completely determined by a small number of
them. 

Let us introduce several operations on functions of conjugacy classes of $G$. Firstly, we identify
functions of a conjugacy class of $G$ and functions on $G$ which are constant on the conjugacy
classes. Thus, we may speak of continuous or integrable functions of a
conjugacy class. This point of view is consistent with our previous definition of a topology and
$\sigma$-field on ${\rm Conj}(G)$ (see Section \ref{section def MHF}). Of
particular interest is the space of square-integrable functions of one conjugacy class of $G$, which
we identify with the space $L^2(G)^G$ of conjugation-invariant square-integrable functions on $G$. If
$p\geq 1$ is an integer, we identify the elements of the $p$-th symmetric tensor power $\Sym^p\left(L^2(G)^{G}\right)$ with symmetric functions of $p$ conjugacy classes. We use the shorthand notation $S^p(G)$ for
$\Sym^p\left(L^2(G)^{G}\right)$.

\begin{definition} For all integers $p,q,r\geq 1$ and $s\geq 2$, the three 
linear mappings $\upsilon :S^r(G)\to S^{r-1}(G)$ , $\beta_1:S^s(G)\to S^{s-2}(G)$ and $\beta_2 : S^p(G)\otimes S^q(G)\to S^{p-1}(G)\otimes S^{q-1}(G)$ are defined by
\begin{eqnarray}
\forall f\in S^r(G), && (\upsilon f)(x_1,\ldots,x_{r-1})=\int_G f(x_1,\ldots,x_{r-1},x^2) \; dx,\\
\forall f\in S^s(G), && (\beta_1 f)(x_1,\ldots,x_{s-2})=\int_G f(x_1,\ldots,x_{s-2},x,x^{-1}) \; dx,
\end{eqnarray}
and, for all $f\in S^p(G)$  and all $f'\in S^q(G)$, 
\begin{equation}
(\beta_2 (f\otimes f'))(x_1,\ldots,x_{p-1},y_1,\ldots,y_{p-1})=\int_G f(x_1,\ldots,x_{p-1},z) f'(y_1,\ldots,y_{q-1},z^{-1}) \; dz.
\end{equation}
\end{definition}

We will now use these linear mappings to formulate the relations between the partition functions of the holonomy field. The operation $\upsilon$ expresses the transformation of the partition function under a unary gluing, and the contractions $\beta_1$ and $\beta_2$ correspond to binary gluings, respectively of two boundary components which lie on the same connected component of the surface and  
two boundary components which lie on two distinct connected components.

Recall that we use the notation $\epsilon\wedge\epsilon'$ for $\epsilon,\epsilon'\in \{-,+\}$ with the following meaning
: $\epsilon\wedge\epsilon'=+$ if $\epsilon=\epsilon'=+$ and $\epsilon\wedge\epsilon'=-$ in all
other cases.

\begin{proposition} \label{algebre Z} Let $(Z^\pm_{p,g,t})_{p,g\geq 0, t>0}$ 
be the partition functions of a Markovian holonomy field $\HF$. \\
1. For all $(\epsilon,g) \in (\{+\}\times 2\NK)\cup (\{-\}\times \NK^*)$, all integers $p\geq
q\geq 1$ and
all real $t>0$, the function $Z^{\epsilon}_{p,g,t}$ is square-integrable with respect to any $q$ of
its variables for any value of the $p-q$ other. \\
2. The following equality holds:
\begin{equation}\label{crosscap}
\upsilon(Z^\epsilon_{p,g,t})=Z^-_{p-1,g+1,t}.
\end{equation}
Moreover, if $p\geq 2$, then
\begin{equation} \label{handle}
\beta_1(Z^\epsilon_{p,g,t})=Z^\epsilon_{p-2,g+2,t}.
\end{equation}
Finally, for all $(\epsilon',g') \in (\{+\}\times 2\NK)\cup (\{-\}\times \NK^*)$, all $p'\geq 1$ and all real $t'>0$, the following equality holds:
\begin{equation}\label{circle}
\beta_2(Z^{\epsilon}_{p,g,t}\otimes Z^{\epsilon'}_{p',g',t'})=Z^{\epsilon\wedge \epsilon'}_{p+p'-2,g+g',t+t'}.
\end{equation}
\end{proposition}
\index{Kolmogorov-Chapman equation}

\begin{proof}Let us start by proving the second assertion. Choose $\epsilon,p,g,t$ as above. It follows from the axiom \Ax{2} that $Z^\epsilon_{p,g,t}:G^p\to \RK^*_+$ is a measurable function. Since it takes non-negative values, the integral which defines $\upsilon (Z^\epsilon_{p,g,t})$ is well defined, possibly infinite. Let us prove that the identity (\ref{crosscap}) holds.

Let $M$ be a surface homeomorphic to $\Sigma^\epsilon_{p,g}$, endowed with a surface measure $\vol$ of total area $t$.
Consider $b\in \BS(M)$ and let $f:M\to M_1$ be a unary gluing along $b$, with joint $\{l,l^{-1}\}=\{f(b),f(b^{-1})\}$. Thus, $M=\Spl_l(M_1)$. The surface $M_1$ is not orientable and it has $p-1$ boundary components. According to the observation made after Definition \ref{def gluing}, it is homeomorphic to the connected sum of a projective plane and the surface obtained by gluing a disk along one boundary component of $M$. Thus, it is homeomorphic to $\Sigma^-_{p-1,g+1}$. Finally, $\vol$ induces on $M_1$ a surface measure $\vol_1$ with total area $t$.

Let $C$ be a set of $G$-constraints on $(M,\{l,l^{-1}\})$. Applying axioms \Ax{3} and then \Ax{6} gives us the relation
\begin{eqnarray*}
\HF_{M_1,\vol_1,\varnothing,C}(1)&=&\int_G \HF_{M_1,\vol_1,\{l,l^{-1}\},C_{l\mapsto x}}(1) \; dx\\
&=& \int_G \HF_{M,\vol,\varnothing,\Spl_l(C_{l\mapsto x})}(1) \; dx.
\end{eqnarray*}

Recall from Definition \ref{split C} that, since we are considering a unary gluing, $\Spl_l(C_{l\mapsto x})$ puts the constraint $\O_{x^2}$ on $b\in \BS(M)$. Translating the last relation in terms of the partition functions gives (\ref{crosscap}).

The proofs of the relations (\ref{handle}) and (\ref{circle}) are very similar. For the relation (\ref{handle}), one considers a binary gluing in which one identifies two boundary components of a connected surface. If the surface is orientable, the two boundary components are identified by an orientation-reversing diffeomorphism. The result of this gluing can be described as the connected sum of a torus and the surface obtained by gluing two disks along two boundary components of the original surface. Thus, it has two boundary components less, and reduced genus increased by $2$. For the relation (\ref{circle}), one considers a binary gluing in which one identifies two boundary components located on two distinct connected components of a surface. Proposition \ref{connected sum} settles the issue of orientation and reduced genus of the resulting surface.

The assertion of square-integrability follows  from (\ref{handle}) and (\ref{circle}) as we shall see now. Firstly, observe that Theorem \ref{diffeo bord} and the axiom \Ax{4} imply that the value of a partition function is not affected by simultaneously replacing each argument by its inverse. Now, for all $x_{q+1},\ldots,x_p \in G$,
\begin{eqnarray*}
&& \int_{G^{q}} Z^\epsilon_{p,g,t}(x_1,\ldots,x_p)^2 \; dx_{1}\ldots dx_q  \\ 
&& \hskip 3cm = \int_{G^{q}} Z^\epsilon_{p,g,t}(x_1,\ldots,x_p) Z^\epsilon_{p,g,t}(x_1^{-1},\ldots,x_p^{-1})\; dx_{1}\ldots dx_q \\
&& \hskip 3cm = (\beta_1^{q-1} \circ \beta_2)(Z^\epsilon_{p,g,t} \otimes Z^\epsilon_{p,g,t})(x_{q+1},x_{q+1}^{-1},\ldots,x_p,x_p^{-1}) \\
&& \hskip 3cm  = Z^\epsilon_{2(p-q),2(g+q-1),2t}(x_{q+1},x_{q+1}^{-1},\ldots,x_p,x_p^{-1}).
\end{eqnarray*}
The last number is indeed finite and the assertion is proved. \end{proof}

\begin{corollary} The partition functions of a Markovian holonomy field are completely determined by the functions $(Z^+_{1,0,t})_{t>0}$ and $(Z^+_{3,0,t})_{t>0}$.
\end{corollary}

\begin{proof}For all non-negative $p$ and $g$, such that $g$ is even and $p+g>0$, the following identity holds:
\begin{equation}\label{Z+ Z1}
Z^+_{p,g,t}=(\beta_1^{\frac{g}{2}} \circ \beta_2^{p+g-1})\left(Z^+_{1,0,\frac{t}{p+g}} \otimes {Z^+_{3,0,\frac{t}{p+g}}}^{\otimes (p+g-1)} \right).
\end{equation}
Then, the equalities
\begin{equation}
\forall p,g\geq 0\; , \; \; Z^-_{p,2g+1,t}=\upsilon(Z^+_{p+1,2g,t}) \mbox{ and } Z^-_{p,2g+2,t}=\upsilon^2(Z^+_{p+2,2g,t})
\end{equation}\label{Z- Z1}
determine all non-orientable partition functions. Finally, the equality
$$Z^+_{0,0,t}=\beta_2(Z^+_{1,0,\frac{t}{2}}\otimes Z^+_{1,0,\frac{t}{2}})$$
determines the only remaining partition function. \end{proof}

An important consequence of Proposition \ref{algebre Z} is that for all $\epsilon,p,g,t$, the function
$Z^{\epsilon}_{p,g,t}$ is the density with respect to the Haar measure of a measure on $G^p$. By the axiom \Ax{7}, it is a probability measure. In the next proposition, we start to study the behaviour of these
probability measures as $t$ tends to $0$, under the assumption that the holonomy field is
stochastically continuous. Recall that we denote by $1$ the unit element of $G$ and that we use the notation
$\delta_{\O_{x}}$ for the unique $G$-invariant probability measure on $\O_{x}$.


\begin{proposition} \label{tqft limite} Let $(Z^{\pm}_{p,g,t})_{p,g,t}$ be the partition functions
of a
stochastically continuous Markovian holonomy field. Then, as $t$ tends to 0, one has the following
weak convergences of measures on $G$.\\
1. $Z^+_{1,0,t}(x)\; dx \build{\Longrightarrow}_{t\to 0}^{} \delta_{1}$.\\
2. $\forall x,y\in G, Z^+_{3,0,t}(x,y,z^{-1})\; dz\build{\Longrightarrow}_{t\to
0}^{}\delta_{\O_{x}}*\delta_{\O_{y}}= \int_{G^2}\delta_{vxv^{-1} wyw^{-1}}\; dv dw$.
\end{proposition}

\begin{proof}1. Let $M$ be the disk of radius 1 centred at the origin in $\RK^2$ endowed with
the Lebesgue measure, which we denote by $\vol$. We denote by $\partial M$ the positively oriented
boundary of $M$. For each $r\in [0,1]$, let $s_{r}$ be the path which goes straight from the origin
to the point $(r,0)$ and let $c_{r}$ be the loop based at $(r,0)$ which goes once counterclockwise
around the circle of radius $r$ centred at the origin. For all $y\in G$, we denote by $(h_{c})_{c\in \Path(M)}$ the
canonical process on $\M(\Path(M),G)$, we consider the measure
$\HF_{M,\vol,\varnothing,\partial M \mapsto y}$ on $\M(\Path(M),G)$ and we denote by $\E_{y}$ the
corresponding expectation. Let $f$ be a continuous invariant function on $G$. We compute
$\E_{y}[f(h_{s_{r}c_{r}s_{r}^{-1}})]$. By the multiplicativity property of $h$, it is equal to
$\E_{y}[f(h_{s_{r}}^{-1}h_{c_{r}}h_{s_{r}})]=\E_{y}[f(h_{c_{r}})]$.  By the axiom \Ax{3}, we can disintegrate this expectation with respect to the value of $h(c_r)$. We find
\begin{eqnarray*}
\E_{y}[f(h_{c_{r}})]&=&\int_{G\times \M(\Path(M),G)} f(h(c_r)) \HF_{M,\vol,\{c_r,c_r^{-1}\},(\partial M \mapsto y,c_r\mapsto x)}(dh) dx\\
&=& \int_G f(x) \HF_{M,\vol,\{c_r,c_r^{-1}\},(\partial M \mapsto y,c_r\mapsto x)}(1) dx.
\end{eqnarray*}
We use now the axioms \Ax{6} and \Ax{5} to split $M$ along $c_r$ and we find
$$\E_{y}[f(h_{c_{r}})]=\int_{G}Z^+_{1,0,\pi r^2}(x) f(x) Z^{+}_{2,0,\pi (1-r^2)}
(x^{-1},y)\; dx.$$
By integrating over $y$ and using the axiom \Ax{7}, we find
\begin{equation}\label{1,0,t cv}
\int_{G} \E_{y}[f(h_{c_{r}})]\; dy = \int_{G}Z^+_{1,0,\pi r^2}(x) f(x).
\end{equation}
Our goal is now to prove that the left-hand side tends to $f(1)$ as $r$ tends to $0$. For this, we
use the stochastic continuity of the holonomy field. Indeed, as $r$ tends to 0, the loop
$s_{r}c_{r}s_{r}^{-1}$ converges with fixed endpoints to the constant loop $c_{0}$ at the origin.
Hence, for all $y\in G$, $\E_{y}[f(h_{c_{r}})]\to \E_{y}[f(h_{c_{0}})]$ as $r$ tends to $0$.  By
the multiplicativity of $h$ and the fact that $c_{0}=c_{0}c_{0}^{-1}$,
the mapping $h_{c_{0}}:\Path(M)\to G$ is identically equal to 1. Hence,  $\E_{y}[f(h_{c_{r}})]
\build{\lra}_{r\to 0}^{}f(1) Z^+_{1,0,\pi}(y) $.
In order to integrate this convergence with respect to $y$, we use the fact that
$$\left|\E_{y}[f(h_{c_{r}})]\right|\leq \parallel f \parallel_{\infty}
Z^+_{1,0,\pi}(y)$$
and the right-hand side is continuous, hence integrable, with respect to $y$. Hence, the
dominated convergence theorem applies and we deduce that the left-hand side of (\ref{1,0,t cv})
tends to $f(1)$ as $r$ tends to 0.

2. Let $M$ be the closed disk of radius 4 centred at the origin in $\RK^2$ from which one has removed
the two open disks of radius 1 centred respectively at the points $\alpha=(2,0)$ and
$\beta=(-2,0)$. We endow $M$ with some density denoted by $\vol$.
Let $a$ (resp. $b$) be the loop which starts at $(1,0)$ (resp. $(-1,0)$) and goes once around the
circle of radius 1 centred at $\alpha$ (resp. $\beta$), counterclockwise. Let $d$ be the path
which goes straight from $(1,0)$ to $(-1,0)$.
Choose $r\in (0,1)$. Consider the union of the two closed disks of radius $1+r$ centred at $\alpha$
and $\beta$ and the rectangle $[-1,1]\times [-r,r]$. Let $c_{r}$ be the loop which starts at
$(2-\sqrt{1+2r},r)$ and bounds this domain with positive orientation. Let $s_{r}$ be the path
which goes straight from $(1,0)$ to $(2-\sqrt{1+2r},r)$. As $r$ tends to 0, the loop
$s_{r}c_{r}s_{r}^{-1}$ converges with fixed endpoints to the loop $dbd^{-1}a$. However, in order to
apply our axioms, we need to replace $c_{r}$ by a loop based at the same point and whose image is
a smooth submanifold of $M$. We do this in such a way that the convergence of $s_{r}c_{r}s_{r}^{-1}$ to 
$dbd^{-1}a$ is preserved. 

We consider the measure $\HF_{M,\vol,\varnothing,C_{x,y,z}}$ on $\M(\Path(M),G)$, where
$C_{x,y,z}$ is characterized by the fact that $C(a)=x$, $C(b)=y$ and $C$ maps the circle of radius
4 centred at the origin to $z$. We denote the corresponding expectation by $\E_{x,y,z}$. Let $f$
be a continuous invariant function on $G$. For all $r\in (0,1)$, we have
$$\E_{x,y,z}[f(h_{s_{r}c_{r}s_{r}^{-1}})]=\int_{G}Z^{+}_{3,0,t_{r}}(x,y,v) f(v)
Z^{+}_{2,0,T-t_{r}}(v^{-1},y)\; dv,$$
where $t_{r}$ is the area of the domain delimited by $a$, $b$ and $c_{r}$ and $T$ is the total
area of $M$. By the same arguments as in the previous proofs, the left-hand side converges as $r$
tends to 0
to $\E_{x,y,z}[f(h_{dbd^{-1}a})]$ and the convergence is dominated with
respect to $z$. Since the two endpoints of $d$ are distinct and the measure $\HF_{M,\vol,\varnothing,C_{x,y,z}}$
is invariant under gauge transformations, the distribution of $h_{d}$ is both left and
right-invariant on $G$. Thus, $h_{d}$ has the uniform distribution on $G$. Hence,
$\E_{x,y,z}[f(h_{dbd^{-1}a})]=\int_{G}f(xwyw^{-1})\; dw.$ By integrating with respect to $z$, we find
$$\int_{G}Z^{+}_{3,0,t_{r}}(x,y,z)f(z) \; dz \build{\lra}_{r\to 0}^{} \int_{G} f(xwyw^{-1})\; dw.$$
\end{proof}

We will use this result to prove that the partition functions a stochastically continuous Markovian holonomy field are completely determined by the functions $(Z^+_{1,0,t})_{t>0}$. Let us introduce two probability measures on $G$.

\index{EAeta@$\eta,\kappa$}
\begin{definition} \label{kappasigma} Let $\eta$ and $\kappa$ be the two invariant probability
measures on $G$ defined respectively by the fact that for all continuous function $f$ on $G$,
\begin{equation}
\int_{G} f \; d\eta =\int_{G^2} f(aba^{-1}b^{-1})\; da db \; \mbox{
and } \; \int_{G} f \; d\kappa = \int_{G} f(a^2)\; da.
\end{equation}
\end{definition}

The letters $\eta$ and $\kappa$ correspond to the words {\em handle} and {\em cross-cap}. We start by proving some important properties of these measures.

Let $\irrep(G)$ denote the set of isomorphism classes of irreducible representations of $G$ over $\CK$.
Given $\alpha \in \irrep(G)$ with character $\chi_{\alpha}$ and a measure $\mu$ on $G$, the Fourier
coefficient $\hat \mu (\alpha)$ is defined by $\hat\mu(\alpha)=\int_{G} \overline{\chi_{\alpha}}
\; d\mu$. Recall that an irreducible representation is said to be complex if its
character is not real valued, otherwise real (resp. quaternionic) if it preserves a non-degenerate
symmetric (resp. skew-symmetric) complex bilinear form.

If $\mu$ is a measure on $G$, we denote by $\mu^{\vee}$ the measure defined by $\int_{G} f \;
d\mu^{\vee}=\int_{G}f(g^{-1})\; \mu(dg)$. By an invariant measure we mean a measure which is
invariant by conjugation. 

\begin{lemma} \label{convol k s}Let $\mu,\nu,\xi$ be three invariant probability measures on $G$. If $\hat\mu(\alpha)=0$ for every complex representation $\alpha$, then $\mu*\nu=\mu*\nu^{\vee}$. In particular,
$\kappa * \xi * \nu= \kappa * \xi * \nu^{\vee}$. Moreover, $\kappa*\eta=\kappa^{*3}$.
\end{lemma}
\index{conjugation!measures invariant by}

\begin{proof} The Fourier coefficients of $\eta$ and $\kappa$ can be computed easily by using the elementary theory of characters: 
$$\forall \alpha \in \irrep(G), \; \hat\eta(\alpha)=\frac{1}{\dim \alpha} \; \mbox{ and }
\; 
\hat\kappa(\alpha)=\left\{\begin{array}{rl} 1 &\mbox{ if } \alpha \mbox{ is real,} \\0 &\mbox{ if }
\alpha \mbox{ is complex,} \\ -1 &\mbox{ if } \alpha \mbox{ is quaternionic.}\end{array}\right.$$
The  Fourier coefficients of the convolution product of two measures is given by the relation $\widehat{\mu * \mu'}(\alpha)=\hat\mu(\alpha)\hat\mu'(\alpha)/\dim \alpha$. Moreover, the Fourier coefficients of $\nu^{\vee}$ are the complex conjugate of those of $\nu$. Hence, on real and quaternionic representations, whose character is real, the Fourier coefficients of $\nu$ are real and agree with those of $\nu^{\vee}$. Hence, $\mu*\nu$ and $\mu*\nu^\vee$ have the same Fourier coefficients. Since they are both invariant, they are equal. The equality $\kappa * \xi * \nu= \kappa * \xi * \nu^{\vee}$ follows immediately. The last assertion is proved by computing the Fourier coefficients of both sides. \end{proof}

\begin{remark} We invite the reader to compare the equality $\kappa * \eta = \kappa^{* 3}$ with the fact that the connected sum of a projective plane and a torus is homeomorphic to the connected sum of three projective planes (or the connected sum of a projective plane and a Klein bottle).
\end{remark}
\index{Klein bottle}

In the next proposition, we use the notation $\mu(f)$ for the integral of a function $f$ against a
measure $\mu$ and we denote by $*$ the convolution of probability measures. 

\begin{proposition} \label{explicit Z} The partition functions of a regular Markovian holonomy field are completely
determined by the functions $Z^+_{1,0,t}$ for $t>0$. One has the following explicit formulas.

For all $p\geq 0$, all $g\geq 0$ even, all $t>0$ and all $x_{1},\ldots,x_{p}\in G$, one has
\begin{equation}\label{valeur Z+}
Z^{+}_{p,g,t}(x_{1},\ldots,x_{p})=\eta^{*\frac{g}{2}}*\delta_{\O_{x_{1}}}*\ldots*\delta_{\O_{x_{p}}}(Z^{+}_{1,0,t}).
\end{equation}
Moreover, for all $p\geq 0$, all $g> 0$, all $t>0$ and all $x_{1},\ldots,x_{p}\in G$, one has
\begin{equation}\label{valeur Z-}
Z^{-}_{p,g,t}(x_{1},\ldots,x_{p})=\kappa^{*g}*\delta_{\O_{x_{1}}}*\ldots*\delta_{\O_{x_{p}}}(Z^{+}_{1,0,t}).
\end{equation}
\end{proposition}

\begin{proof}Let us start by proving (\ref{valeur Z+}) when $g=0$ and $p>0$, by induction on $p$. For $p=1$, it is a consequence of the fact that $Z^+_{1,0,t}$ is invariant by conjugation. Assume that $p>1$ and the result has been proved for $Z^+_{p-1,0,t}$. Then, for all $x_1,\ldots,x_p\in G$ and all $s\in (0,t)$, (\ref{circle}) yields
$$Z^+_{p,0,t}=\int_G Z^+_{p-1,0,t-s}(x_1,\ldots,x_{p-2},y) Z^+_{3,0,s}(x_{p-1},x_p,y^{-1})\; dy.$$
Since the Markovian holonomy field that we consider is Fellerian, the function 
$$(s,x_1,\ldots,x_{p-2},y)\mapsto Z^+_{p-1,0,t-s}(x_1,\ldots,x_{p-2},y)$$
is continuous on the compact set $[0,\frac{t}{2}]\times G^{p-1}$. Hence, when $s$ tends to $0$, it converges uniformly as a function on $G^{p-1}$ towards $Z^+_{p-1,0,t}$. Thus, using the convergence proved in Proposition \ref{tqft limite}, we find
$$Z^+_{p,0,t}=\int_G  Z^+_{p-1,0,t}(x_1,\ldots,x_{p-2},y) \; (\delta_{O_{x_{p-1}}} * \delta_{O_{x_p}})(dy).$$
Using the induction hypothesis, we find
\begin{align*}
Z^+_{p,0,t} &= \int_G \left(\int_{G^{p-1}} Z^+_{1,0,t}(w_1 \ldots w_{p-2} z) \;  \prod_{i=1}^{p-2} \delta_{\O_{x_i}}(dw_i) \delta_{O_y}(dz) \right) (\delta_{O_{x_{p-1}}} * \delta_{O_{x_p}})(dy)\\
&= \int_{G^p} Z^+_{1,0,t}(w_1 \ldots w_{p}) \prod_{i=1}^{p} \delta_{\O_{x_i}}(dw_i)
\end{align*}
because $\delta_{O_{x_{p-1}}} * \delta_{O_{x_p}}$ is already an invariant measure on $G$. This is the expected result. 

Let us now treat the case where $p+g>0$. We have 
\begin{eqnarray*}
Z^+_{p,g,t}(x_1,\ldots,x_p) &=& \beta_1^{\frac{g}{2}}(Z^+_{p+g,0,t})(x_1,\ldots,x_p)\\
&=& \int_{G^{\frac{g}{2}}} Z^+_{p+g,0,t}(x_1,\ldots,x_p,y_1,y_1^{-1},\ldots,y_{\frac{g}{2}},y_{\frac{g}{2}}^{-1}) 
\; dy_1 \ldots dy_{\frac{g}{2}} \\
&& \hskip -3.5cm = \int_{G^{p+\frac{3g}{2}}} Z^+_{1,0,t} (w_1\ldots w_p z_1 z'_1 \ldots z_{\frac{g}{2}} 
z'_{\frac{g}{2}}) \; \prod_{i=1}^{p} \delta_{\O_{x_i}}(dw_i) \prod_{i=1}^{\frac{g}{2}} 
\delta_{\O_{y_i}}(dz_i)\delta_{\O_{y_i^{-1}}}(dz'_i) \prod_{i=1}^{\frac{g}{2}} dy_i.
\end{eqnarray*}
The result follows now from the equality $\int_G \delta_{\O_y} * \delta_{\O_{y^{-1}}} \; dy =\eta$ which one checks easily using the elementary properties of the Haar measure.

In order to prove (\ref{valeur Z+}), we still have to prove that $Z^+_{0,0,t}=Z^+_{1,0,t}(1)$. This follows from the equality $Z^+_{0,0,t}=\int_G Z^+_{1,0,t-s}(y) Z^+_{1,0,s}(y^{-1})\; dy$, Proposition \ref{tqft limite} and the argument of uniform convergence that we have already used above.  

By (\ref{Z- Z1}) and the first part of this proof, we have for all $p,g\geq 0$, all $t>0$ and all $x_1,\ldots,x_p\in G$,
\begin{eqnarray*}
Z^-_{p,2g+1,t}(x_1,\ldots,x_p) &=& \eta^{*g}*\delta_{\O_{x_1}}*\ldots*\delta_{\O_{x_p}}* \left(\int_G \delta_{\O_{y^2}} \; dy \right) (Z^+_{1,0,t})\\
&=& \kappa *  \eta^{*g}*\delta_{\O_{x_1}}*\ldots*\delta_{\O_{x_p}} (Z^+_{1,0,t}),
\end{eqnarray*}
and
\begin{eqnarray*}
Z^-_{p,2g+2,t}(x_1,\ldots,x_p) &=& \eta^{*g}*\delta_{\O_{x_1}}*\ldots*\delta_{\O_{x_p}}* \left(\int_{G^2} \delta_{\O_{y_1^2}}*\delta_{\O_{y_2^2}} \; dy_1 dy_2 \right) (Z^+_{1,0,t})\\
&=& \kappa^{*2} *  \eta^{*g}*\delta_{\O_{x_1}}*\ldots*\delta_{\O_{x_p}} (Z^+_{1,0,t}).
\end{eqnarray*}
The result is now a consequence of the last assertion of Proposition \ref{convol k s}. \end{proof}

It is now easy to complete the result obtained in Proposition \ref{tqft limite}. We leave the proof of the following corollary to the reader. 

\begin{corollary} \label{full tqft limite} Let $(Z^{\pm}_{p,g,t})_{p,g,t}$ be the partition functions
of a regular Markovian holonomy field. Then, as $t$ tends to 0, one has the following
weak convergences of measures on $G$.\\
1. For all $p\geq 0$, all $g\geq 0$ even and all $x_1,\ldots x_{p-1}\in G$,
$$Z^+_{p,g,t}(x_1,\ldots,x_{p-1},x^{-1})\; dx \build{\Longrightarrow}_{t\to 0}^{} \eta^{*\frac{g}{2}}*\delta_{\O_{x_1}}*\ldots *\delta_{\O_{x_{p-1}}}.$$
2. For all $p\geq 0$, all $g> 0$ and all $x_1,\ldots x_{p-1}\in G$,
 $$Z^-_{p,g,t}(x_1,\ldots,x_{p-1},x^{-1})\; dx\build{\Longrightarrow}_{t\to
0}^{} \kappa^{*g}*\delta_{\O_{x_1}}*\ldots *\delta_{\O_{x_{p-1}}}.$$
\end{corollary}

\begin{remark} In the proofs so far, we have used the axiom \Ax{7} only for cylinders. Let us call {\rm A}$'_7$ the axiom \Ax{7} restricted to cylinders. By using Propositions \ref{tqft limite} and \ref{explicit Z}, one could easily prove now that \Ax{7} could be replaced by {\rm A}$'_7$ without affecting the notion of regular Markovian holonomy field.
\end{remark}
\index{Markovian holonomy field!partition function|)}

\section{The L\'evy process associated to a Markovian holonomy field}
\index{L\'evy process|(}

In the previous section, we have reduced the description of the partition functions of a regular Markovian holonomy field to the description of the one-parameter family of functions $Z^{1}_{1,0,t}:G\to[0,+\infty)$, $t>0$. This allows us to state a classification result.
 
\begin{proposition} \label{hol levy} Let $\HF$ be a regular Markovian holonomy field. Then the probability
measures $(Z^+_{1,0,t}(x)dx)_{t>0}$ on $G$ are the one-dimensional distributions of a unique conjugation-invariant L\'evy process issued from the unit element. If the Markovian field is not oriented, then the distribution 
of this L\'evy process is invariant by inversion. Moreover, this L\'evy process determines
completely the partition functions of $\HF$.
\end{proposition}

It is conceivable that a regular Markovian holonomy field is completely determined by its associated L\'evy process, but we have warned the reader in the introduction that we are not yet able to settle this question.\\

\begin{proof}It suffices to prove that the probability measures $\nu_{t}=Z^+_{1,0,t}(x)\; dx$ form a
convolution semigroup. Let us fix $s,t>0$. By (\ref{valeur Z+}),
we have $Z^+_{2,0,s}(x,y)=\int_{G}Z^+_{1,0,t}(xwyw^{-1})$. Now,
\begin{align*}
\nu_{t}*\nu_{s}&=\int_{G} Z^{+}_{1,0,t}(x)Z^{+}_{1,0,s}(x^{-1}y)\; dx \; dy\\
&= \int_{G^2} Z^{+}_{1,0,t}(wxw^{-1})Z^{+}_{1,0,s}(w^{-1}x^{-1}wy)\; dwdx\;  dy\\
&=\int_{G} Z^{+}_{1,0,t}(x) Z^{+}_{2,0,s}(x^{-1},y)\; dx \; dy\\
&= \beta_2(Z^{+}_{1,0,t} \otimes Z^{+}_{2,0,s})(y)\\
&= Z^{+}_{1,0,s+t}(y)\; dy=\nu_{t+s}.
\end{align*}
Proposition \ref{tqft limite} ensures that $\nu_{t}$ tends to the Dirac mass at the unit element
as $t$ tends to 0. Moreover, the conjugation invariance of the partition functions implies that
the measure $\nu_{t}$ is invariant for all $t\geq 0$.

If the Markovian field is not orientable, then for all $t$, it follows from the axiom \Ax{4} applied to an
orientation-reversing diffeomorphism of a disk of area $t$ that
$Z^+_{1,0,t}(x)=Z^+_{1,0,t}(x^{-1})$. It follows that the 1-dimensional distributions of the L\'evy
process are invariant by inversion, hence the distribution of the process itself.

The fact that the measures $(\nu_{t})_{t\geq 0}$ determine the partition functions is the content
of Proposition \ref{explicit Z}.
\end{proof}

Let us recall some classical facts about L\'evy processes in compact Lie groups and use them to prove that the function $Z^+_{1,0,t}$ is positive on the connected component of the identity of $G$ for all $t>0$. Our constant reference in this section is the book of M. Liao \cite{Liao}.

Let $X$ be an arbitrary L\'evy process on $G$ with a conjugation-invariant distribution. Let us describe briefly the
generator of $X$. Let $\gg$ be the Lie algebra of $G$. Let $A:G\lra \gg$ be a smooth mapping such that $A(1)=0$, $d_{1}A=\id_{\gg}$ and, for all $x,y\in G$,
$A(xyx^{-1})=\Ad(x)A(y)$. For example, let $r>0$ be such that the exponential mapping
is a diffeomorphism from the ball $B(0,r)$ in $\gg$ to the ball $B(1,r)$ in $G$. Let $\log$ denote
the inverse mapping. Let $\varphi:[0,+\infty)\lra [0,1]$ be a smooth function with compact support
contained in $[0,r/2]$ and equal to $1$ in a neighbourhood of $0$. Then $A(x)=\log(x)
\varphi(d_{G}(1,x))$ satisfies the required properties. 

In what follows, we identify the elements of $\gg$ with left-invariant vector fields on $G$.

\begin{proposition}\label{generateur X} Let $X$ be a L\'evy process on $G$ whose distribution is invariant by conjugation. Let $\{A_{1},\ldots,A_{d}\}$ denote a basis of $\gg$. Let $\zz$
denote the
centre of $\gg$. There exists a symmetric non-negative definite matrix $(a_{jk})_{j,k\in\{1,\ldots
d\}}$, an element $A_{0}\in \zz$, and a Borel measure $\Pi$ on $G$ which satisfies $\Pi(\{1\})=0$ and
$\int_{G}d_{G}(1,x)^2\;\Pi(dx)<+\infty$, such that the generator $L$ of $X$ is the following : for all $f\in C^2(G)$,
all $g\in G$, 
$$Lf(g)=\frac{1}{2}\sum_{j,k=1}^{d}a_{jk}(A_{j}A_{k}f)(g) +
(A_{0}f)(g) + \int_{G}\left[f(gh)-f(g)-(A(h)f)(g)\right]\; \Pi(dh).$$
The differential operator $L_{D}=\frac{1}{2}\sum_{j,k=1}^{d}a_{jk} A_{j}A_{k}$ and the measure
$\Pi$ are both invariant by conjugation. They are called respectively the diffusive part of the
generator of $X$ and the L\'evy measure of $X$. Both are independent of the choice of the mapping
$A:G\lra \gg$.
\end{proposition}
\index{PAPi@$\Pi$}
\index{conjugation!L\'evy process invariant by}
\index{L\'evy process!generator}

\begin{proof}The unique point in which this presentation differs from that of \cite{Liao} is the fact that
$A_{0}\in \zz$. The mapping $A:G\lra \gg$ has been chosen to be equivariant under the adjoint
action of $G$. This makes the third term of the generator invariant by conjugation. Since $L_{D}$ 
is also invariant, the second term must be invariant as well. This implies that $A_{0}$ belongs to
the invariant subspace of $\gg$ under the adjoint action, that is, $\zz$. \end{proof}


\index{QQQ@$Q_t$}
Let us now assume that, for all $t>0$, the distribution of $X_{t}$
has a density with respect to the Haar measure on $G$, which we denote by $Q_{t}$. 
The function $Q_{t}$ is a central function and, if $X$ is invariant by inversion, it satisfies the
relation $Q_{t}(x)=Q_{t}(x^{-1})$ for all $t>0$ and all $x\in G$.

Let $\irrep(G)$ denote the set of isomorphism classes of irreducible representations of $G$. For
each $\alpha\in \irrep(G)$, let $\chi_{\alpha}:G\lra \CK$ denote the character of $\alpha$. Also,
set 
\begin{equation}\label{coeffs fourier}
\lambda_{\alpha}=-\frac{(L_{D}\chi_{\alpha})(1)}{\chi_{\alpha}(1)},
\delta_{\alpha}=-\frac{(A_{0} \chi_{\alpha})(1)}{\chi_{\alpha}(1)}
\mbox{ and } \pi_{\alpha}=\int_{G}\left(1-\frac{\chi_{\alpha}(x)}{\chi_{\alpha}(1)}\right)\; \Pi(dx).
\end{equation}
The results of \cite[Chapter 4]{Liao} show that $Q_t$ is square-integrable 
for all $t>0$ if and only if, for all $t>0$,
\begin{equation}
\label{cond dens}
\sum_{\alpha\in \irrep(G)} e^{-(\lambda_{\alpha}+\delta_{\alpha}+\pi_{\alpha}) t} \;\chi_{\alpha}(1)^{2}
<+\infty.
\end{equation}

It is also proved in \cite{Liao} that $Q_{t}$ is square-integrable for all $t>0$ if and only if $(t,x)\mapsto
Q_{t}(x)$ is continuous on $(0,+\infty)\times G$. Let us assume that these equivalent properties are
satisfied. In this case, the following expansion is uniformly absolutely convergent on
$[\eta,+\infty)\times G$ for all $\eta>0$: 
\begin{equation}\label{densite char}
Q_{t}(x)=\sum_{\alpha\in \irrep(G)} e^{-(\lambda_{\alpha}+\delta_{\alpha}+\pi_{\alpha}) t} \;\chi_{\alpha}(1)
\overline{\chi_{\alpha}(x)}.
\end{equation}

In the following result, we use the compactness of $G$ to prove that, in this situation, $Q_t$ is positive for all $t>0$.

\begin{proposition} \label{densite positive} Let $G$ be a compact connected Lie group. Let $(X_{t})_{t\geq
0}$ be a L\'evy process on $G$ issued from $1$ and invariant in law by conjugation. Assume that,
for all $t>0$, the distribution of $X_{t}$ has a square-integrable density $Q_{t}$ with respect to
the Haar measure on $G$. Then $(t,x)\mapsto Q_{t}(x)$ is a continuous function on
$(0,+\infty)\times G$ and
$$\forall t>0, \forall x\in G , \; Q_{t}(x)>0.$$
\end{proposition}
\index{L\'evy process!with positive density}

\begin{proof}The continuity property follows from the results presented above. We focus on the assertion of positivity.

We claim that there exists $t_{0}> 0$ such that $Q_{t}(x)>0$ for all $x\in G$ and all
$t\geq t_{0}$. Indeed, since $\int_{G}Q_{1}(x)\; dx=1$, there is an open subset $U$ of $G$ on which
$Q_{1}$ is positive. Hence, for all $n\geq 1$, $Q_{n}$ is positive on $U^n$. Since $G$ is a
compact topological group, there exists $n_{0}>1$ such that $U^{n_{0}}=G$. Then $t_{0}=n_{0}$
satisfies the expected property.

Let $L=L_{D}+A_{0}+L_{J}$ be the generator of $X$ written as the sum of the diffusive part, a
drift and the jump part. Since $A_{0}$ belongs to $\zz$ and $L_{D}$ is invariant by conjugation,
these three operators commute to each other. Let $X^{D}$ and $X^{J}$ be independent L\'evy processes on $G$
with respective generators $L_{D}$ and $L_{J}$. Then we have the identity in distribution
\begin{equation}\label{id dist}
\forall t>0, X_{t}\stackrel{(d)}{=}\exp(tA_{0}) X^D_{t} X^J_{t}.
\end{equation}
The term $\exp(tA_{0})$ modifies the subset of $G$ where $Q_{t}$ is positive by a simple translation.
Hence, we may and will assume that $A_{0}=0$.  

The topological support of the distribution of $X^D_{t}$ does not depend on $t$. We denote 
it by $D$. It is the closure of the exponential of a Lie subalgebra
of $\gg$ which depends on $L_{D}$. Since $L_{D}$ is invariant by conjugation, $D$ is a closed normal subgroup of $G$.

The topological support of the distribution of $X^J_{t}$ does not depend on $t$ either and we denote
it by $J$. It is the closure of the submonoid of $G$ generated the topological support of $\Pi$.
Since $G$ is compact, the closure of the submonoid generated by any element of $G$ contains the inverse of this element. Hence, $J$ is also the closed subgroup generated by the support of $\Pi$. Since
$\Pi$ is invariant, $J$ is also a closed normal subgroup of $G$. In particular, $DJ=JD$ is a
closed subgroup of $G$.

For each $t>0$, set $S_{t}=\{x\in G : Q_{t}(x)>0\}$. We claim that $\overline{S_{t}}= DJ
\overline{S_{t}}$. Indeed, consider $x\in S_{t}$, $d\in D$ and $j\in J$. Let $U,V,W$ be three open
neighbourhoods of $x,d,j$ respectively. We claim that $\int_{UVW}Q_{t}(y)\;dy>0$. Since $(t,x)\mapsto
Q_{t}(x)$ is continuous, there exists $\epsilon>0$ such that $Q_{t-\epsilon}(x)>0$. Now,
$\P(X^D_{\epsilon}\in V)>0$ and $P(X^J_{\epsilon}\in W)>0$. Hence, by (\ref{id dist}), $\P(X_{\epsilon}\in VW)>0$
and
$$\int_{UVW}Q_{t}(y)\; dy \geq \int_{U}Q_{t-\epsilon}(x)\; dx \int_{VW}Q_{\epsilon}(y)\; dy >0.$$
Since this holds for any choice of $U,V,W$, the integral of $Q_{t}$ over any neighbourhood of
$xdj$ is positive. Hence, $xdj\in \overline{S_{t}}$. The claimed equality follows. 

Now, it follows from (\ref{id dist}), after the simplification $A_0=0$, that $\P(\forall t\geq 0, X_{t}\in DJ)=1$. If the
inclusion $DJ\subset G$ was a strict one, we would find a contradiction with the fact that $Q_{t}$ is
eventually everywhere positive on $G$. Hence, $DJ=G$. 

Putting our results together, we find $\overline{S_{t}}=G$ for all $t>0$. Now choose $t>0$, $x\in 
G$ and consider the mapping $y\mapsto Q_{t/2}(y)Q_{t/2}(y^{-1}x)$. It vanishes on
$(G-S_{t/2}) \cup x(G-S_{t/2})^{-1}$ which is the union of two closed sets with empty interior.
This set has thus empty interior, so the mapping which we consider is continuous, non-negative and not
identically zero. By integrating it with respect to $y$, we find $Q_{t}(x)>0$.\end{proof}

If $G$ is not connected, let $G^{0}$ denote the connected component of $1$. It is a normal
subgroup of $G$ and the quotient group $G/G^{0}$ is finite. The measure $\Pi$ 
induces a measure on the group $G/G^{0}$ which is finite excepted possibly on the unit element. This measure, restricted to the complement of the unit element, is the jump measure of the projection of $X$ on this finite group. 

\begin{corollary}\label{support densite} Let $G$ be a compact Lie group. Let $(X_{t})_{t\geq 0}$ be a L\'evy process on $G$ which satisfies the assumptions of Proposition \ref{densite positive}. Let $G^0$ denote the connected component of the unit element of $G$. Let $H$ be the subgroup of $G$ generated by $G^0$ and the support of $\Pi$.
Then for all $t\geq 0$, $X_t\in H$ almost surely and for all $t>0$, all $x\in H$, $Q_t(x)>0$.
\end{corollary}

\begin{proof}Let $\Pi^0$ denote the restriction of $\Pi$ to $G^0$. Since $G^0$ is a normal subgroup of $G$, both $\Pi^0$ and $\Pi-\Pi^0$ are L\'evy measures on $G$ invariant by conjugation. Moreover, $\Pi-\Pi^0$ is a finite measure. Let $X^0$ be the L\'evy process whose generator is that of $X$ in which $\Pi$ is replaced by $\Pi^0$. It is a L\'evy process in $G^0$. Let $X^J$ be the pure jump process with jump measure $\Pi-\Pi^0$. Then the generators of $X^0$ and $X^J$ commute, so that we have in distribution, for all $t\geq 0$, $X_t \build{=}_{}^{d} X^0_t X^J_t$.

For all $\alpha\in\widehat G$, we have $\sup_{x\in G} |\chi_\alpha(x)|=\chi_\alpha(1)$. Hence, changing the measure $\Pi$ by adding or subtracting to it a finite measure of mass $m$ changes each coefficient $\pi_\alpha$ by at most  $2m$. The condition (\ref{cond dens}) is not affected by such a change, so neither is the existence of a square-integrable density. This proves that the process $X^0$ satisfies the assumptions, hence the conclusions, of Proposition \ref{densite positive}. 

The set of connected components of $G$ visited by the process $X^J$ is the set of the connected components of the elements of the submonoid of $G$ generated by the support of $\Pi-\Pi^0$. Since $G$ is compact, this submonoid is also the subgroup generated by the same set. The conclusion follows easily. \end{proof}

\begin{corollary} Let $\HF$ be a regular Markovian holonomy field. There exists a subgroup $H$ of $G$ which contains the connected component of the unit element and  such that 
\begin{equation}\label{Z positive}
\forall t>0, \forall p,g\geq 0,\forall x_1,\ldots,x_p \in H, Z^+_{p,g,t}(x_1,\ldots,x_p)>0.
\end{equation}
\end{corollary}
\index{Markovian holonomy field!partition function}

\begin{proof}By Propositions \ref{hol levy} and \ref{algebre Z}, the L\'evy process associated with $\HF$ satisfies the assumptions of Corollary \ref{support densite}. Hence, (\ref{Z positive}) holds for $Z^+_{1,0,t}$. The general case follows by Proposition \ref{explicit Z}. \end{proof}

From now on, we will always assume that $H=G$.

\begin{definition}\label{def levy adm} Let $(X_t)_{t\geq 0}$ be a L\'evy process on $G$. We say that $X$ is {\em admissible} if it is issued from $1$, invariant in law by conjugation, and if for all $t>0$ the distribution of $X_t$ admits a continuous density $Q_t$ with respect to the Haar measure on $G$, such that the function $(t,x)\mapsto Q_t(x)$ is continuous and positive on $\RK^*_+\times G$.
\end{definition}
\index{L\'evy process!admissible}

Let us discuss briefly the existence of a square-integrable density for the distribution of $X$. If $G$ is a finite
group, this condition is always satisfied. An admissible L\'evy process in this case is
simply a continuous-time Markov chain on $G$ whose jump
distribution is invariant by conjugation and has a support which generates $G$. In the case of the
symmetric group, where every element is conjugated to its inverse, this invariance property implies
that the jump distribution, hence the distribution of $X$, is also invariant by inversion.

If $G$ is connected and $\dim G\geq 1$, an assumption under which the condition (\ref{cond dens}) is always
satisfied is the ellipticity of $X$. In general, the hypoellipticity is sufficient to ensure the
existence of a density, but a conjugation-invariant hypoelliptic process is necessarily elliptic.
Indeed, if $G$ is Abelian, hypoellipticity is equivalent to ellipticity and if $G$ is simple,
the invariance of $X$ implies that the diffusive part of the generator of $X$ must be a
non-negative multiple of the Laplace operator, hence elliptic or zero. The general case is a
combination of these two.

In the case where the process is not elliptic, the distribution of $X$ may or may not have a
density, depending on the jumps of $X$. The discussion of ellipticity and hypo-ellipticity above implies that if $X$ is not elliptic and has no jumps, then $X$ has no square-integrable density. The remark made in the course of the proof of Corollary \ref{support densite} implies that this is still true if the  L\'evy measure of $X$ is finite.

Let us conclude this section by giving an example of an admissible pure jump processes. Let us work on $SU(2)$. Choose a real $s$ and consider the measure $\Pi(dx)=d(1,x)^{s} \; dx$. Since the dimension of $SU(2)$ is $3$, the integral $\int_{SU(2)} d(1,x)^2 \; \Pi(dx)$ converges if and only if $s>-5$ and $\Pi$ is a finite measure for $s>-3$.  
The irreps of $SU(2)$ are labelled by their dimension which can be any positive integer. Accordingly, the Fourier coefficient $\pi_n$, which is given, thanks to Weyl's integration formula, by
$$\pi_n=\frac{2}{\pi}\int_0^\pi \left(1-\frac{\sin(n\theta)}{n\sin \theta}\right) \sin^2(\theta) \theta^s \; d\theta,$$
is non-negative and grows faster than a constant times $n^{-s-3}$. In particular, if $s<-3$, the series $\sum_{n\geq 1} e^{-\pi_n t}$ converges and the condition (\ref{cond dens}) is satisfied.  Finally, for all $s\in (-5,-3)$, the pure jump process on $SU(2)$ with L\'evy measure $\Pi(dx)=d(1,x)^{s} \; dx$ is admissible.
\index{L\'evy process!pure jump admissible}
\index{L\'evy process|)}

\section{A Markovian holonomy field for each L\'evy process}

In this section, we prove the following theorem, which is one of the main results of the present work. Recall Definitions \ref{def SCFMHF} and  \ref{def levy adm}, and Proposition \ref{hol levy}.

\begin{theorem}\label{levy HF} Every admissible L\'evy process is the L\'evy process associated to a regular Markovian holonomy field.  
\end{theorem} 

Whether this regular Markovian holonomy field is unique is a natural question which we hope to be able to answer in a future work. 

In order to prove this theorem, we use the results of the previous chapter. We start by constructing a discrete Markovian holonomy field, prove that it is regular and extend it to a Markovian holonomy field.

\subsection{A discrete Markovian holonomy field}

Let $X$ be an admissible L\'evy process. Let $Q_t$ denote the density of the distribution of $X_t$. Let $(M,\vol,\CS,C)$ be a connected measured marked surface with $G$-constraints. Let $\G$ be a graph on $(M,\CS)$. For each face $F$ of $\G$, recall 
that $\partial F$ is a cycle, oriented or non-oriented depending on the orientability of
$M$. Assume first that $M$ is orientable. For each $h\in \M(\Path(\G),G)$, different choices of the origin of $\partial F$ lead to different elements $h(\partial F)$ of $G$, but these elements belong to the same conjugacy class of $G$. Hence, for all $t>0$, the assumption that the distribution of $X_{t}$ is invariant by conjugation
makes the positive real number $Q_{t}(h(\partial F))$ well-defined. 

If $M$ is non-orientable, then $h(\partial F)$ is defined only up to conjugation and inversion.
In this case, we make the further assumption that the distribution of $X$ is invariant by inversion. Then, for all $t>0$, the non-negative real number $Q_{t}(h(\partial F))$ is also well-defined.

\begin{definition} \label{discrete YM def} Let $X$ be an admissible L\'evy process. Let $(M,\vol,\CS,C)$ be a measured marked surface with $G$-constraints. Let $\G$ be a graph on $(M,\CS)$. We define the following measure
on $(\M(\Path(\G),G),\C)$:
$$\DF^{X,\G}_{M,\vol,\CS,C}(dh)= \prod_{F\in\F}Q_{\vol(F)}(h(\partial F)) \; \U^{\G}_{M,\CS,C}(dh).$$
We denote the collection of these measures by $\DF^X$.
\end{definition}
\index{DADF@$\DF^{X,\G}_{M,\vol,\CS,C}$}
\index{Markovian holonomy field!discrete|(}

\begin{proposition}\label{Ax 1-6} Let $X$ be an admissible L\'evy process. The collection of measures $\DF^X$ satisfies the axioms \AxD{1} to \AxD{6} of a discrete Markovian holonomy field.
\end{proposition}

\begin{proof}For each quadruple $(M,\vol,\CS,C)$, $\DF^{X,\G}_{M,\vol,\CS,C}$ is a measure on the invariant $\sigma$-field of $\M(\Path(\G),G)$. It has a bounded density with respect to the probability measure $\U^\G_{M,\CS,C}$, so that it is a finite measure. Let us prove that the axioms \AxD{1} to \AxD{6} are satisfied. 

The fact that the discrete Markovian holonomy field $\U$ satisfies \AxD{1} and \AxD{3} implies immediately that $\DF^X$ also satisfies them. The argument used for $\U$ in the proof of Proposition \ref{U DMHF} shows that $\DF^X$ satisfies \AxD{2}. Let $\psi$ be a homeomorphism as in the statement of the axiom \AxD{4}. The measure $\DF^{X,\G}_{M,\vol,\CS,C}$ depends only on the combinatorial structure of the graph $\G$, on the cycles which represent the curves of $\CS$,  on the set of $G$-constraints and finally the boundaries and the areas of the faces of $\G$. These characteristics are all preserved by the homeomorphism $\psi$. The axiom \AxD{5} is obviously satisfied. Let us finally check that $\DF^X$ satisfies \AxD{6}. 

Let us denote by $M'$ the surface ${\Spl}_{l}(M)$ and by $\G'$ the graph $\Spl_l(\G)$. Let us also denote by
$D':\M(\Path(\G'),G)\to \RK$ the density of the measure $\DF^{X,\G'}_{M',\vol',\CS',C'}$ with respect to $\U^{\G'}_{M',\CS',C'}$. Then $D'\circ \psi:\M(\Path(\G),G)\to \RK$ is the density of $\DF^{X,\G}_{M,\vol,\CS,C}$ with respect to $\U^{\G}_{M,\CS,C}$. Hence, the property follows from Proposition \ref{prop markov unif}. \end{proof}

We prove now that the collection of measures that we consider satisfies the property of invariance under subdivision.

\begin{proposition}\label{Ax I} Let $X$ be an admissible L\'evy process. The collection of measures $\DF^{X}$ satisfies the axiom \AxD{I} of a discrete Markovian holonomy field.
\end{proposition}

\begin{proof}Consider $(M,\vol,\CS,C)$ endowed with two graphs $\G_1$ and $\G_2$ such that $\G_1\preccurlyeq \G_2$.
Let $r:\M(\Path(\G_2),G)\to \M(\Path(\G_1),G)$ denote the restriction map.
Let us first make the assumption that $\E_{1}\subset \E_{2}$ and choose orientations $\E_1^+$ and $\E_2^+$ of $\G_1$ and $\G_2$ such that $\E_{1}^+\subset \E_{2}^+$. The restriction map $r$ can be thought of as a map from $G^{\E_2^+}$ to $G^{\E_1^+}$. Let us write $\E_2^+=\E_1^+ \cup (\E_2^+\setminus \E_1)^+$ and decompose the generic element of $G^{\E_2^+}$ as $g=(g_1,g_2)$ accordingly. With this notation, $r(g_1,g_2)=g_1$. Let $f:G^{\E_{1}^+}\to \RK$ be a continuous function. We need to prove that 
$$\int_{G^{\E_{2}^{+}}} f (g_1) \;\DF^{X,\G_{2}}_{M,\vol,\CS,C}(dg_1,dg_2)=\int_{G^{\E_{1}^{+}}} f (g_1) \;\DF^{X,\G_{1}}_{M,\vol,\CS,C}(dg_1)$$

We are going to perform the integration on the left-hand side with respect to $dg_2$. For this, let us observe that the curves of $\CS$ belong to $\Path(\G_{1})$, so that the measure $\U^{\G_2}_{M,\vol,\CS,C}(dg_1,dg_2)$ on $G^{\E_2^+}$ can be written as $\U^{\G_1}_{M,\vol,\CS,C}(dg_1) dg_2$. This is in fact an instance of Proposition \ref{inv subd unif}. Hence, if we put together the faces of $\G_{2}$ according to the
face of $\G_{1}$ in which they are contained, we find the following expression for our integral:
\begin{equation}\label{subdivise}
\int_{G^{\E_{1}^+}} f(g_1) \prod_{F_{1}\in\F_{1}}\left[ \int \prod_{\substack{F_{2}\in
\F_{2}\\ F_{2}\subset F_{1}}} Q_{\vol(F_{2})}(g(\partial F_{2})) \prod_{\substack{e\in \E_{2}^+\setminus \E_{1}^+ \\ e\subset F_{1}}} dg_{e} \right] \; \U^{\G_{1}}_{M,\CS,C}(dg_1).
\end{equation}

The integral between the brackets is over $G^{\{e\in \E_{2}^+\setminus \E_{1}^+ : e \subset  F_{1}\}}$. It
suffices to prove that this integral is equal to $Q_{\vol(F_{1})}(h_{\partial F_{1}}(g))$.

We proceed then by induction on the number of faces of $\G_{2}$ contained in $F_{1}$. Let us assume first that this number is 1 and denote by $F_2$ the unique face of $\G_2$ contained in $F_1$. In order to treat this case, we proceed by induction on the number of edges of $\G_{2}$ whose interior is contained in $F_{1}$. If this number is zero, then $F_{1}=F_2$  and the expression between the brackets is exactly $Q_{\vol(F_{1})}(g(\partial F_{1}))$. Now let us assume that there is at least one edge of $\G_{2}$ whose interior is contained in
$F_{1}$. Let us consider a split pattern $M'$ of $(M,\G_1)$ and let $M'_{F_1}$  denote the connected component corresponding to $F_1$. Let $\G'_2$ be the graph on $M'$ induced by $\G_2$. The restriction of $\G'_2$ to $M'_{F_1}$ is a graph with a single face on a disk. By Euler's formula, this graph has the same number of edges and vertices. By assumption, there is at least one vertex of degree at least 3 on the boundary of $M'_{F_1}$. 
Hence, $\G'_2$ has at least one vertex of degree 1, which must be in the interior of $M'_{F_1}$ and hence is sent to a vertex of $\G_2$ of degree 1 contained in $F_1$. Let $e$ be an edge of $\G_{2}$ adjacent to a vertex $v$ of degree 1. The cycle $\partial F_{2}$ contains either the sequence $ee^{-1}$ or the sequence $e^{-1}e$. This sequence can be removed without affecting the value of the expression between the brackets. The cycle $\partial F_{2}$ with the sequence removed is the boundary of the face $F_{2}\cup e((0,1))\cup \{v\}$ of the graph whose set of edges is $\E_{2}\setminus \{e,e^{-1}\}$. This graph has one edge less inside $F_{1}$ than $\G_{2}$
and the result follows by induction.

Let $n\geq 2$ be an integer and let us assume that the
result has been proved when $F_{1}$ contains at most $n-1$ faces of $\G_{2}$. Consider the case
where $F_{1}$ contains $n$ faces of $\G_{2}$.

Let $F_2,F'_2$ be two distinct adjacent faces of $\G_{2}$ which are both
contained in $F_{1}$. The boundaries of $F_2$ and $F'_2$ are respectively of the form $e_{1}\ldots e_{k} e''$ and $(e'')^{-1}e'_{1}\ldots e'_{l}$, where
$\{e_{1},\ldots,e_{k},e'_{1},\ldots,e'_{l}\}\subset \E_{2}^+$ and $e''\in
\E_{2}^{+}\setminus \E_{1}^{+}$. When we integrate with respect to the component of $g$ corresponding to $e''$ between the brackets in (\ref{subdivise}), we find
$$\int_{G} Q_{\vol(F_2)}(g_{e''}g_{e_k}\ldots g_{e_1}) Q_{\vol(F'_2)}(g_{e'_l}\ldots
g_{e'_1}( g_{e''})^{-1}) \; dg_{e''},$$
which, by the Markov property of $X$, is equal to $Q_{\vol(F\cup
F')}(g_{l}'\ldots g_{1}' g_{k}\ldots g_{1})$. We are thus reduced to the graph obtained from
$\G_{2}$ by merging $F_2$ and $F'_2$ along the edge $e''$. By Proposition \ref{erase edge}, the result
of this operation is indeed a graph. The induction hypothesis applied to this new graph yields the desired result.

Finally, let us treat the case where $\E_{1}\not\subset \E_{2}$. In this case, there are vertices
of $\G_{2}$ located on the edges of $\G_{1}$ which are not vertices of $\G_{1}$. Adding these
vertices to $\G_{1}$ and splitting its edges accordingly produces a new graph $\G$ such that $\G_1\preccurlyeq \G \preccurlyeq \G_2$ and $\E\subset \E_2$.
It remains to prove that the restriction $r:\M(\Path(\G),G)\to \M(\Path(\G_1),G)$ sends the measure $\DF^{X,\G}_{M,\vol,\CS,C}$ to $\DF^{X,\G_1}_{M,\vol,\CS,C}$. This follows at once from (\ref{contracte dO}). \end{proof}

Combining Propositions \ref{456 Z}, \ref{Ax 1-6} and \ref{Ax I}, we find the following result.

\begin{proposition} Let $X$ be an admissible L\'evy process. Let $Z^{X,\G}_{M,\vol,\CS,C}$ denote the partition functions associated with the collection of measures $\DF^{X}$. Let $(M,\vol,\CS,C)$ be a measured marked surface with $G$-constraints. Then $Z^{X,\G}_{M,\vol,\CS,C}$ does not depend on the graph $\G$. We denote it by $Z^X_{M,\vol,\CS,C}$.
\end{proposition}

We are now going to compute this partition function. For this, we start by associating a probability measures on $G$ to each connected surface $(M,\varnothing,C)$ with $G$-constraints along the boundary. Recall the definition of the measures $\eta$ and $\kappa$ from Definition \ref{kappasigma} and the properties that they satisfy proved in Proposition \ref{convol k s}.  

\begin{definition}\label{def m} Let $(M,\varnothing,C)$ be a connected surface with $G$-constraints along the boundary. If $M$ is oriented, write $\BS^+(M)=\{b_1,\ldots,b_p\}$ and set
$$\m_{M,\varnothing,C}=\eta^{*\frac{\rg(M)}{2}}*\delta_{C(b_1)}*\ldots*\delta_{C(b_p)}.$$
If $M$ is non-orientable, write $\BS(M)=\{b_1^{\pm 1},\ldots,b_p^{\pm 1}\}$ and set
$$\m_{M,\varnothing,C}=\kappa^{*\rg(M)}*\delta_{C(b_1)}*\ldots*\delta_{C(b_p)}.$$
\end{definition}

\begin{remark}\label{larem} 1. The second definition is meaningful thanks to Lemma \ref{convol k s}. Indeed, the orientation chosen on the boundary components of $M$ does not affect the definition of $\m_{M,\varnothing,C}$. \\
2. Assume that $(M,\varnothing,C)$ is an oriented surface with $G$-constraints. Write $M^\vee$ for the same surface with the opposite orientation. Then $\m_{M^\vee,\varnothing,C}=\m_{M,\varnothing,C}^\vee$.
\end{remark}

\begin{lemma}\label{mesure m} Let $M$ be a connected compact surface. \\
1. The mapping which to a set $C$ of $G$-constraints on $\BS(M)$ associates the probability measure $\m_{M,\varnothing,C}$ on $G$ is continuous from $\Const_G(M,\varnothing)$ to the space of probability measures on $G$ endowed with the topology of weak convergence.\\
2. For all $b\in \BS(M)$, the measure $\int_G \m_{M,\CS,C_{b\to x}}\; dx$ is the Haar measure on $G$.
\end{lemma}

\begin{proof}1. This property follows from the continuity of the mapping $x\mapsto \delta_{\O_x}$ and the continuity of the convolution of measures.

2. This assertion follows from the fact that $\int_G \delta_{\O_x} \; dx$ is the Haar measure on $G$. \end{proof}

\begin{proposition}\label{calcul de Z} Let $X$ be an admissible L\'evy process. Consider the collection of measures $\DF^{X}$. Let $(M,\vol,\CS,C)$ be a connected measured marked surface. Let $(M',\vol',\varnothing,C')$ be a split tubular pattern of $(M,\CS,C)$ endowed with the induced $G$-constraints. Let $M'_1,\ldots,M'_s$ be the connected components of $M'$ and for each $i\in \{1,\ldots,s\}$, let $(M'_i,\varnothing,C'_i)$ be the associated connected surface with constraints. If $M$ is oriented, then $M'$ carries the induced orientation. If $M$ is non-orientable, then let us assign an arbitrary orientation to each orientable connected component of $M'$. Then the following equality holds : 
\begin{equation}\label{valeur ZX}
Z^{X}_{M,\vol,\CS,C}=\prod_{i=1}^s \int_G Q_{\vol'(M'_i)} \; d\m_{M'_i,\varnothing,C'_i}.
\end{equation}
\end{proposition}

\begin{proof}When $M$ is non-orientable, there is a choice made in assigning an orientation to each orientable connected component of $M'$. However, in this case, the distribution of $X$, hence the function $Q_t$ for all $t>0$, is invariant by inversion. Hence, by the second part of Remark \ref{larem}, the right-hand side of (\ref{valeur ZX}) is well defined. 

The proof of this equality is of the same vein as that of Proposition \ref{456 Z}. By \Ax{6}, $Z^{X}_{M,\vol,\CS,C}=Z^X_{M',\vol',\varnothing,C'}$. Then, by \Ax{5}, $Z^X_{M',\vol',\varnothing,C'}=\prod_{i=1}^s 
Z^X_{M'_i,\vol',\varnothing,C'_i}$. The problem is thus reduced to the case of a connected surface with $G$-constraints along the boundary. 

Let us assume that $M$ is connected and $\CS=\varnothing$. In order to compute the partition function in this case, we choose a graph on $M$ with a single face and, by cutting and pasting, transform it so that the boundary of its unique face has a canonical form. Then, we find, for instance if $M$ is non-orientable of reduced genus $g\geq 1$ with $p$ boundary components,
\begin{eqnarray*}
Z^{X}_{M,\vol,\varnothing,C}&=&\int_{G^{g+2p}} Q_{\vol(M)}(a_1^2 \ldots a_g^2 u_1 c_1 u_1^{-1} \ldots u_p c_p u_p^{-1})\\
&&\hskip 3.1cm  da_1 \ldots da_g du_1 \ldots du_p \delta_{C(b_1)}(dc_1) \ldots \delta_{C(b_p)}(dc_p)\\
&=& \int_{G^{g+p}} Q_{\vol(M)}(a_1^2 \ldots a_g^2  c_1  \ldots  c_p) da_1 \ldots da_g \delta_{C(b_1)}(dc_1) \ldots \delta_{C(b_p)}(dc_p)\\
&=&\int_{G} Q_{\vol(M)} d\m_{M,\varnothing,C}.
\end{eqnarray*}
The other cases are similar. \end{proof}

We can summarize our results.

\begin{proposition}\label{DFX dhf} Let $X$ be an admissible L\'evy process. The collection of measures $\DF^{X}$ defined in Definition \ref{discrete YM def} is a discrete Markovian holonomy field. 
\end{proposition}

\begin{proof}For all $(M,\vol,\CS,C)$, the measure $\DF^{X,\G}_{M,\vol,\CS,C}$ is a measure on the cylinder $\sigma$-field of $\M(\Path(\G),G)$, hence it determines by restriction a measure on the invariant $\sigma$-field. By Proposition \ref{calcul de Z}, it is a finite measure. 

The collections of these measures satisfies the axioms \AxD{1} to \AxD{6} by Proposition \ref{Ax 1-6}, \AxD{I} by Proposition \ref{Ax I} and \AxD{7} by the combination of Proposition \ref{calcul de Z}, the second assertion of Lemma \ref{mesure m} and the fact that for all $t>0$, $\int_G Q_t(x)\; dx =1$. \end{proof}

\subsection{A Markovian holonomy field}

Our next goal is to prove that the discrete Markovian holonomy field $\DF^X$ is regular in the sense of Definition \ref{def regular DHF}.

\begin{proposition}\label{cad f} Let $X$ be an admissible L\'evy process. The discrete Markovian holonomy field $\DF^X$ is continuously area-dependent and Fellerian.
\end{proposition}

\begin{proof}Recall the notation of Definition \ref{def regular DHF}. For each $n\geq 0$, we have on $\M(\Path(\G),G)$ the equality of measures
$$\DF^{X,\G_n}_{M,\vol,\CS,C}\circ \psi_n^{-1} (dh) = \prod_{F\in\F}Q_{\vol(\psi_n(F))}(h(\partial F)) \; \U^{\G}_{M,\CS,C}(dh).$$
By assumption, $\vol(\psi_n(F))$ tends to $\vol(F)$ for each $F\in \F$. Moreover, for all segment $[s,t]\subset \RK^*_+$, the mapping $(t,x)\mapsto Q_t(G)$ is uniformly continuous on $[s,t]\times G$. Hence, for all face $F\in \F$, $Q_{\vol(\psi_n(F))}$ converges uniformly to $Q_{\vol(F)}$ as $n$ tends to infinity. The fact that $\DF^X$ is continuously area-dependent follows.

The fact $\DF^X$ is Fellerian, that is, that the partition function $Z^X_{M,\vol,\CS,C}$ depends continuously on $C\in \Const_G(M,\CS)$, follows at once from Proposition \ref{calcul de Z} and the first assertion of Corollary \ref{mesure m}. \end{proof}

In order to prove that $\DF^X$ is stochastically $\frac{1}{2}$-H\"{o}lder continuous, we need to establish the corresponding property for the L\'evy process $X$. 

\begin{proposition} \label{sqrt X} Let $(X_{t})_{t\geq 0}$ be a L\'evy process on the compact
Lie group $G$ issued from 1. Then there exists a constant $K$ such that
$$\forall t\geq 0, \; \E\left[d_{G}(1,X_{t})\right] \leq K \sqrt{t}.$$
\end{proposition}

This property follows from Lemma 3.5 in the book of M. Liao \cite{Liao}, but we still offer a short proof.\\

\begin{proof}We use the It\^{o} formula for L\'evy process on Lie groups, which has been
proved by Applebaum and Kunita \cite{ApplebaumKunita}. We borrow the statement from
\cite[Section 1.4]{Liao}. In fact we use the following weak statement. Let $L$ be the generator of
$X$. Let $f$ be a smooth function on $G$. Then $f$ belongs to the domain of $L$ and there exists a
$L^2$ martingale $M^f$ such that, for all $t\geq 0$,
\begin{equation}\label{martingale}
f(X_{t})=f(X_{0})+M^f_{t}+\int_{0}^{t} Lf(X_{s})\; ds.
\end{equation}
This is the equation (1.18) of \cite{Liao}. We apply it to a function $f$ which is close to the
function $d_{G}(1,\cdot)^2$.

Let $\{A_{1},\ldots,A_{d}\}$ be a basis of the Lie algebra of $G$, which we identify with the
space of left-invariant vector fields on $G$. Let $a_{1},\ldots,a_{d}$
be smooth functions on $G$ such that for all $i,j\in\{1,\ldots,d\}$, $a_{i}(1)=0$ and
$A_{i}a_{j}(1)=\delta_{ij}$. Set $\delta=\sum_{i=1}^{d} a_{i}^{2}$.
It follows readily from the definition of $\delta$ and the fact that $G$ is compact that there
exists a constant $K_{1}$ such that for all $x\in G$, $d_{G}(1,x)^2 \leq K_{1} \delta(x)$. Now
(\ref{martingale}) implies that $\E[\delta(X_{t})]\leq \parallel\! L\delta \!\parallel_{\infty} t$.
Hence, by Jensen's inequality, $\E[d_{G}(1,X_{t})]^2\leq K_{1}\parallel \! L\delta
\! \parallel_{\infty} t$ for all $t\geq 0$.  \end{proof}

In order to deduce the stochastic H\"{o}lder continuity of $\DF^X$ from this property, we need to be able to compare the values of some integrals with and without $G$-constraints. Actually, we introduce random holonomy fields with free boundary conditions. Recall the definition of the uniform measure $\U^\G_{M,\varnothing}$ (Definition \ref{def U free}).  

\begin{definition} Let $X$ be an admissible L\'evy process. Let $(M,\vol)$ be a measured surface endowed with a graph $\G$. We define the measure $\DF^{X,\G}_{M,\vol,\varnothing,\varnothing}$ by setting
$$\DF^{X,\G}_{M,\vol,\varnothing,\varnothing}(dh)= \prod_{F\in\F}Q_{\vol(F)}(h(\partial F)) \; \U^{\G}_{M,\varnothing}(dh).$$
\end{definition}

In the following proofs, we use the fact that the functions $t\mapsto \sup\{Q_t(x) : x\in G\}$ and $t\mapsto \inf\{Q_t(x) : x\in G\}$ are respectively non-increasing and non-decreasing. This follows from $(Q_t)_{t>0}$ being a convolution semigroup of positive continuous functions.

\begin{lemma}\label{relax constraint} Let $(M,\vol,\CS,C)$ be a measured marked surface with $G$-constraints. Let $\G$ be a graph on $(M,\CS)$. 
Consider  ${\E_1}\subset \E$ and ${\F_1}\subset \F$. Assume that ${\E_1}={\E_1}^{-1}$. Assume that for each  $l\in \CS\cup\BS(M)$, at least one edge of ${\E_1}$ is located on $l$, and each face adjacent to an edge of $\E_1$ belongs to $\F_1$. Set $r=\sharp {\F_1}$ and $A=\min\{\vol(F):F\in{\F_1}\}$. Set $K=\sup\left\{\frac{Q_A(x)}{Q_A(y)} : x,y\in G\right\}.$

Let $f:\M(\E,G)\to [0,+\infty)$ be a non-negative continuous function. Assume that $f$ factorizes through the restriction map $\M(\E,G)\to \M(\E\setminus \E_1,G)$. Then
$$ K^{-r}\! \int_{\M(\E,G)} f \;d\DF^{X,\G}_{M,\vol,\varnothing,\varnothing} \leq  \int_{\M(\E,G)} f\; d\DF^{X,\G}_{M,\vol,\CS,C} \leq K^r\! \int_{\M(\E,G)} f \;d\DF^{X,\G}_{M,\vol,\varnothing,\varnothing}.$$
\end{lemma}

\begin{proof}Increasing the number of edges in $\E_1$ can only increase $\F_1$, hence make $A$ smaller and $K$ larger. So, without loss of generality, we may assume that $\E_1$ contains exactly one non-oriented edge on each curve $l\in\CS\cup\BS(M)$ and $\F_1$ is exactly the set of faces adjacent to these non-oriented edges.

Let us choose an orientation $\E^+$ of $\G$ and identify $\M(\E,G)$ with $G^{\E^+}$.  Let us enumerate $\F_1$ as $\{F_1,\ldots,F_r\}$. Let us denote the generic element of $G^{\E^+}$ as $g=(g_1,g_2)$ according to the partition $\E=\E_1 \cup (\E\setminus \E_1)$. The assumption on $f$ expresses that $f(g_1,g_2)$ depends only on $g_2$. 

By (\ref{put constraint}), the integration against $\U^{\G}_{M,\CS,C}$ can be decomposed into the integration with respect to the Haar measure on $\M(\E\setminus \E_1,G)$ and then with respect to the explicitly known conditional distribution of $g_1$ given $g_2$, which we denote by $\U^\G_{M,\CS,C}(dg_1|g_2)$.
 
\begin{align*}
&\int_{\M(\E,G)} f\; \DF^{X,\G}_{M,\vol,\CS,C} = \int_{\M(\E,G)} f(g) \prod_{F\in\F} Q_{\vol(F)}(g(\partial F)) \; \U^\G_{M,\CS,C}(dg)=\\
& \int_{G^{(\E\setminus \E_1)^+}} \! f(g_2) \! \prod_{F\in\F\setminus \F_1} Q_{\vol(F)}(g_2 (\partial F)) \left[ \int_{G^{\E_1^+}}\! \prod_{F\in\F_1} Q_{\vol(F)}(g (\partial F)) \; \U^\G_{M,\CS,C}(dg_1 | g_2) \right]  dg_2.
\end{align*}

Changing the probability measure with respect to which the integral between the brackets is taken can at most multiply the integral by $\frac{\max u}{\min u}$ and most divide it by the same number, where $u$ denotes the integrand. In the present situation, the definition of $K$ implies that $\frac{\max u}{\min u} \leq K^r$.

Hence, focusing for example on the upper bound, we have
\begin{eqnarray*}
\int_{\M(\E,G)} f\; \DF^{X,\G}_{M,\vol,\CS,C} &\leq & K^r \int_{G^{\E^+}} f(g_2) \prod_{F\in\F} Q_{\vol(F)}(g(\partial F)) \; dg_1 dg_2\\
&=& \int_{\M(\E,G)} f \; d\DF^{X,\G}_{M,\vol,\varnothing, \varnothing}.
\end{eqnarray*}

The derivation of the lower bound is similar. \end{proof}

\begin{proposition}\label{1/2 holder} Let $X$ be an admissible L\'evy process. The discrete Markovian field $\DF^X$ is stochastically $\frac{1}{2}$-H\"{older} continuous.
\end{proposition}

\begin{proof}Let $(M,\vol,\gamma,\CS,C)$ be a measured marked surface with $G$-constraints. Write $\CS\cup\BS(M)=\{l_{1},l_1^{-1},\ldots,l_{q},l_q^{- 1}\}$. Let $M_{1},\ldots,M_{s}$ denote the connected components of $M\setminus \CS$. Set $A=\frac{1}{2}\min\{\vol(M_{i}):i\in \{1,\ldots,s\}\}$. For each $i\in \{1,\ldots,q\}$, let us write $l_{i}$ as the product of three edges: $l_{i}=e_{i,1}e_{i,2}e_{i,3}$. Let $L>0$ be such that any Riemannian ball of radius smaller than $L$ intersects at most one curve of $\CS\cup\BS(M)$ and at most two of the edges $\{e_{i,j} : i\in\{1,\ldots,q\}, j\in\{1,2,3\}\}$, and has a Riemannian area smaller than $A$.

Let $l$ be a piecewise geodesic loop such that $\ell(l)< L$ and $l$ bounds a disk which we denote
by $D$. If $l$ bounds two disks, we choose the one included in the ball of radius $L$ centred at the
basepoint of $l$. By assumption on $L$, there is at most one $i\in \{1,\ldots,q\}$ such that $l$ meets one of the edges $e_{i,1},e_{i,2}, e_{i,3}$ and it does not meet the three of them. We may assume that $l$ meets none of the edges $e_{i,j}$ except possibly $e_{1,2}$ and $e_{1,3}$.

Let $\G$ be a graph on $(M,\CS)$ such that the edges $e_{i,j}$ are edges of $\G$. By repeated applications of Proposition  \ref{erase edge}, we may assume that $\G$ has exactly one face in each connected component of $M\setminus (\CS\cup \{l\})$. The number of these components depends on the relative position of $l$ and the curve $l_{1}$. Nevertheless, $\G$ has the following property: each face adjacent to one of the edges $e_{1,1},\ldots,e_{p,1}$ or their inverses has an area greater or equal to $A$. 

Let us define $\E_1=\{e_{1,1},e_{1,1}^{-1},\ldots,e_{q,1},e_{q,1}^{-1}\}$ and $\F_1$ as the subset of $\F$ consisting of all faces adjacent to an edge of $\E_1$. We have $\sharp F_1 \leq 2q$. Set $K_{X,A}=\sup\left\{\frac{Q_A(x)}{Q_A(y)} : x,y\in G\right\}$. Then, by Lemma \ref{relax constraint},
\begin{equation}\label{eqeq}
\int_{\M(\Path(\G),G)} \! d_G(1,h(l)) \; \DF^{X,\G}_{M,\vol,\CS,C}(dh) \leq  K_{X,A}^{2q} \int_{\M(\Path(\G),G)} \! d_G(1,h(l)) \; \DF^{X,\G}_{M,\vol,\varnothing,\varnothing}(dh).
\end{equation}
By the axiom \AxD{I}, we can remove edges from $\G$ so that it becomes a graph $\G_1$ with only two faces, $D$ and another one, denoted by $F$, of area $\vol(M\setminus D)$, without altering the value of the integral above. Hence, by proposition \ref{sqrt X},
\begin{align*}
\mbox{l.h.s. of } (\ref{eqeq}) 
&\leq  K_{X,A}^{2q} \int_{\M(\Path(\G_1),G)} \! d_G(1,h(l)) Q_{\vol(D)}(h(l)) Q_{\vol(M-D)}(h(\partial F)) \; \U^{\G}_{M,\varnothing}(dh)\\
&\leq  K_{X,A}^{2q+1} \int_{G}  d_{G}(1,x) Q_{\vol(D)}(x) \;
dx\\
&= K_{X,A}^{2q+1} \E\left[d_{G}(1,X_{\vol(D)})\right] \\
&\leq K  K_{X,A}^{2q+1} \sqrt{\vol(D)}.
\end{align*}

This is the expected result. \end{proof}

We can conclude this chapter by proving Theorem \ref{levy HF}.

\begin{proof}[Proof of Theorem \ref{levy HF}] Let $X$ be an admissible L\'evy process. Let $\DF^X$ be defined by Definition \ref{discrete YM def}.  By Proposition \ref{DFX dhf}, it is a discrete Markovian holonomy field. By Propositions \ref{cad f} and \ref{1/2 holder}, it is regular. By theorem \ref{main existence}, $\DF^X$ is the restriction of a regular Markovian holonomy field, which we denote by $\HF^X$. By Proposition \ref{calcul de Z}, the L\'evy process associated with $\HF^X$ is indeed $X$.
\end{proof}
\index{Markovian holonomy field!discrete|)}

\chapter{Random ramified coverings}

In this chapter, we investigate the Markovian holonomy field that we have associated to a L\'evy
process in the case where $G$ is a finite group. In this case, the structure of the L\'evy process
is particularly simple. It is a continuous time Markov chain with a jump distribution invariant by
conjugation and, depending on the orientation issue, by inversion.

It turns out that in this case, the canonical process associated to the Markovian holonomy field is
the process of monodromy in a random ramified covering picked under a probability measure which
depends in a simple way on the L\'evy process. This is consistent with the usual heuristic
interpretation of the Yang-Mills measure as a probability measure on the space of connections on a
principal bundle. Indeed, ramified coverings can be naturally interpreted as discrete models for
principal bundles, endowed with a connection which is flat everywhere but at the ramification points,
where it is concentrated.  

\section{Ramified $G$-bundles} 

Let us choose once for all a finite group $G$. Let $(M,\vol,\varnothing,C)$ be a measured surface
with $G$-constraints on the boundary. For the sake of simplicity, we treat the case $\CS=\varnothing$.

Let $Y\subset M\setminus\partial M$ be a finite subset. A principal
$G$-bundle over $M-Y$ is a smooth covering $\pi:P\to M\setminus Y$ of $M\setminus Y$ by a surface $P$ on which $G$
acts freely on the right, by smooth automorphisms of covering 
and transitively on each fibre. The surface $P$ is not compact unless $Y=\varnothing$ and in
general it is not connected. Two
$G$-bundles $\pi:P\to M\setminus Y$ and $\pi':P'\to M\setminus Y$ are
isomorphic if there exists a $G$-equivariant diffeomorphism $h:P\to P'$ such that $\pi'\circ h=\pi$.

A ramified bundle over $M$ with ramification locus $Y$ is a continuous mapping $\pi:P\to M$ from a
surface $P$ such that the restriction of $\pi$ to $\pi^{-1}(M\setminus Y)$ is a covering and, for all $y\in Y$ and all $p\in
\pi^{-1}(y)$, there exists a neighbourhood $U$ of $p$ and an integer $n\geq 1$ such that the mapping
$\pi_{|U}:(U,p)\to (\pi(U),y)$ is conjugated to the mapping $z\mapsto z^{n} : (\CK,0)\to (\CK,0)$.
The integer $n$ is called the order of ramification of $p$. We assume that for all $y\in Y$, there
exists $p\in\pi^{-1}(y)$ whose order of ramification is at least $2$. Two ramified coverings $\pi:P\to M$ and
$\pi':P'\to M$ are isomorphic if there exists a homeomorphism $h:P\to P'$ such that $\pi'\circ h=\pi$.
\index{ramified principal bundle}

From the classical fact that the only connected coverings of finite degree of $\CK^*$ are, up to
isomorphism, the mappings $z\mapsto z^n: \CK^* \to \CK^*$ for $n\geq 1$, it follows that a
principal $G$-bundle $\pi:P\to M\setminus Y$ can always be extended to a ramified covering of $M$ by a
suitable compactification of $P$, and that any two such extensions give rise to isomorphic ramified
coverings. Moreover, it is possible to endow the total space of the ramified covering with a
differentiable structure in such a way that the covering map is smooth. 

\begin{definition} A {\em ramified principal $G$-bundle over
$M$ with ramification locus $Y$} is a smooth ramified covering $\pi:R\to M$ of $M$ with ramification
locus $Y$, together with an action of $G$ on $\pi^{-1}(M\setminus Y)$ which endows the restriction of $\pi$ 
to $\pi^{-1}(M\setminus Y)$ with the structure of a principal $G$-bundle. 

Two ramified principal $G$-bundles $\pi:R\to M$ and $\pi:R'\to M$ with ramification locus $Y$ are
{\em isomorphic} if their restrictions to $M-Y$ are isomorphic as principal $G$-bundles.
\end{definition}

\begin{remark} By the discussion before the definition, two isomorphic ramified $G$-bundles are also
isomorphic as ramified coverings. However, an isomorphism of ramified coverings between two
ramified $G$-bundles is not necessarily an isomorphism of ramified $G$-bundles. Consider for
example, for $n\geq 3$, the trivial $\Z/n\Z$-bundle $R=M\times \Z/n\Z$. The group $\Sy_{n}$ acts on
$R$ by permuting the sheets and this is an action by automorphisms of covering. Nevertheless, only
a cyclic permutation of the sheets is an isomorphism of $\Z/n\Z$-bundle. In general, the group of
covering automorphisms of a principal $G$-bundle is bigger than $G$. This is related to the
fact that the total space of the covering is not always connected, so that the group of covering
automorphisms does not always act freely.
\end{remark}

\begin{remark} In the case where $G$ is the symmetric group $\Sy_{n}$, ramified
$G$-bundles are the same thing as ramified coverings of degree $n$. Indeed, let $\pi:P\to M$ be a principal $\Sy_{n}$-bundle, ramified over $Y$. The group $\Sy_{n}$ acts naturally on $\{1,\ldots,n\}$ and the associated bundle $P\times_{\Sy_{n}} \{1,\ldots,n\}$, which is the quotient of $P\times \{1,\ldots,n\}$ by the relation 
$$(p,k)\simeq (p',k') \Leftrightarrow \exists \sigma \in \Sy_{n},
(p',k')=(p\sigma,\sigma^{-1}(k)),$$
is a ramified covering of $M$ of degree $n$ with ramification locus $Y$. 

Conversely, let $\pi:R\lra M$ be a ramified covering of degree $n$ with ramification locus $Y$. Let $m$ be a point of $M \setminus Y$. By a {\em labelling} of $R$ at $m$ we mean a bijection between $\pi^{-1}(m)$ and the set $\{1,\ldots,n\}$.
For each $m\in M\setminus Y$, let $\Lab_{m}(R)$ denote the set of labellings of $R$ at $m$. The group
$\Sy_{n}$ acts on the right, transitively and freely on $\Lab_{m}(R)$. Then $\Lab(R)=\cup_{m\in
M}\Lab_{m}(R)$ endowed with the natural topology and projection on $M$, is a principal
$\Sy_{n}$-bundle over $M$ ramified over $Y$. At a point $y\in Y$, the ramification type of $R$ is
a partition of $n$ and the monodromy of $\Lab(R)$ around $y$ is the corresponding conjugacy class
of $\Sy_{n}$.

It is easy to check that, if $P$ is a $\Sy_{n}$-bundle, then $\Lab(P\times_{\Sy_{n}}\{1,\ldots,n\})$ is canonically isomorphic to $P$. 
\end{remark}

Let $\RB(M)$ (resp. $\RB(M,Y)$) denote the set of isomorphism classes of ramified principal
G-bundles over $M$ (resp. with ramification locus $Y$). 
\index{RARM@$\RB(M),\RB(M,Y)$}

\begin{definition} A {\em based ramified $G$-bundle} is a pair $(R,p)$ where $\pi:R\to M$ is a
ramified bundle and $p\in R$ is a point such that $\pi(p)$ does not belong to the ramification
locus of $R$. The pair $(R,p)$ is said to be {\em based at $\pi(p)$}.

Two based ramified $G$-bundles $(R,p)$ and $(R',p')$ are isomorphic if there exists an isomorphism
$f:R\to R'$ of ramified $G$-bundles such that $f(p)=p'$.
\end{definition}

The importance of this notion comes from the face that the automorphism group of a based ramified bundle is trivial.

Let $m$ be a point of $M$. We denote by $\BRB_{m}(M)$ (resp. $\BRB_{m}(M,Y)$) the set of
isomorphism
classes of based ramified $G$-bundles based at $m$ (resp. with ramification locus $Y$). 
\index{RARMM@$\RB_m(M),\RB_m(M,Y)$}



\section{Monodromy of ramified $G$-bundles} 
\index{ramified principal bundle!monodromy}
\label{sec monodromy}

Consider $R\in \RB(M,Y)$. Choose $m\in M\setminus Y$. Choose
$p\in
\pi^{-1}(m)$. For each loop $l\in \Loop_{m}(M)$ which
does not meet $Y$, the lift of $l$ starting at $p$ finishes at $pg$ for a unique $g\in G$, called
the monodromy of $R$ along $l$ with respect to $p$. This monodromy depends only on the homotopy
class of $l$ in $M-Y$. Hence, the choice of $p$ determines a group homomorphism
$\hol_{p}:\pi_{1}(M\setminus Y,m)\lra G$, which characterizes the based ramified $G$-bundle $(R,p)$
up to isomorphism. Another choice of $p$ would lead to another homomorphism, which differs from $\hol_{p}$ by composition by an inner
automorphism of $G$. The class of the homomorphism $\hol_{p}$ modulo inner automorphisms of $G$
characterizes the ramified $G$-bundle $R$ up to isomorphism (see \cite{Steenrod}).

The words {\em monodromy} and {\em holonomy} are synonymous in this work, but we use the first in the context of ramified bundles and the second in the framework of Markovian holonomy fields. 

An automorphism of the ramified $G$-bundle $R$ is completely determined by the point to which it sends $p$. This point is of the form $pg$ for a unique $g$, hence the choice of $p$ allows us also to identify $\Aut(R)$ with a subgroup of $G$ which we denote by $\Aut_{p}(R)$. Let $\Hol_p(R)$ be the image of the homomorphism $\hol_{p}:\pi_{1}(M\setminus Y,m)\to G$. Then $\Aut_p(R)$ is the centralizer of $\Hol_p(R)$. Again, changing $p$ to $ph$ for some $h\in G$
conjugates $\Hol_p(R)$ and $\Aut_{p}(R)$ by $h$.
\index{ramified principal bundle!automorphisms}

In order to study $\RB(M,Y)$, it is convenient to choose a system of generators of $\pi_{1}(M\setminus Y,m)$. To do this, let us first assume that $Y$ is not empty and set $k=\#Y$. Let us choose on $M$ a graph $\G$ such that $m$ is a vertex of $\G$ and each face of $\G$ contains exactly one point of $Y$.  Throughout this chapter, we use the notation $\rg=\rg(M)$, $\p=\p(M)$, and $\f=\f(\G)$ when there is no ambiguity. In the present situation, $\f=k$. By Lemma \ref{lg is free}, the group $\RL_v(\G)$ is naturally isomorphic to $\pi_1(M\setminus Y,m)$. According to Proposition \ref{tame generators}, let us choose a tame system $\Gen=\{a_1,\ldots,a_{\rg},c_1,\ldots,c_{\p},l_1,\ldots,l_k\}$ of generators of $\RL_v(\G)$, associated with a certain word $w$ in the free group of rank $\rg$. 
\index{GAG@$\mathscr G$}
\index{OOOR@$\O(R,y)$}
For all ramification point $y\in Y$, we denote by $\O(R,y)$ the conjugacy class of the monodromy
along the facial lasso whose meander goes around $y$. This is also the conjugacy class of the
monodromy along any small loop which circles once around $y$, positively if $M$ is oriented. In particular, it does not depend on the choice of $\G$. 

Recall that the surface $M$ is endowed with a set $C$ of $G$-constraints along its boundary. Thus,
to each oriented connected component $b$ of $\partial M$, the $G$-constraints $C$ associate a
conjugacy class $C(b)$ of $G$. Let us write $\BS(M)=\{b_1,b_1^{-1},\ldots,b_\p,b_\p^{-1}\}$ and, for
all $i\in\{1,\ldots,\p\}$, $\O_i=C(b_i)$. We define the sets $\RB(M,C)$,
 (resp. $\RB(M,Y,C)$, $\BRB_m(M,C)$, $\BRB_m(M,Y,C)$) as the sets of isomorphism classes of
ramified $G$-bundles (resp. with ramification locus $Y$, based at $m$, based at $m$ with
ramification locus $Y$) such that the monodromy along $b_i$ belongs to $\O_i$ for all $i\in\{1,\ldots, \p\}$.  

Let us use the information gathered so far to build concrete models for the various spaces of
isomorphism classes of ramified $G$-bundles. Let us define
\begin{eqnarray*}
 \H(M,k,C,w)=&&\\
&& \hskip -2cm \left\{(a_{1},\ldots,a_{\rg},c_{1},\ldots,c_{\p},d_{1},\ldots,d_{k})
\in G^{\rg}\times \O_1 \times \ldots \times \O_\p \times (G\setminus \{1\})^{k} :\right.\\
&&\hskip 4cm\left.  \vbox{\vspace{0.4cm}}w(a_{1},\ldots, a_{\rg}) c_{1}\ldots c_{\p} d_{1}\ldots
d_{k}=1\right\}.
\end{eqnarray*}
\index{HAH@$\H(M,k,C,w)$}
We denote by $a_{i},c_{i},d_{i}:\H(M,k,C,w)\to G$ the obvious coordinate mappings. 

If $Y$ is empty, then we choose $\G$ with a single face. Then the appropriate concrete model is the following space:
\begin{eqnarray*}
\H(M,0,C,w)=&&\\
&& \hskip-2cm \left\{(a_{1},\ldots,a_{\rg},c_{1},\ldots,c_{\p})
\in G^{\rg}\times \O_1 \times \ldots \times \O_\p : w(a_{1},\ldots, a_{\rg}) c_{1}\ldots c_{\p} =1\right\}.
\end{eqnarray*}

The group $G$ acts on $\H(M,k,C,w)$ by simultaneous conjugation on each factor. Let us
consider 
the following diagram :
\begin{equation}\label{identif}
\xymatrix{\BRB_m(M,Y,C) \ar[r]^{\sim} \ar[d] & \H(M,k,C,w)\ar[d] \\
\RB(M,Y,C)\ar[r]^{\sim} &\H(M,k,C,w)/G}
\end{equation}
The vertical arrow on the left is the map which forgets the base point. The vertical arrrow on the
right is the quotient map. The top horizontal arrow is given by the monodromy with respect to the 
base point along the elements of $\Gen$. The bottom horizontal arrow is also given by this
monodromy, but since no base point is specified, it is defined up to global conjugation.

This diagram is commutative and, according to the discussion at the beginning of this section, its
horizontal arrows are bijections.

The preimage of an element $R\in \RB(M,Y,C)$ by the vertical arrow consists in $\frac{\# G}{\# \Aut(R)}$
elements. It follows that, for all function $f:\RB(M,Y,C)\lra \CK$, which can alternatively be seen
as an invariant function on $\H(M,k,C)$, we have the counting formula
\begin{equation}\label{counting}
\sum_{R\in\RB(M,Y,C)}\frac{1}{\# \Aut(R)}\; f(R) =\frac{1}{\#G}\sum_{(R,p)\in\RB_{m}(M,Y,C)} f(R)=
\frac{1}{\# G}\sum_{h\in\H(M,k,C,w)} f(h).
\end{equation}
\index{ramified principal bundle!counting}

\section{Measured spaces of ramified $G$-bundles} Let us start by putting a topology on the sets
of ramified $G$-bundles. For each based ramified $G$-bundle $(R,p)$ based at $\pi(p)=m$ with
ramification locus $Y$, and each open subset $U$ of $M\setminus \{m\}$ containing $Y$, we define
$$ {\mathcal V}((R,p),U)=\{(R',p')\in \BRB_m(M,C) : (R,p)_{| M\setminus U} \simeq (R',p')_{| M\setminus U}
\mbox{ as } G\mbox{-bundles}\}.$$
The sets ${\mathcal V}((R,p),U)$ form a basis of a topology on $\BRB_m(M,C)$ and
from now on we consider this space endowed with that topology. Similarly, we endow $\RB(M,C)$
with the topology generated by the sets 
$$ {\mathcal V}(R,U)=\{R'\in \RB(M,C) : R_{| M\setminus U} \simeq R'_{| M\setminus U}\},$$
where $U$ contains the ramification locus of $R$. These topologies make the projection
$\BRB_{m}(M,C)\to \RB(M,C)$ continuous. However, observe that the number of ramification points
is not a continuous function with respect to these topologies, it is only lower semi-continuous. In
fact, these topologies are the roughest which make the monodromy along any loop on $M$ a continuous
functions on its definition set. 
\index{ramified principal bundle!topology}

\index{YYY@$\Y(M)$}
Let ${\Y(M)}$ denote the set of finite subsets of $M$. For each $k\geq 0$, let $\Delta_{k}\subset
M^k$ denote the subset of $M^k$ on which at least two components are equal. We endow $\Y(M)$ with the
topology which makes the bijection $\Y(M)\simeq \bigsqcup_{k\geq 0} (M^k\setminus \Delta_{k})/\Sy_{k}$ a
homeomorphism. Once again, the natural mapping $\Ram: \RB(M,C)\lra \Y(M)$ which
associates to a covering its ramification locus is not continuous. 

It is now time to introduce the L\'evy process. Let $X$ be a continuous-time
Markov chain on $G$, with jump measure invariant by conjugation, and also invariant by inversion
if $M$ is non-orientable. Its L\'evy measure $\Pi$ is a finite invariant measure supported by $G\setminus \{1\}$. 

We denote by $\Pi_1$ the probability measure $\frac{\Pi}{\Pi(G)}$ on $G$. We define now the weight of a ramified $G$-bundle with respect to $\Pi$. Recall that if $R$ is ramified at $y$, then $\O(R,y)$ denotes the conjugacy class of the monodromy of a small circle around $y$, positively oriented if $M$ is oriented.

\begin{definition} Consider $R\in \RB(M)$. Let $Y$ denote the ramification locus of $R$. The {\em
$\Pi$-weight} of $R$ is the non-negative real number $\Pi_1(R)$defined as follows : 
$$\Pi_1 (R)=\prod_{y\in Y} \frac{\Pi_1(\O(R,y))}{\# \O(R,y)}.$$
If $R$ is represented by an element $h$ of $\H(M,k,C,w)$, then $\Pi_1(R)=\prod_{i=1}^{k}\Pi_1(\{d_{i}(h)\})$. 
\end{definition}
\index{PAPi@$\Pi(R)$}
\index{ramified principal bundle!weight}

The notion of weight of a ramified $G$-bundle allows us to define positive measures on the spaces
of bundles. The choice of the normalization will be justified by later results.

\begin{definition}\label{def mes RB}  The Borel measure $\BRBm_{M,m,Y,C}^{X}$ on $\BRB_m(M,Y,C)$ is
defined by 
$$\BRBm_{M,m,Y,C}^{X}=\frac{\#G^{1-\rg}}{\#\O_{1}\ldots \#\O_{\p}}\sum_{(R,p)\in\BRB_m(M,Y,C)}
\Pi_1(R)\; \delta_{(R,p)}.$$
By the left vertical arrow of (\ref{identif}), this measure is projected on the Borel
measure $\RBm_{M,Y,C}^{X}$ on $\RB(M,Y,C)$  defined by
$$\RBm_{M,Y,C}^{X}=\frac{\#G^{2-\rg}}{\#\O_{1}\ldots \#\O_{\p}}\sum_{R\in\RB(M,Y,C)}
\frac{\Pi_1(R)}{\# \Aut(R)}\; \delta_{R}.$$
\end{definition}
\index{RARBmeas@$\BRBm_{M,m,Y,C}^{X},\RBm_{M,Y,C}^{X}$}

Thanks to the counting formula (\ref{counting}), we can roughly bound above the total mass of $\BRBm_{M,m,Y,C}^{X}$ by
\begin{equation}\label{borne masse}
\BRBm_{M,m,Y,C}^{X}(1)\leq \frac{\#G^{1-\rg}}{\#\O_{1}\ldots \O_{p}} \sum_{h\in \H(M,k,C)}
\Pi_1 (h)\leq \#G.
\end{equation}

Our next objective is to put measures on $\RB_m(M,C)$ and $\RB(M,C)$, the sets in which the
ramification locus is not fixed. We have endowed both spaces with topologies. Thus, they carry a
Borel $\sigma$-field. Let ${\mathcal M}_{+}(\RB(M,C))$ and $\M_{+}(\RB_m(M,C))$ denote the spaces of
positive Borel measures on $\RB(M,C)$ and $\RB_{m}(M,C)$ respectively, endowed with the
topology of weak convergence.

\begin{proposition} \label{continu en Y} The mapping from $\Y(M)$ to $\M_{+}(\RB(M,C))$ which
sends $Y$ to $\RBm_{M,Y,C}^{X}$ is continuous. Similarly, the mapping from $\Y(M)$ to
 $\M_{+}(\RB_m(M,C))$ which sends $Y$ to $\BRBm_{M,m,Y,C}^{X}$
is continuous on its definition set.
\end{proposition}

\begin{proof}We prove only the first statement. The second one is very similar.

By definition of the topology on $\Y(M)$, it suffices to prove that the mapping from
$M^k-\Delta_{k}$ to $\M_{+}(\RB(M,C))$ which sends $Y=(y_{1},\ldots,y_{k})$ to
$\RBm_{M,Y,C}^{X}$ is continuous for all $k\geq 0$.
Consider $k\geq 0$, $Y=\{y_{1},\ldots,y_{k}\}$ and a bounded continuous function
$f:\RB(M,C)\lra \RK$. Choose $\epsilon>0$. For simplicity, assume that $M$ is endowed
with a Riemannian metric.

Since $\RB(M,Y,C)$ is a finite set, the continuity of $f$ implies the existence of $r>0$ such
that the balls $B(y_{i},r)$ are contained in $M\setminus \partial M$, pairwise disjoint and such that
the neighbourhood $U=B(y_{1},r)\times  \ldots \times B(y_{k},r)$ of $Y$ in $M^k\setminus \Delta_{k}$
satisfies 
$$\forall R \in \RB(M,Y,C), \forall R' \in {\mathcal V}(R,U) ,\;  |f(R')-f(R)| <
\frac{\epsilon}{\# \RB(M,Y,C)\;  \Pi(G)}.$$
Let $Y'=\{y'_{1},\ldots,y'_{k}\}$ be an element of $U$. Let $\phi$ be a diffeomorphism of $M$ such
that $\phi_{|M\setminus U}=\id_{M\setminus U}$ and $\phi(y_{i})=y'_{i}$ for all $i\in\{1,\ldots,k\}$. For each
bundle $\pi:R\lra M$ belonging to $\RB(M,Y,C)$, the bundle $\phi(R)=(\phi\circ \pi :
R\lra M)$ belongs to $\RB(M,Y',C)$. Replacing $\phi$ by its inverse in the definition of
$\phi:\RB(M,Y,C)\lra \RB(M,Y',C)$ yields the inverse mapping, hence $\phi$ is
a bijection. Moreover, for each $i\in\{1,\ldots,k\}$, $\O(\phi(R),y'_{i})=\O(R,y_{i})$,
so that $\Pi(\phi(R))=\Pi(R)$. Also, the conjugation by $\phi$ determines an isomorphism
between $\Aut(R)$ and $\Aut(\phi(R))$. Finally, $R$ and $\phi(R)$ are isomorphic outside $U$.
Altogether,
$$\left| \RBm_{M,Y',C}^{X}(f)-\RBm_{M,Y,C}^{X}(f)\right| \leq
\sum_{R\in\RB(M,Y,C)} \frac{\Pi(R)}{\#\Aut(R)}\; \left|
f(\phi(R))-f(R)\right|<\epsilon.$$
Since $k$, $Y$, $f$ and $\epsilon$ were arbitrary, the result follows. \end{proof}

We choose for the ramification locus a very simple probability distribution which incorporates the measure $\vol$ on $M$. Let $\PPP$ be the distribution of a Poisson point process of intensity $\Pi(G) \vol$ on $M$. It
is a Borel probability measure on $\Y(M)$. Moreover, for all $m\in M$, $\PPP(\{Y : m\in Y\})=0$.
According to Proposition \ref{continu en Y}, the following definition is legitimate.
\index{XXXi@$\Xi$}

\begin{definition} The Borel measures $\RBm_{M,\vol,C}^{X}$ on $\RB(M,C)$
and $\BRBm_{M,m,\vol,C}^{X}$ on $\BRB_m(M,C)$ are
defined by
$$\RBm_{M,\vol,C}^{X}=\int_{\Y(M)} \RBm_{M,Y,C}^{X}\; \PPP(dY) =
\int_{\Y(M)} \left(\sum_{R\in
\RB(M,Y,C)} \frac{\Pi(R)}{\#\Aut(R)}\; \delta_{R} \right)\; \PPP(dY),$$
\begin{align*}
\BRBm_{M,m,\vol,C}^{X}&=\int_{\Y(M)} \BRBm_{M,m,C}^{X}\;
\PPP(dY) \\
&= \int_{\Y(M)} \left(\frac{1}{\# G}\sum_{(R,p)\in \BRB_m(M,Y,C)}  \Pi(R) \;
\delta_{(R,p)} \right)\; \PPP(dY).
\end{align*}
\end{definition}
\index{ramified principal bundle!random}

Since $\PPP(\{Y:m\in Y\})=0$, the subset of $\RB(M,C)$ which consists in bundles ramified over
$m$ is negligible for the measure $\RBm_{M,\vol,C}^{X}$. Hence, the measure $\BRBm_{M,m,\vol,C}^{X}$
projects on $\RBm_{M,\vol,C}^{X}$ by the left vertical arrow of the diagram (\ref{identif}). In
particular, these measures have the same total mass. Thanks to (\ref{borne masse}), this total mass is finite.
Hence, $\RBm_{M,\vol,C}^{X}$ and $\BRBm_{M,m,\vol,C}^{X}$ are finite measures. We will denote by $\NRBm_{M,\vol,C}^{X}$ and $\NBRBm_{M,m,\vol,C}^{X}$ the
corresponding probability measures.

Although this is not absolutely necessary, let us compute $\BRBm_{M,m,\vol,C}^{X}(1)$. Recall that (\ref{densite char}) gives an expression of the density of the $1$-dimensional marginals of the L\'evy process $X$ with respect to the uniform measure on $G$: setting, for all $\alpha \in \irrep(G)$, $\widehat \Pi(\alpha)=\sum_{x\in G} \overline{\chi_\alpha(x)} \Pi(\{x\})$, we have
\begin{equation}\label{Qt fini}
\forall t>0, \forall x\in G, \; Q_t(x)=e^{-t\Pi(G)} \sum_{\alpha\in \irrep(G)} e^{t\frac{\widehat\Pi(\alpha)}{\chi_\alpha(1)}} \chi_\alpha(1) \chi_\alpha(x).
\end{equation}
In the present context, this equality can be checked by an elementary computation, using the following formula, which we will need again later and which is proved by using the standard properties of characters.

\begin{lemma}\label{magique} For all $k\geq 1$ and all $x\in G$, the following equality holds:
$$\sum_{x_1,\ldots,x_k \in G} \Pi(\{x_1\}) \ldots \Pi(\{x_k\}) {\mathbbm 1}_{x_1\ldots x_k=x} =\frac{1}{\# G} \sum_{\alpha\in \irrep(G)} \left(\frac{\widehat \Pi(\alpha)}{\chi_\alpha(1)}\right)^k \chi_\alpha(1)\chi_\alpha(x).$$
\end{lemma}

We can now compute the mass of $\BRBm_{M,m,\vol,C}^{X}(1)$.

\begin{proposition} The total mass of the measure $ \BRBm_{M,m,\vol,C}^{X}$ is equal to
\begin{equation}\label{masse BB}
\frac{1}{{\# G}^{\rg} \prod_{i=1}^{\p} \# \O_i} \sum_{\substack {a_1,\ldots,a_{\rg} \in G  \\ c_1\in \O_1,\ldots,c_{\p} \in \O_{\p}}} Q_{\vol(M)}(w(a_1\ldots a_{\rg}) c_1 \ldots c_{\p}),
\end{equation}
which, with the notation of Definition \ref{def m}, is none other than
$$\int_G Q_{\vol(M)} \; d\m_{M,\varnothing,C}.$$
\end{proposition}

\begin{proof} Choose $Y\in \Y(M)$. Set $k=\# Y$. Choose a graph $\G$, a vertex $v$ of $\G$ and a tame system of generators of $\RL_v(\G)$ associated with some word $w$, like we did in Section \ref{sec monodromy}. This determines a bijection $\R_m(M,Y,C) \simeq \H(M,k,C,w)$. By the counting formula (\ref{counting}), 
\begin{align*}
\BRBm_{M,m,Y,C}^{X}(1) &=\frac{{\# G}^{1-\rg}}{\prod_{i=1}^{\p} \# \O_i} \sum_{h\in \H(M,k,C,w)} \prod_{i=1}^k \Pi_1(\{d_i(h)\})\\
&\hskip -1cm =\frac{{\# G}^{1-\rg}}{\prod_{i=1}^{\p} \# \O_i} \sum_{\substack {a_1,\ldots,a_{\rg} \in G  \\ c_1\in \O_1,\ldots,c_{\p} \in \O_{\p}}} \sum_{d_1,\ldots,d_k \in G} \prod_{i=1}^k \frac{\Pi(\{d_i(h)\})}{\Pi(G)} {\mathbbm 1}_{d_1\ldots d_k= w(a_1 \ldots a_\rg) c_1 \ldots c_\p}.
\end{align*}
By Lemma \ref{magique}, this is equal to
$$\frac{1}{{\# G}^{\rg}\prod_{i=1}^{\p} \# \O_i} \sum_{\substack {a_1,\ldots,a_{\rg} \in G  \\ c_1\in \O_1,\ldots,c_{\p} \in \O_{\p}\\ \alpha \in \irrep(G)}} \frac{1}{\Pi(G)^k} \left(\frac{\widehat \Pi(\alpha)}{\chi_\alpha(1)}\right)^k \chi_\alpha(1)\chi_\alpha(w(a_1 \ldots a_\rg) c_1 \ldots c_\p).$$
Integrating this expression with respect to $Y$ under the probability measure $\Xi$ amounts to replacing $k$ by a Poisson random variable with parameter $\Pi(G) \vol(M)$ and taking the expectation. Using (\ref{Qt fini}), we find that this expectation is equal to (\ref{masse BB}). \end{proof}

\section{The monodromy process as a Markovian holonomy field} 

Let $m$ be a point of $M$. Let $l\in \Loop_{m}(M)$ be a loop based at $m$. Since $l$ is
rectifiable, its range is negligible for the measure $\vol$. Hence, the ramification locus of a
ramified $G$-bundle based at $m$ distributed according to the probability measure $\NBRBm_{M,m,\vol,C}^{X}$
is almost surely disjoint from the range of $l$. The mapping $P_{l}:\RB_{m}(M,C)\to G$ which sends
a pair $(R,p)$ to the monodromy of $R$ along $l$ with respect to $p$
is defined on the subset where
the ramification locus is disjoint from $l$ and thus is a well-defined random variable under the
probability measure $\NBRBm_{M,m,\vol,C}^{X}$.

Let $l_{1},l_{2}\in \Loop_{m}(M)$ be two loops. Let $(R,p)$ be an element of $\R_{m}(M,C)$. Let
$g_{1}$ and $g_{2}$ be the monodromies of $l_{1}$ and $l_{2}$ respectively. Let us compute the
monodromy of $l_{1}l_{2}$. The point $p$ is sent to $pg_{1}$ by the parallel transport along
$l_{1}$. Then, on one hand the parallel transport along $l_{2}$ sends $p$ to $pg_{2}$ and on the
other hand the the parallel transport commutes to the action of $G$ on the right on $\pi^{-1}(m)$.
Thus, the parallel transport along $l_{2}$ sends $pg_{1}$ to $pg_{2}g_{1}$. It appears that
monodromies are multiplied in the reversed order of concatenation. Coming back to the
probabilistic setting, this implies that
$$\forall l_{1},l_{2}\in \Loop_{m}(M)\; , \; \; P_{l_{1}l_{2}}=P_{l_{2}}P_{l_{1}} \mbox{ almost
surely}.$$
It is even easier to check that for all $l\in \Loop_{m}(M)$, $P_{l^{-1}}=P_{l}^{-1}$ almost surely.

Thanks to Proposition \ref{take proj lim}, these two relations ensure that the collection of
random variables $(P_{l})_{l\in\Loop_{m}(M)}$ defined on the probability space
$(\BRB_{m}(M,C),\NBRBm_{M,m,\vol,C}^{X})$ determines a probability measure on the space
$(\M(\Loop_{m}(M),G),\C)$. We denote this probability measure by $\NMF_{M,m,\vol,C}^{X}$.
By restriction, this probability measure is also defined on the invariant $\sigma$-field and, by
Lemma \ref{chemins lacets}, determines a probability measure on the measurable space
$(\M(\Path(M),G),\I)$, which we denote by $\NMF_{M,(m),\vol,C}^{X}$. Finally, we define a finite
measure on $(\M(\Path(M),G),\I)$ by
$$\MF_{M,(m),\vol,C}^{X}=\BRBm_{M,\vol,C}^{X}(1) \NMF_{M,(m),\vol,C}^{X}.$$

\begin{lemma} The measure $\MF_{M,(m),\vol,C}^{X}$ on $(\M(\Path(M),G),\I)$ does not
depend on the point $m$. We denote it by $\MF_{M,\vol,C}^{X}$.
\end{lemma}
\index{MAMF@$\MF_{M,\vol,C}^{X}$}

\begin{proof} By Definition \ref{def inv sigma field}, it suffices to show that, if $l_{1},\ldots,l_{n}$ are loops on $M$ based at the same point $m_{0}$, and $f:G^n\to  \CK$ is a function invariant under the action of $G$ by diagonal
conjugation, then the distribution of $f(h(l_{1}),\ldots,h(l_{n}))$ under
$\NMF_{M,(m),\vol,C}^{X}$ does not depend on $m$. By definition, this distribution is that of
$f(P_{cl_{1}c^{-1}},\ldots,P_{cl_{n}c^{-1}})$ under $\BRBm_{M,m,\vol,C}^{X}$, where $c$ is an
arbitrary path from $m$ to $m_{0}$. Let $m$ and $m'$ be two points. Let us chose a path $c$ from
$m$ to $m_{0}$ and a path $c'$ from $m'$ to $m$. It suffices to prove that the distributions of 
$f(P_{cl_{1}c^{-1}},\ldots,P_{cl_{n}c^{-1}})$ under $\BRBm_{M,m,\vol,C}^{X}$ and 
$f(P_{c'cl_{1}c^{-1}{c'}^{-1}},\ldots,P_{c'cl_{n}c^{-1}{c'}^{-1}})$ under $\BRBm_{M,m',\vol,C}^{X}$
coincide. 

Let $Y$ be a finite subset of $M$ which does not meet $c'$. Let $(R,p')$ be an element of
$\BRB_{m'}(M,Y,C)$. Then, for each $i\in\{1,\ldots,n\}$, the monodromy of $R$ along
$c'cl_{i}c^{-1}{c'}^{-1}$ relatively to $p'$ is equal to the monodromy of $R$ along
$cl_{i}c^{-1}$ relatively to the image of $p'$ by parallel transport along $c'$, which we denote
by $p$. Thus, it suffices to prove that the mapping $\RB_{m'}(M,Y,C)\to \RB_{m}(M,Y,C)$ which
sends $(R,p')$ to $(R,p)$, where $p$ is the image of $p'$ by parallel transport along $c'$, sends
the measure $\BRBm_{M,m',Y,C}$ to the measure $\BRBm_{M,m,Y,C}$. This follows from the definition 
of these measures and the fact that the mapping which we consider is a bijection which preserves the
$\Pi$-weight. \end{proof}

The main result of this section is the following.

\begin{theorem} \label{holo mono} The finite measures $\HF^{X}_{M,\vol,\varnothing,C}$ and $\MF^{X}_{M,\vol,C}$
on the measurable space $(\M(\Path(M),G),\I)$ are equal. 
\end{theorem}

This theorem expresses, at least when the surfaces carry only $G$-constraints along their boundary, the fact that the holonomy process associated with the Markovian holonomy field $\HF^{X}$ is the monodromy process associated to a random ramified $G$-bundle taken under the appropriate distribution.

The proof of this theorem consists in two main steps. In the first step, we prove that the
monodromy process is stochastically continuous. Then, we prove that the holonomy process and the
monodromy process coincide in distribution on the set of piecewise geodesic loops for some
Riemannian metric on $M$.

\begin{proposition}\label{MF stoch cont} The measure $\MF^{X}_{M,\vol,C}$ is stochastically
continuous, in the sense that is satisfies the first property of Definition \ref{def SCFMHF}.
\end{proposition}

\begin{proof}It suffices to prove that for all $m\in M$, all $l\in \Loop_{m}(M)$ and all sequence
$(l_{n})_{n\geq 0}$ of loops based at $m$ converging to $l$, the sequence $(P_{l_{n}})_{n\geq
0}$ converges in measure to $P_{l}$. Let us
endow $M$ with a Riemannian metric and choose $m$, $l$ and $(l_{n})_{n\geq 0}$ as above. We assume
that all loops are parametrized at constant speed, so that the sequence of parametrized paths
$(l_{n})_{n\geq 0}$ converges uniformly to $l$.

Choose $\epsilon>0$. For each $r>0$, let ${\mathcal N}_{r}(l)$ denote the $r$-neighbourhood
of the image of $l$. Since the distribution of the ramification locus $\Ram(R)$ of $R$ under
the finite measure $\BRBm^{X}_{M,m,\vol,C}$ is absolutely continuous with respect to $\Xi$,
$$\lim_{r\to 0} \BRBm^{X}_{M,m,\vol,C}\left(\{(R,p) : \Ram(R)\cap {\mathcal N}_{r}(l) \neq
\varnothing \}\right) = 0.$$
Choose $r>0$ such that this probability is smaller than $\epsilon$. Assume also that $r$ is
smaller than the convexity radius of our Riemannian metric on $M$. Finally, let $n_{0}$ be such
that $n\geq n_{0}$ implies $d_{\infty}(l_{n},l)<r$. Then, if $n\geq n_{0}$, $l_{n}$ and $l$
are homotopic inside ${\mathcal N}_{r}(l)$ which, with probability greater than $1-\epsilon$,
does not contain any ramification point. Hence,
$$\forall n\geq n_{0}, \; \BRBm^{X}_{M,m,\vol,C}\left(\{(R,p): P_{l_{n}}(R)\neq P_{l}(R)\}\right) <\epsilon.$$
Since $\epsilon$ is arbitrary, this proves that $P_{l_{n}}$ converges to $P_{l}$ in
measure. \end{proof}

Theorem \ref{holo mono} asserts the equality of two finite measures. We
consider the two stochastic processes $(H_{l})_{l\in \Loop_{m}(\G)}$ and $(P_{l})_{l\in\Loop_{m}(\G)}$ which
are both the canonical process on $\M(\Loop_{m}(M),G)$, the first considered under the measure
$\HF^{X}_{M,\vol,\varnothing,C}$ and the second under the measure
$\MF^{X}_{M,(m),\vol,C}$. Although these measures are not in general probability measures, we use
the language of stochastic processes for $H$ and $P$.

By Proposition \ref{MF stoch cont} and Theorem \ref{main extension}, it
suffices to endow $M$ with a Riemannian metric and to show that the restrictions of $P$ and $H$ to
piecewise geodesic loops agree in distribution. For this, as we have already observed several
times, it suffices to show that they agree in distribution when restricted to the set of loops in a
graph with piecewise geodesic edges, or in fact any graph.

\begin{proposition} Let $\G=(\V,\E,\F)$ be a graph on $M$ such that $m\in \V$. The families of
random variables $(P_{l})_{l\in\Loop_{m}(\G)}$ and $(H_{l})_{l\in \Loop_{m}(\G)}$
have the same distribution.
\end{proposition}

\begin{proof}It suffices to prove that the equality holds when the processes are restricted to a family of
loops which generate the group $\RL_{m}(M)$. Consider a tame family of generators $\Gen=\{a_1,\ldots,a_{\rg},c_1,\ldots,c_{\p},l_1,\ldots,l_{\f}\}$ of $\RL_m(\G)$ associated with a word $w$. The loop $l_{\f}$ is a function of all other loops, so that it suffices to compute the distribution of
$H=(H_{a_{1}},\ldots,H_{a_{\rg}},H_{c_{1}},\ldots,H_{c_{\p}}, H_{l_{1}},\ldots,H_{l_{{\f}-1}})$. 
Let us choose
$h=(g_{a_{1}},\ldots,g_{a_{\rg}},g_{c_{1}},\ldots,g_{c_{\p}},g_{l_{1}},\ldots,g_{l_{{\f}-1}})$
in $G^{\rg}\times \O_{1}\times \ldots \times \O_{\p}\times G^{{\f}-1}$. By
Proposition \ref{tame generators}, 
\begin{align*}
\HF^X_{M,\vol,\varnothing,C}(H=h)&=  \frac{\#G^{1-\rg-{\f}}}{\prod_{i=1}^{\p}\#\O_{i}} \prod_{i=1}^\f Q_{\vol(F_{i})}(g_{l_{i}}),
\end{align*}
where we have set  $g_{l_{\f}}=w(g_{a_1},\ldots,g_{a_{\rg}}) g_{c_1}\ldots
g_{c_{\p}} (g_{l_1}\ldots g_{l_{{\f}-1}})^{-1}$.

Now let us compute the corresponding quantity for the monodromy field. Let $Y$ be a finite subset
of $M$ which does not meet $\G$. Let refine $\G$ inside each face which meets $Y$ in order to get a new graph $\G'$, finer than $\G$, such that each face of $\G'$ either does not meet $Y$ and is equal to a face of $\G$, or contains exactly one point of $Y$. 

By applying Proposition \ref{tame generators} in each face of a split pattern of $\G$, we can construct a tame family of generators $\Gen'$ of the group of reduced loops of $\G'$ which is finer than $\Gen$ in the sense that for each face $F$, the facial lasso of $\Gen$ corresponding to $F$ is the product in a certain order of the facial lassos of $\Gen'$ corresponding to the faces of $\G'$ contained in $F$. 

For each $i\in \{1,\ldots,{\f}\}$, let us write $Y_{i}=Y\cap F_{i}=\{y_{i,1},\ldots,y_{i,k_{i}}\}$.  Let $\{l_{i,j} : i\in\{1,\ldots,{\f}\}, j\in\{1,\ldots,k_i\}\}$ be the set of facial lassos of $\Gen'$ indexed accordingly. We may assume that for all $i\in\{1,\ldots,{\f}\}$, $l_i=l_{i,1}\ldots l_{i,k_i}$· 

The family $\Gen'$ determines a bijection between $\RB_m(M,Y,C)$ and $\H(M,k,C,w)$. This allows us to compute
the distribution of the random variable
$$M=(M_{a_{1}},\ldots,M_{a_{\rg}},M_{c_{1}}, 
\ldots,M_{c_{\p}},M_{l_{1}},\ldots,M_{l_{{\f}-1}})$$
 under $\BRBm^X_{M,m,Y,C}$. We find
\begin{align}
\BRBm^X_{M,m,Y,C} (M=h) &=\nonumber \\
&\hskip -1.5cm  \frac{\#G^{1-\rg}}{\prod_{i=1}^{p} \#\O_{i}} \prod_{i=1}^{{\f}} \sum_{g_{l_{i,1}},\ldots,g_{l_{i,k_i}} \in G}
 \Pi_1(\{g_{l_{i,1}}\})\ldots \Pi_1(\{g_{l_{i,k_{i}}}\}) {\mathbbm 1}_{g_{l_{i,k_i}}\ldots g_{l_{i,1}}=g_{l_i}}.
\label{Mgamma}
\end{align}
Using Lemma \ref{magique}, we find that the quantity (\ref{Mgamma}) is equal to
$$\frac{\#G^{1-\rg-{\f}}}{\prod_{i=1}^{\p}\#\O_{i}} 
\prod_{i=1}^{{\f}} \sum_{\alpha_{i} \in \irrep(G)} \left(  \frac{\widehat\Pi(\alpha_{i})}{\Pi(G)\chi_{\alpha_{i}}(1)}\right)^{k_{i}}
\chi_{\alpha_{i}}(1) \chi_{\alpha_{i}}(g_{l_i}).$$
By integrating this expression with respect to $Y$ under the measure $\Xi$, we find
\begin{align*}
\BRBm_{M,m,\vol,C}^{X}(M=h) &= \frac{\#G^{1-\rg-\f}}{\prod_{i=1}^{\p} \#\O_{i}}
\prod_{i=1}^{{\f}} e^{-\vol(F_i) \Pi(G)} \sum_{\alpha \in \irrep(G)} e^{\vol(F_i) \frac{\widehat \Pi(\alpha)}{\chi_\alpha (1)}} \chi_\alpha(1) \chi_\alpha(g_{l_i}) \\
&= \frac{\#G^{1-\rg-\f}}{\prod_{i=1}^{\p} \#\O_{i}}
\prod_{i=1}^{{\f}} Q_{\vol(F_{i})}(g_{l_i}).
\end{align*}
This proves that  $\BRBm^X_{M,m,\vol,C}(M=h)=\HF_{M,\vol,\varnothing,C}^{X}(H=h)$. \end{proof}

\printindex

\bibliographystyle{amsplain}
\bibliography{BiblioCHM}

\end{document}